\documentclass{amsart}
\usepackage{amssymb,amsmath,amsthm,amscd}
\usepackage{leftidx,caption}
\usepackage{graphicx,mathrsfs}
\usepackage{color}
\usepackage{microtype}
\usepackage{hyperref}
\usepackage[usenames,dvipsnames,svgnames,table]{xcolor}
\usepackage{tikz}
\usepackage{comment}
\usepackage{epstopdf} 
\usepackage{float}
\usepackage{tikz-cd}
\usepackage[all]{xy}
\newtheorem{lemma}{Lemma}[section]
\newtheorem{theorem}[lemma]{Theorem}
\newtheorem{corollary}[lemma]{Corollary}
\newtheorem{proposition}[lemma]{Proposition}

\theoremstyle{definition}
\newtheorem{definition}[lemma]{Definition}
\theoremstyle{remark}
\newtheorem{remark}[lemma]{Remark}
\newcommand{\C}{\mathbb{C}}
\newcommand{\D}{\mathbb{D}}
\newcommand{\N}{\mathbb{N}}
\newcommand{\Q}{\mathbb{Q}}
\newcommand{\R}{\mathbb{R}}
\newcommand{\Z}{\mathbb{Z}}

\newcommand{\cC}{\mathcal{C}}

\DeclareMathOperator{\re}{Re}
\DeclareMathOperator{\im}{Im}
\DeclareMathOperator{\Int}{int}

\renewcommand{\epsilon}{\varepsilon}
\renewcommand{\phi}{\varphi}

\DeclareMathOperator{\ind}{ind}

\numberwithin{equation}{section}

\date{\today}

\begin{document}

\title[Schwarz reflections and the Tricorn]{Schwarz reflections and the Tricorn}

\author[S.-Y. Lee]{Seung-Yeop Lee}
\address{Department of Mathematics and Statistics, University of South Florida, Tampa, FL 33620, USA}
\email{lees3@usf.edu}

\author[M. Lyubich]{Mikhail Lyubich}
\address{Institute for Mathematical Sciences, Stony Brook University, NY, 11794, USA}
\email{mlyubich@math.stonybrook.edu}
\thanks{M. Lyubich was partially supported by NSF grants DMS-1600519 and 1901357.}

\author[N. G. Makarov]{Nikolai G. Makarov}
\address{Department of Mathematics, California Institute of Technology, Pasadena, California 91125, USA}
\email{makarov@caltech.edu}

\author[S. Mukherjee]{Sabyasachi Mukherjee}
\address{School of Mathematics, Tata Institute of Fundamental Research, 1 Homi Bhabha Road, Mumbai 400005, India}
\email{sabya@math.tifr.res.in}
\thanks{S. Mukherjee was supported by the Institute for Mathematical Sciences at Stony Brook University, Department of Atomic Energy, Government of India, under project no.12-R\&D-TFR-5.01-0500, an endowment from Infosys Foundation, and SERB research grant SRG/2020/000018 during parts of the work on this project.}

\maketitle

\begin{abstract}
We continue our exploration of the family $\mathcal{S}$ of Schwarz reflection maps with respect to the cardioid and a circle which was initiated in our earlier work. We prove that there is a natural combinatorial bijection between the geometrically finite maps of this family and those of the basilica limb of the Tricorn, which is the connectedness locus of quadratic anti-holomorphic polynomials. We also show that every geometrically finite map in $\mathcal{S}$ arises as a conformal mating of a unique geometrically finite quadratic anti-holomorphic polynomial and a reflection map arising from the ideal triangle group. We then follow up with a combinatorial mating description for the periodically repelling maps in $\mathcal{S}$. Finally, we show that the locally connected topological model of the connectedness locus of $\mathcal{S}$ is naturally homeomorphic to such a model of the basilica limb of the Tricorn.
\\

{\sc Titre.} Sur les r{\'e}flections de Schwarz et la Tricorn \\

{\sc R\'esum\'e.} Nous poursuivons l'exploration d'une famille $\mathcal{S}$ d'applications de r{\'e}flections de Schwarz par rapport la cardio{\"i}de et par rapport au cercle, qui a {\'e}t{\'e} initi{\'e}e dans des travaux ant{\'e}rieurs.
Nous prouvons qu'il y a une bijection naturelle de nature combinatoire entre les applications g{\'e}om{\'e}triquement finies de cette famille et celles du 
membre associ{\'e} {\`a} la basilique de la Tricorn, qui est le lieu de connexit{\'e} des applications polynomiales anti-holomorphes de degr{\'e} deux.
Nous prouvons aussi que toute application g{\'e}om{\'e}triquement finie dans $\mathcal{S}$ provient d'un accouplement conforme entre un polyn{\^o}me quadratique anti-holomorphe g{\'e}om{\'e}triquement fini uniquement d{\'e}termin{\'e} avec une application associ{\'e}e {\`a} un groupe engendr{\'e} par les r{\'e}flections par rapport aux c{\^o}t{\'e}s d'un triangle id{\'e}al. 
Nous continuons avec une description d'un accouplement combinatoire pour les applications p{\'e}riodiquement r{\'e}pulsives de $\mathcal{S}$. Enfin, nous montrons que le mod{\`e}le topologique locallement connexe du lieu de connexit{\'e} de $\mathcal{S}$ est naturellement hom{\'e}omorphe {\`a} un tel mod{\`e}le du membre associ{\'e} {\`a} la basilique de la Tricorn.
\end{abstract}

\setcounter{tocdepth}{1}
\tableofcontents

\section{Introduction}\label{intro}

\subsection{Quadrature domains, and Schwarz reflections}\label{intro_1_subsec}
A domain in the complex plane is called a quadrature domain if the Schwarz reflection map with respect to its boundary extends meromorphically to its interior. They first appeared in the work of Davis \cite{Dav74}, and independently in the work of Aharonov and Shapiro \cite{AS1,AS2,AS}. Since then, quadrature domains have played an important role in diverse areas of complex analysis such as quadrature identities \cite{Dav74, AS, Sak82, Gus}, extremal problems for conformal mapping \cite{Dur83, ASS99, She00}, Hele-Shaw flows \cite{Ric72, EV92, GV06}, Richardson's moment problem \cite{Sak78, EV92, GHMP00}, free boundary problems \cite{Sha92, Sak91, CKS00}, subnormal and hyponormal operators \cite{GP17}.
Moreover, topology of quadrature domains has important applications in physics, and leads to interesting classes of dynamical systems generated by Schwarz reflection maps.
 
It is known that except for a finite number of {\it singular} points (cusps and double points), the boundary of a quadrature domain consists of finitely many disjoint real analytic curves. Every non-singular boundary point has a neighborhood where the local reflection in $\partial\Omega$ is well-defined. The (global) Schwarz reflection $\sigma$ is an anti-holomorphic continuation of all such local reflections.

Round discs on the Riemann sphere are the simplest examples of quadrature domains. Their Schwarz reflections are just the usual circle reflections. Further examples can be constructed using univalent polynomials or rational functions. Namely, if $\Omega$ is a {\it simply connected} domain and $\phi : \D \to\Omega$ is a univalent map from the unit disc onto $\Omega$, then $\Omega$ is a quadrature domain if and only if $\phi$ is a rational function. In this case, the Schwarz reflection $\sigma$ associated with $\Omega$ is semi-conjugate by $\phi$ to reflection in the unit circle.

\begin{figure}[ht!]
\captionsetup{width=0.96\linewidth}
\centering
\includegraphics[scale=0.2]{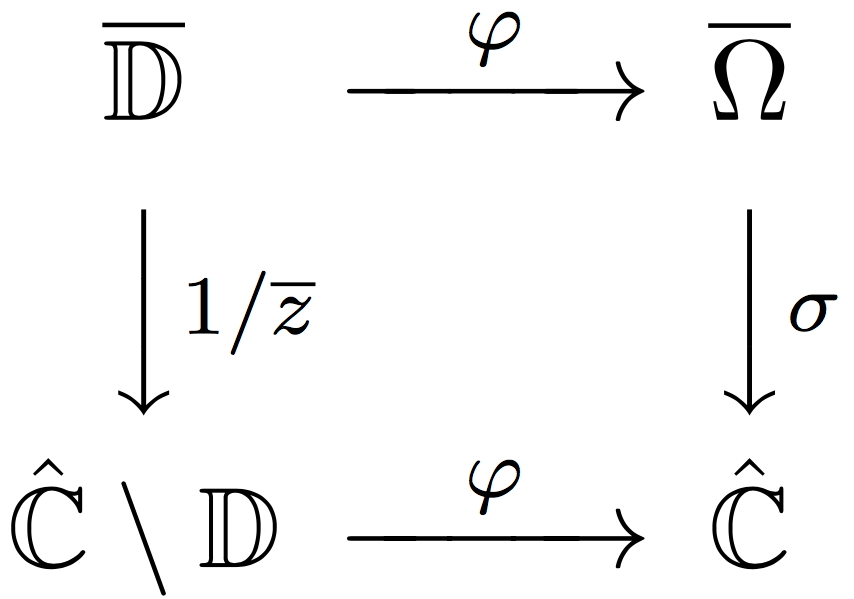}
\caption{The rational map $\phi$ semi-conjugates the reflection map $1/\overline{z}$ of $\D$ to the Schwarz reflection map 
$\sigma$ of $\Omega$.}
\label{comm_diag_schwarz}
\end{figure}

To a disjoint union of quadrature domains, one can associate a dynamical system generated by the corresponding Schwarz reflections. In \cite{LM}, dynamical ideas were applied to the theory of quadrature domains with some physical implications. A systematic exploration of such dynamical systems was then launched in \cite{LLMM1}. Two specific examples of Schwarz reflection maps (associated with simply connected quadrature domains) that appeared in \cite{LLMM1} are Schwarz reflections of the interior of the {\it cardioid} curve and the exterior of the {\it deltoid} curve,
$$\left\{\frac z2-\frac{z^2}4:~|z|<1\right\}\quad{\rm and}\quad
\left\{\frac{1}{z}+\frac{z^2}{2}:~|z|<1\right\}.$$
One of the principal goals of that paper was to develop a general method of producing conformal matings between groups and anti-polynomials using Schwarz reflection maps. In particular, it was proved in \cite[\S 4]{LLMM1} that the Schwarz reflection map of the deltoid is a mating of the \emph{ideal triangle group} and the anti-polynomial $\overline{z}^2$.

\subsection{Coulomb gas ensembles and algebraic droplets}

Quadrature domains naturally arise in the study of $2$-dimensional Coulomb gas models. Consider $N$ electrons located at points $\lbrace z_j\rbrace_{j=1}^N$ in the complex plane, influenced by a strong ($2$-dimensional) external electrostatic field arising from a uniform non-zero charge density. Let the scalar potential of the external electrostatic field be $N\cdot Q:\C\to\R\cup\{+\infty\}$ (note that the scalar potential is rescaled so that it is proportional to the number of electrons). We assume that $Q=+\infty$ outside some compact set $L$, and finite (but not identically zero) on $L$ (see Figure~\ref{electron}). Since the charge density of the potential $Q$ is assumed be uniform (and non-zero), it follows that $Q$ has constant Laplacian on $L$. It follows that we can write $Q(z)= \vert z\vert^2-H(z)$ (on $L$), where $H$ is harmonic. In various physically interesting cases, the scalar potential is assumed to be \emph{algebraic}; i.e., $\frac{\partial H}{\partial z}$ is a rational function.  

\begin{figure}[ht!]
\captionsetup{width=0.96\linewidth}
\centering
\includegraphics[scale=0.2]{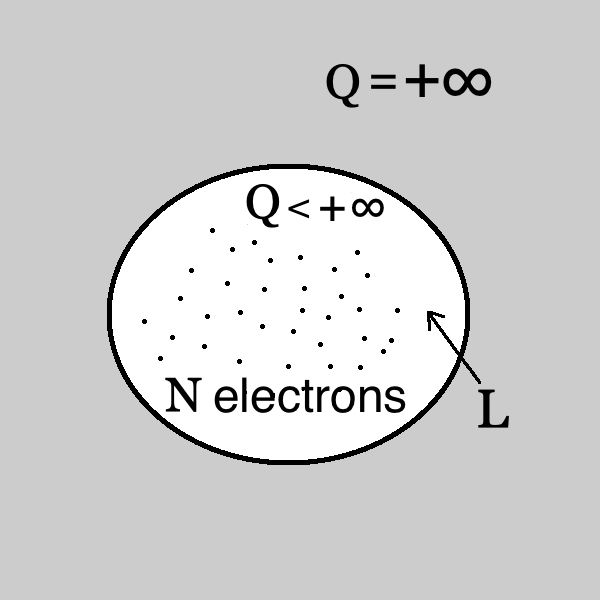}
\caption{The external potential $Q$ is infinite outside a compact set $L$. In the limit, the electrons condensate on a compact subset $T$ (of $L$) which is called a droplet.}
\label{electron}
\end{figure}

The combined energy of the system resulting from particle interaction and external potential is: $$\mathfrak{E}_Q(z_1,\cdots,z_N)=\displaystyle\sum_{i\neq j} \ln\vert z_i-z_j\vert^{-1}+N\sum_{j=1}^N Q(z_j).$$ In the equilibrium, the states of this system are distributed according to the Gibbs measure with density $$\frac{\exp(-\mathfrak{E}_Q(z_1,\cdots,z_N))}{Z_N},$$ where $Z_N$ is a normalization constant known as the ``partition function".

An important topic in statistical physics is to understand the limiting behavior of the ``electron cloud" as the number of electrons $N$ grows to infinity. Under some mild regularity conditions on $Q$, in the limit the electrons condensate on a compact subset $T$ of $L$, and they are distributed according to the normalized area measure of $T$ \cite{EF, HM}. Thus, the probability measure governing the distribution of the limiting electron cloud is completely determined by the shape of the ``droplet" $T$.

If $Q$ is algebraic, the complementary components of the droplet $T$ are quadrature domains \cite{LM}. For example, the deltoid (the compact set bounded by the deltoid curve) is a droplet in the physically interesting case of the (localized) ``cubic" external potential, see \cite{Wie02}. 

The $2D$ Coulomb gas model described above is intimately related to logarithmic potential theory with an algebraic external field \cite{ST97} and the corresponding random normal matrix models, where the same probability measure describes the distribution of eigenvalues \cite{TBAZW, ABWZ}.

We will call any compact set $T\subset\C$ such that $T^c:=\widehat{\C}\setminus T$ is a disjoint union of quadrature domains an {\it algebraic droplet}, and we define the corresponding Schwarz reflection as a map $\sigma:~\overline{T^c} \to \widehat\C$. Iteration of Schwarz reflections was used in \cite{LM} to answer some questions concerning topology and singular points of quadrature domains and algebraic droplets, as well as the Hele-Shaw flow\footnote{It is a flow on the space of planar compact sets such that the normal velocity of the boundary of a compact set is proportional to its harmonic measure. This is an example of the so-called DLA or diffusion-limited aggregation process.} of such domains. At the same time, these dynamical systems seem to be quite interesting in their own right, and in this paper we will take a closer look at them.

\subsection{Anti-holomorphic dynamics}
In this paper, we will be mainly concerned with dynamics of Schwarz reflection maps (which are anti-holomorphic) with a single critical point. It is not surprising that their dynamics is closely related to the dynamics of quadratic anti-holomorphic polynomials (anti-polynomials for short). 

The dynamics of quadratic anti-polynomials and their connectedness locus, the Tricorn, was first studied in \cite{CHRS} (note that they called it the Mandelbar set). Their numerical experiments showed structural differences between the Mandelbrot set and the Tricorn. However, it was Milnor who first observed the importance of the Tricorn; he found little Tricorn-like sets as prototypical objects in the parameter space of real cubic polynomials \cite{M3}, and in the real slices of rational maps with two critical points \cite{M4}. Nakane followed up by proving that the Tricorn is connected \cite{Na1}, in analogy to Douady and Hubbard's classical proof of connectedness of the Mandelbrot set. Later, Nakane and Schleicher \cite{NS} studied the structure and dynamical uniformization of hyperbolic components of the Tricorn. It transpired from their study that while the even period hyperbolic components of the Tricorn are similar to those in the Mandelbrot set (the connectedness locus of quadratic holomorphic polynomials), the odd period ones have a completely different structure. Then, Hubbard and Schleicher \cite{HS} proved that the Tricorn is not pathwise connected, confirming a conjecture by Milnor.
Techniques of anti-holomorphic dynamics were used in \cite{KS, BeEr, BMMS} to answer certain questions with physics motivation.

Recently, the topological and combinatorial differences between the Mandelbrot set and the Tricorn have been systematically studied by several people. The combinatorics of external dynamical rays of quadratic anti-polynomials was studied in \cite{Sa} in terms of orbit portraits, and this was used in \cite{MNS} where the bifurcation phenomena, boundaries of odd period hyperbolic components, and the combinatorics of parameter rays were described. It was proved in \cite{IM1} that many rational parameter rays of the Tricorn non-trivially accumulate on persistently parabolic regions, and Misiurewicz parameters are not dense on the boundary of the Tricorn. These results are in stark contrast with the corresponding features of the Mandelbrot set. In this vein, Gauthier and Vigny \cite{GV} constructed a bifurcation measure for the Tricorn, and proved an equidistribution result for Misiurewicz parameter with respect to the bifurcation measure. As another fundamental difference between the Mandelbrot set and the Tricorn, it was proved in \cite{IM2} that the Tricorn contains infinitely many baby Tricorn-like sets, but these are not dynamically homeomorphic to the Tricorn. Tricorn-like sets also showed up in the recent work of Buff, Bonifant, and Milnor on the parameter space of a family of rational maps with real symmetry \cite{BBM1} (also see \cite{CFG}).

\subsection{Ideal triangle group and the reflection map $\rho$} 

The \emph{ideal triangle group} $\mathcal{G}$ is generated by the reflections in the sides of a hyperbolic triangle $\Pi$ (in the open unit disk $\D$) with zero angles. Denoting the (anti-M{\"o}bius) reflection maps in the three sides of $\Pi$ by $\rho_1$, $\rho_2$, and $\rho_3$, we have $$\mathcal{G}=\langle\rho_1, \rho_2, \rho_3: \rho_1^2=\rho_2^2=\rho_3^2=\mathrm{id}\rangle<\mathrm{Aut}(\D).$$ $\Pi$ is a fundamental domain of the group. The tessellation of $\D$ by images of the fundamental domain under the group elements is shown in Figure~\ref{tessellation_pic}.

In order to model the dynamics of Schwarz reflection maps, we define a map $$\rho:\D\setminus\Int{\Pi}\to\D$$ by setting it equal to $\rho_k$ in the connected component of $\D\setminus\Int{\Pi}$ containing $\rho_k(\Pi)$ (for $k=1,2,3$). The map $\rho$ extends to an orientation-reversing double covering of $\mathbb{T}=\partial\D$ admitting a Markov partition $\mathbb{T}=[1,e^{2\pi i/3}]\cup[e^{2\pi i/3},e^{4\pi i/3}]\cup[e^{4\pi i/3},1]$ with transition matrix $$M:=\begin{bmatrix} 0 & 1 & 1\\ 1 & 0 & 1\\ 1 & 1 & 0\end{bmatrix}.$$

The anti-doubling map $$m_{-2}:\R/\Z\to\R/\Z,\ \theta\mapsto-2\theta$$ (which models the dynamics of quadratic anti-polynomials on their Julia sets) admits the same Markov partition as above with the same transition matrix. This allows one to construct a circle homeomorphism $\mathcal{E}:\mathbb{T}\to\mathbb{T}$ that conjugates the reflection map $\rho$ to the anti-doubling map $m_{-2}$. The conjugacy $\mathcal{E}$, which is a version of the Minkowski's question mark function, serves as a connecting link between the dynamics of Schwarz reflections and that of quadratic anti-polynomials (see the article by Shaun Bullett in \cite[\S 7.8]{BrFa} for a detailed exposition of the Minkowski's question mark function, and \cite[\S 4.4.2]{LLMM1} for an explicit relation between Minkowski's question mark function and $\mathcal{E}$). The conjugacy $\mathcal{E}$ plays a crucial role in the paper (see Section~\ref{ideal_triangle} for details).

\subsection{Dynamical decomposition: tiling and non-escaping sets} 

Let us now describe the basic dynamical objects associated with iteration of Schwarz reflection maps. Given an algebraic droplet $T$ and the corresponding reflection $\sigma: \overline{T^c}\to \widehat\C$, we partition $\widehat\C$ into two invariant sets. The first one is an open set called the {\it tiling set}. It is the set of all points that eventually escape to $T$ (where $\sigma$ is no longer defined). Alternatively, the tiling set is the union of all ``tiles", the {\it fundamental tile} $T$ (with singular points removed) and the components of all its pre-images under the iterations of $\sigma$. The second invariant set is the {\it non-escaping} set, the complement of the tiling set; it is analogous to the filled in Julia set in polynomial dynamics. The dynamics of $\sigma$ on the non-escaping set is much like the dynamics of an anti-holomorphic rational-like map (which may be ``pinched"). On the other hand, the dynamics on the tiling set exhibits features of reflection groups; this is particularly evident if the tiling set is {\it unramified} (i.e., if it does not contain any critical point of $\sigma$).

This is precisely the case if $T$ is the deltoid. Figure~\ref{deltoid_reflection_pic} shows the tiling and the non-escaping sets as well as their common boundary, which is simultaneously analogous to the Julia set of an anti-polynomial and to the limit set of a group. In fact, it was shown in \cite[\S 4]{LLMM1} that the Schwarz reflection $\sigma$ of the deltoid is the \emph{unique conformal mating} of the anti-polynomial $z\mapsto \overline{z}^2$ and the reflection map $\rho$ in the following sense: the conformal dynamical systems
$$\rho: \overline{\mathbb D}\setminus\Int{\Pi}\to  \overline{\mathbb D}$$
and
$$f_0: \widehat\C\setminus\mathbb D\to \widehat\C\setminus\mathbb D,\ z\mapsto\overline{z}^2$$ can be glued together by the circle homeomorphism $\mathcal{E}$ (which conjugates $\rho$ to $f_0$ on $\mathbb{T}$) to yield a (partially defined) topological map $\eta$ on a topological $2$-sphere. There exists a unique conformal structure on this $2$-sphere which makes $\eta$ an anti-holomorphic map conformally conjugate to $\sigma$.

\begin{figure}[ht!]
\captionsetup{width=0.96\linewidth}
\begin{center}
\includegraphics[scale=0.4]{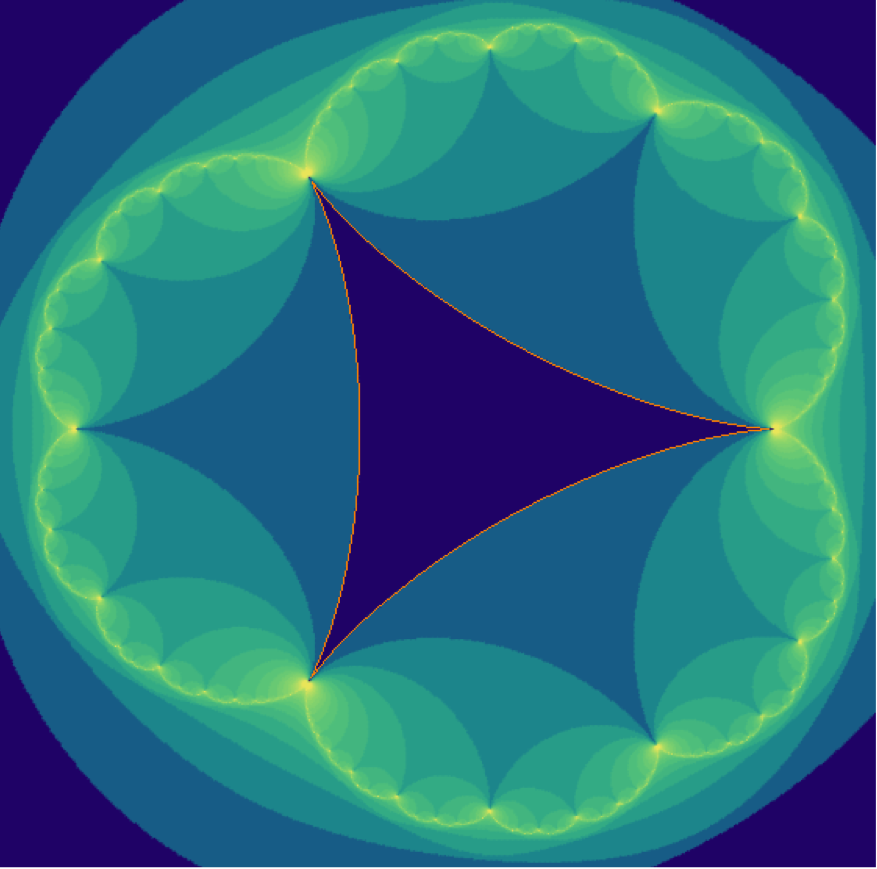}
\end{center}
\caption{Schwarz dynamics of the deltoid with tiles of various ranks shaded. The interior and the exterior of the bright green Jordan curve are the tiling set and the non-escaping set, respectively.}
\label{deltoid_reflection_pic}
\end{figure}

\subsection{Circle-and-cardioid family and main results} One of the principal goals of this paper is to study the dynamics of the following one-parameter family of Schwarz reflection maps. Consider the cardioid 
$$
\heartsuit:=\phi(\D),\ \textrm{where}\ \phi(u)= u/2-u^2/4;
$$ 
i.e., $\heartsuit$ is the domain bounded by the cardioid curve defined in Subsection~\ref{intro_1_subsec}.
For each complex number $a\in\C\setminus(-\infty,-1/12)$, the circle centered at $a$ circumscribing the cardioid $\heartsuit$ touches $\partial\heartsuit$ at exactly one point. Let $r_a$ be the radius of this circumcircle, $T_a$ be the \emph{droplet} $\overline{B(a,r_a)}\setminus\heartsuit$, and let $F_a$ denote the corresponding Schwarz reflection map: the circle reflection $\sigma_a$ in the exterior of $B(a,r_a)$ and the reflection $\sigma$ with respect to the cardioid in its interior (see Figure~\ref{c_and_c_fig}). This family of Schwarz reflections maps $F_a$ is denoted by $\mathcal{S}$ and is referred to as the \emph{circle-and-cardioid family}.

\begin{figure}[ht!]
\captionsetup{width=0.96\linewidth}
\centering
\includegraphics[scale=0.12]{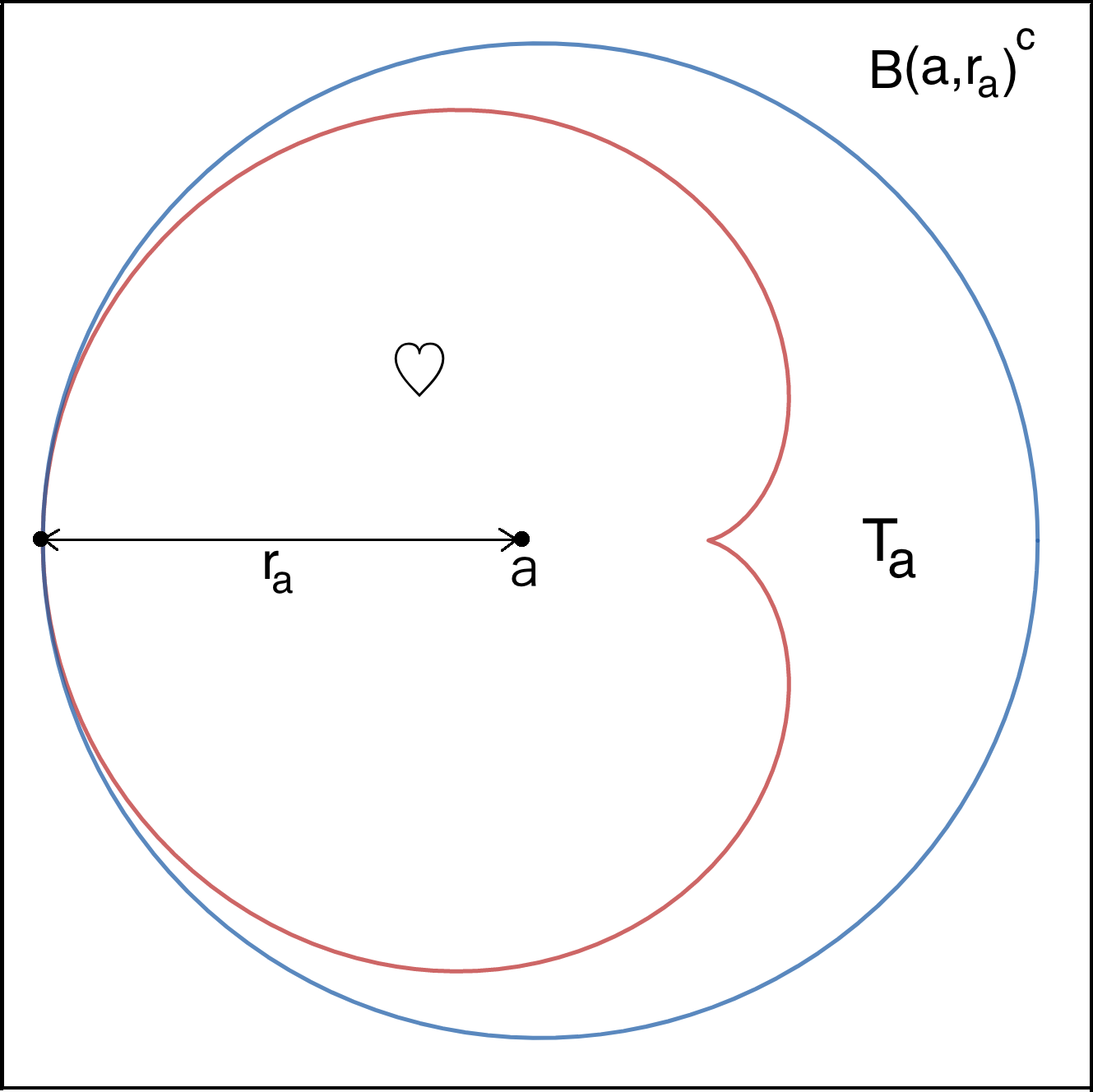}
\caption{The circle $\{\vert w-a\vert=r_a\}$ is the circumcircle to the cardioid $\heartsuit$ centered at $a$. The domain of $F_a$ is the union of $\overline{\heartsuit}$ and the exterior disk $B(a,r_a)^c:=\widehat{\C}\setminus B(a,r_a)$.}
\label{c_and_c_fig}
\end{figure}

The Schwarz reflection map $F_a$ is unicritical; indeed, the circle reflection map $\sigma_a$ is univalent, while the cardioid reflection map $\sigma$ has a unique critical point at the origin. Note that the droplet $T_a$ has two singular point on its boundary. Removing these two singular point from $T_a$, we obtain the \emph{desingularized droplet} $T_a^0$ (which is also called the \emph{fundamental tile}). The \emph{non-escaping set} of $F_a$ (denoted by $K_a$) consists of all points that do not escape to the fundamental tile $T_a^0$ under iterates of $F_a$, while the \emph{tiling set} of $F_a$ (denoted by $T_a^\infty$) is the set of points that eventually escape to $T_a^0$. The connected components of $\Int{K_a}$ are called \emph{Fatou components}. The boundary of the tiling set is called the \emph{limit set}, and is denoted by $\Gamma_a$. It is instructive to think of the tiling set, the non-escaping set, and the limit set of $F_a$ as the analogues of the basin of infinity, the filled Julia set, and the Julia set (respectively) of a polynomial.

As in the case of quadratic polynomials, the non-escaping set of $F_a$ is {\it connected} if and only if it contains the unique critical point of $F_a$; i.e., the critical point does not escape to the fundamental tile. If the critical point of $F_a$ does not escape to the fundamental tile $T_a^0$, then the conformal map $\psi_a$ from $T_a^0$ onto $\Pi$ extends to a biholomorphism between the tiling set $T_a^\infty$ and the unit disk $\D$. Moreover, the extended map $\psi_a$ conjugates $F_a$ to the reflection map $\rho$. On the other hand, if the critical point escapes to the fundamental tile, the corresponding non-escaping set is totally disconnected (see Figure~\ref{various_limit_sets}).

This leads to the notion of the {\it connectedness locus} $\cC(\mathcal{S})$ as the set of parameters with connected non-escaping sets (see Figure~\ref{conn_locus_tessellation}). Equivalently, $\cC(\mathcal{S})$ is exactly the set of parameters for which the tiling set is unramified.

The geometrically finite maps (i.e., maps with attracting/parabolic cycles, and maps with non-escaping, strictly pre-periodic critical point) of $\mathcal{S}$ are of particular importance. They belong to the connectedness locus $\cC(\mathcal{S})$, and their topological and analytic properties are more tractable. For instance, if $F_a$ is geometrically finite, then the limit set $\Gamma_a$ of $F_a$ is locally connected, and the area of $\Gamma_a$ is zero.

For any $c_0$ in the Tricorn with a locally connected Julia set, one can glue the conformal dynamical systems
$$\rho: \overline{\mathbb D}\setminus\Int{\Pi}\to  \overline{\mathbb D}\quad \mathrm{and}\quad f_{c_0}: \mathcal{K}_{c_0}\to\mathcal{K}_{c_0},\ z\mapsto\overline{z}^2+c_0$$ (where $\mathcal{K}_{c_0}$ is the filled Julia set of $f_{c_0}$) by (a factor of) the circle homeomorphism $\mathcal{E}$ yielding a (partially defined) topological map $\eta$ on a topological $2$-sphere. We say that $F_{a_0}$ is the (unique) conformal mating of $\rho$ and the quadratic anti-polynomial $f_{c_0}$ if this topological $2$-sphere admits a (unique) conformal structure that turns $\eta$ into an anti-holomorphic map conformally conjugate to $F_{a_0}$.

\begin{figure}[ht!]
\captionsetup{width=0.96\linewidth}
\centering
\includegraphics[scale=0.27]{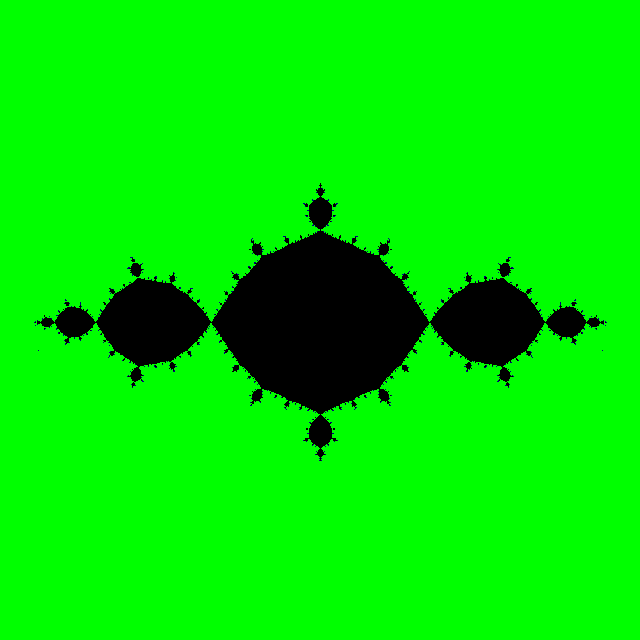}\ \includegraphics[scale=0.2]{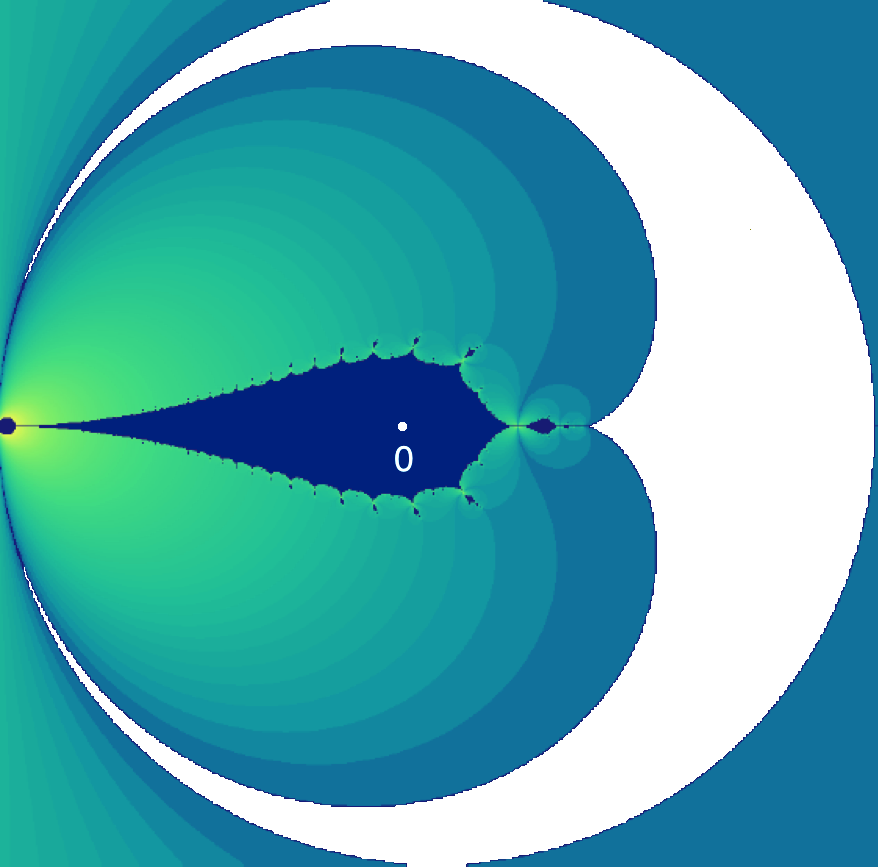}
\caption{Left: The filled Julia set of the map $\overline{z}^2-1$. The map has a super-attracting cycle of period $2$. Right: The part of the non-escaping set of $F_0$ inside the cardioid (in dark blue) with the critical point $0$ marked. The map has a super-attracting cycle of period $2$.}
\label{schwarz_basilica_1}
\end{figure}

The current paper has the dual objective of producing a topological model of the parameter space of the family $\mathcal{S}$, and proving that every geometrically finite map in $\mathcal{S}$ is a conformal mating of a unique geometrically finite quadratic anti-polynomial and the reflection map $\rho$ arising from the ideal triangle group (see Subsection~\ref{C_and_C_subsec} for the definition of $\rho$).

\begin{theorem}[Geometrically finite maps are mating]\label{filled_julia_top_conjugacy}
Every geometrically finite map in $\mathcal{S}$ is a conformal mating of a unique geometrically finite quadratic anti-polynomial and the reflection map $\rho$. 
\end{theorem}

Let us mention in this respect that in the 1990s, Bullett and Penrose discovered holomorphic correspondences that are matings of quadratic holomorphic polynomials and the modular group \cite{BP}. More recently, Bullett and Lomonaco studied the dynamics of such correspondences and showed that they also appear as matings of certain rational maps and the modular group \cite{BuLo2,BuLo1}. The conclusion of Theorem~\ref{filled_julia_top_conjugacy} can be viewed as a similar phenomenon in the anti-holomorphic world, which produces in a simple and systematic way an abundant supply of such examples.

The proof of Theorem~\ref{filled_julia_top_conjugacy} requires a thorough understanding of the relation between the geometrically finite maps in $\mathcal{S}$ and those in the \emph{basilica limb} $\mathcal{L}$ of the Tricorn (see Subsection~\ref{sec_basilica_limb} for the definition of the basilica limb of the Tricorn). We establish such a relation via combinatorial models of the maps which we briefly describe below.

In usual holomorphic dynamics, the uniformization of the basin of infinity of an (anti-)polynomial extends continuously to the Julia set, provided that the Julia set is connected and locally connected. Similarly, for parameters $a$ in the connectedness locus $\cC(\mathcal{S})$, there is a dynamically defined conformal isomorphism $\psi_a$ between the tiling set $T_a^\infty$ and the unit disk $\mathbb{D}$ that conjugates $F_a$ to the reflection map $\rho$ (see Subsection~\ref{C_and_C_subsec}, also compare \cite[Proposition~5.38]{LLMM1}). Moreover, if the limit set of such an $F_a$ is locally connected, then $\psi_a^{-1}$ extends continuously to the limit set. This yields a topological model of the non-escaping set of $F_a$ as the quotient of the closed unit disk by a geodesic lamination (analogous to polynomial laminations). 

To glue the action of the reflection map $\rho$ with that of quadratic anti-polynomials, we use a topological conjugacy $\mathcal{E}$ (see Subsection~\ref{C_and_C_subsec}) between $\rho$ (which models the external dynamics of the maps in $\mathcal{S}$) and the anti-doubling map $\theta\mapsto-2\theta$ on the circle (which models the external dynamics of quadratic anti-polynomials). The desired relation between the geometrically finite maps mentioned above is achieved by showing that $\mathcal{E}$ induces a bijective correspondence between the laminations of geometrically finite maps in $\mathcal{S}$ and those of geometrically finite maps in $\mathcal{L}$.

\begin{theorem}[Combinatorial bijection between geometrically finite maps]\label{thm_comb_bijec_pcf}
There exists a natural bijection between the geometrically finite parameters in $\mathcal{S}$ and those in $\mathcal{L}$ such that the laminations of the corresponding maps are related by $\mathcal{E}$ and the dynamics on the respective periodic Fatou components are conformally conjugate.
\end{theorem}

The above bijection is called ``combinatorial straightening", and is denoted by $\chi$. While the existence of the map $\chi$ follows from well-known realization results in polynomial dynamics, the proof of the fact that $\chi$ is a bijection lies at the heart of the technical difficulties of this paper. 

Injectivity of $\chi$ is equivalent to ``combinatorial rigidity" of geometrically finite maps in $\mathcal{S}$; more precisely, one needs to prove that geometrically finite maps in $\mathcal{S}$ are completely determined by their combinatorial models (or laminations) and suitable conformal invariants associated with them (see Subsection~\ref{rigidity_hyp_para}). We establish such rigidity results via a ``Pullback Argument''. In fact, due to certain geometric features of the quadrature domains under consideration, the proof also involves an analysis of the boundary behavior of conformal maps near cusps and double points. 

\begin{figure}[ht!]
\captionsetup{width=0.96\linewidth}
\begin{center}
\includegraphics[scale=0.45]{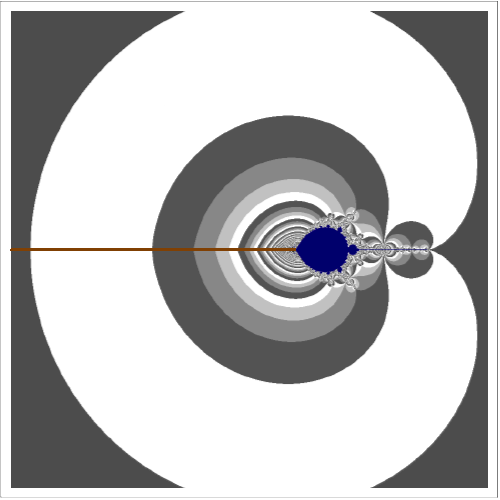}
\end{center}
\caption{The escape locus $\C\setminus\left((-\infty,-1/12)\cup\cC(\mathcal{S})\right)$ is tessellated by parameter tiles, a few of which are marked. The brown line stands for the slit $(-\infty,-1/12)$, and the connectedness locus $\cC(\mathcal{S})$ is shown in blue. The initial parameter $-1/12$ of $\cC(\mathcal{S})$ corresponds to the fat basilica parameter of the Tricorn.}
\label{conn_locus_tessellation}
\end{figure}

On the other hand, surjectivity of $\chi$ amounts to finding geometrically finite maps in $\cC(\mathcal{S})$ with prescribed laminations and conformal data. Note that since Schwarz reflection maps are not defined on the whole sphere (they are not even branched covers), there is no ``Thurston Realization Theorem'' available for such classes of maps. To circumvent this problem, we design a suitable machinery for description of the external structure of $\cC(\mathcal{S})$ (the ``escape locus'' of $\mathcal{S}$); namely, we uniformize the escape locus of $\mathcal{S}$, and tessellate the escape locus by dynamically meaningful tiles. More precisely, for every map $F_a$ with a disconnected limit set (i.e., when $a\notin\cC(\mathcal{S})$), there is a conformal map $\psi_a$, defined on a proper subset of the tiling set containing the critical value $\infty$, that conjugates $F_a$ to $\rho$ (see Subsection~\ref{C_and_C_subsec}). Using the map $\psi_a$, we prove the following result (compare Figure~\ref{conn_locus_tessellation}). 

\begin{theorem}[Uniformization of the escape locus]\label{thm_unif_exterior_conn_locus}
The map $$\pmb{\Psi}:\C\setminus((-\infty,-1/12)\cup\cC(\mathcal{S}))\to\D_2,$$ $$a\mapsto\psi_a(\infty)$$ is a homeomorphism, where $\D_2$ is a simply connected subset of the unit disk $\D$.
\end{theorem}

The proof of Theorem~\ref{thm_unif_exterior_conn_locus} has some features in common with the proof of connectedness of the Mandelbrot set, but lack of holomorphic parameter dependence of the maps $F_a$ adds significant subtlety to the situation forcing us to adopt a more topological route. The uniformization $\pmb{\Psi}$ allows us to use the tessellation of the unit disk arising from the ideal triangle group to produce a tessellation of the escape locus of $\mathcal{S}$. One can then define ``external parameter rays" for $\mathcal{S}$ as ``spines'' of suitable sequences of tiles in the tessellation of the escape locus. Equivalently, these external rays are obtained by pulling back a Cayley graph of the ideal triangle group via the uniformization $\pmb{\Psi}$. We remark that although the parameter rays and tiles of $\mathcal{S}$ are reminiscent of usual ray-equipotential structures in escape loci of polynomial parameter spaces, the parameter rays of $\mathcal{S}$ are not defined as pre-images of radial lines (see Definition~\ref{para_ray_schwarz} and \cite[\S 2]{LLMM1} for the precise definition). Finally, the surjectivity part of Theorem~\ref{thm_comb_bijec_pcf} is established by realizing geometrically finite maps with prescribed combinatorics (in $\mathcal{S}$) as limit points of suitable external parameter rays of $\mathcal{S}$.

The tessellation of the escape locus and the study of the landing/accumulation properties of the external rays of $\mathcal{S}$ not only play a key role in the proof of bijection between geometrically finite maps mentioned above, but also enable us to study the connectedness locus $\cC(\mathcal{S})$ from outside. Indeed, these results combined with our knowledge of the corresponding situation for the Tricorn allow us to demonstrate that the lamination model of the connectedness locus of the circle-and-cardioid family precisely corresponds to that of the basilica limb of the Tricorn under the circle homeomorphism $\mathcal{E}$ (see Figure~\ref{schwarz_basilica_limb} for pictures of the two parameter spaces in question). This proves that the locally connected models of the two connectedness loci are homeomorphic (see Subsection~\ref{sec_basilica_limb} for the definition of the \emph{abstract basilica limb} $\widetilde{\mathcal{L}}$ of the Tricorn, and Subsection~\ref{model_homeo} for that of the \emph{abstract connectedness locus} $\widetilde{\cC(\mathcal{S})}$ of $\mathcal{S}$).

\begin{figure}[ht!]
\captionsetup{width=0.96\linewidth}
\begin{center}
\includegraphics[scale=0.18]{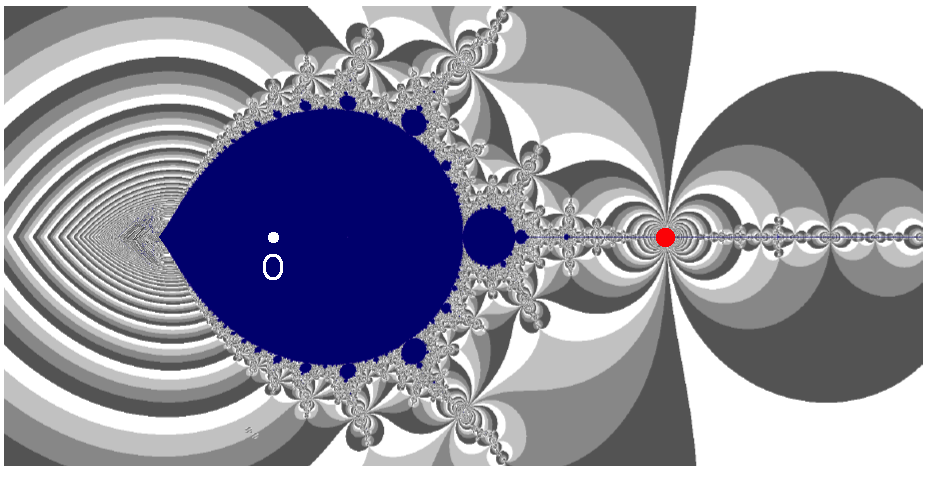}\ \includegraphics[scale=0.34]{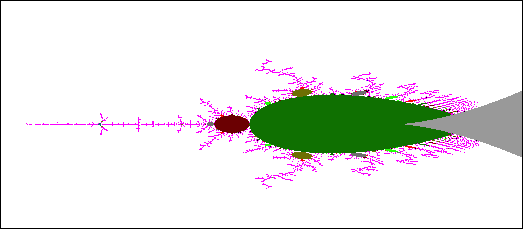}
\end{center}
\caption{Left: A blow-up of $\cC(\mathcal{S})$ around the principal hyperbolic component (having its center at $a=0$) is shown. The marked red parameter (tip of the doubling Mandelbrot copy) corresponds to a Misiurewicz map for which the critical point $0$ lands at a fixed point in three iterates. Right: The region to the left of the grey region (which is a part of the principal hyperbolic component of the Tricorn) is the real basilica limb of the Tricorn.}
\label{schwarz_basilica_limb}
\end{figure}

\begin{theorem}[Homeomorphism between models]\label{thm_model_homeo}
The map $\mathcal{E}$ induces a homeomorphism between the abstract connectedness locus $\widetilde{\cC(\mathcal{S})}$ of the family $\mathcal{S}$ and the abstract basilica limb $\widetilde{\mathcal{L}}$ of the Tricorn.
\end{theorem}

Combining Theorem~\ref{thm_unif_exterior_conn_locus} and Theorem~\ref{thm_model_homeo}, one obtains a description of the parameter space of the circle-and-cardioid family as a ``combinatorial mating" of the basilica limb of the maps with (a part of) the unit disk $\D$ equipped with its tessellation arising from the ideal triangle group (compare \cite{Dud} for an analogous mating description of the parameter space of a certain family of quadratic rational maps).

To conclude, it is worth mentioning that there are several compelling reasons for adopting a combinatorial approach to describe the topology of the connectedness locus $\cC(\mathcal{S})$. The ``external map'' of $F_a$ is given by the map $\rho$, which is a two-to-one covering of the circle with three \emph{parabolic} fixed points. On the other hand, the external map of a quadratic anti-polynomial is given by $\theta\mapsto-2\theta$, which is a two-to-one covering of the circle with three \emph{repelling} fixed points. As there is no quasisymmetric conjugacy between a parabolic and a repelling fixed point, one cannot quasiconformally straighten $F_a$ to a quadratic anti-polynomial. In fact, there exists no (anti-)rational map of degree two with three parabolic fixed points (alternatively, there is no (anti-)Blaschke map with more than one parabolic fixed point). Consequently, maps in $\mathcal{S}$ cannot be quasiconformally straightened to any family of (anti-)rational maps. 

In addition to the above obstacles, there are intrinsic issues with anti-holomorphic parameter spaces that make straightening maps ill-behaved (see Section~\ref{straightening_discont}, also compare \cite[Theorem~1.1]{IM2}). Since the Tricorn is known to be non-locally connected (with quite non-uniform wiggly features in various places), one needs to work with its locally connected topological model to develop a tractable theory. On the other hand, there are deep MLC-type problems of combinatorial rigidity for the Tricorn that are still open (compare \cite[\S 38]{L6}), and go beyond the scope of the current work. Any progress in this direction would bring our topological models closer to the actual connectedness loci.

\subsection{Organization of the paper}

In Section~\ref{sec_back}, we give a detailed description of the dynamics of quadratic anti-polynomials and their connectedness locus, the Tricorn. Although many techniques used in the study of the Mandelbrot set can be adapted to investigate the Tricorn, lack of holomorphic parameter dependence adds complexity to the situation. Moreover, lack of quasi-conformal rigidity of parameters on the boundary of the Tricorn results in various topological subtleties. We discuss some of the essential topological differences between the Mandelbrot set and the Tricorn, and record all the results that we will need in the paper.

Sections~\ref{ideal_triangle} contains a description of the ideal triangle group, the associated reflection map $\rho$, and the conjugacy $\mathcal{E}$ between $\rho$ and $\overline{z}^2$. In Section~\ref{sec_quad_domain}, we briefly review the basic definitions and some general properties of quadrature domains and Schwarz reflection maps.

Section~\ref{schwarz_circle_cardioid} is a recapitulation of the circle-and-cardioid family $\mathcal{S}$ that was introduced in \cite{LLMM1}. Some basic dynamical results about the maps in the circle-and-cardioid family (that were proved in \cite{LLMM1}) are also collected here.
 
In Section~\ref{unif_exterior_conn_locus}, we begin our study of the parameter space of $\mathcal{S}$. The main goal of this section is to prove Theorem~\ref{thm_unif_exterior_conn_locus}, which states that the conformal position of the escaping critical value produces a homeomorphism between the escape locus of $\mathcal{S}$ and a suitable simply connected domain. 

In Section~\ref{hyperbolic}, we describe the structure of hyperbolic components of $\cC(\mathcal{S})$. As in the case of the Tricorn, the hyperbolic components of odd period vastly differ from their even-period counterparts. 

Section~\ref{comb_rigidity_orbit_portraits} concerns combinatorics of geometrically finite maps. In Subsection~\ref{comb_orbit_portrait}, we introduce an important combinatorial object called \emph{orbit portraits} which, as in the polynomial case, describes the landing patterns of dynamical rays landing on a periodic orbit. Subsequently in Subsection~\ref{rigidity_hyp_para}, we prove a number of crucial rigidity statements (Theorems~\ref{rigidity_center},~\ref{rigidity_misi}) to the effect that PCF parameters in $\mathcal{S}$ are uniquely determined by their combinatorial models (orbit portraits for centers of hyperbolic components, and laminations for Misiurewicz parameters). This subsection also contains some rigidity statements for hyperbolic and parabolic maps. 

In Section~\ref{odd_period_parabolic}, we carry out a detailed study of the landing/accumulation properties of parameters rays of $\mathcal{S}$ at (pre-)periodic angles (under $\rho$). This requires a complete combinatorial understanding of parabolic parameters of $\cC(\mathcal{S})$. The odd period parabolic parameters of $\cC(\mathcal{S})$ and the structure of bifurcations across such parameters are studied in Subsection~\ref{odd_para_doubling}. Subsection~\ref{para_orbit_portrait} contains a combinatorial realization result for parabolic parameters as landing/accumulation points of parameter rays at periodic angles. In Subsection~\ref{para_pre_per}, we investigate landing properties of parameter rays of $\mathcal{S}$ at strictly pre-periodic angles. In particular, we characterize parameter rays landing at Misiurewicz parameters in terms of combinatorial properties of their dynamical planes. The results of this section play an important role in the proofs of our main theorems. 

In Section~\ref{comb_bijec_pcf}, we define the combinatorial straightening map $\chi$ on all geometrically finite maps of $\mathcal{S}$. More precisely, we send a geometrically finite map $F_a$ to the unique geometrically finite map of the Tricorn so that the homeomorphism $\mathcal{E}$ sends the lamination of the former to that of the latter, and the conformal conjugacy classes of the first return map to the characteristic Fatou components of the corresponding maps are the same. The fact that such a member of the Tricorn can be found follows from the combinatorial structure of the corresponding laminations, landing properties of external parameter rays of the Tricorn, and our understanding of the closures of the hyperbolic components. 

Thanks to the combinatorial rigidity results proved in Subsection~\ref{rigidity_hyp_para}, the above combinatorial straightening map turns out to be injective.

We proceed to show that the straightening map is surjective onto all geometrically finite maps of the Tricorn. We use landing properties of external parameter rays of $\mathcal{S}$ prepared in Subsection~\ref{para_pre_per} to show that Misiurewicz maps in $\mathcal{S}$ with prescribed lamination can be found as limit points of suitable parameter rays. To achieve the goal for hyperbolic parameters in $\mathcal{S}$, we first realize parabolic parameters using results from Subsection~\ref{para_orbit_portrait}. Since parabolic parameters lie on boundaries of hyperbolic components, this allows us to realize hyperbolic parameters by perturbing parabolic parameters inside hyperbolic components. This part of the argument involves a thorough understanding of odd period hyperbolic components of $\cC(\mathcal{S})$ and their bifurcation structure. This yields our desired combinatorial bijection between the geometrically finite maps of $\mathcal{S}$ and $\mathcal{L}$, and completes the proof of Theorem~\ref{thm_comb_bijec_pcf}. 

In the next Section~\ref{sec_PCF_mating}, we put together Theorem~\ref{thm_comb_bijec_pcf} and standard techniques in holomorphic dynamics to complete the proof of Theorem~\ref{filled_julia_top_conjugacy}. 

In Subsection~\ref{model_homeo}, we construct a locally connected model for $\cC(\mathcal{S})$, and use the landing properties of parameter rays to complete the proof of Theorem~\ref{thm_model_homeo}.

The final Section~\ref{straightening_discont} is devoted to a discussion of possible analytic improvements of the straightening map $\chi$. In fact, we show that the map $\chi$ has ``built-in" discontinuities. It is worth mentioning that discontinuity of straightening maps is typical in anti-holomorphic dynamics, and is related to ``non-universality" of certain conformal invariants (compare \cite[\S 8]{IM2}). 

In Appendix~\ref{unique_app}, we use softer arguments (that avoid conformal removability) to demonstrate that the deltoid Schwarz reflection map naturally arises as the conformal mating of $\overline{z}^2$ and the map $\rho$ arising from the ideal triangle group. We also employ similar arguments to show that the circle-and-cardioid family is canonical in the sense that, if a quadratic anti-polynomial $\overline{z}^2+c$ in the real basilica limb of the Tricorn is conformally mateable with the refection map $\rho$, then (up to M{\"o}bius conjugation) the corresponding conformal mating necessarily lies in the circle-and-cardioid family. We use this fact to prove uniqueness of the conformal mating of the Basilica anti-polynomial $\overline{z}^2-1$ and $\rho$.

 Appendix~\ref{conf_asymp_appendix} uses a classical result of Warschawski to prepare some analytic tools regarding boundary behavior of conformal maps near cusps. These technical results play a crucial role in the rigidity theorems proved in Section~\ref{comb_rigidity_orbit_portraits}.

\bigskip

\noindent\textbf{Acknowledgements.} The authors are grateful to the anonymous referee for detailed comments and suggestions for improvement.

\section{Background on holomorphic dynamics}\label{sec_back} 

\noindent \textbf{Notation:} 
\begin{itemize}
\item The complex conjugation map on the Riemann sphere is denoted by $\iota$. 
\item The complex conjugate of a set $X\subset{\C}$ will be denoted by $\iota(X)$, while $\overline{X}$ will stand for the topological closure of $X$. \item $B(a,r)$ (respectively $\overline{B}(a,r)$) will stand for the open (respectively closed) disk centered at $a$ with radius $r$. 
\item The complement of $X\subset\widehat{\C}$ will be denoted by $X^c$; i.e., $X^c:=\widehat{\C}\setminus X$. 
\end{itemize}
\medskip

This preliminary section will be a brief survey of several fundamental results on the dynamics and parameter spaces of quadratic polynomials and anti-polynomials. The results collected in this section will be repeatedly used in the rest of the paper.

\subsection{Dynamics of complex quadratic polynomials: the Mandelbrot set}\label{mandel}

The quadratic polynomial family is undoubtedly the most well-studied family of maps in holomorphic dynamics. Although it is the simplest family of nonlinear holomorphic maps, their dynamics and parameter space turn out be highly non-trivial. In particular, many powerful methods were developed to study the parameter space of these maps leading to remarkable theorems. These results and methods laid the foundation of the study of general parameter spaces of holomorphic maps, and act as a strong motivational factor in our investigation.

With this in mind, we will briefly review some aspects of the dynamics and parameter space of complex quadratic polynomials in this subsection. We will mainly touch upon the concepts and results that are directly (or indirectly) related to our study of the dynamics and parameter spaces of Schwarz reflections. 

Any complex quadratic polynomial can be affinely conjugated to a map of the form $p_c(z)=z^2+c$. The \emph{filled Julia set} $K(p_c)$ is defined as the set of all points that remain bounded under all iterations of $p_c$. The boundary of the filled Julia set is defined to be the \emph{Julia set} $J(p_c)$. The complement of the filled Julia set is called the \emph{basin of infinity}, and is denoted by $A_\infty(c)$.

For a periodic orbit (equivalently, a cycle) $\mathcal{O} = \{ z_1, z_2, \cdots, z_n \}$ of $p_c$, we denote
by $\mu \left(c,\mathcal{O}\right):=(p_c^{\circ n})'(z_1)$ the
\emph{multiplier} of $\mathcal{O}$ (the definition is independent of the choice of $z_i$). A periodic orbit $\mathcal{O}$ of $p_c$ is called \emph{super-attracting}, \emph{attracting}, \emph{neutral}, or \emph{repelling} if $\mu(c,\mathcal{O})=0$, $0<\vert\mu(c,\mathcal{O})\vert<1$, $\vert\mu(c,\mathcal{O})\vert=1$, or $\vert\mu(c,\mathcal{O})\vert>1$ (respectively). A neutral cycle is called \emph{parabolic} if the associated multiplier is a root of unity. Otherwise, it is called \emph{irrational}.

Every quadratic polynomial $p_c$ has two finite fixed points, and the sum of the multipliers of these two fixed points is $2$.

We now recall the existence of local uniformizing coordinates for attracting cycles of holomorphic maps. Let $z_0$ be an attracting fixed point of a holomorphic map $g$; i.e., $0<\vert g'(z_0)\vert<1$. By a classical theorem of Koenigs, there exists a conformal change of coordinates $\kappa$ in a neighborhood $U$ of $z_0$ such that $\kappa(z_0)=0$ and $\kappa\circ g(z) = g'(z_0)\kappa(z)$, for all $z$ in $U$. The map $\kappa$ is known as Koenigs linearizing coordinate, and it is unique up to multiplication by a non-zero complex number \cite[Theorem~8.2]{M1new}.

For every quadratic polynomial, $\infty$ is a super-attracting fixed point. It is well-known that there is a conformal map $\phi_c$ near $\infty$ such that $\displaystyle \lim_{z\to \infty} \phi_c(z)/z=~1$ and $\phi_c\circ p_c(z)= \phi_c(z)^2$ \cite[Theorem~9.1]{M1new}. The map $\phi_c$ is called the \emph{B{\"o}ttcher coordinate} for $p_c$. The function $G_c(z):=\log\vert\phi_c\vert$ can be extended as a subharmonic function to the entire complex plane such that it vanishes precisely on $K(p_c)$ and has a logarithmic singularity at $\infty$. In other words, $G_c$ is the Green's function of $K(p_c)$ \cite[Corollary~9.2, Definition~9.6]{M1new}. The level curves of $G_c$ are called \emph{equipotentials} of $p_c$. As for the conformal map $\phi_c$ itself, it extends as a conformal isomorphism to an equipotential containing $c$, when $K(p_c)$ is disconnected, and extends as a biholomorphism from $\widehat{\mathbb{C}} \setminus K(p_c)$ onto $\widehat{\mathbb{C}} \setminus \overline{\mathbb{D}}$ when $K(p_c)$ is connected \cite[Theorem~9.3, Theorem~9.5]{M1new}. The \emph{dynamical ray} $R_c(t)$ of $p_c$ at an angle $t\in\R/\Z$ is defined as the pre-image of the radial line at angle $t$ under $\phi_c$. Evidently, $p_c$ maps the dynamical ray $R_c(t)$ to the dynamical ray $R_c(2t)$.

We say that the dynamical ray $R_c(t)$ of $p_c$ lands if $\overline{R_c(t)} \cap K(p_c)$ is a singleton; this unique point is called the landing point of $R_c(t)$. It is worth mentioning that for a quadratic polynomial with connected filled Julia set, every dynamical ray at a periodic angle (under multiplication by $2$) lands at a repelling or parabolic periodic point, and conversely, every repelling or parabolic periodic point is the landing point of at least one periodic dynamical ray \cite[\S 18]{M1new}. 

The filled Julia set $K(p_c)$ of a quadratic polynomial is either connected, or totally disconnected. In fact, $K(p_c)$ is connected if and only the unique critical point $0$ does not lie in the basin of infinity $A_\infty(c)$ \cite[Expos{\'e}~III, \S1, Proposition~1]{orsay} \cite[Chapter~VIII, Theorem~1.1]{CG1}. This dichotomy leads to the notion of the connectedness locus.

\begin{definition}
The \emph{Mandelbrot set} $\mathcal{M}$ is the connectedness locus of complex quadratic polynomials; i.e., $$\mathcal{M}:=\{c\in\C: K(p_c)\ \mathrm{is\ connected}\}.$$
\end{definition}

\begin{figure}[ht!]
\captionsetup{width=0.96\linewidth}
\begin{center}
\includegraphics[scale=0.35]{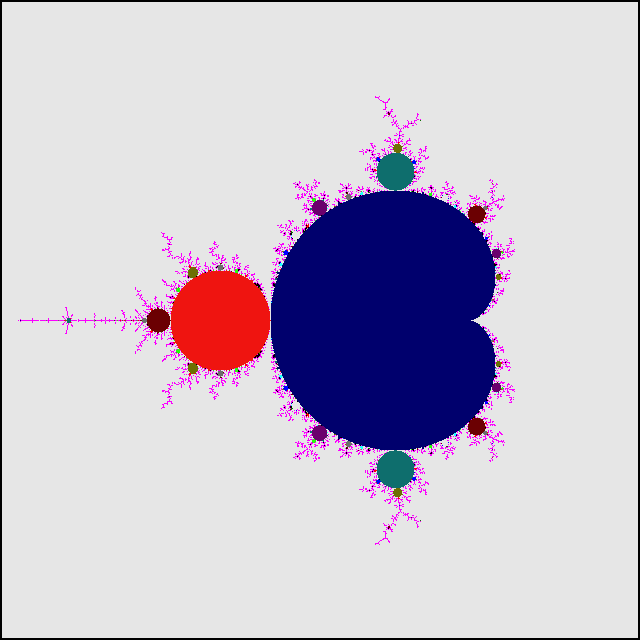}
\end{center}
\caption{Mandelbrot set, the connectedness locus of quadratic holomorphic polynomials $z^2+c$.}
\label{mandel_pic}
\end{figure}

A parameter $c\in\mathcal{M}$ is called \emph{hyperbolic} if $p_c$ has a (necessarily unique) attracting cycle. A connected component $H$ of the set of all hyperbolic parameters is called a \emph{hyperbolic component} of $\mathcal{M}$. The following classical result is a consequence of holomorphic parameter dependence of the maps $p_c$ \cite[Expos{\'e}~XIX, Theorem~1]{orsay} \cite[Chapter~VIII, Theorem~1.4, Theorem~2.1]{CG1}.

\begin{theorem}[Hyperbolic components of $\mathcal{M}$]\label{mandel_hyp}
Every hyperbolic component of $\mathcal{M}$ is a connected component of the interior of $\mathcal{M}$. For every hyperbolic component $H$, there exists some $n\in\N$ such that each $p_c$ in $H$ has an attracting cycle of period $n$. The multiplier of the unique attracting cycle of $p_c$ (where $c\in H$) defines a biholomorphism from $H$ onto the open unit disk $\D$ in the complex plane. This map is called the \emph{multiplier map} of $H$, and is denoted by $\mu_H$.

Moreover, every parameter on the boundary of a hyperbolic component $H$ has a unique neutral cycle of period dividing $n$. The derivative of $p_c^{\circ n}$ at this neutral cycle defines a continuous extension of $\mu_H$ up to $\partial H$.
\end{theorem}

The above proposition directly leads to the notion of centers and roots of hyperbolic components. 

\begin{definition}
\label{Def:RootsAndCenters}
The \emph{center} of a hyperbolic component $H$ is the unique parameter in $H$ which has a super-attracting cycle (equivalently, where the multiplier map vanishes).

The \emph{root} of $H$ is the unique parameter $c\in\partial H$ with $\mu_H(c)=1$. 
\end{definition}

The unique hyperbolic component of period one (also called the principal hyperbolic component of $\mathcal{M}$) will be of particular importance to us. We denote it by $\heartsuit$. A straightforward computation shows that the inverse of the multiplier map of $\heartsuit$ takes the form $$\phi:\D\to\heartsuit$$ $$\phi(\mu)= \mu/2-\mu^2/4.$$

Roots of hyperbolic components are intimately related to bifurcation phenomena in $\mathcal{M}$. If $c$ is a parabolic parameter of $\mathcal{M}$ such that $p_c$ has a $k$-periodic cycle of multiplier $e^{2\pi ip/q}$ (where $\mathrm{gcd}(p,q)=1$ and $q\geq 2$), then $c$ lies on the boundary of a hyperbolic component $H$ of period $k$ and a hyperbolic component $H'$ of period $kq$. Moreover, $c$ is the root of $H'$ \cite[Expos{\'e}~XIV, \S 5, Proposition 5]{orsay} .

The following theorem plays a basic role in the study of $\mathcal{M}$. For a proof, see  \cite[Expos{\'e}~VIII, \S I.3, Theorem~1]{orsay}.

\begin{theorem}[Connectedness of $\mathcal{M}$]\label{mandel_connected}
The map $\Phi : \mathbb{C} \setminus \mathcal{M} \rightarrow \mathbb{C} \setminus \overline{\mathbb{D}}$, defined by $c \mapsto \phi_c(c)$ (where $\phi_c$ is the B\"{o}ttcher coordinate for $p_c$) is a biholomorphism. In particular, the Mandelbrot set is compact, connected, and full.
\end{theorem}

The above theorem allows one to define parameter rays of the Mandelbrot set as pre-images of radial lines under $\Phi$. More precisely, the \emph{parameter ray} of $\mathcal{M}$ at angle $\theta$ is defined as
$$\mathcal{R}_\theta :=\{ \Phi^{-1}\left(re^{2\pi i\theta}\right), r>1 \}.$$
If $\displaystyle \lim_{r \rightarrow 1^+} \Phi^{-1}\left(re^{2\pi i\theta}\right)$ exists, we say
that $\mathcal{R}_\theta$ lands. The parameter rays of the Mandelbrot set have been profitably used to reveal its combinatorial and topological structure. In particular, it is known that all parameter rays of $\mathcal{M}$ at rational angles land. For a complete description of landing patterns of rational parameter rays and the corresponding structure theorem of the Mandelbrot set, see \cite[Theorem~1.1]{S1a}.

One of the most conspicuous features of the Mandelbrot set is its self-similarity. In \cite{DH2}, Douady and Hubbard developed the theory of polynomial-like maps to study this self-similarity. They proved the ``straightening theorem'' that, under certain circumstances, allows one to study a sufficiently large iterate of a polynomial by associating a simpler dynamical system, namely a polynomial of smaller degree, to it \cite[Chapter~I, Theorem~1]{DH2}. They used it to explain the existence of infinitely many small homeomorphic copies of the Mandelbrot set in itself \cite[Chapter~II, Proposition~14]{DH2}. It is worth mentioning that continuity of the straightening map from a baby Mandelbrot set to the original Mandelbrot set is an essential consequence of quasiconformal rigidity of parameters on the boundary of $\mathcal{M}$ \cite[Chapter~I, Proposition~7]{DH2}. In fact, straightening maps are typically discontinuous in the parameter spaces of higher degree polynomials \cite{I}.

For a more general and comprehensive discussion of straightening of quadratic-like families, see \cite[Chapter 6]{L6}. 

We refer the readers to \cite{orsay} for an account of the early development of the subject, and to \cite{L6} for a more comprehensive account containing many sophisticated results on the Mandelbrot set.

\subsection{Dynamics of quadratic anti-polynomials: the Tricorn}\label{sec_tricorn}

In this Section, we recall some known results on the dynamics of quadratic anti-polynomials, and their parameter space. The reason to include this fairly detailed survey is twofold. Since the Schwarz reflection maps are anti-holomorphic and they depend only real-analytically (and not holomorphically) on the parameter, some of the purely holomorphic techniques used to study the Mandelbrot set fail to work in this setting. It is, therefore, not too surprising that the tools required to study the dynamics and parameter space of quadratic anti-polynomials find widespread applications in our study of the parameter space of Schwarz reflections. Secondly, some important topological features of the parameter space of anti-polynomials differ from their holomorphic counterpart. These differences serve as a mental guide in our analysis. Readers familiar with anti-holomorphic dynamics (or unwilling go through this lengthy exposition) may skip to Subsection~\ref{sec_basilica_limb} where the \emph{abstract basilica limb} of the Tricorn is defined, and come back to this section whenever required. 

Any quadratic anti-polynomial, after an affine change of coordinates, can be written in the form $f_c(z) = \overline{z}^2 + c$ for $c \in \mathbb{C}$. In analogy to the holomorphic case, the set of all points that remain bounded under all iterations of $f_c(z) = \overline{z}^2 + c$ (for $c \in \mathbb{C}$) is called the \emph{filled Julia set} $\mathcal{K}_c$. The boundary of the filled Julia set is defined to be the \emph{Julia set} $\mathcal{J}_c$. This leads, as in the holomorphic case, to the notion of \emph{connectedness locus} of quadratic anti-polynomials:

\begin{definition}
The \emph{Tricorn} is defined as $\mathcal{T} = \{ c \in \mathbb{C} : \mathcal{K}_c$ is connected$\}$. 
\end{definition} 

\begin{figure}[ht!]
\captionsetup{width=0.96\linewidth}
\begin{center}
\includegraphics[width=0.4\linewidth]{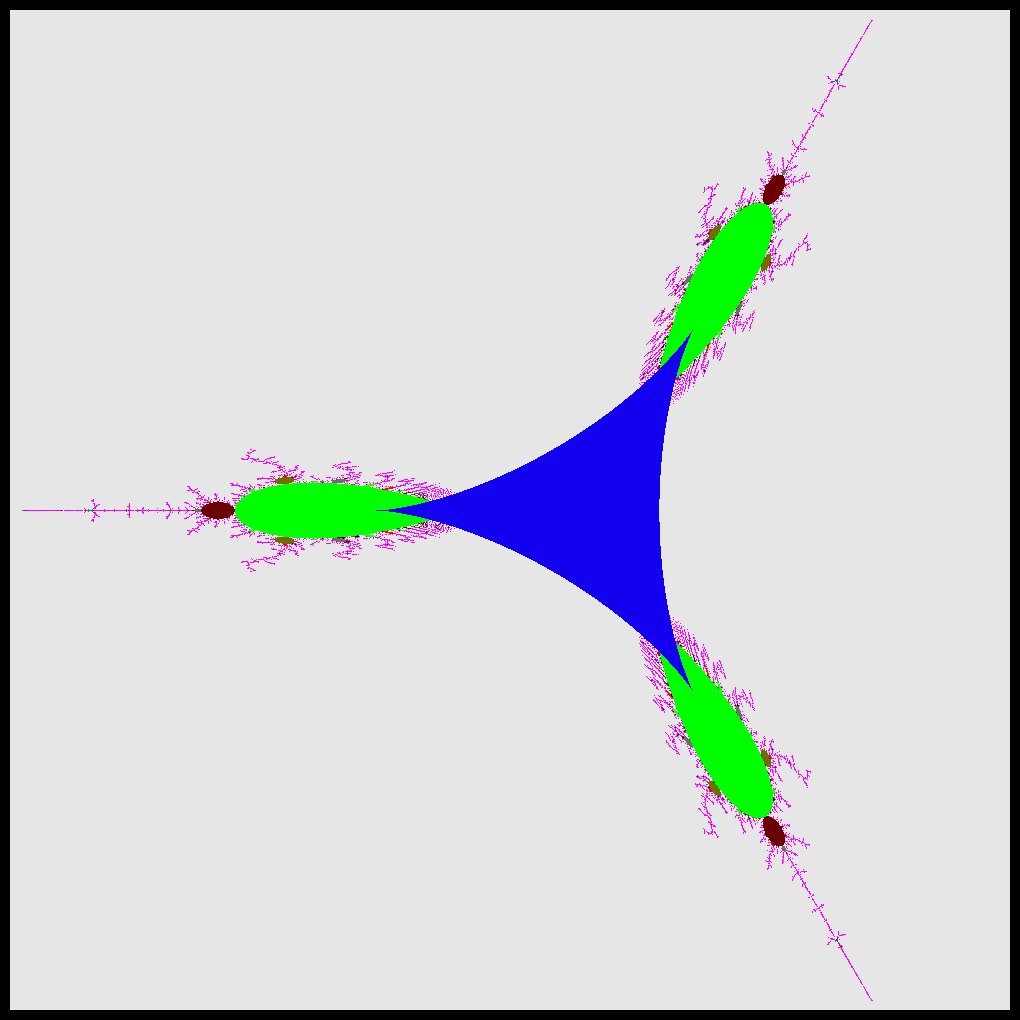}\quad \includegraphics[width=0.4\linewidth]{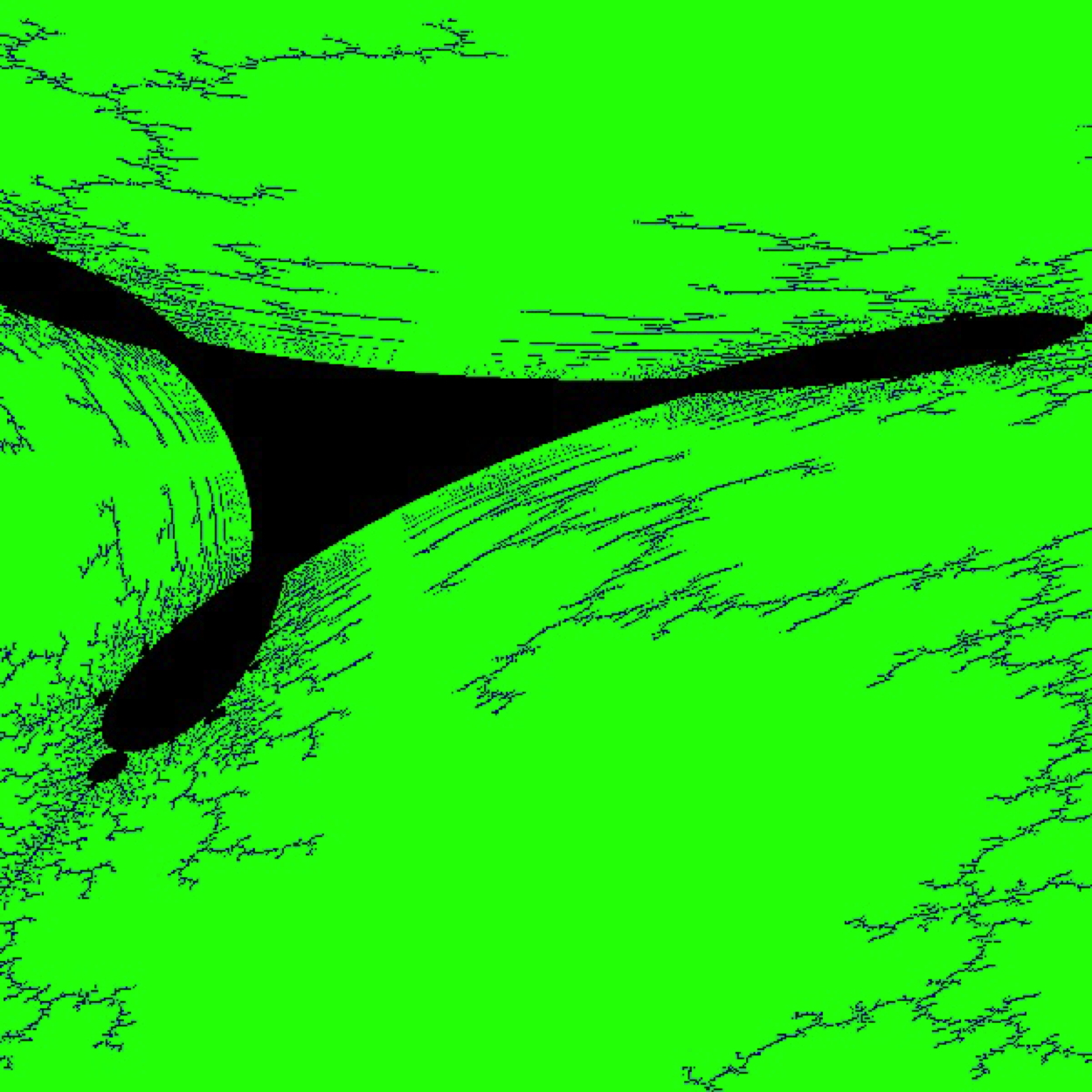} 
\end{center}
\caption{Left: Tricorn, the connectedness locus of quadratic anti-polynomials $\overline{z}^2+c$. Right: Wiggling of umbilical cord on an odd period hyperbolic component of the Tricorn.}
\label{tricorn_pic}
\end{figure}

\begin{remark}\label{symmetry}
The anti-polynomials $f_c$ and $f_{\omega c}$ are conformally conjugate via the linear map $z \mapsto  \omega z$, where $\omega = \exp(\frac{2\pi i}{3})$. It follows that $\mathcal{T}$ has a $3$-fold rotational symmetry (see Figure~\ref{tricorn_pic}).
\end{remark} 

\begin{remark}
The Tricorn (see Figure~\ref{tricorn_pic}) can be thought of as an object of intermediate complexity between one dimensional and higher dimensional parameter spaces. Combinatorially speaking, Douady's famous `plough in the dynamical plane, and harvest in the parameter space' principle continues to stand us in good stead since our maps are unicritical and our parameter space is still real two-dimensional. However, the iterates of a quadratic anti-polynomial $f_c$ only depend real-analytically on the parameter $c$ (unlike the iterates of $z^2+c$, which depend holomorphically on $c$). This results in significant topological differences between the Tricorn and the Mandelbrot set. Note that since the second iterate of $f_c$ is $(z^2+\overline{c})^2+c$, the space of quadratic anti-polynomials can be viewed as the real slice $\{(a,b)\in\C^2: a=\overline{b}\}$ of the family of biquadratic polynomials $\{(z^2+a)^2+b: a,b\in\C\}$. The polynomials $(z^2+a)^2+b$ generically have two infinite critical orbits, much like cubic polynomials. Hence, the dynamics and parameter space of quadratic anti-polynomials also resemble in many respects the connectedness locus of cubic polynomials.
\end{remark}

\subsubsection{Dynamical plane of anti-polynomials}\label{anti_poly_dyn_subsubsec}

Similar to the holomorphic case, we have a notion of B{\"o}ttcher coordinates for anti-polynomials. By \cite[Lemma~1]{Na1}, there is a conformal map $\phi_c$ near $\infty$ that conjugates $f_c$ to $\overline{z}^2$. As in the holomorphic case, $\phi_c$ extends conformally to an equipotential containing $c$, when $c\notin \mathcal{T}$, and extends as a biholomorphism from $\widehat{\mathbb{C}} \setminus \mathcal{K}_c$ onto $\widehat{\mathbb{C}} \setminus \overline{\mathbb{D}}$ when $c \in \mathcal{T}$.

\begin{definition}\label{dyn_ray}
The dynamical ray $R_c(\theta)$ of $f_c$ at an angle $\theta$ is defined as the pre-image of the radial line at angle $\theta$ under $\phi_c$.
\end{definition}

The dynamical ray $R_c(\theta)$ maps to the dynamical ray $R_c(-2\theta)$ under $f_c$. It follows that, at the level of external angles, the dynamics of $f_c$ can be studied by looking at the simpler map $$m_{-2}:\R/\Z\to\R/\Z,\ m_{-2}(\theta)=-2\theta.$$ We refer the readers to \cite[\S 3]{NS}, \cite{Sa} for details on the combinatorics of the landing pattern of dynamical rays for unicritical anti-polynomials.

For an anti-holomorphic germ $g$ fixing a point $z_0$, the quantity $\frac{\partial g}{\partial\overline{z}}\vert_{z_0}$ is called the \emph{multiplier} of $g$ at the fixed point $z_0$. One can use this definition to define multipliers of periodic orbits of anti-holomorphic maps (compare \cite[\S 1.1]{Sa1}). A cycle is called attracting (respectively, super-attracting or parabolic) if the associated multiplier has absolute value between $0$ and $1$ (respectively, is $0$ or a root of unity). A map $f_c$ is called hyperbolic (respectively, parabolic) if it has a (super-)attracting (respectively, parabolic) cycle. A connected component of the set of all hyperbolic parameters is called a \emph{hyperbolic component} of $\mathcal{T}$.

\begin{remark}\label{anti_fixed_points_count}
Recall that every quadratic polynomial $p_c$ has three distinct fixed points in $\widehat{\mathbb{C}}$ (except for $c=\frac14$, which has a simple fixed point at $\infty$, and a double fixed point at $\frac{1}{2}$). On the other hand, the number of distinct fixed points of a quadratic anti-polynomial $f_c$ drops from $5$ to $3$ as $c$ exits the principal hyperbolic component; i.e., the hyperbolic component of period one (the central blue region in Figure~\ref{tricorn_pic}). More precisely, for $c$ in the principal hyperbolic component, $f_c$ has two attracting and three repelling fixed points (in $\widehat{\mathbb{C}}$), while for $c$ outside the closure of the principal hyperbolic component, $f_c$ has one attracting and two repelling fixed points. 

It is worth mentioning that the Lefschetz fixed point index of an attracting (respectively, repelling) fixed point of $f_c$ (which is defined as the winding number of $f_c(z)-z$ along the boundary of a small disc centered at a fixed point) is $+1$ (respectively, $-1$). By the Lefschetz-Hopf Theorem, the sum of the indices of all fixed points must be $-1$ (since the topological degree of $f_c$ on $\widehat{\C}$ is $-2$). Thus outside the closure of the principal hyperbolic component, the loss of an attracting fixed point of $f_c$ must be accompanied by the loss of a repelling fixed point. This explains why the number of fixed points of $f_c$ must drop by two. 
\end{remark}

A (super-)attracting cycle of $f_c$ belongs to the interior of $\mathcal{K}_c$, and a parabolic cycle lies on the boundary of $\mathcal{K}_c$ (see \cite[\S 5, Theorem~5.2]{M1new}). Moreover, a parabolic periodic point necessarily lies on the boundary of a Fatou component (i.e., a connected component of $\Int{\mathcal{K}_c}$) that contains an attracting petal of the parabolic germ such that the forward orbit of every point in the component converges to the parabolic cycle. In the attracting (respectively, parabolic) case, the forward orbit of the critical point $0$ converges to the attracting (respectively, parabolic) cycle. In either case, the unique Fatou component containing the critical value is called the \emph{characteristic} Fatou component.

It is well-known that if $f_c$ has a connected Julia set, then all rational dynamical rays of $f_c$ land at repelling or parabolic (pre-)periodic points. This allows us to introduce an important combinatorial object that will play a key role later in the paper. 

\begin{definition}\label{def_rat_lami}
The \emph{rational lamination} of a quadratic anti-polynomial $f_c$ (with connected Julia set) is defined as an equivalence relation on $\mathbb{Q}/\mathbb{Z}$ such that $\theta_1 \sim \theta_2$ if and only if the dynamical rays $R_c(\theta_1)$ and $R_c(\theta_2)$ land at the same point of $J(f_c)$. The rational lamination of $f_c$ is denoted by $\lambda(f_c)$.
\end{definition}

The next proposition lists the basic properties of rational laminations.

\begin{proposition}\label{prop_rat_lam}
The rational lamination $\lambda(f_c)$ of a quadratic anti-polynomial $f_c$ satisfies the following properties.

\begin{enumerate}
\item $\lambda(f_c)$ is closed in $\Q/\Z\times\Q/\Z$.

\item Each  $\lambda(f_c)$-equivalence class $A$ is a finite subset of $\Q/\Z$.

\item If $A$ is a $\lambda(f_c)$-equivalence class, then $m_{-2}(A)$ is also a $\lambda(f_c)$-equivalence class.

\item If $A$ is a $\lambda(f_c)$-equivalence class, then $A\mapsto m_{-2}(A)$ is consecutive reversing.

\item $\lambda(f_c)$-equivalence classes are pairwise unlinked.
\end{enumerate}
\end{proposition}
\begin{proof}
The proof of \cite[Theorem~1.1]{kiwi} applies mutatis mutandis to the anti-holomorphic setting.
\end{proof}

\begin{definition}\label{formal_rat_lami}
An equivalence relation $\lambda$ on $\Q/\Z$ satisfying the conditions of Proposition~\ref{prop_rat_lam} is called a \emph{formal rational lamination} under $m_{-2}$.
\end{definition}

\subsubsection{Uniformization of the exterior of the Tricorn}

The following result was proved by Nakane \cite{Na1}.

\begin{theorem}[Real-analytic uniformization]\label{RealAnalUniformization}
The map $\Phi : \mathbb{C} \setminus \mathcal{T} \rightarrow \mathbb{C} \setminus \overline{\mathbb{D}}$, defined by $c \mapsto \phi_c(c)$ (where $\phi_c$ is the B\"{o}ttcher coordinate near $\infty$ for $f_c$) is a real-analytic diffeomorphism. In particular, the Tricorn is connected.
\end{theorem}

The previous theorem also allows us to define parameter rays of the Tricorn. 
\begin{definition}
The \emph{parameter ray} at angle $\theta$ of the Tricorn $\mathcal{T}$, denoted by $\mathcal{R}_{\theta}$, is defined as $\{ \Phi^{-1}(r e^{2 \pi i \theta}) : r > 1 \}$, where $\Phi$ is the real-analytic diffeomorphism from the exterior of $\mathcal{T}$ to the exterior of the closed unit disc in the complex plane constructed in Theorem~\ref{RealAnalUniformization}.
\end{definition}

\subsubsection{Uniformization of hyperbolic components} 

Recall that a map $f_c$ is called hyperbolic if it has a (super-)attracting cycle, and a hyperbolic component of $\mathcal{T}$ is defined as a connected component of the set of all hyperbolic parameters. Note that if $c$ lies in a hyperbolic component of odd (respectively even) period of $\mathcal{T}$, then the first return map of an attracting Fatou component of $f_c$ is anti-holomorphic (respectively holomorphic). Due to this dichotomy, one needs to study the topology of odd and even period hyperbolic components of $\mathcal{T}$ separately. The hyperbolic component of period $1$ can be studied by explicit computation \cite[Lemma~5.2]{NS} and is in some sense atypical. Hence we restrict our attention to higher period components, which is all we need in this paper.

Let $H$ be a hyperbolic component of even period $k$ of $\mathcal{T}$. For $c\in H$, the $k$-periodic attracting cycle of $f_c$ splits into two distinct attracting cycles of period $\frac{k}{2}$ under $f_c^{\circ 2}$. These two attracting cycles of $f_c^{\circ 2}$ have complex conjugate multipliers. Let $z_c$ be the attracting periodic point in the critical value Fatou component. We define $\mu_c:=(f_c^{\circ k})'(z_c)$. The map $c\mapsto (f_c^{\circ k})'(z_c)$ is called the \emph{multiplier map} of the hyperbolic component of even period $k$.

For $c\in H$, the restriction of $f_{c}^{\circ k}$ to the Fatou component $U_c$ containing $c$ is a degree $2$ proper holomorphic map. Moreover, $f_{c}^{\circ k}$ has a unique fixed point on $\partial U_c$. Choosing a Riemann map of $U_c$ that maps the attracting periodic point to $0$ and the unique boundary fixed point to $1$, we obtain a conjugacy between $f_c^{\circ k}\vert_{U_c}$ and a holomorphic Blaschke product of degree $2$ on $\D$. By construction, such a Blaschke product must be of the form $B_{a,\theta}^{+}(z)=e^{2\pi i\theta} z\frac{(z-a)}{(1-\overline{a}z)}$, where $a\in\mathbb{D}$ and $\theta=\theta(a)\in\R$ is selected so that $z=1$ is fixed by $B_{a,\theta}^{+}$. The unique such Blaschke product with a super-attracting fixed point is $B_{0,0}^{+}$. Let $\mathcal{B}^{+}$ be the space of all holomorphic Blaschke products $B_{a,\theta}^{+}$ where $a\in\mathbb{D}$ and $B_{a,\theta}^{+}(1)=1$.

Now let $H$ be a hyperbolic component of odd period $k\neq 1$ with center $c_0$. As before, for $c\in H$, let $z_c$ be the attracting periodic point of $f_c$ contained in the critical value Fatou component $U_c$. Let $\mathrm{Jac}(f_c^{\circ k},z_c)$ be the Jacobian determinant of $f_c^{\circ k}$ at $z_c$. A simple computation shows that $z_c$ is a periodic point of $f_c^{\circ 2}$ of period $k$, and the associated multiplier 
$$
(f_c^{\circ 2k})'(z_c)=-\mathrm{Jac}(f_c^{\circ k},z_c)=\left|\frac{\partial f_c^{\circ k}}{\partial\overline{z}}(z_c)\right|^2
$$ 
is real and positive (compare \cite[\S 1.1]{Sa1}). Clearly, one has to work a bit harder to define a meaningful conformal invariant that uniformizes a hyperbolic component $H$ of odd period. Unlike in the even period case, the natural conformal invariant for maps with odd period attracting cycles is not a purely local quantity; it uses the conformal position of the orbit of the critical point. The following definition was introduced in \cite[\S 6]{IM2} (see \cite[\S 5]{NS} for an equivalent formulation).

For $c\in H\setminus \lbrace c_0\rbrace$, there are two distinct critical orbits of the second iterate $f_c^{\circ 2}$ converging to an attracting cycle. One can choose two representatives of these two critical orbits (e.g. $c$ and $f_c^{\circ k}(c)$) in a fundamental domain (in the critical value Fatou component), and consider their ratio in a Koenigs linearizing coordinate. More precisely, let $\kappa_c : U_c\to \mathbb{C}$ be a Koenigs linearizing coordinate for $f_c^{\circ 2k}$ near $z_c$; i.e., $\kappa_c (f_c^{\circ 2k}(z))=(f_c^{\circ 2k})'(z_c) \kappa_c(z)$ for all $z\in U_c$. We define
\begin{align*}
\zeta_H(c)
&:=\frac{\kappa_c(f_c^{\circ k}(c))}{\kappa_c(c)}\;.
\end{align*} 
Since a Koenigs linearizing coordinate is unique up to multiplication by a non-zero number, the above ratio is independent of the choice of $\kappa_c$. At the center $c_0$, we define $\zeta_H(c_0)=0$. The map $\zeta_H$ is called the \emph{Koenigs ratio map} of the hyperbolic component $H$ of odd period $k$.

For $c\in H$, the restriction of $f_{c}^{\circ k}$ to the Fatou component $U_c$ containing $c$ is a degree $2$ proper anti-holomorphic map. Moreover, $f_{c}^{\circ k}$ has three fixed points on $\partial U_c$. Exactly one of them is a cut point of the Julia set, this point is called the dynamical root point of $f_c$ on $\partial U_c$. Choosing a Riemann map of $U_c$ that maps the attracting periodic point to $0$ and the dynamical root point to $1$, we obtain a conjugacy between $f_c^{\circ k}\vert_{U_c}$ and an anti-holomorphic Blaschke product of degree $2$ on $\D$. By construction, such a Blaschke product must be of the form $B^{-}_{a,\theta}(z)=e^{2\pi i\theta}\overline{z}\frac{(\overline{z}-a)}{(1-\overline{az})}$, where $a\in\mathbb{D}$ and $\theta=\theta(a)\in\R$ is selected so that $z=1$ is fixed by $B_{a,\theta}^{-}$. The unique such Blaschke product with a super-attracting fixed point is $B_{0,0}^{-}$. Let $\mathcal{B}^{-}$ be the space of all anti-holomorphic Blaschke products $B_{a,\theta}^{-}$ where $a\in\mathbb{D}$ and $B_{a,\theta}^{-}(1)=1$.

A direct calculation (or the Schwarz lemma) shows that $0$ is necessarily an attracting fixed point for every Blaschke product in $\mathcal{B}^{\pm}$. Clearly, both Blaschke product spaces $\mathcal{B}^{\pm}$ are simply connected as their common parameter space is the open unit disc $\mathbb{D}$. Thus, the spaces $\mathcal{B}^{\pm}$ can be endowed with real-analytic manifold structures (the appearance of $a$ and $\overline{a}$ in the definition of $B_{a,\theta}^\pm$ is an obstruction to the existence of a complex structure on $\mathcal{B}^{\pm}$). For both families of Blaschke products, we can define the multiplier/Koenigs ratio of the attracting fixed point. The next lemma elucidates the mapping properties of the multiplier/Koenigs ratio maps defined on $\mathcal{B}^{\pm}$ \cite[Lemma~5.4]{NS}.

\begin{lemma}\label{Blaschke}
The Blaschke product model spaces $\mathcal{B}^{\pm}$ are simply connected. Moreover, the Koenigs ratio map (respectively, the multiplier map) of the attracting fixed point defines a real-analytic $3$-fold branched covering from $\mathcal{B}^{-}$ (respectively a real-analytic diffeomorphism from $\mathcal{B}^{+}$) onto $\mathbb{D}$. 
\end{lemma}

The above discussion shows that we can associate a unique element of $\mathcal{B}^{-}$ (respectively $\mathcal{B}^{+}$) to every $f_c$ in an odd (respectively even) period hyperbolic component $H$. We thus have a map $\eta_H$ from $H$ to $\mathcal{B}^{-}$ or $\mathcal{B}^{+}$. The following theorem, which gives a dynamical uniformization of the hyperbolic components, was proved in \cite[Theorem~5.6, Theorem~5.9]{NS} (cf. \cite[\S 5]{M12}).

\begin{theorem}[Uniformization of hyperbolic components]\label{hyp_unif}
Let $H$ be a hyperbolic component. The map $\eta_H:H\to\mathcal{B}^-$ (respectively, $\mathcal{B}^+$) is a real-analytic diffeomorphism.

\begin{enumerate}
\item If $H$ is of odd period, then $\eta_H:H\to\mathcal{B}^-$ respects the Koenigs ratio of the attracting cycle. In particular, the Koenigs ratio map is a real-analytic $3$-fold branched covering from $H$ onto the unit disk, ramified only over the origin.

\item If $H$ is of even period, then $\eta_H:H\to\mathcal{B}^+$ respects the multiplier of the attracting cycle. In particular, the multiplier map is a real-analytic diffeomorphism from $H$ onto the unit disk.
\end{enumerate}
\end{theorem}

\subsubsection{Bifurcation from even period hyperbolic components} We will now review some facts about neutral parameters and boundaries of hyperbolic components of the Tricorn. The following proposition states that every neutral (in particular, parabolic) parameter lies on the boundary of a hyperbolic component of the same period (see \cite[Theorem~2.1]{MNS}).

\begin{proposition}[Neutral parameters on boundary]\label{ThmIndiffBdyHyp} 
If $f_{c_0}(z) = \overline{z}^2+c_0$ has an neutral periodic point of period $k$, then every neighborhood of $c_0$ contains parameters with attracting periodic points of period $k$, so the parameter $c_0$ is on the  boundary of a hyperbolic component of period $k$ of the Tricorn. 

Moreover, every neighborhood of $c_0$ contains parameters for which all period $k$ orbits are repelling. 
\end{proposition}

Using Theorem~\ref{hyp_unif}, one can define internal rays of hyperbolic components of $\mathcal{T}$. If $H$ is a hyperbolic component of even period, then all internal rays of $H$ land \cite[Lemma~2.19]{IM2}. If $H$ does not bifurcate from a hyperbolic component of odd period, then the landing point of the internal ray at angle $0$ is a parabolic parameter with an even-periodic parabolic cycle. This parameter is called the \emph{root} of $H$.

The bifurcation structure of even period hyperbolic components of the Tricorn is analogous to that in the Mandelbrot set. The following theorem was proved in \cite[Theorem~1.1]{MNS}.

\begin{theorem}[Bifurcations from even periods]\label{ThmEvenBif} 
If a quadratic anti-polynomial $f_c$ has a $2k$-periodic cycle with multiplier 
$e^{2\pi ip/q}$ with $\mathrm{gcd}(p,q)=1$, then $c$ sits on the boundary of a hyperbolic component of 
period $2kq$ of the Tricorn (and is the root thereof). 
\end{theorem}

\subsubsection{Bifurcation from odd period hyperbolic components}\label{bif_odd_per_hyp_subsubsec}

We now turn our attention to the odd period hyperbolic components of the Tricorn. One of the main features of anti-holomorphic parameter spaces is the existence of abundant parabolics. In particular, the boundaries of odd period hyperbolic components of the Tricorn consist only of parabolic parameters \cite[Lemma~2.5]{MNS}.

\begin{proposition}[Neutral dynamics of odd period]\label{LemOddIndiffDyn}  
The boundary of a hyperbolic component of odd period $k$ consists 
entirely of parameters having a parabolic orbit of exact period $k$. In suitable 
local conformal coordinates, the $2k$-th iterate of such a map has the form 
$z\mapsto z+z^{q+1}+\ldots$ with $q\in\{1,2\}$. 
\end{proposition}

This leads to the following classification of odd periodic parabolic points.

\begin{definition}\label{DefCusp}
A parameter $c$ will be called a {\em parabolic cusp} if it has a parabolic 
periodic point of odd period such that $q=2$ in the previous proposition. Otherwise, it is called a \emph{simple} parabolic parameter.
\end{definition}

In holomorphic dynamics, the local dynamics in attracting petals of parabolic periodic points is well-understood: there is a local coordinate $\psi^{\mathrm{att}}$ which conjugates the first-return dynamics to translation by $+1$ in a right half plane \cite[\S 10]{M1new}. Such a coordinate $\psi^{\mathrm{att}}$ is called a \emph{Fatou coordinate}. Thus, the quotient of the petal by the dynamics is isomorphic to a bi-infinite cylinder, called the \emph{{\'E}calle cylinder}. Note that Fatou coordinates are uniquely determined up to addition of a complex constant. 

In anti-holomorphic dynamics, the situation is at the same time restricted and richer. Since the real eigenvalues of an anti-holomorphic map at a neutral fixed point are $\pm 1$, neutral dynamics of odd period is always parabolic. In particular, for a neutral periodic point of odd period $k$, the $2k$-th iterate is holomorphic with multiplier $+1$. On the other hand, additional structure is given by the anti-holomorphic intermediate iterate. 

\begin{proposition}[Fatou coordinates] \cite[Lemma~2.3]{HS}\label{normalization of fatou}
Suppose $z_0$ is a parabolic periodic point of odd period $k$ of $f_c$ with only one petal (i.e.\ $c$ is not a cusp), and $U$ is a periodic Fatou component with $z_0 \in \partial U$. Then there is an open subset $V \subset U$ with $z_0 \in \partial V$, and $f_c^{\circ k} (V) \subset V$ so that for every $z \in U$, there is an $n \in \mathbb{N}$ with $f_c^{\circ nk}(z)\in 
V$. Moreover, there is a univalent map $\psi^{\mathrm{att}} \colon V \to \mathbb{C}$ with $\psi^{\mathrm{att}}(f_c^{\circ k}(z)) = \overline{\psi^{\mathrm{att}}(z)}+1/2$, and $\psi^{\mathrm{att}}(V)$ contains a right half plane. This map $\psi^{\mathrm{att}}$ is unique up to horizontal translation. 
\end{proposition}

\begin{remark}
Note that the above proposition applies more generally to anti-holomorphic neutral periodic points such that the attracting petal(s) has (have) odd period.
\end{remark}

The map $\psi^{\mathrm{att}}$ will be called an  \emph{anti-holomorphic Fatou coordinate} for the petal $V$. The anti-holomorphic iterate interchanges both ends of the {\'E}calle cylinder, so it must fix one horizontal line around this cylinder (the \emph{equator}). The change of coordinate has been so chosen that the equator is the projection of the real axis.  We will call the vertical Fatou coordinate the \emph{{\'E}calle height}. The {\'E}calle height vanishes precisely on the equator. Of course, the same can be done in the repelling petal as well. We will refer to the equator in the attracting (respectively repelling) petal as the attracting (respectively repelling) equator. The existence of this distinguished real line, or equivalently an intrinsic meaning to {\'E}calle height, is specific to anti-holomorphic maps. 

The {\'E}calle height of the critical value plays a special role in anti-holomorphic dynamics. The next theorem, which is proved in  \cite[Theorem~3.2]{MNS}, proves the existence of real-analytic arcs of simple parabolic parameters on the boundaries of odd period hyperbolic components of the Tricorn.

\begin{theorem}[Parabolic arcs]\label{parabolic arcs}
Let $\widetilde{c}$ be a simple parabolic parameter of odd period. Then $\widetilde{c}$ is on a parabolic arc in the  following sense: there exists a real-analytic arc $\mathscr{C}$ of simple parabolic parameters $c(h)$ (for $h\in\mathbb{R}$) with quasiconformally equivalent but conformally distinct dynamics of which $\widetilde{c}$ is an interior point, and the {\'E}calle height of the critical value of $f_{c(h)}$ is $h$. 
\end{theorem}

The real-analytic arc of simple parabolic parameters constructed in the previous theorem is called a \emph{parabolic arc}, and the real-analytic map $c:\R\to\mathscr{C}$ is called it \emph{critical {\'E}calle height parametrization}.

\begin{remark}[Queer arcs]\label{queer_arc_rem}
It is worth mentioning that most of the topological differences between the Mandelbrot set and the Tricorn arise from the existence of quasiconformally conjugate parabolic parameters on the boundary of the Tricorn (while no two distinct parameters on the boundary of the Mandelbrot set are quasiconformally conjugate; compare Theorem~\ref{parabolic arcs} and \cite[Chapter~I, Proposition~7]{DH2}). We do not know whether there are any non-trivial quasiconformal conjugacy classes on the boundary of the Tricorn other than odd period parabolic arcs. This question has connections with the ``no invariant line fields" conjecture; in particular, non-existence of invariant line fields would imply that the parabolic arcs are the only non-trivial quasiconformal conjugacy classes on the boundary of $\mathcal{T}$.
\end{remark}

Let $f : U \rightarrow \mathbb{C}$ be a holomorphic function on a connected open set $U\ \left(\subset \mathbb{C}\right)$, and $\widehat{z}\in U$ be an isolated fixed point of $f$. Then, the residue fixed point index of $f$ at $\widehat{z}$ is defined to be the complex number

\begin{align*}
\displaystyle \iota(f, \widehat{z}) &= \frac{1}{2\pi i} \oint \frac{dz}{z-f(z)}.
\end{align*}
where we integrate in a small loop in the positive direction around $\widehat{z}$. If the multiplier $\mu:=f'(\widehat{z})$ is not equal to $+1$, then a simple computation shows that $\iota(f, \widehat{z}) = 1/(1-\mu)$. If $z_0$ is a parabolic fixed point with multiplier $+1$, then in local holomorphic coordinates the map can be written as $f(w) = w + w^{q+1} + \alpha w^{2q+1} + \cdots$ (putting $\widehat{z}=0$). A simple calculation shows that $\alpha$ equals the parabolic fixed point index. It is easy to see that the fixed point index does not depend on the choice of complex coordinates, so it is a conformal invariant (compare \cite[\S 12]{M1new}). 

By the fixed point index of a periodic orbit of odd period of $f_c$, we will mean the holomorphic fixed point index of the second iterate $f_c^{\circ 2}$ at that periodic orbit.

Let $\mathscr{C}$ be a parabolic arc of odd period $k$ and $c:\R\to\mathscr{C}$ be its critical {\'E}calle height parametrization (compare Theorem~\ref{parabolic arcs}). For any $h$ in $\mathbb{R}$, let us denote the residue fixed point index of the unique parabolic cycle of $f_{c(h)}^{\circ 2}$ by $\ind_{\mathscr{C}}(f_{c(h)}^{\circ 2})$. This defines a function $$\ind_{\mathscr{C}}: \mathbb{R}\to\mathbb{C},\ h\mapsto \ind_{\mathscr{C}}(f_{c(h)}^{\circ 2}).$$

Every parabolic arc limits at a parabolic cusp (of the same period) on each end. Moreover, in the dynamical plane of a parabolic cusp, the double parabolic points are formed by the merger of a simple parabolic point with a repelling point. The sum of the fixed point indices at the simple parabolic point and the repelling point converges to the fixed point index (which is necessarily a finite number) of the double parabolic point of the cusp parameter. This observation leads to the following asymptotic behavior of the parabolic fixed point index towards the ends of parabolic arcs (see \cite[Proposition~3.7]{HS} for a proof).

\begin{proposition}[Fixed point index on parabolic arc]\label{index goes to infinity}
The function $\ind_{\mathscr{C}}$ is real-valued and real-analytic. Moreover, $$\lim_{h\to\pm\infty}\ind_{\mathscr{C}}(h)=+\infty.$$
\end{proposition}

Note that in the Mandelbrot set, bifurcation from one hyperbolic component to another occurs across a single point. The following theorem is one of the instances of the topological differences between the Mandelbrot set and the Tricorn \cite[Theorem~3.8, Corollary~3.9]{HS}, \cite[Lemma~2.12]{IM2} (see Figure~\ref{bif_pic}).  

\begin{figure}[ht!]
\captionsetup{width=0.96\linewidth}
\includegraphics[scale=0.34]{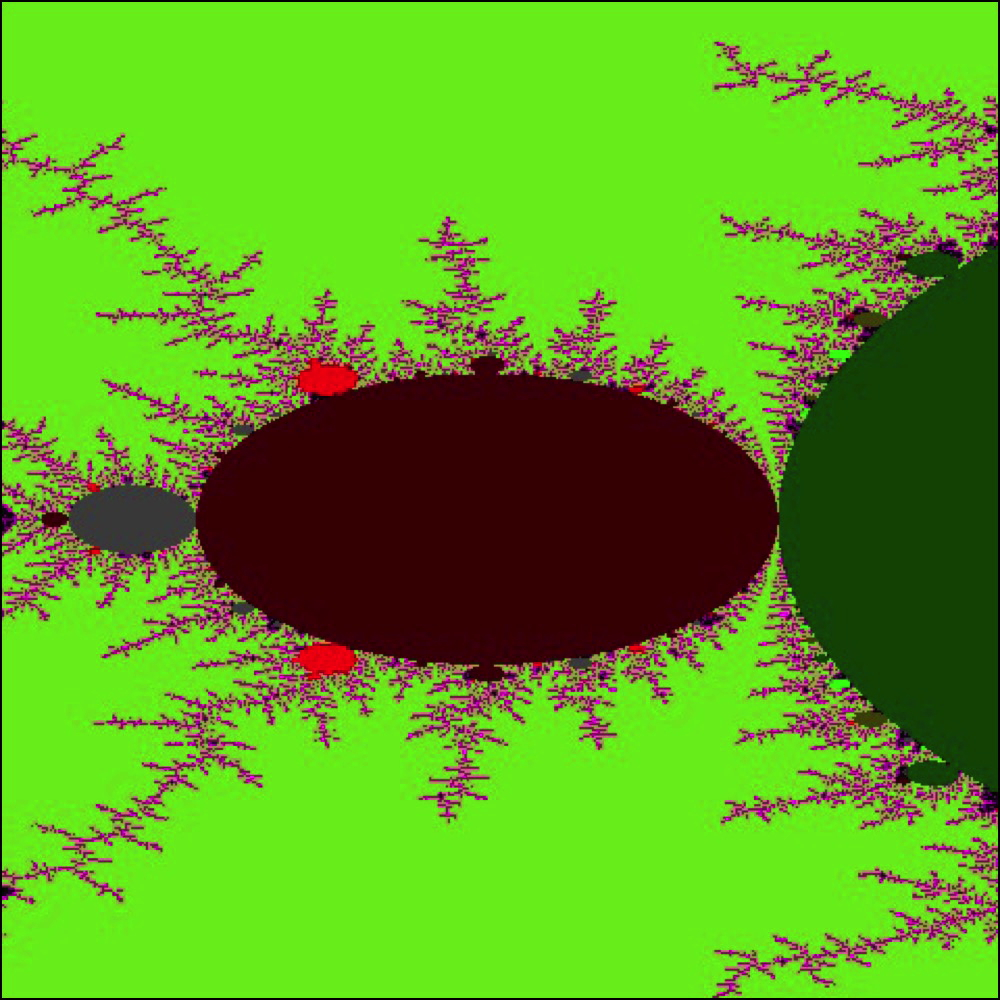} \includegraphics[scale=0.32]{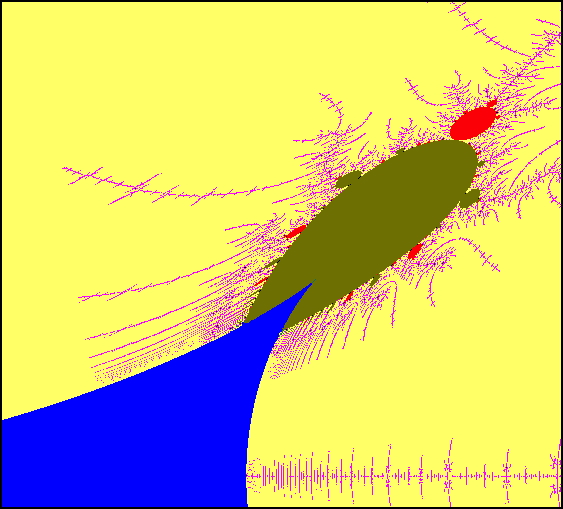} 
\caption{Left: A hyperbolic component of even period bifurcating from another hyperbolic component of even period across a point. Right: A hyperbolic component of even period bifurcating from a hyperbolic component of odd period across arcs.}
\label{bif_pic}
\end{figure}

\begin{theorem}[Bifurcations along arcs]\label{ThmBifArc}
Every parabolic arc of period $k$ intersects the boundary of a hyperbolic component of period $2k$ along an arc consisting of the set of parameters where the parabolic fixed point index is at least $1$. In particular, every parabolic arc has, at both ends, an interval of positive length at which bifurcation from a hyperbolic component of odd period $k$ to a hyperbolic component of period $2k$ occurs.
\end{theorem}

Here is a brief dynamical description of the nature of bifurcation across an odd period hyperbolic component. As a parameter approaches a parabolic arc from the interior of a hyperbolic component of odd period $k$, the corresponding attracting $k$-cycle tends to merge with a repelling $k$-cycle (lying on the boundary of the immediate basin of the attracting $k$-cycle). This produces the simple parabolic $k$-cycle for parameters lying on the parabolic arc. When such a parabolic parameter (lying on the parabolic arc) with fixed point index greater than $1$ (respectively, less than $1$) is slightly perturbed outside the $k$-periodic hyperbolic component, the parabolic $k$-cycle splits into an attracting (respectively, repelling) cycle of period $2k$.

Similarly, as a parameter approaches a parabolic cusp from the interior of a hyperbolic component of odd period $k$, the corresponding attracting $k$-cycle tends to merge with two distinct repelling $k$-cycles (both lying on the boundary of the immediate basin of the attracting $k$-cycle). This produces the double parabolic $k$-cycle for a parabolic cusp. When the parabolic cusp is slightly perturbed outside the $k$-periodic hyperbolic component, the double parabolic $k$-cycle splits into an attracting cycle of period $2k$ and a repelling cycle of period $k$.

To conclude this subsection, note that we have associated two important conformal invariants with odd period parabolic parameters; namely, the residue fixed point index of its parabolic cycle and the critical {\'E}calle height. There is no known explicit relation between these two invariants. However, some partial information is collected in the following proposition \cite[Corollary~2.21]{IM2}. 

We can assume without loss of generality that the set of parameters on $\mathscr{C}$ across which bifurcation from $H$ to a hyperbolic component $H'$ (of period $2k$) occurs is precisely $c[h_0,+\infty)$; i.e.\ $\mathscr{C}\cap \partial H'= c[h_0,+\infty)$.

\begin{proposition}\label{index_increasing_1}
The function
\begin{center}
$\begin{array}{rccc}
  \ind_{\mathscr{C}}: & [h_0,+\infty) & \to & [1,+\infty) \\
  & h & \mapsto & \ind_{\mathscr{C}}(f_{c(h)}^{\circ 2}).\\
   \end{array}$
  \end{center}
is strictly increasing, and hence a bijection. In particular, the bifurcating region $c[h_0,+\infty)$ can be parametrized by the fixed point index of the unique parabolic cycle.
\end{proposition} 
Following \cite{MNS}, we classify parabolic arcs into two types.

\begin{definition}\label{DefRootArc} 
We call a parabolic arc a \emph{root arc} if, in the dynamics of any 
parameter on this arc, the parabolic orbit disconnects the Julia set. 
Otherwise, we call it a \emph{co-root arc}.
\end{definition}

\subsubsection{Orbit portraits} 

\begin{definition}
Let $f_c$ be a parabolic map. The \emph{characteristic parabolic point} of $f_c$ is the unique parabolic point on the boundary of the characteristic Fatou component of $f_c$ (i.e., the Fatou component containing the critical value).
\end{definition}

Orbit portraits were introduced by Goldberg and Milnor as a combinatorial tool to describe the patterns of all periodic dynamical rays landing on a periodic cycle of a complex quadratic polynomial \cite{Go, GM1, M2a}. The usefulness of orbit portraits stems from the fact that these combinatorial objects contain substantial information on the connection between the dynamical and the parameter planes of the maps under consideration. Orbit portraits for quadratic anti-polynomials were studied in \cite{Sa}. 

\begin{definition}\label{def_orbit_portrait_anti_quad}
For a cycle $\mathcal{O} = \lbrace z_1 , z_2 ,\cdots$, $z_p\rbrace$ of $f_c$, let $\mathcal{A}_i$ be the set of angles of dynamical rays landing at $z_i$. The collection $\mathcal{P} = \lbrace \mathcal{A}_1 , \mathcal{A}_2, \cdots, \mathcal{A}_p \rbrace$ is called the \emph{orbit portrait} associated with the orbit $\mathcal{O}$.
\end{definition}

\begin{theorem}\cite[Theorem~2.6]{Sa}\label{complete_anti-holomorphic}
Let $f_c$ be a quadratic anti-polynomial, and $\mathcal{O} = \lbrace z_1 , z_2 ,\cdots$, $z_p\rbrace$ be a periodic orbit such that at least one rational dynamical ray lands at some $z_j$, $j\in\{1,\cdots,p\}$. Then the associated orbit portrait (which we assume to be non-trivial; i.e., $\vert\mathcal{A}_i\vert\geq 2$) $\mathcal{P} = \lbrace \mathcal{A}_1 , \mathcal{A}_2, \cdots, \mathcal{A}_p \rbrace$ satisfies the following properties:
\begin{enumerate}
\item Each $\mathcal{A}_j$, $j\in\{1,\cdots,p\}$, is a finite non-empty subset of $\mathbb{Q}/\mathbb{Z}$.

\item For each $j\in\Z/p\Z$, the map $m_{-2}$ maps $\mathcal{A}_j$ bijectively onto $\mathcal{A}_{j+1}$, and reverses their cyclic order.

\item For each $i\neq j$, the sets $\mathcal{A}_i$ and $\mathcal{A}_j$ are unlinked.

\item Each $\theta \in \mathcal{A}_j$, $j\in\{1,\cdots,p\}$, is periodic under $m_{-2}$, and there are four possibilities for their periods: 
\begin{enumerate}

\item If $p$ is even, then all angles in $\mathcal{P}$ have the same period $rp$ for some $r\geq 1$.

\item If $p$ is odd, then one of the following three possibilities must be realized:
\begin{enumerate}
\item $\vert \mathcal{A}_j\vert = 2$, and both angles have period $p$.

\item $\vert \mathcal{A}_j\vert = 2$, and both angles have period $2p$.

\item $\vert \mathcal{A}_j\vert = 3$; one angle has period $p$, and the other two angles have period $2p$.
\end{enumerate}
\end{enumerate}
\end{enumerate}
\end{theorem}

\begin{definition}\label{def_orbit_portrait}
A finite collection $\mathcal{P} = \{\mathcal{A}_1,$ $\mathcal{A}_2,$ $\cdots,$ $\mathcal{A}_p \}$ of non-empty finite subsets of $\mathbb{Q}/\mathbb{Z}$ satisfying the conditions of Theorem~\ref{complete_anti-holomorphic} is called a \emph{formal orbit portrait} under the anti-doubling map $m_{-2}$ (in short, an $m_{-2}$-FOP).
\end{definition}

By \cite[Theorem~3.1]{Sa}, every formal orbit portrait is realized by some $f_c$.

\begin{theorem}[Realization of orbit portraits outside $\mathcal{T}$]\label{realization_orbit_portrait_outside}
Let $\mathcal{P} = \{\mathcal{A}_1,$ $\mathcal{A}_2,$ $\cdots,$ $\mathcal{A}_p \}$ be a formal orbit portrait under the anti-doubling map $m_{-2}$. Then there exists some $c\in\mathbb{C}\setminus\mathcal{T}$, such that $f_c$ has a repelling periodic orbit with associated orbit portrait $\mathcal{P}$.
\end{theorem}

Among all the complementary arcs of the various $\mathcal{A}_j$, there is a unique one of minimum length. This shortest arc $\mathcal{I}_{\mathcal{P}}$ is called the \emph{characteristic arc} of the orbit portrait, and the two angles $\{t^{-},t^{+}\}$ at the ends of this arc are called its \emph{characteristic angles}.

The following theorem will play an important role later in the paper.

\begin{theorem}[Realization of orbit portraits at parabolic parameters]\label{realization_orbit_portrait_parabolic}
Let $\mathcal{P}= \{\mathcal{A}_1,$ $\mathcal{A}_2,$ $\cdots,$ $\mathcal{A}_p \}$ be a formal orbit portrait under the anti-doubling map $m_{-2}$ with characteristic angles $t^-$ and $t^+$.

1) Suppose that $p$ is odd, and $t^\pm$ have period $2p$. Then the parameter rays $\mathcal{R}_{t^-}$ and $\mathcal{R}_{t^+}$ accumulate on a common root parabolic arc $\mathscr{C}$ such that for every parameter $c\in\mathscr{C}$, $f_c$ has a parabolic cycle of period $p$ and the orbit portrait associated with the parabolic cycle of $f_c$ is $\mathcal{P}$. 

2) Suppose that $p$ is even. Then the parameter rays $\mathcal{R}_{t^-}$ and $\mathcal{R}_{t^+}$ land at a common parabolic parameter $c$ (whose parabolic cycle has period $p$) such that the orbit portrait associated with the parabolic cycle of $f_c$ is $\mathcal{P}$. 
\end{theorem}

\begin{proof}
1) By \cite[Lemma~2.9]{Sa}, we have that $\mathcal{A}_1=\{t^-,t^+\}$, and hence $t^+=(-2)^pt^-$. It now follows from \cite[Lemma~4.1]{IM1} that the parameter rays $\mathcal{R}_{t^-}$ and $\mathcal{R}_{t^+}$ accumulate on a common root parabolic arc $\mathscr{C}$. Hence, in the dynamical plane of every $c\in\mathscr{C}$, the dynamical rays $R_c(t^-)$ and $R_c(t^+)$ land at the characteristic parabolic point. Finally, by \cite[Lemma~4.8]{MNS}, these are the only dynamical rays landing at the characteristic parabolic point of $f_c$ (for $c\in\mathscr{C}$). This proves that for every parameter $c\in\mathscr{C}$, the map $f_c$ has a parabolic cycle with associated orbit portrait $\mathcal{P}$. 

2) Arguing as in \cite[Lemma~4.1]{IM1}, we can conclude that $\mathcal{R}_{t^-}$ and $\mathcal{R}_{t^+}$ either accumulate on a common root arc $\mathscr{C}$ or land at a common parabolic parameter $c$ of even parabolic period.

We will first show that the former possibility cannot occur. For definiteness, we assume that $\{t^-,t^+\}\subset\mathcal{A}_1$. Let us suppose that $\mathcal{R}_{t^-}$ and $\mathcal{R}_{t^+}$ accumulate on a common root arc $\mathscr{C}$ of period $k$, and fix some $c'\in\mathscr{C}$. Then, the dynamical rays $R_{c'}(t^+)$ and $R_{c'}(t^-)$ land at the characteristic parabolic point of $f_{c'}$, which has odd period $k$. It follows that $t^+=(-2)^kt^-$, and both these angles $t^\pm$ have period $2k$. It is now easy to see that $p$ must divide $k$ (otherwise, $t^+$ would be contained in some $\mathcal{A}_{i}$ different from $\mathcal{A}_1$). But this is impossible as $p$ is even and $k$ is odd.

Therefore, the parameter rays $\mathcal{R}_{t^-}$ and $\mathcal{R}_{t^+}$ must land at a common parabolic parameter $c$ of even parabolic period. Then, the corresponding dynamical rays $R_c(t^+)$ and $R_c(t^-)$ land at the characteristic parabolic point of $f_c$, which has even period. We denote the actual orbit portrait associated with the parabolic cycle of $f_c$ by $\mathcal{P}'$. Since both the orbit portraits $\mathcal{P}$ and $\mathcal{P}'$ have even orbit period, it follows by \cite[Lemma~3.3]{Sa} that each of them is either primitive or satellite (compare \cite[Lemma~2.7]{M2a}). The proof of \cite[Lemma~2.8]{M2a} now applies verbatim to show that $\mathcal{P}=\mathcal{P}'$. This completes the proof.
\end{proof}

Let $H$ be a hyperbolic component of even period $k$ such that $H$ does not bifurcate from an odd period hyperbolic component. Let $\mathcal{A}_1$ be the set of angles of the dynamical rays landing at the dynamical root of $f_{c}$ (where $c\in H$ or $c$ is the root point of $H$). Then, the first return map of the dynamical root either fixes every angle in $\mathcal{A}_1$ and $\vert\mathcal{A}_1\vert=2$, or permutes the angles in $\mathcal{A}_1$ transitively. Moreover, the characteristic angles $t^-$ and $t^+$ of the orbit portrait $\mathcal{P}$ generated by $\mathcal{A}_1$ are precisely the two adjacent angles in $\mathcal{A}_1$ (with respect to circular order) that separate $0$ from $c$, and bound a sector of angular width less that $\frac12$. The root point of $H$ is the landing point of exactly two parameter rays at angles $t^-$ and $t^+$.

Let us now look at the connection between orbit portraits associated with parabolic parameters on the boundary of an odd period hyperbolic component $H$ and the angles of parameter rays accumulating on $\partial H$. Suppose that the period of $H$ is $k$ and its center is $c_0$. The first return map of the closure of the characteristic Fatou component of $c_0$ fixes exactly three points on its boundary. Only one of these fixed points disconnects the Julia set, and is the landing point of two distinct dynamical rays at $2k$-periodic angles. Let the set of the angles of these two rays be $S' = \{\alpha_1,\alpha_2 \}$. Then, $\alpha_2=(-2)^k\alpha_1$, and $S'$ is the set of characteristic angles of the corresponding orbit portrait. Each of the remaining two fixed points is the landing point of precisely one dynamical ray at a $k$-periodic angle; let the collection of the angles of these rays be $S = \{ \theta_1, \theta_2\}$. We can, possibly after renumbering, assume that $0 < \alpha_1 < \theta_1 < \theta_2 < \alpha_2$ and $\alpha_2 - \alpha_1 < \frac{1}{2}$. Then, these angles satisfy the following relation (see \cite[Lemma~3.5]{Sa}) 

\begin{equation}
(2^{k}+1)(\theta_1-\alpha_1)=(\alpha_2-\alpha_1)=(2^k+1)(\alpha_2-\theta_2).
\label{char_rays_relation_1}
\end{equation}

\begin{figure}[ht!]
\captionsetup{width=0.96\linewidth}
\includegraphics[scale=0.275]{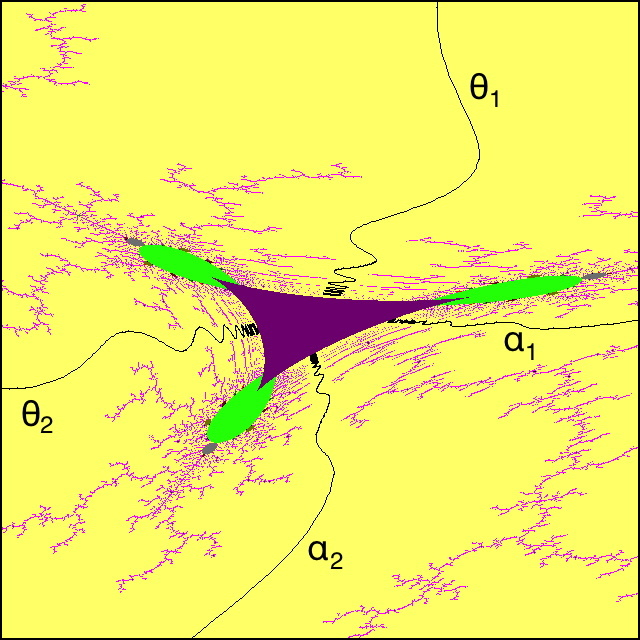} \includegraphics[scale=0.275]{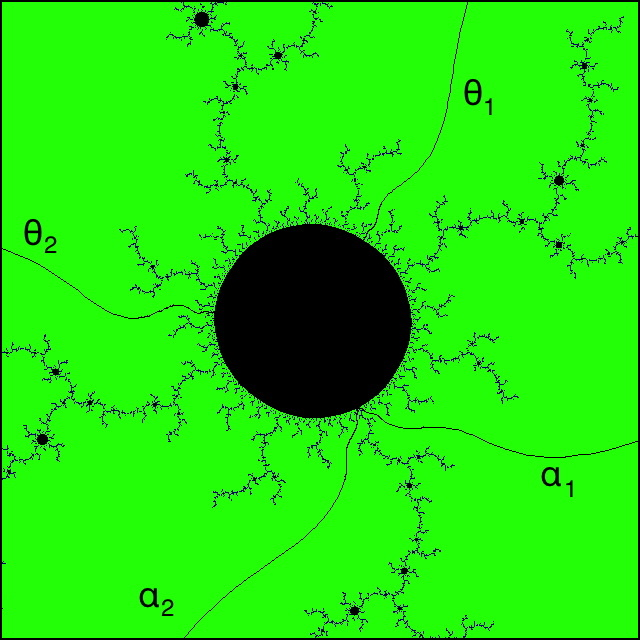} 
\caption{Left: Parameter rays accumulating on the boundary of a hyperbolic component of period $5$ of the Tricorn. Right: The corresponding dynamical rays landing on the boundary of the characteristic Fatou component in the dynamical plane of the center of the same hyperbolic component.}
\label{all_rays}
\end{figure}

\subsubsection{Boundaries of odd period hyperbolic components} By \cite[Theorem~1.2]{MNS}, $\partial H$ is a simple closed curve consisting of three parabolic arcs, and the same number of cusp points such that every arc has two cusp points at its ends. Exactly one of these three parabolic arcs (say, $\mathscr{C}_3$) is a root arc, and the parameter rays at angles $\alpha_1$ and $\alpha_2$ accumulate on this arc. The characteristic parabolic point in the dynamical plane of any parameter on this root arc is the landing point of precisely two dynamical rays at angles $\alpha_1$ and $\alpha_2$. The rest of the two parabolic arcs (say, $\mathscr{C}_1$ and $\mathscr{C}_2$) on $\partial H$ are co-root arcs. Each of these co-root arcs contains the accumulation set of exactly one parameter ray at an angle $\theta_i$, and the characteristic parabolic point in the dynamical plane of any parameter on this co-root arc is the landing point of precisely one dynamical ray at angle $\theta_i$ (compare Figure~\ref{all_rays}). 

At the parabolic cusp on $\partial H$ where $\mathscr{C}_1$ and $\mathscr{C}_2$ meet, the characteristic parabolic point is the landing point of exactly two dynamical rays at angles $\theta_1$ and $\theta_2$. The same is true at the center of the hyperbolic component of period $2k$ that bifurcates from $H$ across this parabolic cusp. Moreover, these angles are the characteristic angles of the corresponding orbit portrait.

On the other hand, at the parabolic cusp where $\mathscr{C}_1$ and $\mathscr{C}_3$ (respectively, $\mathscr{C}_2$ and $\mathscr{C}_3$) meet, the characteristic parabolic point is the landing point of precisely three dynamical rays at angles $\alpha_1$, $\alpha_2$ and $\theta_1$ (respectively, $\alpha_1$, $\alpha_2$ and $\theta_2$). As before, the same is true at the center of the hyperbolic component of period $2k$ that bifurcates from $H$ across this parabolic cusp. The characteristic angles of the corresponding orbit portrait are $\alpha_1$ and $\theta_1$ (respectively, $\theta_2$ and $\alpha_2$).

\begin{theorem}[Boundaries of odd period hyperbolic components]\label{Exactly d+1}
The boundary of every hyperbolic component of odd period of $\mathcal{T}$ is a topological triangle having parabolic cusps as vertices and parabolic arcs as sides.
\end{theorem}

\subsubsection{Misiurewicz parameters}\label{sec_misi_Tricorn}

A \emph{Misiurewicz} parameter of the Tricorn is a parameter $c$ such that the critical point $0$ is strictly pre-periodic. For a Misiurewicz parameter, the critical point eventually maps on a repelling cycle. By classification of Fatou components, the filled Julia set of such a map has empty interior. Moreover, the Julia set of a Misiurewicz parameter is locally connected \cite[Expos{\'e} III, Proposition~4, Theorem~1]{orsay}, and has measure zero \cite[Expos{\'e} V, Theorem 3]{orsay}. 

\begin{theorem}[Parameter rays landing at Misiurewicz parameters]\label{Tricorn_para_misi_ray} 
Every parameter ray of the Tricorn at a pre-periodic angle (under $m_{-2}$) lands at a Misiurewicz parameter such that in its dynamical plane, the corresponding dynamical ray lands at the critical value. Conversely, every Misiurewicz parameter $c$ of the Tricorn is the landing point of a finite (non-zero) number of parameter rays at pre-periodic angles (under $m_{-2}$) such that the angles of these parameter rays are exactly the external angles of the dynamical rays that land at the critical value $c$ in the dynamical plane of $f_c$. 
\end{theorem}
\begin{proof} 
A proof of the corresponding results for the Mandelbrot set and the necessary modifications required to adapt the proof in the anti-holomorphic setting can be found in \cite[Theorem~1.1 (pre-periodic case)]{S1a} and the remark thereafter. 

Alternatively, see \cite[Theorem~7.3]{GV} for the first part of the result (also compare \cite[Theorem~37.35]{L6}). For the converse, let $\mathcal{A}$ be the set of angles of dynamical rays landing at the critical value $c$ of a Misiurewicz polynomial $f_c$. Pick $\theta\in\mathcal{A}$. If $c'$ is the landing point of $\mathcal{R}_\theta$, then the dynamical ray $R_{c'}(\theta)$ lands at the critical value $c'$ of $f_{c'}$. But then, the holomorphic polynomials $f_{c}^{\circ 2}$ and $f_{c'}^{\circ 2}$ have a common critical portrait in the sense of \cite{Po}. It now follows by \cite[Theorem~1.1]{Po} that $f_{c}^{\circ 2}=f_{c'}^{\circ 2}$; i.e., $c=c'$. Therefore, for each $\theta\in\mathcal{A}$, the parameter ray $\mathcal{R}_\theta$ lands at the Misiurewicz parameter $c$. By the first part, no other parameter ray at a pre-periodic angle can land at $c$.
\end{proof}

Let $c_0$ be a Misiurewicz parameter, and $\mathcal{A}'$ be the set of angles of the dynamical rays of $f_{c_0}$ landing at the critical point $0$. The set of angles of the dynamical rays that land at the critical value $c_0$ is then given by $\mathcal{A}:=m_{-2}(\mathcal{A}')$. Moreover, $m_{-2}$ is two-to-one from $\mathcal{A}'$ onto $\mathcal{A}$. All other equivalence classes of $\lambda(f_{c_0})$ are mapped bijectively onto its image class by $m_{-2}$. Note also that all angles in $\mathcal{A}'$ are strictly pre-periodic. It is easy to see that the existence of a unique equivalence class (of $\lambda(f_{c_0})$) that maps two-to-one onto its image class under $m_{-2}$  characterizes the pre-periodic lamination of Misiurewicz maps. A formal rational lamination satisfying this condition is said to be of Misiurewicz type.

The next theorem shows that every formal rational lamination of Misiurewicz type is realized as the rational lamination of a unique Misiurewicz map $f_c$.

\begin{theorem}[Realization of rational laminations]\label{rat_lam_realized}
Every formal rational lamination of Misiurewicz type is realized as the rational lamination of a unique Misiurewicz map $f_c$ in $\mathcal{T}$.
\end{theorem}
\begin{proof}
Let $\lambda$ be a formal rational lamination of Misiurewicz type. As $\lambda$ is of Misiurewicz type, there exists a unique $\lambda$-class $\mathcal{A}'$ (consisting of strictly pre-periodic angles under $m_{-2}$) such that $m_{-2}$ maps $\mathcal{A}'$ two-to-one onto $\mathcal{A}:=m_{-2}(\mathcal{A}')$.

It is easy to see that $\lambda$ satisfies the properties of \cite[Theorem~1.1]{kiwi} with $d=4$, and hence, there exists a degree $4$ holomorphic polynomial $P$ with associated rational lamination $\lambda$. Moreover, there are exactly three $\lambda$-classes on which $m_4$ (i.e., multiplication by $4$ modulo one) acts in a two-to-one fashion. It follows that $P$ has three distinct simple critical points $\{\alpha_1,\alpha_2,\alpha_3\}$ such that $P(\alpha_1)=P(\alpha_2)$. By \cite[Lemma~3.1]{Sa2}, $P$ is a biquadratic polynomial; i.e., $P(z)=(z^2+a)^2+b$, for some $a,b\in\C$, $a\neq0$. Moreover, the critical points of $P$ are strictly pre-periodic. 

Let $\theta\in\mathcal{A}$, and $c\in\mathcal{T}$ be the landing point of $\mathcal{R}_\theta$. Then, the dynamical ray $R_c(\theta)$ lands at the critical value $c$ of $f_c$. It is now easy to see that the PCF holomorphic polynomials $P$ and $f_c^{\circ 2}$ have a common critical portrait in the sense of \cite{Po}. Once again, \cite[Theorem~1.1]{Po} implies that $P=f_c^{\circ 2}$; i.e., $a=\overline{c}$ and $b=c$. Therefore, $\lambda(f_c)$ is equal to the rational lamination of $P$; i.e., $\lambda(f_c)=\lambda$. The uniqueness statement follows by \cite[Theorem~1.1]{Po}.
\end{proof}

\subsubsection{Global topological structure}\label{sec_global_top}

There are various topological differences between the Mandelbrot set and the Tricorn. Here, we collect some of these important differences.

Note that the Mandelbrot set is conjectured to be locally connected. This is known in many cases; e.g. at most finitely renormalizable parameters with no non-repelling cycles \cite{H1}, parameters in embedded baby Mandelbrot sets satisfying the secondary limbs conditions \cite{L3}, etc. The following theorem, which is in stark contrast to the corresponding situation for $\mathcal{M}$, was first proved in \cite[Theorem~6.2]{HS} and improved in \cite[Theorem~1.2]{IM2} (see Figure~\ref{tricorn_pic}).

\begin{theorem}\label{Tricorn_non_lc}
The Tricorn is not path connected. Moreover, no non-real hyperbolic component of odd period can be connected to the principal hyperbolic component by a path.
\end{theorem}

It should be mentioned that unlike the Mandelbrot set, not every (external) parameter ray of the Tricorn lands at a single point \cite{IM1} (see Figure~\ref{all_rays}). 

\begin{theorem}[Non-landing parameter rays]\label{most rays wiggle}
The accumulation set of every parameter ray accumulating on the boundary of a hyperbolic component of \emph{odd} period (except period one) of $\mathcal{T}$ contains an arc of positive length. The fixed rays at angles $0$, $1/3$ and $2/3$ land on the boundary of the principal hyperbolic component.
\end{theorem}

The next result, which was proved in \cite[\S 5]{IM1}, is also in contrast with the corresponding situation for $\mathcal{M}$. 

\begin{theorem}[Non-density of Misiurewicz parameters]\label{thm_misi_not_dense}
Misiurewicz parameters are not dense on the boundary of $\mathcal{T}$. Indeed, there are points on the boundaries of the period $1$ and period $3$ hyperbolic components of $\mathcal{T}$ that cannot be approximated by Misiurewicz parameters. 
\end{theorem}

Non-density of Misiurewicz parameters on $\partial\mathcal{T}$ can be clearly seen in Figure~\ref{limb_pic}(left); for instance, the landing point of the parameter ray at angle $0$ is not a limit point of Misiurewicz parameters.

Another salient difference between $\mathcal{M}$ and $\mathcal{T}$ is that the straightening map for ``baby Tricorns'' is always discontinuous \cite[Theorem~1.1]{IM2}. The discontinuity phenomena is related to non-local connectivity and existence of quasiconformally conjugate parameters on the boundary of the Tricorn.

\begin{theorem}[Discontinuity of straightening in the Tricorn]\label{Straightening_discontinuity_Tricorn}
Let $c_0$ be the center of a hyperbolic component $H$ of odd period (other than $1$) of $\mathcal{T}$, and $\mathcal{R}(c_0)$ be the corresponding $c_0$-renormalization locus (i.e., the baby Tricorn based at $H$). Then the straightening map $\chi_{c_0} : \mathcal{R}(c_0) \rightarrow \mathcal{T}$ is discontinuous at infinitely many parameters.
\end{theorem}

\subsubsection{The real basilica limb}\label{sec_basilica_limb}

Let us now define the real basilica limb of the Tricorn. Of course, one can give a more general definition of limbs, which can be found in \cite[\S 6]{MNS}. Let us denote the hyperbolic component of period one of $\mathcal{T}$ by $H_0$.

\begin{definition}\label{def_basilica_limb}
The connected component of $\left(\mathcal{T}\setminus\overline{H_0}\right)\cup\{-\frac{3}{4}\}$ intersecting the real line is called the \emph{real basilica limb} of the Tricorn, and is denoted by $\mathcal{L}$.
\end{definition}

\begin{figure}[ht!]
\captionsetup{width=0.96\linewidth}
\begin{tikzpicture}
\node[anchor=south west,inner sep=0] at (0,0) {\includegraphics[width=0.46\textwidth]{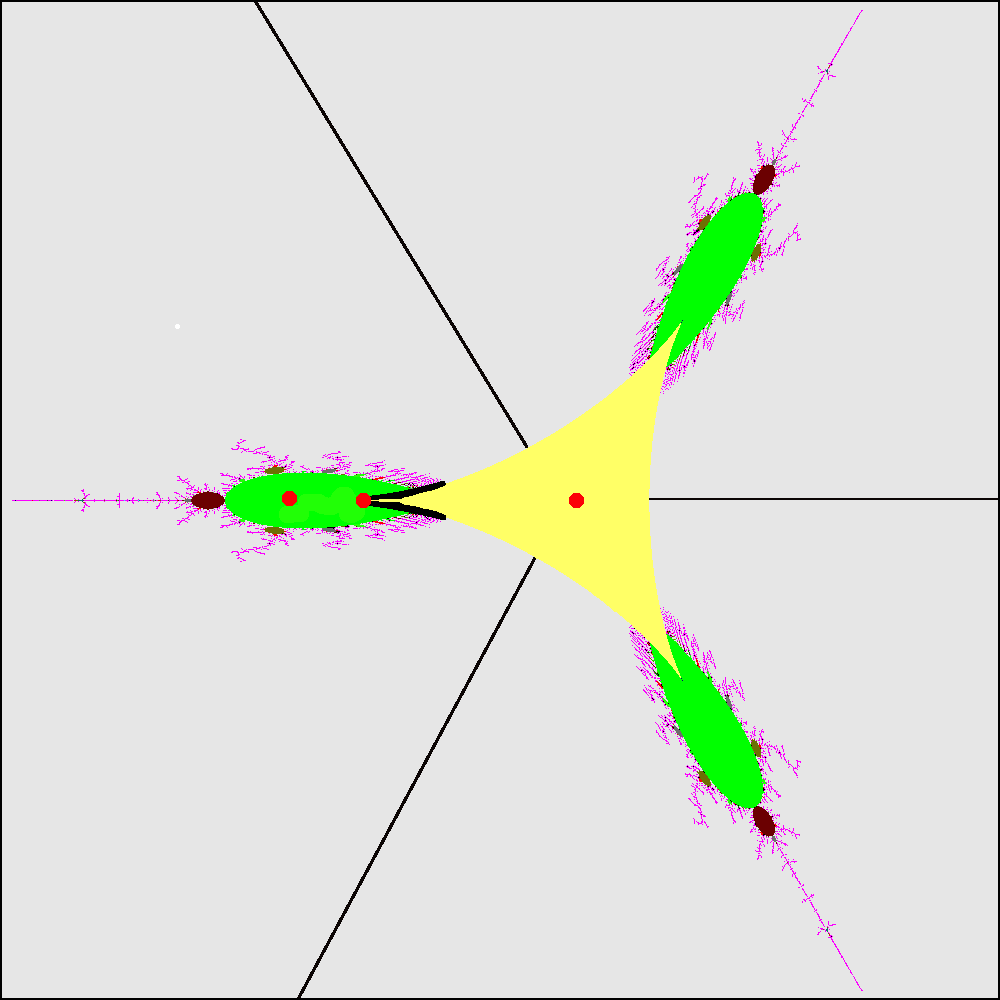}};
\node[anchor=south west,inner sep=0] at (6,0) {\includegraphics[width=0.48\textwidth]{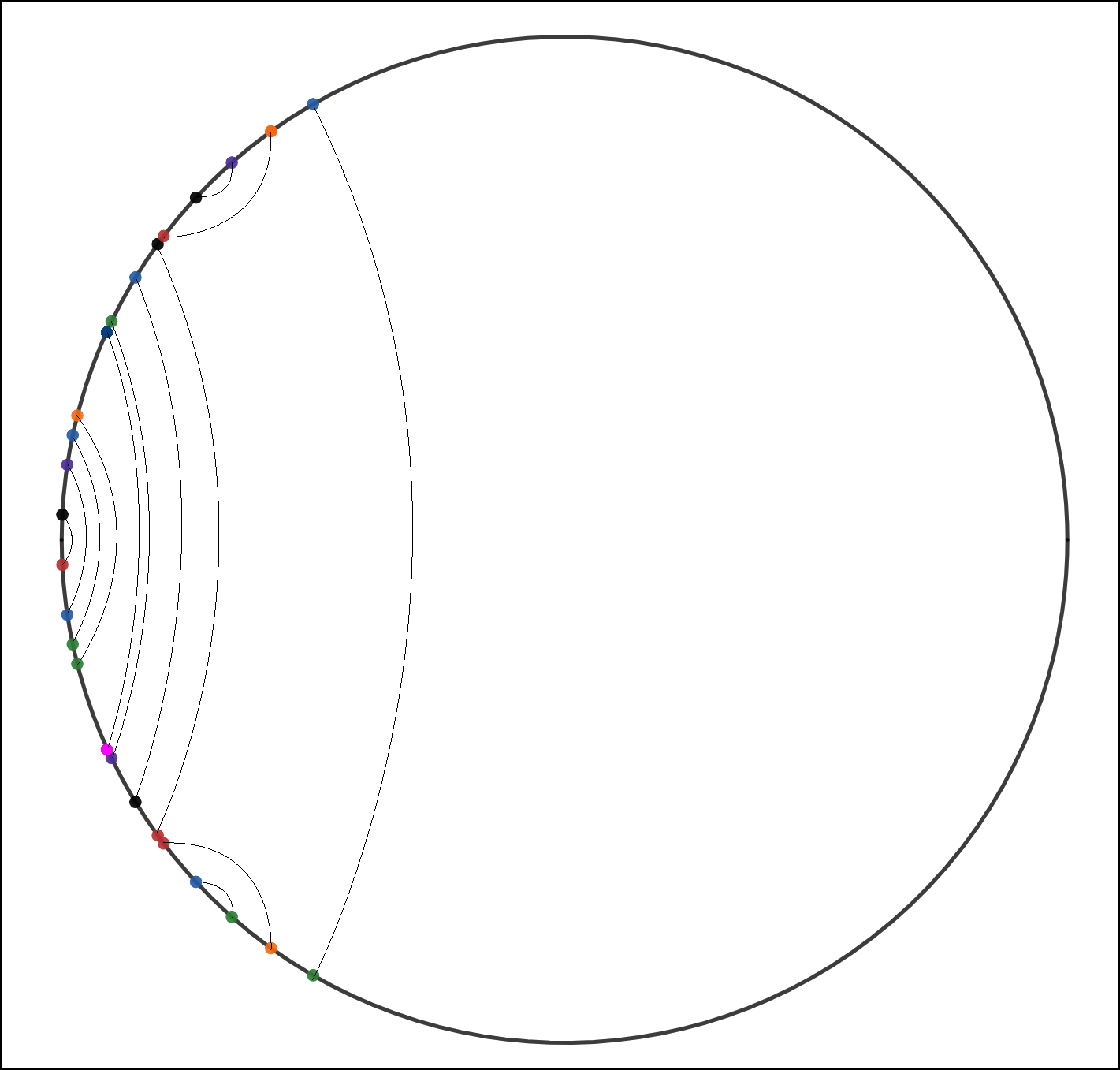}};
\node at (3.36,2.7) {\begin{small}$0$\end{small}};
\node at (2.5,0.4) {\begin{small}$\mathcal{R}_{\frac23}$\end{small}};
\node at (2.5,5) {\begin{small}$\mathcal{R}_{\frac13}$\end{small}};
\node at (5,2.7) {\begin{small}$\mathcal{R}_{0}$\end{small}};
\node at (3.35,3.66) {\begin{small}$\mathscr{C}_1$\end{small}};
\node at (3.35,2.16) {\begin{small}$\mathscr{C}_2$\end{small}};
\draw [->,line width=0.5pt] (0.6,1.84) to (1.64,2.88);
\node at (0.5,1.66) {\begin{small}$-1$\end{small}};
\draw [->,line width=0.5pt] (1.32,1.88) to (2.06,2.86);
\node at (1.2,1.64) {\begin{small}$-\frac34$\end{small}};
\end{tikzpicture}
\caption{Left: The real basilica limb of the Tricorn is shown. The sub-arcs of the parabolic arcs $\mathscr{C}_1, \mathscr{C}_2$ that are contained in $\overline{\mathcal{L}}$ are highlighted in black. These two sub-arcs constitute $\overline{\mathcal{L}}\setminus\mathcal{L}$. Right: A few identifications in the construction of the abstract basilica limb $\widetilde{\mathcal{L}}$.}
\label{limb_pic}
\end{figure}

$\mathcal{L}$ is precisely the set of parameters $c$ in $\mathcal{T}$ such that in the dynamical plane of $f_c$, the rays $R_c(1/3)$ and $R_c(2/3)$ land at a common point (i.e., $1/3\sim2/3$ in $\lambda(f_c)$ for all $c\in\mathcal{L}$). 

Recall that the parameter rays of the Tricorn at angles $1/3$ and $2/3$ are denoted by $\mathcal{R}_{1/3}$ and $\mathcal{R}_{2/3}$. Suppose that these two parameter rays land at the parabolic arcs $\mathscr{C}_1$ and $\mathscr{C}_2$ respectively (compare \cite[Lemma~3.1]{IM1}). The closure of these two parabolic arcs intersect at $-\frac{3}{4}$.

The real basilica limb $\mathcal{L}$ is not compact. There are limit points of $\mathcal{L}$ on the parabolic arcs $\mathscr{C}_1$ and $\mathscr{C}_2$ (of period one), but these points do not lie in $\mathcal{L}$. Moreover, $\overline{\mathcal{L}}\setminus\mathcal{L}$ is precisely the union of two sub-arcs of $\mathscr{C}_1$ and $\mathscr{C}_2$ (compare Figure~\ref{limb_pic}(left)).

\subsubsection{The abstract basilica limb}\label{abs_basi_limb_subsec} We conclude this subsection with the construction of a locally connected model of the real basilica limb of the Tricorn. Let $\gamma$ be the hyperbolic geodesic of $\D$ connecting $1/3$ and $2/3$ on $\partial\D\cong\R/\Z$. We denote the connected component of $\D\setminus\gamma$ not containing $0$ by $\D_2$. The locally connected model of $\mathcal{L}$ will be defined as the quotient of $\overline{\D}_2$ under a suitable equivalence relation. 

We will first construct an equivalence relation on $\partial\D\cap\partial\D_2$. We identify the angles of all rational parameter rays of $\mathcal{T}$ that land at a common (parabolic or Misiurewicz) parameter or accumulate on a common root parabolic arc of $\mathcal{L}$ (see Figure~\ref{limb_pic}(right)). We also identify $1/3$ and $2/3$. This defines an equivalence relation on $\Q/\Z\cap\partial\D_2$. We then consider the smallest closed equivalence relation on $\partial\D\cap\partial\D_2$ generated by the above relation. Take the hyperbolic convex hull of each of these equivalence classes in $\overline{\D}$. This yields a geodesic lamination of $\D_2$ (by hyperbolic geodesics of $\D$). Finally, consider the quotient of $\overline{\D}_2$ by collapsing each hyperbolic convex hull obtained above to a single point. The resulting continuum is called the \emph{abstract basilica limb} $\widetilde{\mathcal{L}}$ (see \cite[\S 9.4.2]{L6} for a general discussion on the construction of pinched disk models of planar continua). 

We will now give a description of $\widetilde{\mathcal{L}}$ as a quotient space of $\mathcal{L}$.

\begin{definition}
i) Two parameters $c$ and $c'$ in $\mathcal{L}$ are called \emph{combinatorially equivalent} if $f_{c}$ and $f_{c'}$ have the same rational lamination.

ii) The \emph{combinatorial class} $\mathrm{Comb}(c)$ of $c\in\mathcal{L}$ is defined as the set of all parameters in $\mathcal{L}$ that are combinatorially equivalent to $c$.

iii) A combinatorial class $\mathrm{Comb}(c)$ is called \emph{periodically repelling} if for every $c'\in\mathrm{Comb}(c)$, each periodic orbit (excluding $\infty$) of the anti-polynomial $f_{c'}$ is repelling.
\end{definition}

The following proposition gives a complete description of the non-repelling combinatorial classes of $\mathcal{L}$. 

\begin{proposition}[Classification of combinatorial classes]\label{comb_class_Tricorn}
Every combinatorial class $\mathrm{Comb}(c)$ of $\mathcal{L}$ is of one of the following four types.
\begin{enumerate}
\item $\mathrm{Comb}(c)$ consists of an even period hyperbolic component (that does not bifurcate from an odd period hyperbolic component), its root point, and the irrationally neutral parameters on its boundary,

\item $\mathrm{Comb}(c)$ consists of an even period hyperbolic component (that bifurcates from an odd period hyperbolic component), the unique parabolic cusp and the irrationally neutral parameters on its boundary,

\item $\mathrm{Comb}(c)$ consists of an odd period hyperbolic component and the parabolic arcs on its boundary,

\item $\mathrm{Comb}(c)$ is periodically repelling.
\end{enumerate}
\end{proposition}

\begin{remark}\label{repelling_comb}
It is conjectured that every periodically repelling combinatorial class of $\mathcal{L}$ is a point. This is known in many cases; e.g. for all Misiurewicz parameters \cite{S4}, at most finitely renormalizable parameters with no non-repelling cycles \cite{H1}, parameters in embedded baby Mandelbrot sets satisfying the secondary limbs conditions \cite{L3}, etc. 
\end{remark}

The \emph{abstract basilica limb} $\widetilde{\mathcal{L}}$ is obtained from $\mathcal{L}$ by

\begin{enumerate}
\item identifying all points in each periodically repelling combinatorial class of $\mathcal{L}$,

\item identifying all points in the non-bifurcating sub-arc of every parabolic arc of $\mathcal{L}$, and

\item identifying all points in $(\overline{\mathcal{L}}\setminus\mathcal{L})\cup\lbrace-\frac{3}{4}\rbrace$.
\end{enumerate}

We refer the readers to \cite{NS,HS,MNS,IM1,IM2} for a more comprehensive account of the combinatorics and topology of the Tricorn.

\section{Ideal triangle group}\label{ideal_triangle}

We now briefly review the basic properties of the ideal triangle group and the associated reflection map $\rho$, which play an important role in our study of the ``escaping'' dynamics of Schwarz reflection maps.

Consider the open unit disk $\D$ in the complex plane, and the hyperbolic geodesics $\widetilde{C}_1$, $\widetilde{C}_2$, and $\widetilde{C}_3$ (in $\D$) connecting the third roots of unity. These geodesics bound a closed ideal triangle (in the topology of $\D$), which we call $\Pi$.
Reflections with respect to the circular arcs $\widetilde{C}_i$ are anti-conformal involutions (hence automorphisms) of $\D$, and we call them $\rho_1$, $\rho_2$, and $\rho_3$. The maps $\rho_1$, $\rho_2$, and $\rho_3$ generate a subgroup $\mathcal{G}$ of $\mathrm{Aut}(\D)$. The group $\mathcal{G}$ is called the \emph{ideal triangle group}. As an abstract group, it is given by the generators and relations 
$$
\langle\rho_1, \rho_2, \rho_3: \rho_1^2=\rho_2^2=\rho_3^2=\mathrm{id}\rangle.
$$ 
Note that $\Pi$ is a fundamental domain of $\mathcal{G}$. The tessellation of $\D$ arising from $\mathcal{G}$ will play an important role in this paper (see Figure~\ref{tessellation_pic}). 
We will denote the connected component of $\D\setminus \Pi$ containing $\Int{\rho_i(\Pi)}$ by $\D_i$. Note that $\D_1\cup \D_2\cup \D_3=\D\setminus\Pi$.

\begin{figure}[ht!]
\captionsetup{width=0.96\linewidth}
\begin{center}
\includegraphics[scale=0.06]{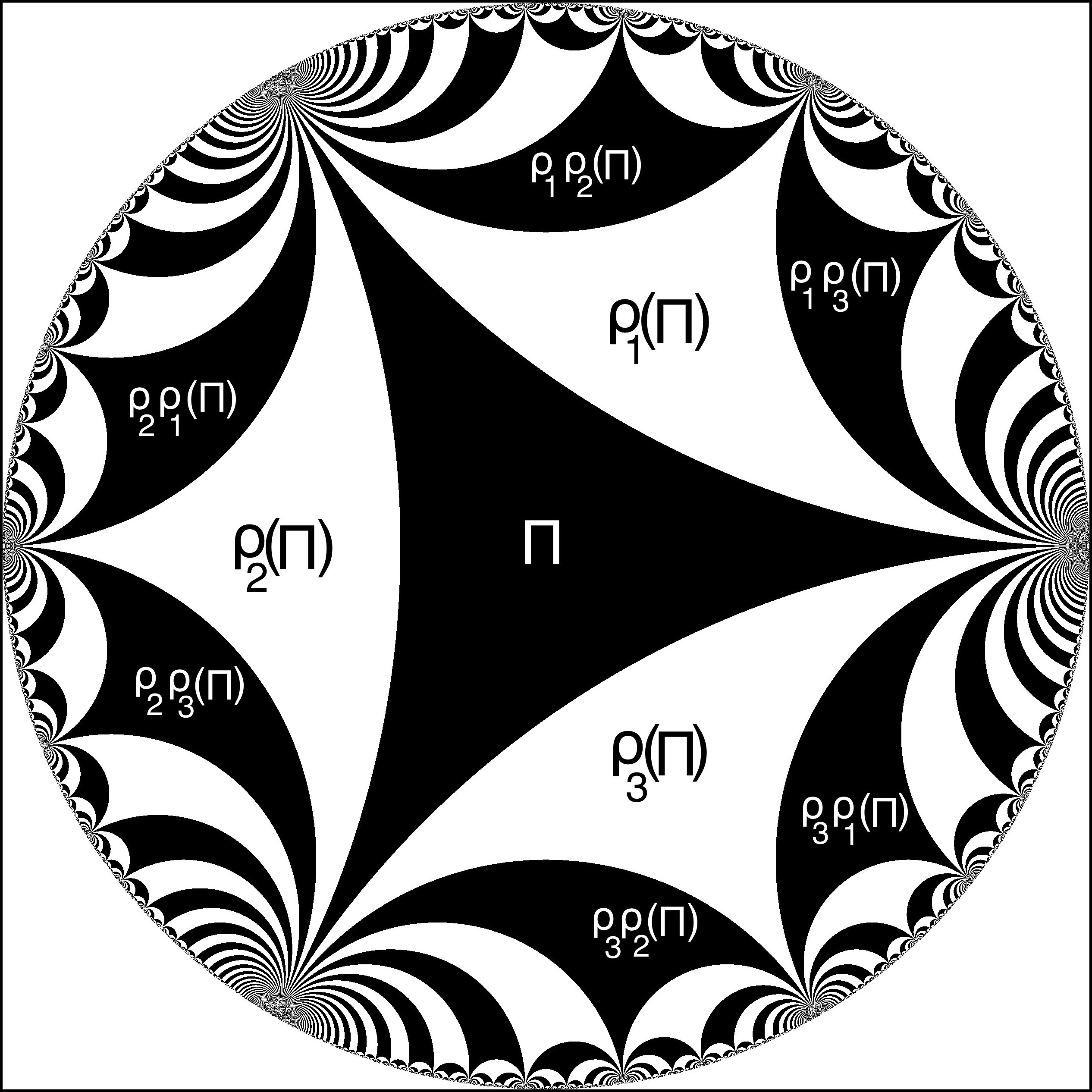} \includegraphics[scale=0.06]{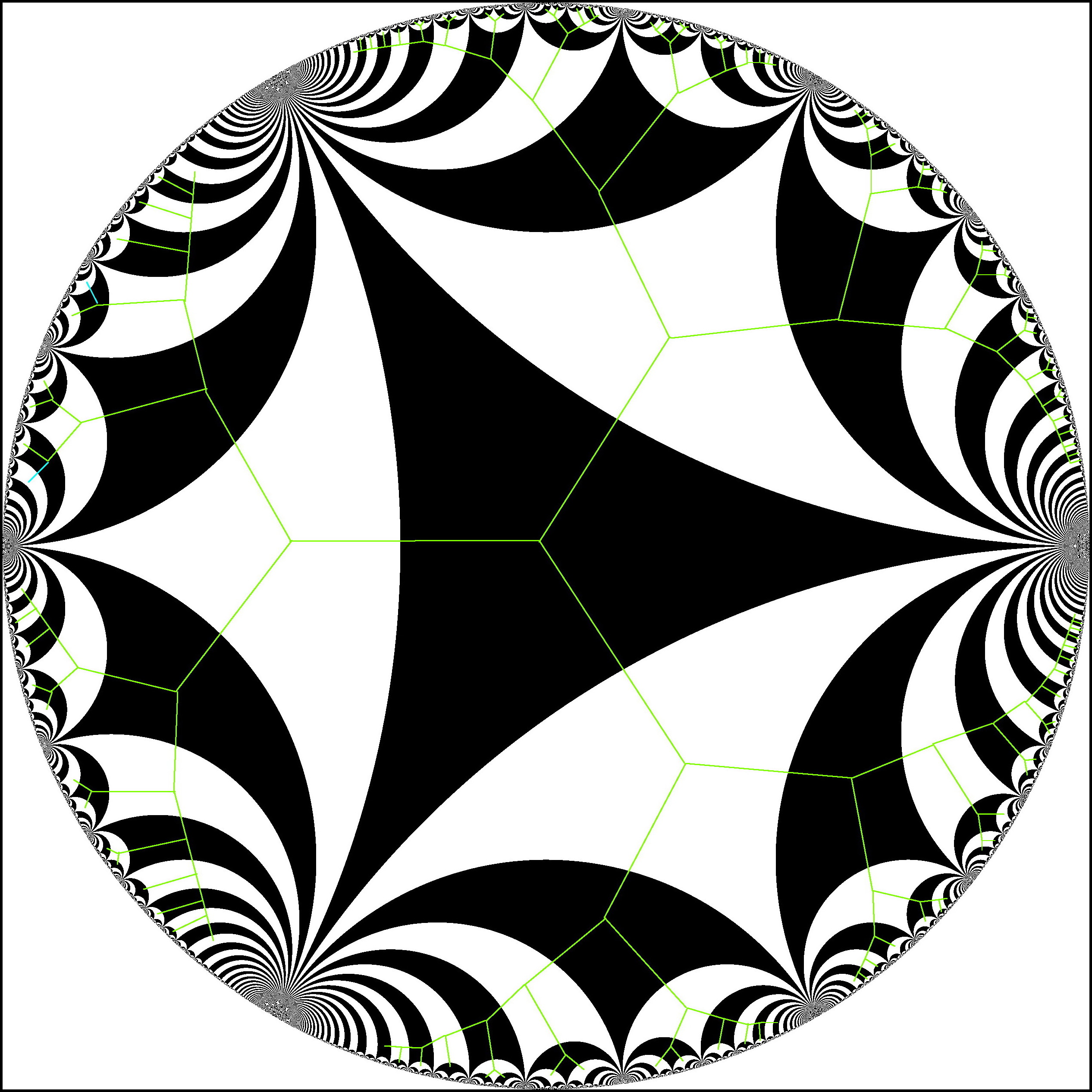}
\end{center}
\caption{Left: The $\mathcal{G}$-tessellation of $\D$ and the formation of a few initial tiles are shown. Right: The dual tree to the $\mathcal{G}$-tessellation of $\D$.}
\label{tessellation_pic}
\end{figure}

We now associate with the ideal triangle group $\mathcal{G}$ a piecewise reflection map 
$$
\rho:\overline{\D}\setminus\Int{\Pi}\to\overline{\D},\
z \mapsto \begin{array}{lll}
                    \rho_i(z) & \mbox{if}\ z\in \overline{\D}_i,\ i\in\{1,2,3\}.
                                          \end{array}
$$
Clearly, $\rho$ restricts to a $C^1$, expansive, orientation-reversing, double covering of $\mathbb{T}$. In fact, the orbit of any point in $\overline{\D}$ under $\mathcal{G}$ coincides with its grand orbit under the reflection map $\rho$. In other words, $\rho$ is \emph{orbit equivalent} to $\mathcal{G}$.

The map $\rho\vert_{\mathbb{T}}$ admits a Markov partition $\mathbb{T}=(\partial\D_1\cap\mathbb{T})\cup(\partial\D_2\cap\mathbb{T})\cup(\partial\D_3\cap\mathbb{T})$ with transition matrix $M=\mathbf{1} -\mathrm{Id}$, where $\mathbf{1}$ is the $3\times 3$ matrix with all entries equal to $1$, and $\mathrm{Id}$ is the $3\times 3$ identity matrix.

Let $W:=\lbrace 1,2,3\rbrace$. An element $(i_1,i_2,\cdots)\in W^{\mathbb{N}}$ is called $M$-admissible if $M_{i_k,i_{k+1}}=1$, for all $k\in\mathbb{N}$. We denote the set of all $M$-admissible words in $W^{\mathbb{N}}$ by $M^\infty$. One can similarly define $M$-admissibility of finite words. Expansivity of $\rho$ implies that there is a continuous surjection $Q:M^\infty\to \mathbb{T}$ that semi-conjugates the left-shift map to $\rho$.

\begin{definition}\label{def_tiles}
The images of the fundamental domain $\Pi$ under the elements of $\mathcal{G}$ are called \emph{tiles}. More precisely, for any $M$-admissible word $(i_1,\cdots,i_k)$, we define the tile $$T^{i_1,\cdots,i_k}:=\rho_{i_1}\circ\cdots\circ\rho_{i_k}(\Pi).$$
\end{definition}

It follows from the definition that $T^{i_1,\cdots,i_k}$ consists of all those $z\in\D$ such that $\rho^{\circ (n-1)}(z)\in\D_{i_n}$, for $n=1,\cdots,k$ and $\rho^{\circ k}(z)\in\Pi$. In other words, $$T^{i_1,\cdots,i_k}=\cap_{n=1}^{k}\rho^{-(n-1)}(\D_{i_n})\cap\rho^{-k}(\Pi).$$

Let us set $d_{\D}(0,\rho_{1}(0))=d_{\D}(0,\rho_{2}(0))=d_{\D}(0,\rho_{3}(0))=:\ell$. For any $M$-admissible sequence $(i_1,i_2,\cdots)$, let us consider the sequence $\{0,\rho_{i_1}(0),\rho_{i_1}\circ\rho_{i_2}(0),\cdots)$. Clearly, the hyperbolic distance (in $\D$) of any two consecutive points in this sequence is $\ell$. Connecting consecutive points of this sequence by hyperbolic geodesics of $\D$, we obtain a curve in $\D$ that lands at $Q((i_1, i_2,\cdots))\in\mathbb{T}$ (i.e., at the point on $\mathbb{T}$ whose $\rho$-itinerary with respect to the above Markov partition is $\left(i_1, i_2,\cdots\right)$).

\begin{definition}\label{rays_triangle}
The curve constructed above is called a $\mathcal{G}$-ray at angle $Q((i_1,i_2,\cdots)).$\footnote{Here, we identify $\mathbb{T}$ with $\R/\Z$.}
\end{definition}

We refer the readers to \cite[\S 2]{LLMM1} for a more detailed description of the symbolic dynamics of $\rho$.

The map $\rho$ is closely related to the anti-doubling map $m_{-2}:\theta\mapsto-2\theta$ on $\R/\Z$. In fact, there exists a homeomorphism $\mathcal{E}$ of the circle that conjugates $\rho$ to $\overline{z}^2$ (see \cite[\S 2]{LLMM1}). This topological conjugacy will be important in more ways than one in the rest of the paper.

\section{Quadrature domains and Schwarz reflections}\label{sec_quad_domain}

Although we will deal with explicit quadrature domains and Schwarz reflection maps in this paper, we would like to remind the readers the general definitions of these objects. For a more detailed exposition on quadrature domains and Schwarz reflection maps, we refer the readers to \cite{QD,LM} and the references therein.

Let $\iota$ denote the complex conjugation map.

\begin{definition}
Let $\Omega\subsetneq\widehat{\C}$ be a domain such that $\infty\notin\partial\Omega$ and $\Int{\overline{\Omega}}=\Omega$. A \emph{Schwarz function} of $\Omega$ is a meromorphic extension of $\iota\vert_{\partial\Omega}$ to all of $\Omega$. More precisely, a continuous function $S:\overline{\Omega}\to\widehat{\C}$ of $\Omega$ is called a Schwarz function of $\Omega$ if it satisfies the following two properties:
\begin{enumerate}
\item $S$ is meromorphic on $\Omega$,

\item $S=\iota$ on $\partial \Omega$.
\end{enumerate}
\end{definition}

It is easy to see from the definition that a Schwarz function of a domain (if it exists) is unique. 

\begin{definition}\label{quad_domain_def}
A domain $\Omega\subsetneq\widehat{\C}$ with $\infty\notin\partial\Omega$ and $\Int{\overline{\Omega}}=\Omega$ is called a \emph{quadrature domain} if $\Omega$ admits a Schwarz function.
\end{definition}

Therefore, for a quadrature domain $\Omega$, the map $\sigma:=\iota\circ S:\overline{\Omega}\to\widehat{\C}$ is an anti-meromorphic extension of the Schwarz reflection map with respect to $\partial \Omega$ (the reflection map fixes $\partial\Omega$ pointwise). We will call $\sigma$ the \emph{Schwarz reflection map of} $\Omega$.

The next proposition characterizes simply connected quadrature domains (see \cite[Theorem~1]{AS}).

\begin{proposition}[Simply connected quadrature domains]\label{simp_conn_quad}
A simply connected domain $\Omega\subsetneq\widehat{\C}$ with $\infty\notin\partial\Omega$ and $\Int{\overline{\Omega}}=\Omega$ is a quadrature domain if and only if the Riemann uniformization $f:\mathbb{D}\to\Omega$ extends to a rational map on $\widehat{\C}$. 

In this case, the Schwarz reflection map $\sigma$ of $\Omega$ is given by $f\circ(1/\overline{z})\circ(f\vert_{\D})^{-1}$. Moreover, if the degree of the rational map $f$ is $d$, then $\sigma:\sigma^{-1}(\Omega)\to\Omega$ is a (branched) covering of degree $(d-1)$, and $\sigma:\sigma^{-1}(\Int{\Omega^c})\to\Int{\Omega}^c$ is a (branched) covering of degree $d$.
\end{proposition}

\section{The circle-and-cardioid family}\label{schwarz_circle_cardioid}

In this section, we will recapitulate the definitions and various important facts about the family of Schwarz reflection maps with respect to the cardioid and a circumscribing circle which was introduced in \cite{LLMM1}.

Proposition~\ref{simp_conn_quad} immediately shows that the principal hyperbolic component $\heartsuit$ of the Mandelbrot set (also called the main cardioid) is a quadrature domain. Indeed, it admits a polynomial Riemann uniformization 
$$\phi:\D\to\heartsuit,\quad \mu\mapsto\mu/2-\mu^2/4.$$ The Riemann uniformization $\phi$ semi-conjugates the Schwarz reflection map $\sigma$ of $\heartsuit$ to the reflection map $1/\overline{z}$ of the unit disk. This yields an explicit description of $\sigma$.
 
 \begin{equation}
\sigma(\phi(\mu))=\phi(1/\overline{\mu})\quad \mathrm{i.e.,}\ \sigma\left(\frac{\mu}{2}-\frac{\mu^2}{4}\right)=\left(\frac{2\overline{\mu}-1}{4\overline{\mu}^2}\right)
\label{schwarz_cardioid}
\end{equation}
for each $\mu\in\overline{\D}$.

This allows us to study the basic mapping properties of the map $\sigma$, see \cite[\S 5.1]{LLMM1}. In particular, $\sigma$ has a unique critical point at $0$.

\subsection{The circle-and-cardioid family: dynamical and parameter planes}\label{C_and_C_subsec} 

We are now ready to describe the main object of this paper, namely the circle-and-cardioid family. For $a\in\C$, let $B(a,r_a)$ be the smallest disk containing $\heartsuit$ and centered at $a$; i.e., $\partial B(a,r_a)$ is the circumcircle to $\heartsuit$.
According to \cite[Proposition~5.9]{LLMM1}, we have the following dichotomy:
\begin{itemize}
\item for $a\in\left(-\infty,-1/12\right)$, the circumcircle $\partial B(a,r_a)$ touches $\partial\heartsuit$ at exactly two points, and
\item for any $a\in\C\setminus\left(-\infty,-1/12\right)$, the circumcircle $\partial B(a,r_a)$ touches $\partial\heartsuit$ at exactly one point.
\end{itemize}
In order to extract triangle group structure from our family of Schwarz reflection maps, we will restrict to parameters  $a\in\C\setminus\left(-\infty,-1/12\right)$.
Let $$\Omega_a := \heartsuit\cup\overline{B}(a,r_a)^c.$$ We now define our dynamical system $F_a:\overline{\Omega}_a\to\widehat{\C}$ as, 
$$
w \mapsto \left\{\begin{array}{ll}
                    \sigma(w) & \mbox{if}\ w\in\overline{\heartsuit}, \\
                    \sigma_a(w) & \mbox{if}\ w\in B(a,r_a)^c, 
                                          \end{array}\right. 
$$
where $\sigma$ is the Schwarz reflection of $\heartsuit$, and $\sigma_a$ is reflection with respect to the circle $\vert w-a\vert=r_a$. It follows from our previous discussion that $0$ is the only critical point of $F_a$. We will call this family of maps $\mathcal{S}$; i.e., $$\mathcal{S}:=\left\{F_a:\overline{\Omega}_a\to\widehat{\C}:a\in\C\setminus (-\infty,-1/12)\right\}.$$

Let $T_a:=\Omega_a^c= \overline{B}(a,r_a)\setminus\heartsuit$ (which we call the \emph{droplet}). Note that $\partial T_a$ has two singular points; namely $\alpha_a$ (a double point) and $\frac{1}{4}$ (a cusp). Both of them are fixed points of $F_a$. We define the \emph{fundamental tile} (or \emph{desingularized droplet}) of $F_a$ as $T_a^0:= T_a\setminus\{\alpha_a,\frac{1}{4}\}$. Then, the restriction $F_a:F_a^{-1}(T_a^0)\to T_a^0$ is a degree $3$ covering.

\begin{figure}[ht!]
\captionsetup{width=0.96\linewidth}
\centering
\includegraphics[scale=0.258]{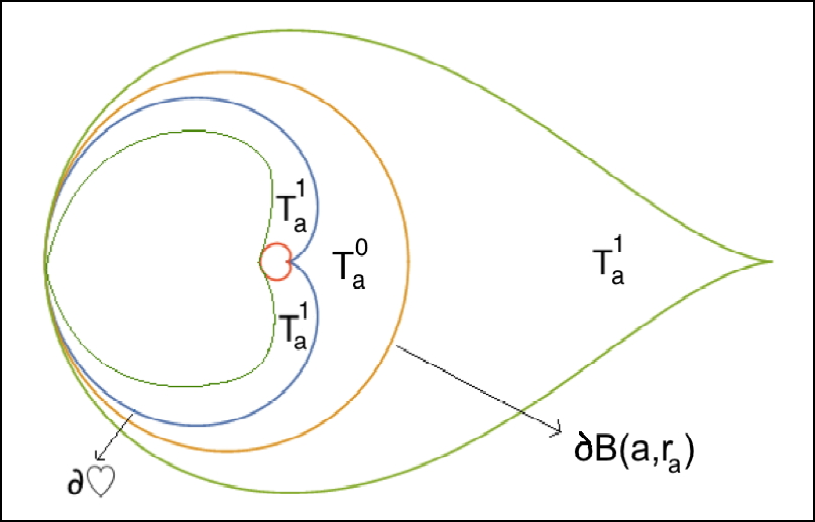} \includegraphics[scale=0.11]{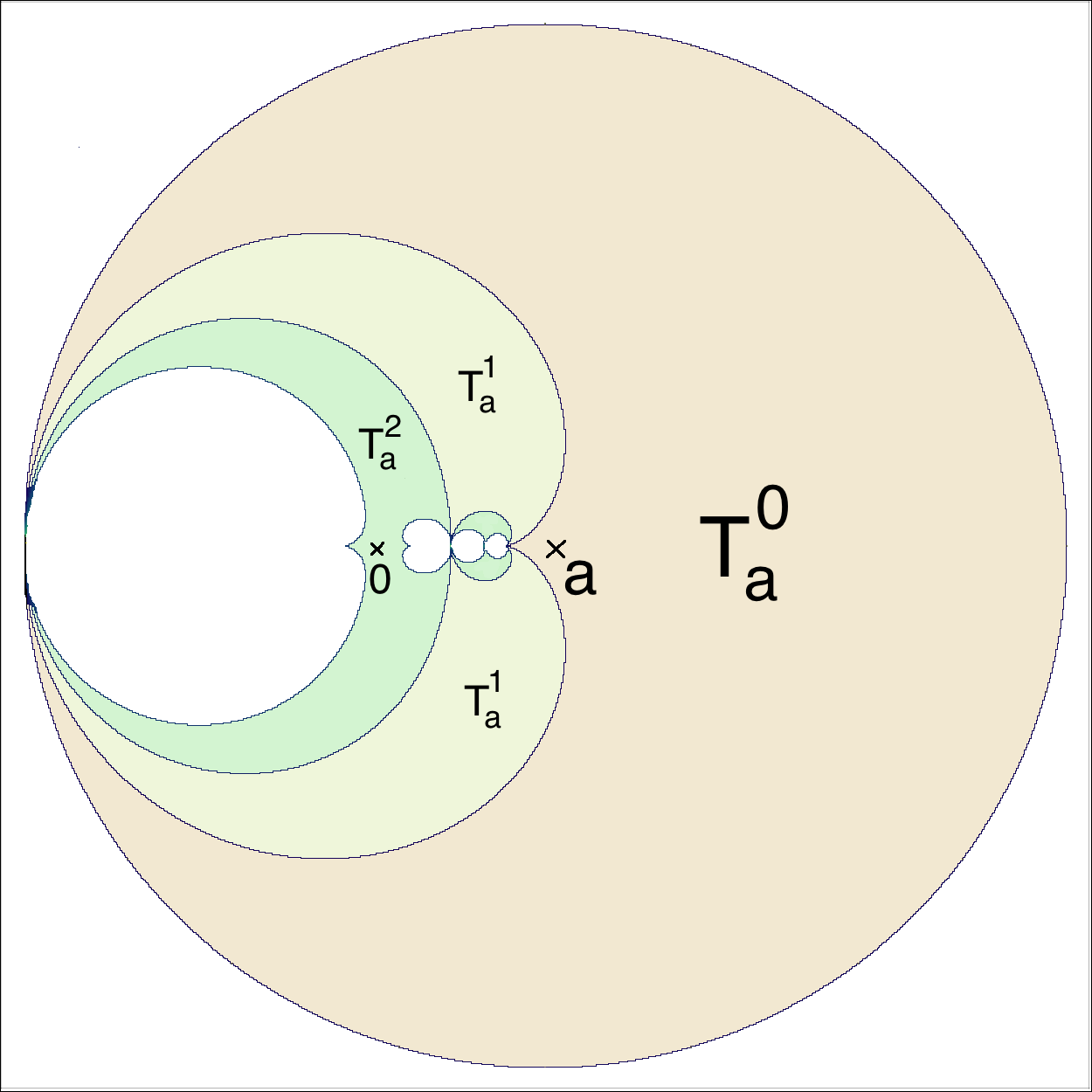}
\caption{Left: The tiles of rank $0$ and $1$ are labelled as $T_a^0$ and $T_a^1$ respectively. Right: Pictured are some of the initial tiles for a parameter $a$ such that the critical value $\infty$ escapes to the fundamental tile $T^0_a$ in one iterate; i.e., $F_a(\infty)=a\in T_a^0$. A tile of rank $2$ is ramified (i.e., it contains the critical point $0$) and hence disconnects the non-escaping set $K_a$. For such a parameter $a$, the external conjugacy $\psi_a$ extends conformally to all the tiles of the first rank (and all non-degenerate tiles of higher rank).}
\label{initial_tiles}
\end{figure}

\begin{proposition}\cite[Proposition~5.10]{LLMM1}\label{cardioid_circle_branched_cover}
For each $a\in\C\setminus(-\infty,-\frac{1}{12})$, the map $F_a: F_a^{-1}(\Omega_a)\to\Omega_a$ is a two-to-one branched covering, branched only at $0$.
\end{proposition}

\begin{definition}

\noindent\begin{enumerate}\upshape
\item For any $k~\geq~0$, the connected components of $F_a^{-k}(T_a^0)$ are called \emph{tiles} (of $F_a$) of rank $k$. The unique tile of rank $0$ is $T_a^0$.

\item The \emph{tiling set} $T_a^\infty$ of $F_a$ is defined as the set of points in $\overline{\Omega}_a$ that eventually escape to $T_a^0$; i.e., $T_a^\infty=\cup_{k=0}^\infty F_a^{-k}(T_a^0)$. Equivalently, the tiling set is the union of all tiles.

\item The \emph{non-escaping set} $K_a$ of $F_a$ is the complement $\widehat{\C}\setminus T_a^\infty$. Connected components of $\Int{K_a}$ are called \emph{Fatou components} of $F_a$. All iterates of $F_a$ are defined on $K_a$. 

\item The boundary of $T_a^\infty$ is called the \emph{limit set} of $F_a$, and is denoted by $\Gamma_a$.
\end{enumerate}
\end{definition}

The tiling set and the non-escaping set yield an invariant partition of the dynamical plane of $F_a$.

\begin{proposition}\cite[Proposition~5.19, Corollary~5.20]{LLMM1}\label{escape_connected}
For each $a\in\C\setminus(-\infty,-\frac{1}{12})$, the tiling set $T_a^\infty$ is an open connected set, and hence the non-escaping set $K_a$ is closed. Consequently, each Fatou component of $F_a$ is simply connected.
\end{proposition} 

Thanks to next lemma, one can perform quasiconformal deformations in the family $\mathcal{S}$.

\begin{lemma}\cite[Lemma~5.29]{LLMM1}\label{schwarz_qcdef}
Let $a\in\C\setminus (-\infty,-\frac{1}{12})$, and $\nu$ be an $F_a$-invariant Beltrami coefficient on $\widehat{\C}$, and $\Phi$ be a quasiconformal map integrating $\nu$ such that $\Phi$ fixes $0, \frac{1}{4},$ and $\infty$. Let $b=\Phi(a)$. Then, $b\in\C\setminus (-\infty,-\frac{1}{12})$, $\Phi(\Omega_a)=\Omega_b$, and $F_b=\Phi\circ F_a\circ\Phi^{-1}$ on $\Omega_b$.
\end{lemma}

The Fatou components of maps in $\mathcal{S}$ can be classified using classical quasiconformal deformation techniques and adapting the Fatou-Julia theory (for rational maps) for the current setting.

\begin{proposition}\cite[Proposition 5.30, Corollary 5.33]{LLMM1}\label{fatou_classification}
Let $a\in\C\setminus(-\infty,-\frac{1}{12})$. Then the following hold true.

1) Every Fatou component of $F_a$ is eventually periodic.

2) For $a\neq-\frac{1}{12}$, every periodic Fatou component of $F_a$ is either the (immediate) basin of attraction of an attracting/parabolic cycle, or a Siegel disk.

3) For $a=-\frac{1}{12}$, $F_a$ has a $2$-cycle of Fatou components such that every point in these components converges (under the dynamics) to the singular point $\alpha_a=-\frac34$ through a period two cycle of attracting petals (intersecting the real line). These are the only periodic Fatou components of $F_a$.
\end{proposition}

Let us also mention the relation between (closures of) various types of Fatou components and the post-critical set.

\begin{proposition}\cite[Proposition 5.32]{LLMM1}\label{fatou_critical}
1) If $F_a$ has an attracting or parabolic cycle, then the forward orbit of the critical point $0$ converges to this cycle.  

2) If $U$ is a Siegel disk of $F_a$, then $\partial U\subset\overline{\left\{F_a^{\circ n}(0)\right\}_{n\geq0}}$. Every Cremer point (i.e., an irrationally neutral, non-linearizable periodic point) of $F_a$ is also contained in $\overline{\left\{F_a^{\circ n}(0)\right\}_{n\geq0}}$.

3) For $a=-\frac{1}{12}$, the critical orbit of $F_a$ converges to the singular point $-\frac34$.
\end{proposition}

\begin{definition}\label{def_depth}
Let $a\in\C\setminus\left(-\infty,-1/12\right)$ be such that the critical point $0$ escapes to the fundamental tile $T^0_a$ under iterations of $F_a$. Then the smallest positive integer $n(a)$ such that $F_a^{\circ n(a)}(\infty)\in T_a^0$ is called the \emph{depth} of $a$.
\end{definition} 

The next result shows that the non-escaping set $K_a$ of $F_a$ is connected if and only if $K_a$ contains the unique critical point $0$. It also reveals the group structure in the tiling set of $F_a$.

\begin{proposition}\cite[Propositions~5.23, 5.38, 5.52]{LLMM1}\label{schwarz_group}
1) If the critical point of $F_a$ does not escape to the fundamental tile $T_a^0$, then $K_a$ is connected. Moreover, the conformal map $\psi_a$ from $T_a^0$ onto $\Pi$ extends to a biholomorphism between the tiling set $T_a^\infty$ and the unit disk $\D$. The extended map $\psi_a$ conjugates $F_a$ to the reflection map $\rho$. 

2) If the critical point of $F_a$ escapes to the fundamental tile, then the corresponding non-escaping set $K_a$ is a Cantor set. In this case, there exists a conformal map 
$$
\psi_a:\displaystyle\bigcup_{k=0}^{n(a)} F_a^{-k}(T_a^0)\to \displaystyle\bigcup_{k=0}^{n(a)} \rho^{-k}(\Pi)
$$ 
conjugating
$$F_a:\displaystyle\bigcup_{k=1}^{n(a)} F_a^{-k}(T_a^0)\to\displaystyle\bigcup_{k=0}^{n(a)-1} F_a^{-k}(T_a^0)\quad \mathrm{to}\quad \rho:\displaystyle\bigcup_{k=1}^{n(a)} \rho^{-k}(\Pi)\to\displaystyle\bigcup_{k=0}^{n(a)-1}\rho^{-k}(\Pi),$$ where $n(a)$ denotes the depth of $a$. 
\end{proposition}

This leads to the definition of the \emph{connectedness locus} of the family $\mathcal{S}$.

\begin{definition}\label{escape_locus_def}
The \emph{connectedness locus} of the family $\mathcal{S}$ is defined as $$\cC(\mathcal{S})=\{a\in\C\setminus(-\infty,-1/12): 0\notin T_a^\infty\}=\{a\in\C\setminus(-\infty,-1/12): K_a\ \textrm{is\ connected}\}.$$ The complement of the connectedness locus in the parameter space is called the \emph{escape locus}.
\end{definition}

\begin{proposition}\cite[Propositions~5.25,~5.26]{LLMM1}\label{conn_locus_in_cardioid}
$\cC(\mathcal{S})\subset \heartsuit\cup\{\frac{1}{4}\}$. Moreover, $\cC(\mathcal{S})$ is closed in $\C\setminus(-\infty,-\frac{1}{12})$.
\end{proposition}

\begin{figure}[ht!]
\captionsetup{width=0.96\linewidth}
\begin{tikzpicture}
\node[anchor=south west,inner sep=0] at (0,6) {\includegraphics[width=0.46\textwidth]{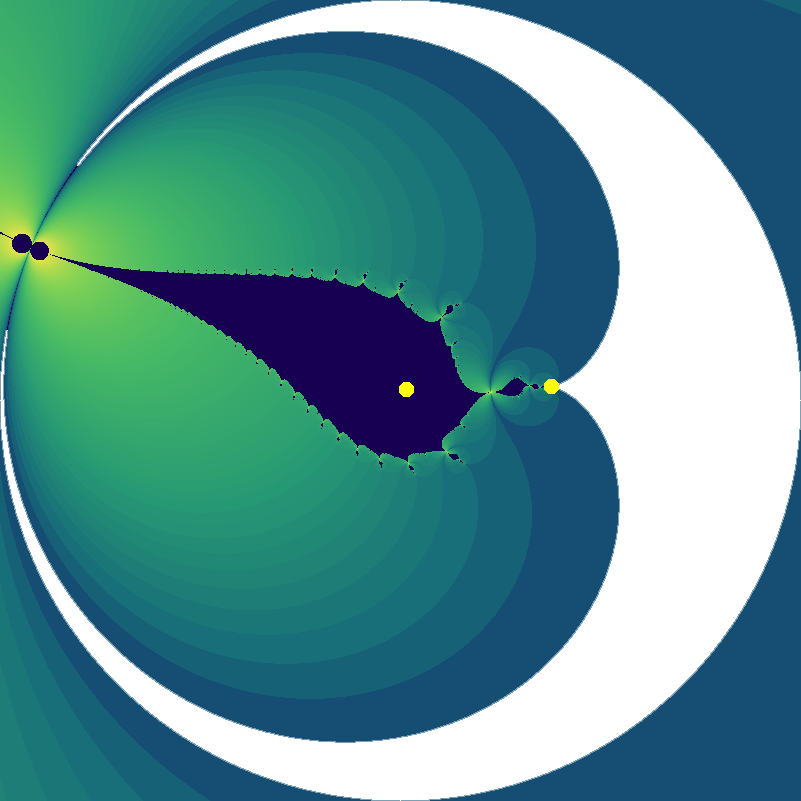}};
\node[anchor=south west,inner sep=0] at (6,6) {\includegraphics[width=0.46\textwidth]{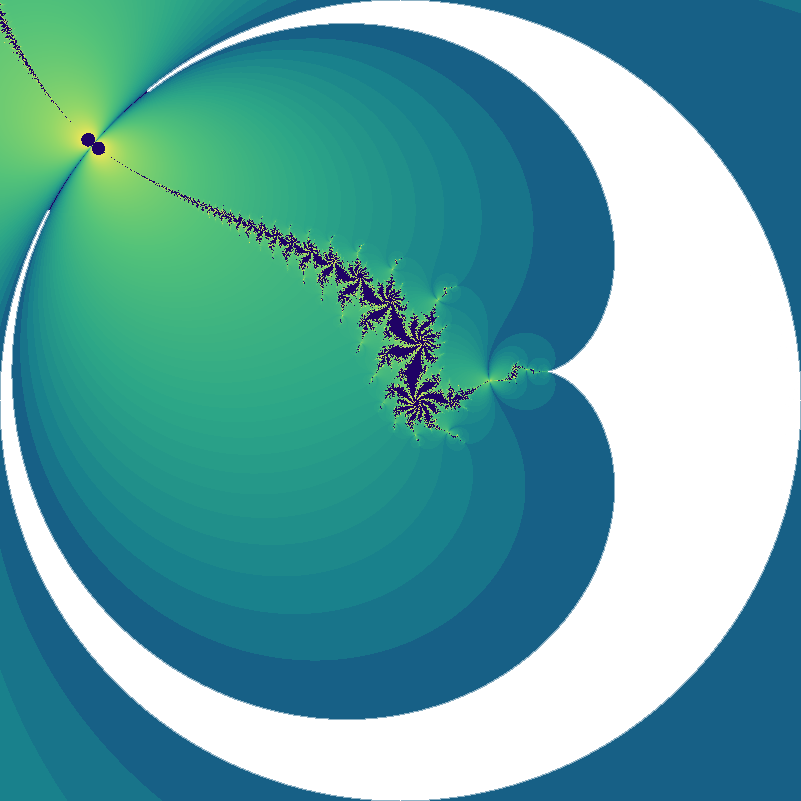}};
\node[anchor=south west,inner sep=0] at (0,0) {\includegraphics[width=0.46\textwidth]{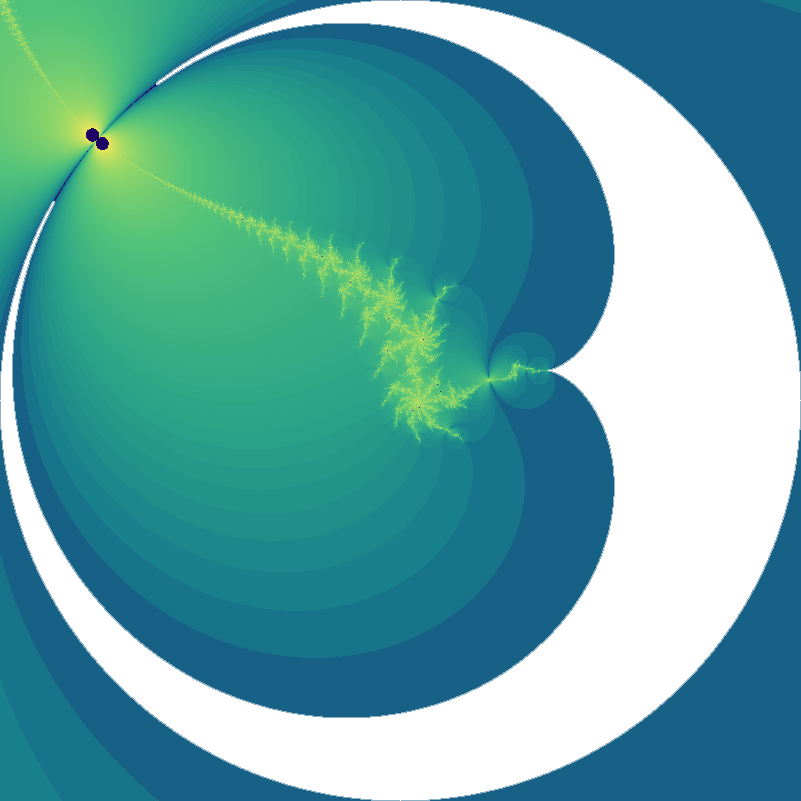}};
\node[anchor=south west,inner sep=0] at (6,0) {\includegraphics[width=0.45\textwidth]{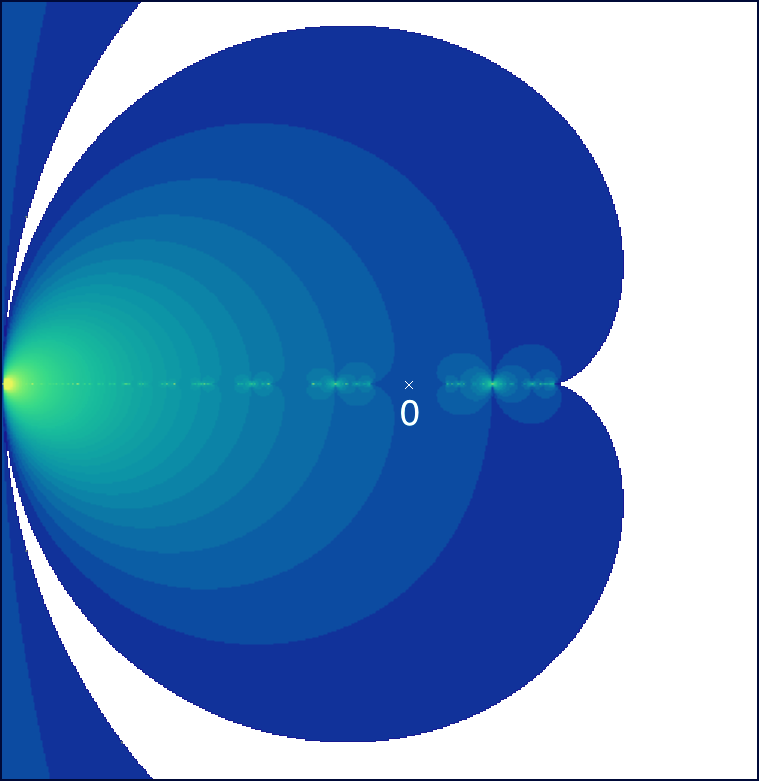}};
\node at (2.96,8.75) {\begin{small}\textcolor{white}{$0$}\end{small}};
\node at (4.4,9) {\begin{tiny}$1/4$\end{tiny}};
\node at (4.3,6.95) {\begin{small}$\partial\heartsuit$\end{small}};
\node at (5.16,11.5) {\begin{small}\textcolor{white}{$\partial B(a,r_a)$}\end{small}};
\node at (0.4,9.66) {\begin{small}$\alpha_a$\end{small}};
\end{tikzpicture}
\caption{Parts of the dynamical planes of various parameters are shown. The white regions denote the tile of rank $0$. The non-escaping sets in the top row have non-empty interior, while the one in the bottom left figure has empty interior. The bottom right figure shows a totally disconnected non-escaping set.}
\label{various_limit_sets}
\end{figure}

\begin{definition}\label{tile_code}
i) Let $a\in\cC(\mathcal{S})$. For any $M$-admissible word $(i_1,\cdots,i_k)$, the dynamical tile $T_a^{i_1,\cdots,i_k}$ is defined as $$T_a^{i_1,\cdots,i_k}:=\psi_a^{-1}(T^{i_1,\cdots,i_k}).$$ 

ii) Let $a\in\C\setminus\left(\left(-\infty,-\frac{1}{12}\right)\cup\cC(\mathcal{S})\right)$. For any $M$-admissible word $(i_1,\cdots,i_k)$, the dynamical tile $T_a^{i_1,\cdots,i_k}$ is defined as above, provided $T^{i_1,\cdots,i_k}$ is in the image of $\psi_a$.
\end{definition}

We can use the map $\psi_a$ to define dynamical rays for the maps $F_a$.

\begin{definition}\label{dyn_ray_schwarz}
The pre-image of a $\mathcal{G}$-ray at angle $\theta$ under the map $\psi_a$ is called a $\theta$-dynamical ray of $F_a$.
\end{definition}

We denote the set of (pre-)periodic points of $\rho:\mathbb{S}^1\to\mathbb{S}^1$ by $\mathrm{Per}(\rho)$ (see Section~\ref{ideal_triangle} for the definition of the piecewise reflection map $\rho$). The landing property of (pre-)periodic dynamical rays plays an important role in our study.

\begin{proposition}\cite[Proposition~5.44]{LLMM1}\label{per_rays_land}
Let $a\in\cC(\mathcal{S})$, and $\theta\in\mathrm{Per}(\rho)$. Then the following statements hold true.

1) The dynamical $\theta$-ray of $F_a$ lands on $\Gamma_a$. 

2) The $0$-ray of $F_a$ lands at $\frac{1}{4}$, while the $\frac{1}{3}$ and $\frac{2}{3}$-rays land at $\alpha_a$. No other ray lands at $\frac14$ or $\alpha_a$. The iterated pre-images of the $0, \frac13,$ and $\frac23$-rays land at the  iterated pre-images of $\frac14$ and $\alpha_a$.

3) Let $\theta\in\mathrm{Per}(\rho)\setminus\cup_{n\geq0}\rho^{-n}\left(\left\{0,\frac{1}{3},\frac{2}{3}\right\}\right)$. Then, the dynamical ray of $F_a$ at angle $\theta$ lands at a repelling or parabolic (pre-)periodic point on $\Gamma_a$. 
\end{proposition}

The following converse to the previous proposition is also important.

\begin{proposition}\cite[Proposition~5.45]{LLMM1}\label{rep_para_landing_point}
Let $a\in\cC(\mathcal{S})$. Then, every repelling and parabolic periodic point of $F_a$ is the landing point of finitely many (at least one) dynamical rays. Moreover, all these rays have the same period under $F_a^{\circ 2}$.
\end{proposition}

\begin{remark}
Unlike in the holomorphic situation, it is not true that all rays landing at a periodic point have the same period under $F_a$ (this is already false for quadratic anti-polynomials, see \cite[Theorems~2.6,~3.1]{Sa1}). 
\end{remark}

Let $a\in\cC(\mathcal{S})$. Proposition~\ref{per_rays_land} allows us to define a landing map $$L_a:\mathrm{Per}(\rho)\to \Gamma_a$$ that associates to every (pre-)periodic angle $\theta$ (under $\rho$) the landing point of the $\theta$-dynamical ray of $F_a$.

\begin{definition}\label{def_preper_lami}
For $a\in\cC(\mathcal{S})$, the \emph{pre-periodic lamination} of $F_a$ is defined as the equivalence relation on $\mathrm{Per}(\rho)\subset\R/\Z$ such that $\theta, \theta'\in\mathrm{Per}(\rho)$ are related if and only if $L_a(\theta)=L_a(\theta')$. We denote the pre-periodic lamination of $F_a$ by $\lambda(F_a)$.
\end{definition}

The next definition and the subsequent proposition relates pre-periodic laminations of $F_a$ to rational laminations of quadratic anti-polynomials (see Definition~\ref{def_rat_lami}). 

\begin{definition}[]\label{push_lami}
The \emph{push-forward} $\mathcal{E}_{\ast}(\lambda(F_a))$ of the pre-periodic lamination of $F_a$ is defined as the image of $\lambda(F_a)\subset\mathrm{Per}(\rho)\times\mathrm{Per}(\rho)$ under $\mathcal{E}\times\mathcal{E}$. Clearly, $\mathcal{E}_{\ast}(\lambda(F_a))$ is an equivalence relation on $\Q/\Z$. Similarly, the pullback $\mathcal{E}^{\ast}(\lambda(f_c))$ of the rational lamination of a quadratic anti-polynomial $f_c$ is defined as the pre-image of $\lambda(f_c)\subset\Q/\Z\times\Q/\Z$ under $\mathcal{E}\times\mathcal{E}$.
\end{definition}

\begin{proposition}\cite[Proposition 5.51]{LLMM1}\label{prop_preper_lami}
Let $a\in\cC(\mathcal{S})$, and $\lambda(F_a)$ be the pre-periodic lamination associated with $F_a$. Then, $\lambda(F_a)$ satisfies the following properties.
\begin{enumerate}
\item $\lambda(F_a)$ is closed in $\mathrm{Per}(\rho)\times\mathrm{Per}(\rho)$.

\item Each $\lambda(F_a)$-equivalence class $A$ is a finite subset of $\mathrm{Per}(\rho)$.

\item If $A$ is a  $\lambda(F_a)$-equivalence class, then $\rho(A)$ is also a $\lambda(F_a)$-equivalence class.

\item If $A$ is a  $\lambda(F_a)$-equivalence class, then $A\mapsto\rho(A)$ is consecutive reversing.

\item $\lambda(F_a)$-equivalence classes are pairwise unlinked.
\end{enumerate}

Consequently, the push-forward $\mathcal{E}_{\ast}(\lambda(F_a))$ is a formal rational lamination (under $m_{-2}$).
\end{proposition}

\subsection{Geometrically finite maps}\label{geom_fin_sec}

Recall that for any $a\in\C\setminus(-\infty,-\frac{1}{12})$, the map $F_a$ has two fixed points $\alpha_a$ and $\frac{1}{4}$ such that $F_a$ does not admit an anti-holomorphic extension in a neighborhood of these points. The dynamical behavior of $F_a$ near these fixed points were analyzed in \cite[Propositions~5.12,~5.13,~5.15]{LLMM1}. In what follows, we will refer to these fixed points as \emph{singular points}. On the other hand, the terms \emph{cycle/periodic orbit} will be reserved for periodic points of period greater than one (these are contained in $\Omega_a$).

The map $F_a$ is called hyperbolic (respectively, parabolic) if it has a (super-)attracting (respectively, parabolic) cycle. It is easy to see that a (super-)attracting cycle of $F_a$ belongs to the interior of $K_a$, and a parabolic cycle lies on the boundary of $K_a$ (e.g. by \cite[\S 5, Theorem~5.2]{M1new}). Moreover, a parabolic periodic point necessarily lies on the boundary of a Fatou component (i.e., a connected component of $\Int{K_a}$) that contains an attracting petal of the parabolic germ such that the forward orbit of every point in the component converges to the parabolic cycle through the attracting petals.

A \emph{Misiurewicz} parameter in the family $\mathcal{S}$ is a parameter $a$ such that the critical point $0$ is non-escaping and strictly pre-periodic under iterations of $F_a$.

\begin{figure}[ht!]
\captionsetup{width=0.96\linewidth}
\centering
\includegraphics[scale=0.36]{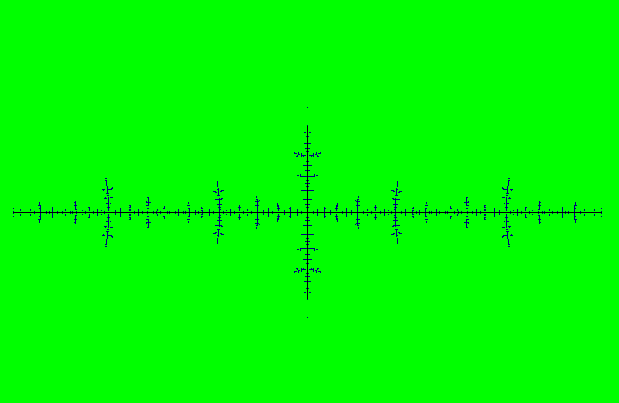}\ \includegraphics[scale=0.15]{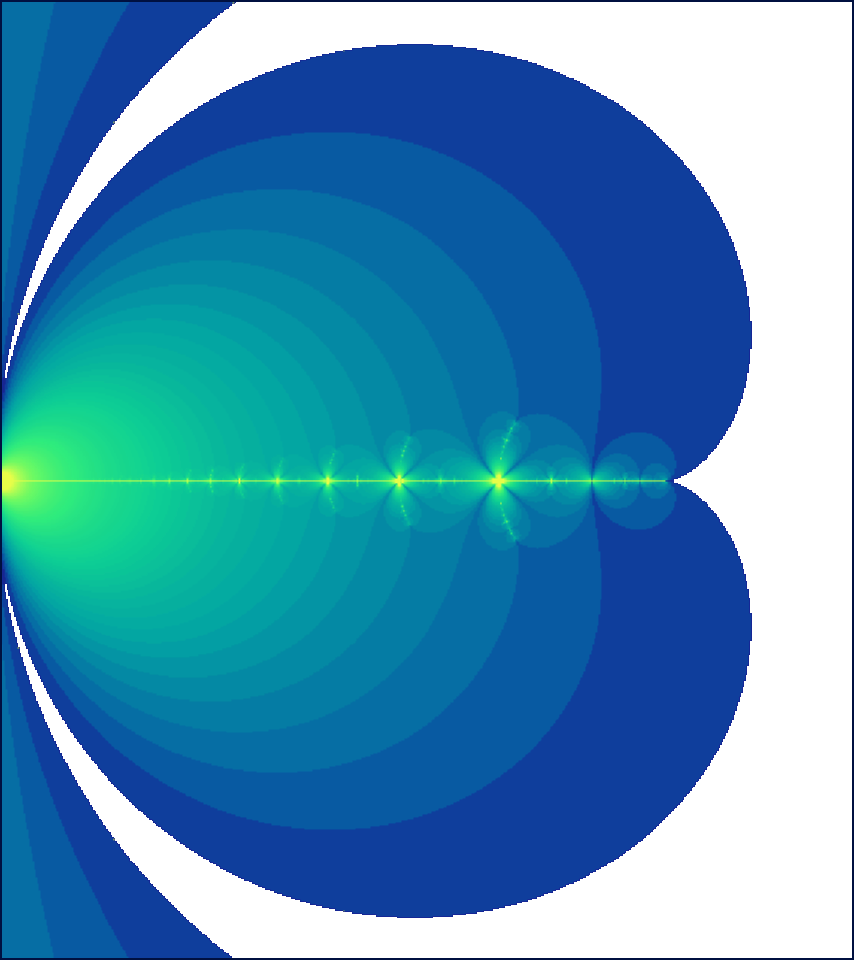}
\caption{Left: The dynamical plane of the Misiurewicz map $\overline{z}^2-1.5437\cdots$, for which the critical point $0$ lands at the $\alpha$-fixed point in $3$ iterates. Right: A part of the dynamical plane of the Misiurewicz map $F_a$ where $a=\frac{5}{36}$. The critical point $0$ of this map lands at the fixed point $\alpha_a$ in $3$ iterates.}
\label{misi_pic_1}
\end{figure}

The main properties of geometrically finite maps (i.e., hyperbolic, parabolic, and Misiurewicz maps) in the family $\mathcal{S}$ are summarized below.

\begin{theorem}\cite[Theorem~1.4, Propositions~6.6,~6.9]{LLMM1}\label{geom_finite_limit_set}
Let $F_a$ be geometrically finite. Then the following hold.
\begin{enumerate}

\item $a\in\cC(\mathcal{S})$.

\item The limit set $\Gamma_a$ of $F_a$ is locally connected. 

\item The area of the limit set $\Gamma_a$ is zero.

\item The iterated pre-images of the cardioid cusp are dense in the limit set $\Gamma_a$. Moreover, repelling periodic points of $F_a$ are dense in $\Gamma_a$.

\item If $a$ is a Misiurewicz parameter, then $\Int{K_a}=\emptyset$; i.e., $K_a=\Gamma_a$. Moreover, $\Gamma_a$ is a dendrite.
\end{enumerate}
\end{theorem}

\section{Tessellation of the escape locus}\label{unif_exterior_conn_locus}

In this section, we will construct a homeomorphism from the escape locus of the family $\mathcal{S}$ (see Definition~\ref{escape_locus_def}) to a suitable simply connected domain. This will yield a dynamically defined tessellation of the exterior of the connectedness locus (in the spirit of a ray-equipotential structure of escape loci of polynomials). The proof will be reminiscent of the proof of connectedness of the Mandelbrot set, but lack of holomorphic parameter dependence will add some complexity to the situation. 

In order to prove the main result of this subsection, we need an auxiliary lemma. Note that for $a\in\left(-\infty,-1/12\right)$, the disk $B(a,r_a)$ touches $\heartsuit$ at exactly two points; i.e., $\Int{(\overline{B}(a,r_a)\setminus\heartsuit)}$ consists of two connected components \cite[Proposition~5.9]{LLMM1}. Exactly one of these two connected components is disjoint from the positive real axis, and we denote the closure of this component by $K_a^{-}$. Note that $K_a^-\cap\R_{-}=\left[q_a,-\frac{3}{4}\right]$, for some $q_a(<0)\in\partial B(a,r_a)$ (compare Figure~\ref{real_slit}). As before, we denote the Schwarz reflection map with respect to $\partial\heartsuit$ by $\sigma$ and reflection with respect to $\partial B(a,r_a)$ by $\sigma_a$. Using these maps, we now define $F_a$ (for $a\in\left(-\infty,-1/12\right)$) on $\left[-\infty,q_a\right]\cup\left[-\frac{3}{4},0\right]$ as 

\begin{figure}[ht!]
\captionsetup{width=0.96\linewidth}
\centering
\includegraphics[scale=0.2]{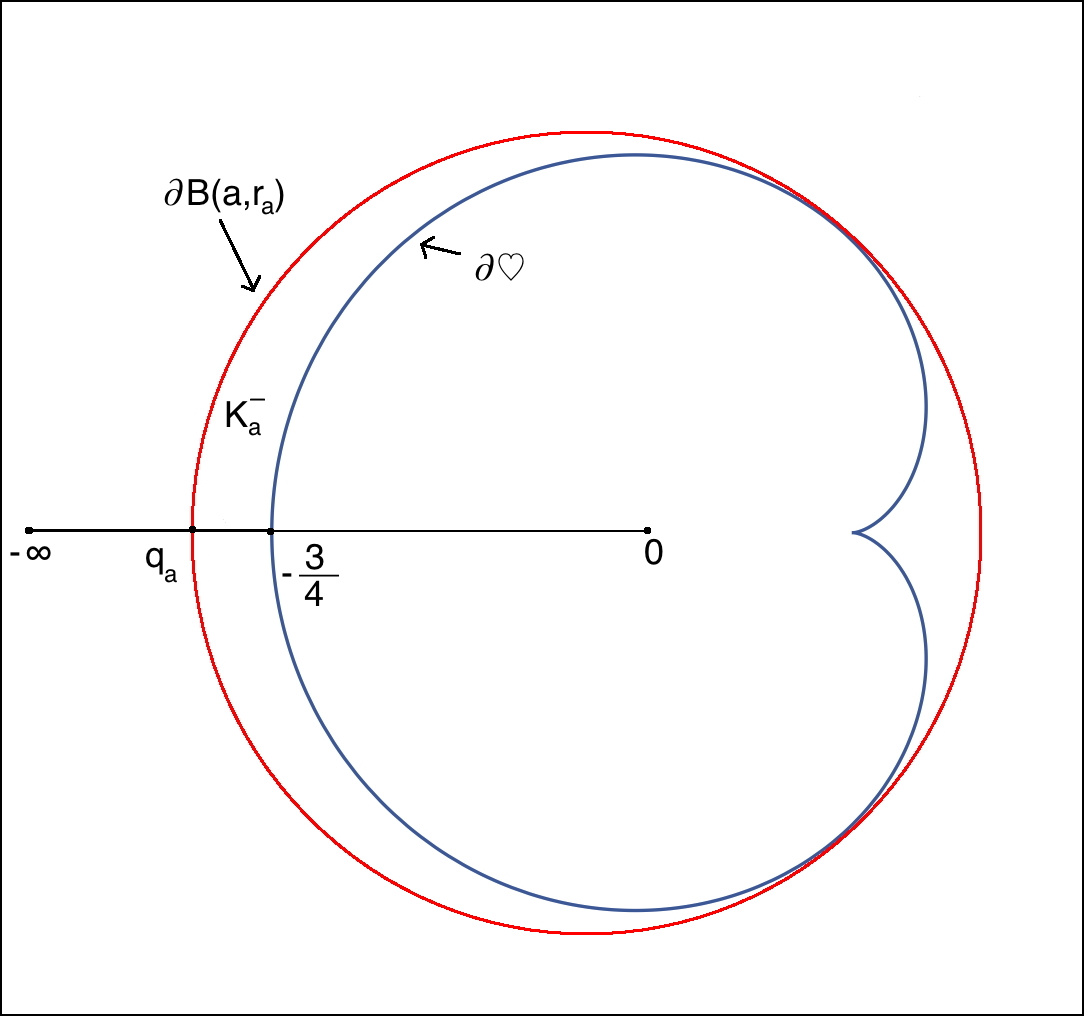}
\caption{For each $a\in\left(-\infty,-1/12\right)$, the disk $B(a,r_a)$ touches $\heartsuit$ at exactly two points. The closure $K_a^-$ of one of the two complementary components of $B(a,r_a)^c\cup\overline{\heartsuit}$ intersects the negative real axis at the points $q_a$ and $-\frac{3}{4}$.}
\label{real_slit}
\end{figure}

$$
w \mapsto \left\{\begin{array}{ll}
                    \sigma(w) & \mbox{if}\ w\in\left[-\frac{3}{4},0\right], \\
                    \sigma_a(w) & \mbox{if}\ w\in \left[-\infty,q_a\right]. 
                                         \end{array}\right. 
$$

\begin{lemma}\label{slit_negative}
For all $a\in\left(-\infty,-1/12\right)$, we have that $F_a^{\circ n}(\infty)\in\R_{-}$ whenever $F_a^{\circ n}(\infty)$ is defined.
\end{lemma}
\begin{proof}
This follows from the simple observation that $F_a(\left[-\frac{3}{4},0\right])=\left[-\infty,-\frac{3}{4}\right]\subset\R_{-}$ and $F_a(\left[-\infty,q_a\right])=\left[q_a,a\right]\subset\R_{-}$.
\end{proof}

Note that for all $a\in\C\setminus\left(\left(-\infty,-1/12\right)\cup\cC(\mathcal{S})\right)$, the critical value $\infty$ is contained in the domain of the conjugacy $\psi_a$. We will now show that the conformal position of the critical value $\infty$ yields the desired uniformization of the escape locus. Recall that the connected component of $\D\setminus\Pi$ containing $\Int{\rho_2(\Pi)}$ is denoted by $\D_2$ (where $\Pi$ is the ideal triangle in $\D$ with vertices at $1,\omega,$ and $\omega^2$; and $\rho_2$ is reflection in the side of $\Pi$ connecting $\omega$ and $\omega^2$).

\begin{proof}[Proof of Theorem~\ref{thm_unif_exterior_conn_locus}]
First note that by Proposition~\ref{conn_locus_in_cardioid}, the set $\left(-\infty,-1/12\right)~\cup~\cC(\mathcal{S})$ is closed in the plane. So the complement of this set is open in $\C$.

Since $\infty$ lies outside $\overline{B}(a,r_a)$, it follows that $\pmb{\Psi}(a):=\psi_a(\infty)\in\D_2$ for each $a\in\C\setminus\left(\left(-\infty,-1/12\right)\cup\cC(\mathcal{S})\right)$. More precisely, $\pmb{\Psi}(a)\in T^{k_1,\cdots,k_{n(a)}}$ with $k_1=2$.

Note also that as $a$ runs over $\C\setminus\left((-\infty,-1/12)\cup\cC(\mathcal{S})\right)$, the fundamental tile $T_a^0$ changes continuously, and hence so does the conformal map $\psi_a$ restricted to $\displaystyle\cup_{k=0}^{n(a)} F_a^{-k}(T_a^0)\ni\infty$. It follows that $\pmb{\Psi}(a)=\psi_a(\infty)$ depends continuously on $a$.

Our plan is to show that $\pmb{\Psi}$ is proper and injective (cf. \cite[\S 3]{Na1}). Since $\pmb{\Psi}$ is continuous, the Invariance of Domain Theorem would then imply that $\pmb{\Psi}$ is an open map, and hence a homeomorphism onto its image (cf. \cite[Theorem~2B.3]{Hat02}). Finally, properness of $\pmb{\Psi}:\C\setminus((-\infty,-1/12)\cup\cC(\mathcal{S}))\to\D_2$ would force the image of $\pmb{\Psi}$ to be the entire simply connected domain $\D_2$.

\begin{lemma}\label{prop_proper}
The map $\pmb{\Psi}$ is proper.
\end{lemma}
\begin{proof}
We will consider three different cases. 

Let us first choose a sequence $\{a_k\}_k\subset\C\setminus((-\infty,-1/12)\cup\cC(\mathcal{S}))$ such that $\vert a_k\vert\to\infty$. We can assume that each $a_k$ lies outside $\overline{\heartsuit}$; i.e., $a_k\in\Int{T_a^0}$ and hence $\infty$ lies in some tile of first generation. Now consider the domain $U_k:=\Int{(T_{a_k}^0\cup F_{a_k}^{-1}(T_{a_k}^0))}$, which is mapped biholomorphically onto $\Int{(\Pi\cup\rho^{-1}(\Pi))}$ by $\psi_{a_k}$. Since $U_k$ contains $a_k$ and $\infty$, it follows that $\psi_{a_k}(a_k)$ and $\psi_{a_k}(\infty)$ are contained in $\Int{(\Pi\cup\rho^{-1}(\Pi))}$ for each $k$. Moreover, the assumption $\vert a_k\vert\to\infty$ implies that $a_k$ is uniformly bounded away from $\partial\heartsuit\setminus\{\alpha_a,\frac{1}{4}\}$ in the hyperbolic metric of $U_k$ for all $k$. Hence, $\psi_{a_k}(a_k)$ is uniformly bounded away from $\widetilde{C}_1\cup\widetilde{C}_3$ in the hyperbolic metric of $\Int{(\Pi\cup\rho^{-1}(\Pi))}$ for all $k$. If $\psi_{a_k}(a_k)$ converges to $\partial\Pi\cap\partial\D$, then $\psi_{a_k}(\infty)$ converges to $\partial\D_2$ (as $\rho_2(\psi_{a_k}(a_k))=\psi_{a_k}(\infty)$). Otherwise, $\psi_{a_k}(a_k)$ and $\psi_{a_k}(\infty)$ are bounded away from the boundary of $\Int{(\Pi\cup\rho^{-1}(\Pi))}$. Since the spherical distance between $a_k$ and $\infty$ tends to $0$ as $k$ increases, it follows that the hyperbolic distance of $\psi_{a_k}(a_k)$ and $\psi_{a_k}(\infty)$ (with respect to the hyperbolic metric of $\Int{(\Pi\cup\rho^{-1}(\Pi))}$) must also converge to $0$. But this implies that both the sequences $\{\psi_{a_k}(a_k)\}_k$ and $\{\psi_{a_k}(\infty)\}_k$  accumulate on $\widetilde{C}_2$. In either case, we conclude that $\{\pmb{\Psi}(a_k)\}_k$ tends to the boundary of $\D_2$.

Now let $\{a_k\}_k\subset\C\setminus((-\infty,-1/12)\cup\cC(\mathcal{S}))$ be a sequence accumulating on $\cC(\mathcal{S})$. Suppose that $\{\pmb{\Psi}(a_k)\}_k$ converges to some $u\in\D_2$. Then, $\{\psi_{a_k}(\infty)\}_k$ is contained in a compact subset $\mathcal{K}$ of $\D_2$. After passing to a subsequence, we can assume that $\mathcal{K}$ is contained in a single tile of $\D$ (arising from $\mathcal{G}$). But this implies that each $a_k$ has a common depth $n_0$, and $\psi_{a_k}(F_{a_k}^{\circ n_0}(\infty))$ is contained in the compact set $\rho^{\circ n_0}(\mathcal{K})\subset\Pi$ for each $k$. Note that the map $F_a$, the fundamental tile $T_a^0$ as well as (the continuous extension of) the conformal map $\psi_a:T_a^0\to\Pi$ change continuously with the parameter as $a$ runs over $\C\setminus\left(-\infty,-1/12\right)$. Therefore, for every accumulation point $a_0$ of $\{a_k\}_k$, the point $F_{a_0}^{\circ n_0}(\infty)$ belongs to the compact set $\psi_{a_0}^{-1}(\rho^{\circ n_0}(\mathcal{K}))$. In particular, the critical value of $F_{a_0}$ lies in the tiling set $T_{a_0}^\infty$. This contradicts the assumption that $\{a_k\}_k$ accumulates on $\cC(\mathcal{S})$, and proves that $\{\pmb{\Psi}(a_k)\}_k$ must accumulate on the boundary of $\D_2$.
\begin{figure}[ht!]
\captionsetup{width=0.96\linewidth}
\centering
\includegraphics[scale=0.25]{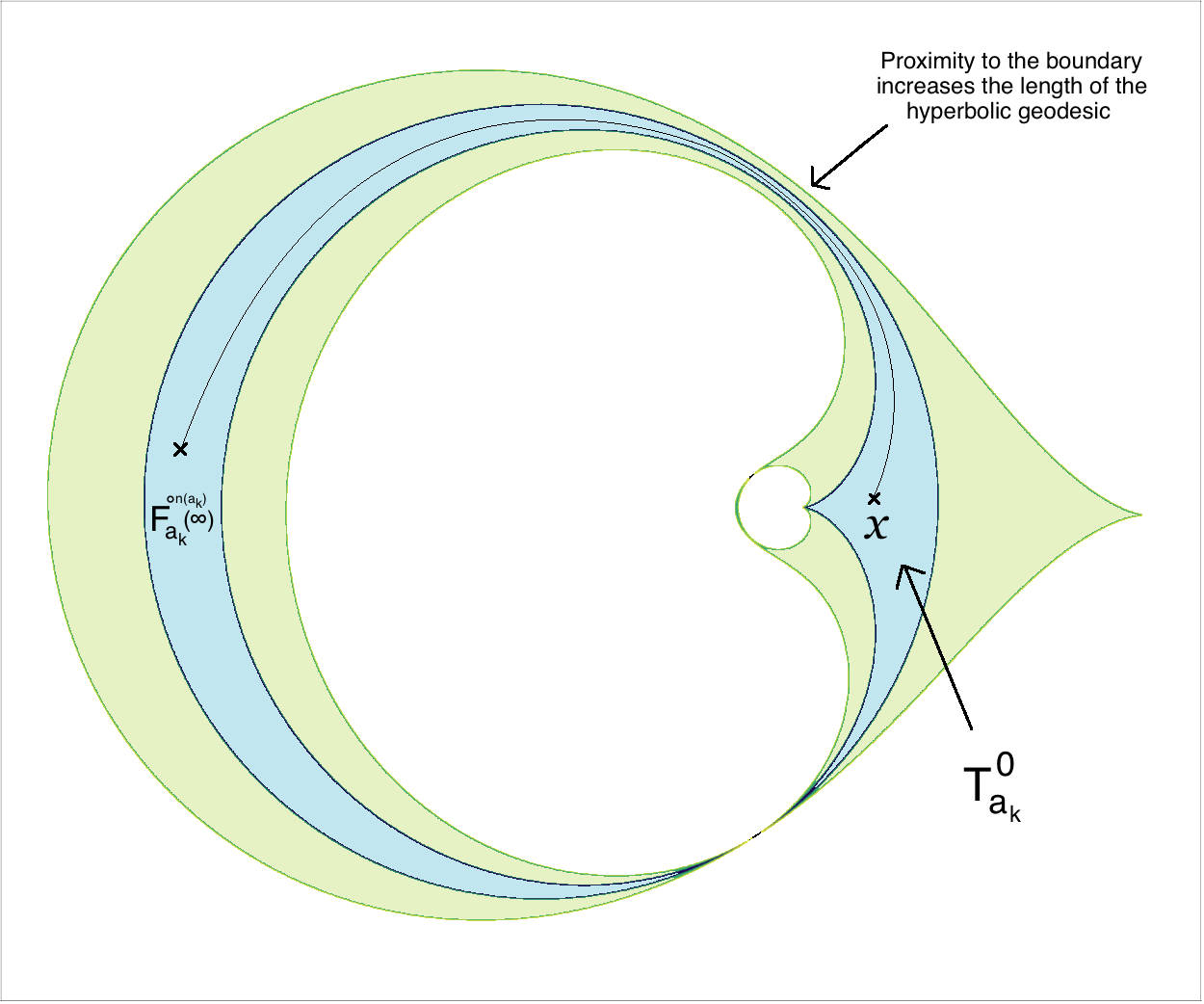}
\caption{For a non-real parameter $a_k$ sufficiently close to $\left(-\infty,-1/12\right)$, the fundamental tile $T_{a_k}^0$ of $F_{a_k}$ has a very narrow channel. Hence, a part of the hyperbolic geodesic of $U_k$ (the shaded domain in the figure) connecting $x$ and $F_{a_k}^{\circ n(a_k)}(\infty)$ lies close to the boundary of $U_k$. As a consequence, the hyperbolic length of this geodesic is very large.}
\label{long_geodesic}
\end{figure}

Finally, let $\{a_k\}_k$ be a sequence of parameters converging to some parameter $a_0$ in $\left(-\infty,-1/12\right)$. Recall that $n(a_k)$ is the smallest positive integer such that $F_{a_k}^{\circ n(a_k)}(\infty)\in T_{a_k}^0$. For $k$ sufficiently large, $F_{a_k}$ is a small perturbation of $F_{a_0}$. Hence by Lemma~\ref{slit_negative}, we have that $\mathrm{Re}(F_{a_k}^{\circ n(a_k)}(\infty))<-\frac{3}{4}$ and $\mathrm{Im}(F_{a_k}^{\circ n(a_k)}(\infty))\approx 0$ (compare Figure~\ref{long_geodesic}). On the other hand, we can choose a point $x\in\mathrm{hull}(\heartsuit)\setminus\heartsuit$ (where `hull' stands for convex hull) such that $x\in\Int{T_{a_k}^0}$, and $\psi_{a_k}(x)$ is contained in a fixed compact subset of $\Int{\Pi}$ for all $k$. We now consider the domain $U_k:=~\Int{(T_{a_k}^0\cup F_{a_k}^{-1}(T_{a_k}^0))}$, which contains both $F_{a_k}^{\circ n(a_k)}(\infty)$ and $x$. Observe that a part of the hyperbolic geodesic (in $U_k$) connecting $x$ and $F_{a_k}^{\circ n(a_k)}(\infty)$ passes through a very thin channel (whose thickness decreases as $k$ increases and gets pinched in the limit), and lies extremely close to the boundary of $U_k$ (compare Figure~\ref{long_geodesic}). Therefore, the hyperbolic distance (in $U_k$) between $x$ and $F_{a_k}^{\circ n(a_k)}(\infty)$ goes to $+\infty$ as $k$ tends to $+\infty$. Since $\psi_{a_k}$ is a conformal isomorphism between $U_k$ and $\Int{(\Pi\cup\rho^{-1}(\Pi))}$, it follows that  the hyperbolic distance (in $\Int{(\Pi\cup\rho^{-1}(\Pi))}$) between $\psi_{a_k}(F_{a_k}^{\circ n(a_k)}(\infty))$ and $\psi_{a_k}(x)$ goes to $+\infty$ as $k$ tends to $+\infty$. Consequently, $\psi_{a_k}(F_{a_k}^{\circ n(a_k)}(\infty))$ converges to the boundary $\partial (\Pi\cup\rho^{-1}(\Pi))$ as $k$ tends to $+\infty$. But $\psi_{a_k}(F_{a_k}^{\circ n(a_k)}(\infty))$ is contained in $\Pi$. Therefore, $\psi_{a_k}(F_{a_k}^{\circ n(a_k)}(\infty))$ must converge to one of the (non-trivial) third roots of unity as $k$ tends to $+\infty$. In fact, Lemma~\ref{slit_negative} implies that for $k$ sufficiently large, each $\psi_{a_k}(F_{a_k}^{\circ n}(\infty))$ ($1\leq n\leq n(a_k)$) is close to one of the non-trivial third roots of unity, and hence the same is true for $\pmb{\Psi}(a_k)=\psi_{a_k}(\infty)$. It follows that $\pmb{\Psi}(a_k)$ converges to one of the non-trivial third roots of unity (which lies on $\partial\D\cap\partial\D_2$) as $k$ tends to $+\infty$. 

This completes the proof of the fact that $\pmb{\Psi}$ is a proper map.
\end{proof}

\begin{lemma}\label{prop_loc_inv}
The map $\pmb{\Psi}$ is injective.
\end{lemma}
\begin{proof}
Suppose that there exist parameters $a_1, a_2\in\C\setminus\left((-\infty,-1/12)\cup\cC(\mathcal{S})\right)$ such that $\pmb{\Psi}(a_1)=\pmb{\Psi}(a_2)$. This implies that, in particular, they have the same depth; i.e., $n(a_1)=n(a_2)$. We denote the common depth by $n_0$.

According to Proposition~\ref{schwarz_group}, for $i\in\{1,2\}$, there exist conformal maps 
$$
\psi_{a_i}:\displaystyle\bigcup_{k=0}^{n_0} F_{a_i}^{-k}(T_{a_i}^0)\longrightarrow \displaystyle\bigcup_{k=0}^{n_0} \rho^{-k}(\Pi)
$$ 
conjugating
$$
F_{a_i}:\displaystyle\bigcup_{k=1}^{n_0} F_{a_i}^{-k}(T_{a_i}^0)\longrightarrow\displaystyle\bigcup_{k=0}^{n_0-1} F_{a_i}^{-k}(T_{a_i}^0)\quad \mathrm{to}\quad \rho:\displaystyle\bigcup_{k=1}^{n_0} \rho^{-k}(\Pi)\longrightarrow\displaystyle\bigcup_{k=0}^{n_0-1}\rho^{-k}(\Pi).
$$ 
We define a conformal homeomorphism 
$$
\psi_{a_1}^{a_2}:=\psi_{a_2}^{-1}\circ\psi_{a_1}:\displaystyle\bigcup_{k=0}^{n_0} F_{a_1}^{-k}(T_{a_1}^0)\longrightarrow \bigcup_{k=0}^{n_0} F_{a_2}^{-k}(T_{a_2}^0)
$$
that conjugates $F_{a_1}$ to $F_{a_2}$. Note that $\psi_{a_1}^{a_2}$ extends continuously to the cusp $\frac14$ of $\partial T^0_{a_1}$, and sends it to the cusp $\frac14$ of $\partial T^0_{a_2}$. Moreover, the assumption that $\psi_{a_1}(\infty)=\psi_{a_2}(\infty)$ implies that $\psi_{a_1}^{a_2}$ maps the unique critical value $\infty$ of $F_{a_1}$ (with associated simple critical point at $0$) to that of $F_{a_2}$. Hence, the map $\psi_{a_1}^{a_2}$ can be lifted via the maps $F_{a_1}, F_{a_2}$ to a conformal homeomorphism from $\displaystyle\bigcup_{k=0}^{n_0+1} F_{a_1}^{-k}(T_{a_1}^0)$ onto $\displaystyle\bigcup_{k=0}^{n_0+1} F_{a_2}^{-k}(T_{a_2}^0)$ such that it maps $\frac14$ to $\frac14$. Due to the equivariance property of $\psi_{a_1}^{a_2}$ (with respect to $F_{a_1}, F_{a_2}$) and the normalization of the lift, this lifted map agrees with $\psi_{a_1}^{a_2}$ on their common domain of definition. Abusing notation, we denote this extension of $\psi_{a_1}^{a_2}$ by
$$
\psi_{a_1}^{a_2}:\displaystyle\bigcup_{k=0}^{n_0+1} F_{a_1}^{-k}(T_{a_1}^0)\longrightarrow \bigcup_{k=0}^{n_0+1} F_{a_2}^{-k}(T_{a_2}^0).
$$
We also note that $\psi_{a_1}^{a_2}$ sends the unique critical point $0$ of $F_{a_1}$ to the unique critical point $0$ of $F_{a_2}$.

By \cite[Proposition~5.9]{LLMM1}, the curves $\partial\heartsuit$ and $\partial B(a_i,r_{a_i})$ have a simple tangency at $\alpha_{a_i}$, for $i\in\{1,2\}$. Moreover, $\frac14$ is a $(3,2)$-cusp on $\partial\heartsuit$.
Hence, Lemmas~\ref{asymp_lin_1_lem} and~\ref{asymp_lin_2_lem} and conformality of $F_{a_i}$ at the iterated preimages of $\alpha_{a_i},\frac14$ imply that the conformal map $\psi_{a_1}^{a_2}$ is asymptotically linear near the singular points $\alpha_{a_1}, \frac14$ and at their iterated preimages under $F_{a_1}$. The same results also show that $\psi_{a_1}^{a_2}$ extends as a quasiconformal homeomorphism $\mathfrak{g}_1$ of $\widehat{\C}$ (cf. Lemma~\ref{qc_conjugacy_alpha}). By construction, $\mathfrak{g}_1$ conjugates $F_{a_1}$ to $F_{a_2}$ on $\displaystyle\bigcup_{k=1}^{n_0+1} F_{a_1}^{-k}(T_{a_1}^0)$, and is conformal on the interior of $\displaystyle\bigcup_{k=0}^{n_0+1} F_{a_1}^{-k}(T_{a_1}^0)$.

With the quasiconformal homeomorphism $\mathfrak{g}_1$ at our disposal, a standard pullback argument as in Proposition~\ref{rigidity_center} (cf. \cite[Theorem~5.1]{LMM1}) can be employed to construct a sequence of $K$-quasiconformal maps $\{\mathfrak{g}_r\}_{r\geq 1}$ of $\widehat{\C}$ such that for all $r\geq 1$,
\begin{enumerate}
\item\label{lift_cond} $F_{a_2}\circ\mathfrak{g}_{r+1}= \mathfrak{g}_{r}\circ F_{a_1}$ on $\widehat{\C}\setminus\Int{T^0_{a_1}}$,

\item\label{conf_cond} $\mathfrak{g}_r$ is conformal on the interior of $\displaystyle\bigcup_{k=0}^{n_0+r} F_{a_1}^{-k}(T^0_{a_1})$,

\item\label{matching_cond} $\mathfrak{g}_{r+1}=\mathfrak{g}_{r}$ on $\displaystyle\bigcup_{k=0}^{n_0+r} F_{a_1}^{-k}(T^0_{a_1})$, and

\item $\mathfrak{g}_r(0)=0,\ \mathfrak{g}_r(\infty)=\infty,\ \mathfrak{g}_r(\frac14)=\frac14$. 
\end{enumerate}
By compactness of the family of $K$-quasiconformal homeomorphisms and Properties~\eqref{lift_cond},~\eqref{matching_cond} of the sequence $\{\mathfrak{g}_r\}$, there exists a quasiconformal homeomorphism $\mathfrak{g}_\infty$ of $\widehat{\C}$ that conjugates $F_{a_1}$ to $F_{a_2}$ on their tiling sets. Since the non-escaping sets $K_{a_1}, K_{a_2}$ are Cantor sets (see Proposition~\ref{schwarz_group}), it follows (by continuity of $\mathfrak{g}_\infty$) that the conjugacy relation holds on the entire domains of definition of $F_{a_1}, F_{a_2}$. Also, Condition~\eqref{conf_cond} implies that $\mathfrak{g}_\infty$ is conformal on the tiling set $T_{a_1}^\infty$. 
 
Moreover, the proof of Theorem~\ref{geom_finite_limit_set} (part 3) applies mutatis mutandis to Schwarz reflection maps in the escape locus $\C\setminus\left((-\infty,-1/12)\cup\cC(\mathcal{S})\right)$ and shows that the non-escaping sets $K_{a_1}, K_{a_2}$ have zero area.
It now follows by Weyl's lemma that $\mathfrak{g}_\infty$ is a M{\"o}bius map of $\widehat{\C}$. Finally, since the M{\"o}bius conjugacy $\mathfrak{g}_\infty$ fixes $0$, $\frac14$, and $\infty$, it must be the identity map. Hence, $F_{a_1}$ and $F_{a_2}$ are conjugate via the identity map; i.e., $a_1=a_2$. 
\end{proof}

The theorem now readily follows from the previous two lemmas.
\end{proof}

As an immediate application of Theorem~\ref{thm_unif_exterior_conn_locus}, we can define parameter tiles that yield a tessellation of the escape locus (see Figure~\ref{conn_locus_tessellation}).

\begin{definition}\label{para_tiles}
For an $M$-admissible word $(i_1,\cdots,i_k)$ with $i_1=2$, the \emph{parameter tile} $\mathfrak{T}^{i_1,\cdots,i_k}$ is defined as $$\mathfrak{T}^{i_1,\cdots,i_k}:=\pmb{\Psi}^{-1}(T^{i_1,\cdots,i_k}).$$
\end{definition}

Finally, we define external parameter rays of $\mathcal{S}$ via the map $\pmb{\Psi}$. 

\begin{definition}\label{para_ray_schwarz}
The pre-image of a $\mathcal{G}$-ray at angle $\theta$ (where $\theta\in(1/3,2/3)$) under the map $\pmb{\Psi}$ is called a \emph{$\theta$-parameter ray} of $\mathcal{S}$.
\end{definition}

\begin{remark}
It follows from the proof of properness of $\pmb{\Psi}$ that every parameter ray of $\mathcal{S}$ accumulates on $\cC(\mathcal{S})$; i.e., none of them accumulates on the slit $(-\infty,-1/12)$.
\end{remark}

The set of all parameter rays of $\mathcal{S}$ form a binary tree (compare Figure~\ref{tessellation_pic}). Note that if $a$ lies on a parameter ray at angle $\theta$, then in the dynamical plane of $F_a$, the critical value $\infty$ lies on a dynamical ray at angle $\theta$. This duality will play an important role in the rest of the paper.

\section{Hyperbolic components}\label{hyperbolic}

We now discuss the structure of hyperbolic parameters in $\cC(\mathcal{S})$.

Since $F_a$ depends real-analytically on $a$, a straightforward application of the Implicit Function Theorem shows that attracting periodic points can be locally continued as real-analytic functions of $a$. Hence, the set of hyperbolic parameters form an open set. A connected component of the set of all hyperbolic parameters is called a \emph{hyperbolic component}. It is easy to see that every hyperbolic component $H$ has an associated positive integer $n$ such that each parameter in $H$ has an attracting cycle of period $n$. We refer to such a component as a hyperbolic component of period $n$.

A \emph{center} of a hyperbolic component is a parameter $a$ for which $F_a$ has a super-attracting periodic cycle; i.e., the unique critical point $0$ is periodic. 

If $F_a$ has an attracting periodic cycle, then the critical point $0$ of $F_a$ is attracted by the attracting cycle (see Proposition~\ref{fatou_critical}). Moreover, we can associate a dynamically defined conformal invariant to every hyperbolic map $F_a$; namely multiplier if the attracting cycle (of $F_a$) has even period, and Koenigs ratio if the attracting cycle (of $F_a$) has odd period (see Subsection~\ref{anti_poly_dyn_subsubsec} for the corresponding definitions for anti-polynomials). 

The hyperbolic components in $\cC(\mathcal{S})$ are parametrized by the Blaschke product spaces $\mathcal{B}^{\pm}$, which model the first return map of the dynamics to the connected component of $\textrm{int}(K_a)$ containing $0$. The following theorem describes the topology and dynamical uniformizations of hyperbolic components in $\cC(\mathcal{S})$.

\begin{theorem}[Dynamical uniformization of hyperbolic components]\label{unif_hyp_schwarz}
Let $H$ be a hyperbolic component in $\cC(\mathcal{S})$.
\begin{enumerate}
\item If $H$ is of odd period, then there exists a homeomorphism $\widetilde{\eta}_H:H\to\mathcal{B}^-$ that respects the Koenigs ratio of the attracting cycle. In particular, the Koenigs ratio map is a real-analytic $3$-fold branched covering from $H$ onto the open unit disk, ramified only over the origin.

\item If $H$ is of even period, then there exists a homeomorphism $\widetilde{\eta}_H:H\to\mathcal{B}^+$ that respects the multiplier of the attracting cycle. In particular, the multiplier map is a real-analytic diffeomorphism from $H$ onto the open unit disk.
\end{enumerate}
In both cases, $H$ is simply connected and has a unique center.
\end{theorem}
\begin{proof}
See \cite[Theorem~5.6, Theorem~5.9]{NS} for a proof of the corresponding facts for quadratic anti-polynomials. It is straightforward to adapt the proof in our case. The main idea is to change the conformal dynamics of the first return map of a periodic Fatou component. More precisely, one can glue any Blaschke product belonging to the family $\mathcal{B}^{\pm}$ in the connected component of $\textrm{int}(K_a)$ containing $0$ by quasiconformal surgery. This gives the required homeomorphism between $H$ and $\mathcal{B}^{\pm}$.

However, there is an important detail here. Since the original dynamics $F_a$ is modified only in a part of the connected component of $\textrm{int}(K_a)$ containing $0$ (this is precisely where an iterate of $F_a$ is replaced by a Blaschke product), the resulting quasiregular modification $G_a$ shares some of the mapping properties of $F_a$. In particular, $G_a$ sends $\overline{B}(a,r_a)^c$ to $B(a,r_a)$ and maps $G_a^{-1}(\heartsuit)$ to $\heartsuit$ as a univalent map. Hence, we can adapt the proof of Lemma~\ref{schwarz_qcdef} to show that $G_a$ is quasiconformally conjugate to some map $F_b$ in our family $\mathcal{S}$.  
\end{proof}

Using Relation~\eqref{schwarz_cardioid}, it is easy to compute the centers of some low period hyperbolic components of $\cC(\mathcal{S})$.

\textbf{E}xamples. i) Since $F_a(0)=\infty$, there is no super-attracting map in $\mathcal{S}$ with a fixed critical point. Hence, there is no hyperbolic component of period $1$ in $\cC(\mathcal{S})$. 

\begin{figure}[ht!]
\captionsetup{width=0.96\linewidth}
\centering
\includegraphics[scale=0.41]{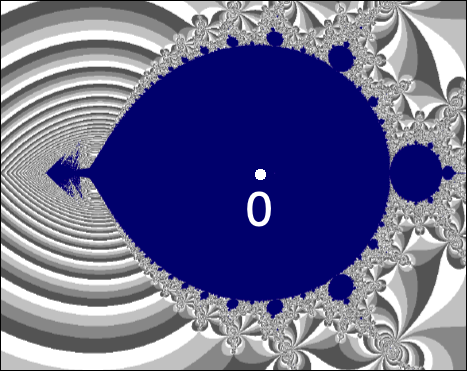} \includegraphics[scale=0.175]{schwarz_basilica.png}
\caption{Left: The hyperbolic component of period two (in blue) with its center $0$ marked. Right: The part of the non-escaping set of $F_0$ inside the cardioid (in dark blue) with the critical point $0$ marked.}
\label{schwarz_basilica_2}
\end{figure}

ii) The unique parameter with a super-attracting $2$-cycle in $\cC(\mathcal{S})$ is $a=0$. Indeed, the critical orbit of the map $F_0$ is given by $0\leftrightarrow\infty$. It is not hard to see that the pullbacks of the leaf joining $1/3$ and $2/3$ (under $\rho$) generate the pre-periodic lamination of the corresponding non-escaping set $K_0$. Hence, its pre-periodic lamination is homeomorphic to the rational lamination of the Basilica anti-polynomial $\overline{z}^2-1$. 

In particular, $0$ is the center of the unique hyperbolic component of period two of $\cC(\mathcal{S})$ (see Figure~\ref{schwarz_basilica_2}).

iii) The center of the unique period $3$ hyperbolic component is $\frac{3}{16}$ (see Figure~\ref{airplane_dynamical} and Figure~\ref{hyp_three}). Its critical orbit is given by $0\mapsto\infty\mapsto3/16\mapsto0$.

\begin{figure}[ht!]
\captionsetup{width=0.96\linewidth}
\centering
\includegraphics[scale=0.194]{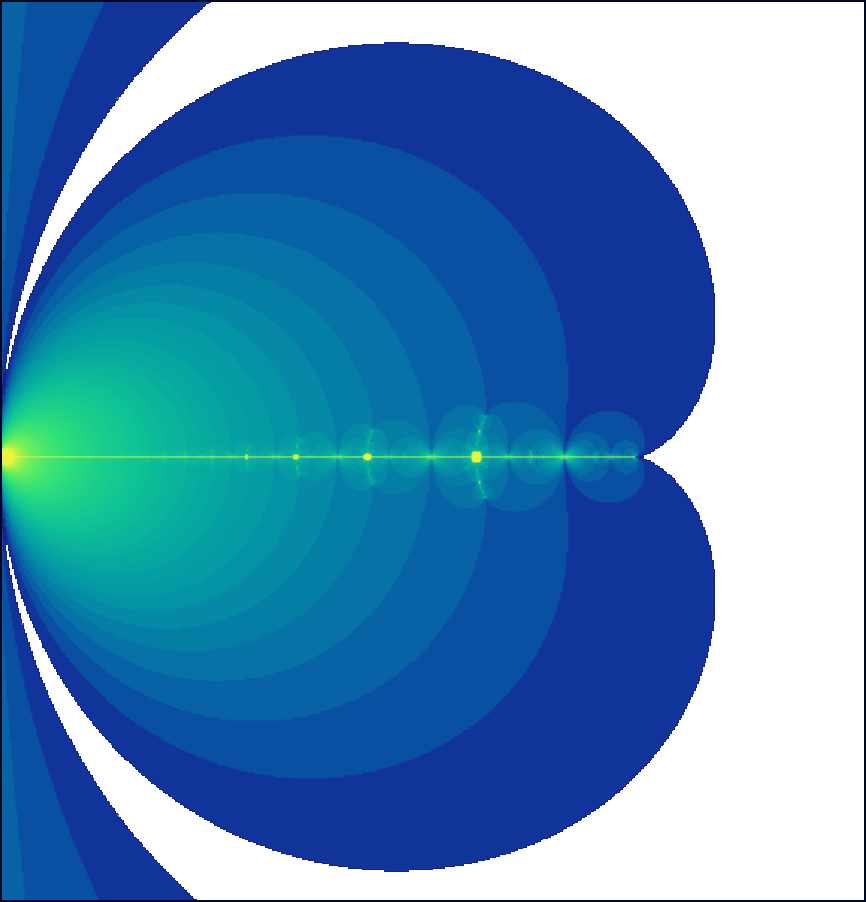} \includegraphics[scale=0.22]{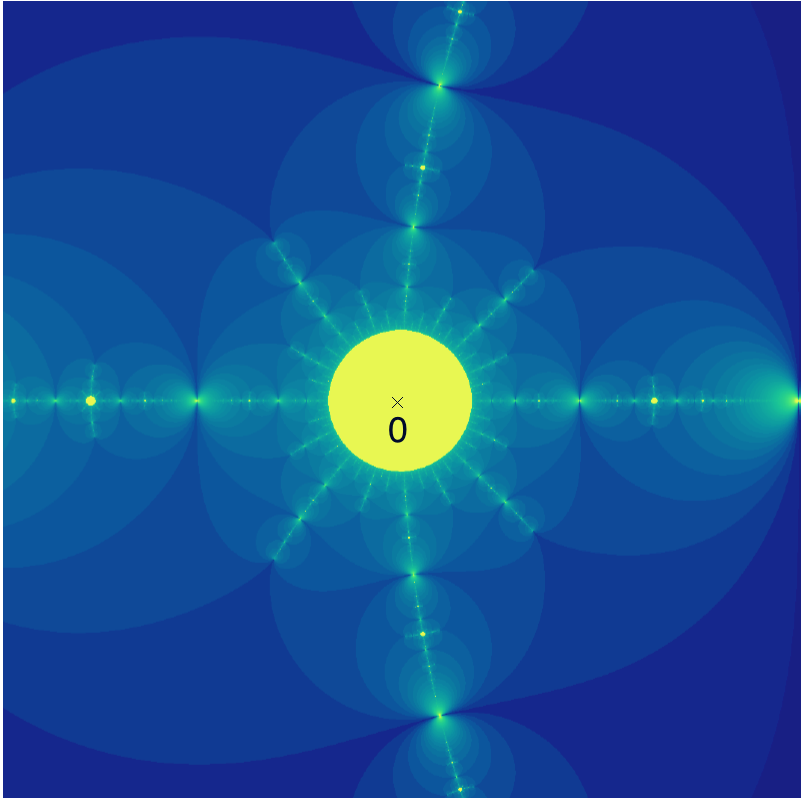}
\caption{Left: A part of the non-escaping set of $a=3/16$. The corresponding map has a super-attracting $3$-cycle. Right: A blow-up of the same non-escaping set around the Fatou component containing the critical point $0$.}
\label{airplane_dynamical}
\end{figure}

Similarly, the centers of the two period $4$ hyperbolic components are $\frac{2}{9}$ (primitive) and $\frac{\sqrt{52}-5}{36}$ (satellite).

\begin{figure}[ht!]
\captionsetup{width=0.96\linewidth}
\centering
\includegraphics[scale=0.185]{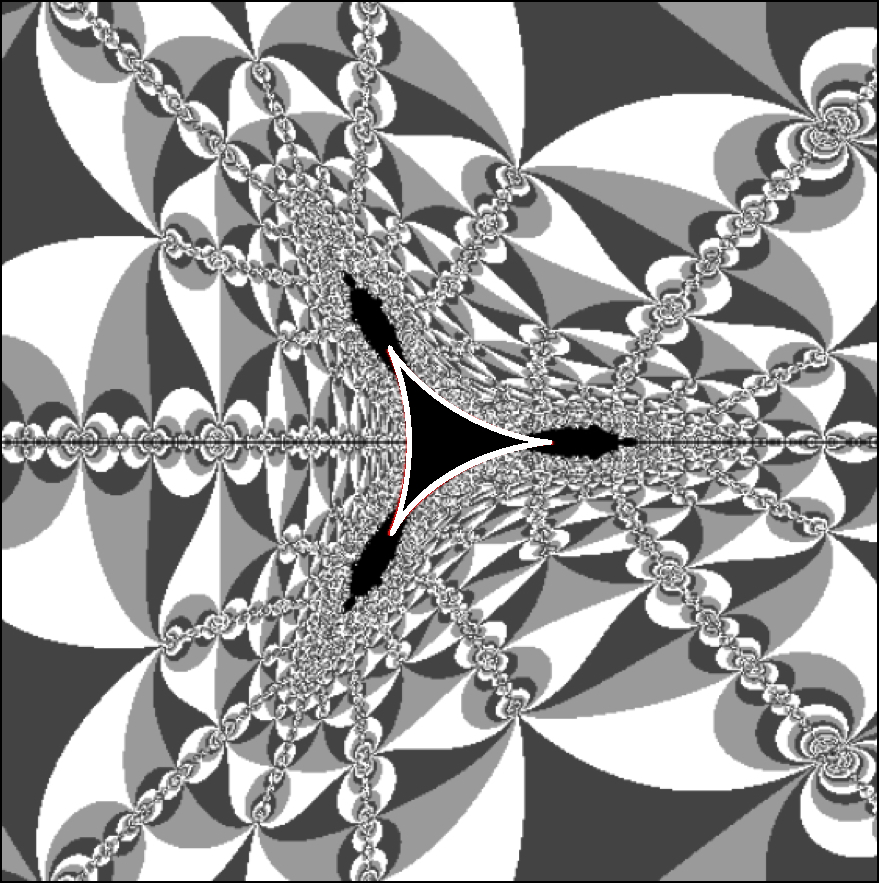} \includegraphics[scale=0.18]{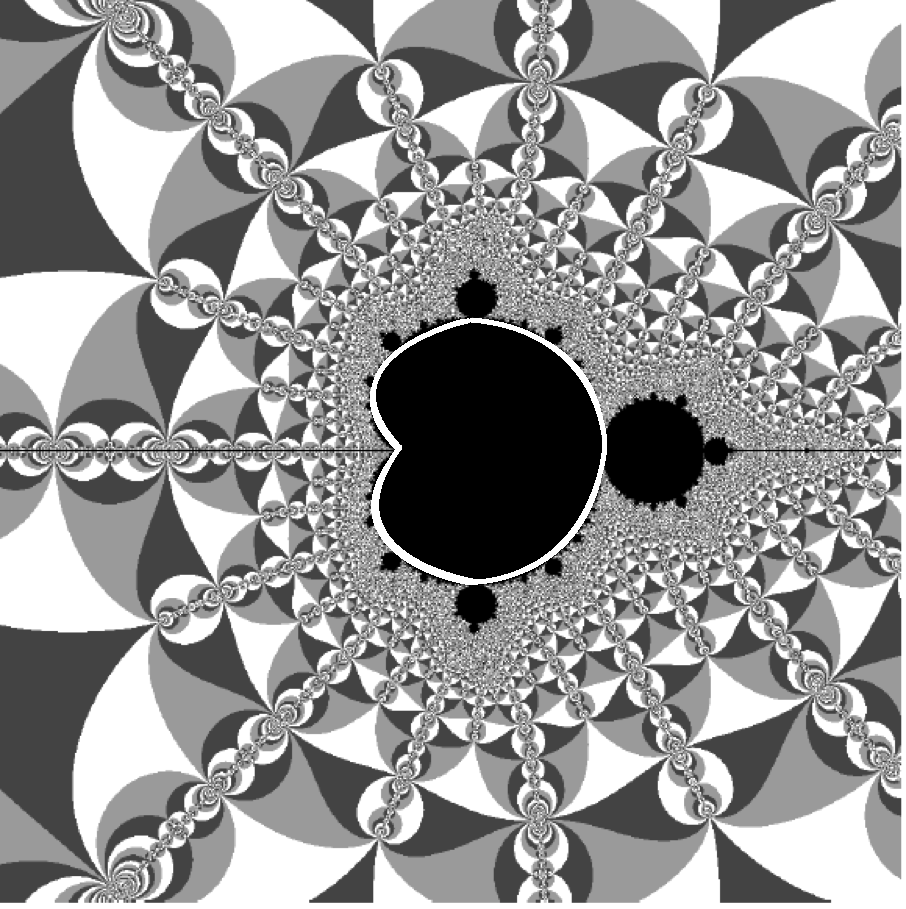}
\caption{Left: The region enclosed by the deltoid-shaped curve is the unique hyperbolic component of period $3$ of $\cC(\mathcal{S})$ with its center at $3/16$. There are bifurcations to period $6$ hyperbolic components at the ends of each parabolic arc. Right: The region enclosed by the cardioid-shaped curve is a hyperbolic component of period $4$ centered at $2/9$.}
\label{hyp_three}
\end{figure}

We will conclude this subsection with a brief description of neutral parameters and boundaries of hyperbolic components of  \emph{even} period of $\cC(\mathcal{S})$. The following proposition states that every neutral (in particular, parabolic) parameter lies on the boundary of a hyperbolic component of the same period.

\begin{proposition}[Neutral parameters on boundary]\label{ThmIndiffBdyHyp_schwarz} 
If $F_{a_0}$ has a neutral periodic point of period $k$, then every neighborhood of $a_0$ contains parameters with attracting periodic points of period $k$, so the parameter $a_0$ is on the  boundary of a hyperbolic component of period $k$ of $\cC(\mathcal{S})$. 
\end{proposition}
\begin{proof}
See \cite[Theorem~2.1]{MNS} for a proof in the Tricorn family. Since the proof given there only uses local dynamical properties of anti-holomorphic maps near neutral periodic points, it applies to the family $\mathcal{S}$ as well.
\end{proof}

Using Theorem~\ref{unif_hyp_schwarz}, one can define internal rays of hyperbolic components of $\cC(\mathcal{S})$. If $H$ is a hyperbolic component of even period, then the proof of \cite[Lemma~2.19]{IM2} can be adapted to show that all internal rays of $H$ at rational angles land (note that by Proposition~\ref{rigidity_para_even}, the accumulation set of such a ray is necessarily a finite set). If $H$ does not bifurcate from a hyperbolic component of odd period, then the landing point of the internal ray at angle $0$ is a parabolic parameter with an even-periodic parabolic cycle. This parameter is called the \emph{root} of $H$.

The next theorem describes the bifurcation structure of even period hyperbolic components of $\cC(\mathcal{S})$. Once again, its proof in the Tricorn family is given in \cite[Theorem~1.1]{MNS}, which can be easily adapted for our setting.

\begin{theorem}[Bifurcations from even period hyperbolic components]\label{ThmEvenBif_schwarz} 
If $F_a$ has a $2k$-periodic cycle with multiplier 
$e^{2\pi ip/q}$ with $\mathrm{gcd}(p,q)=1$, then the parameter $a$ sits on the boundary of a hyperbolic component of 
period $2kq$ (and is the root thereof) of $\cC(\mathcal{S})$. 
\end{theorem}

\section{Combinatorial rigidity of geometrically finite maps}\label{comb_rigidity_orbit_portraits}

\subsection{Combinatorics of dynamical rays: orbit portraits}\label{comb_orbit_portrait}

We now delve into a combinatorial study of hyperbolic and parabolic maps in $\mathcal{S}$. The following landing property of dynamical rays for hyperbolic and parabolic maps in $\mathcal{S}$ follows from the preparation in Section~\ref{schwarz_circle_cardioid}.

\begin{proposition}[Landing/bifurcation of dynamical rays]\label{rays_land_schwarz}
1) Let $a\in\cC(\mathcal{S})$ and $F_a$ be a hyperbolic or parabolic map. Then, every $M$-admissible sequence of tiles $\{T_a^{i_1},T_a^{i_1,i_2},\cdots\}$ shrinks to a point on $\Gamma_a$. In particular, every dynamical ray of $F_a$ lands at some point of $\Gamma_a$, and conversely.

2) If $a\notin\cC(\mathcal{S})$, then every dynamical ray of $F_a$ either bifurcates or lands at some point on $K_a$. 
\end{proposition}
\begin{proof}
1) This follows from Theorem~\ref{geom_finite_limit_set} and Carath{\'e}odory's theorem on boundary extension of (inverses of) Riemann maps.

2) This follows from the fact that the accumulation set of a ray is connected and that the limit set of a parameter $a\notin\cC(\mathcal{S})$ is a Cantor set (Proposition~\ref{schwarz_group}).

\end{proof}

Clearly, if the dynamical ray of $F_a$ at angle $\theta$ lands at a point $w\in \Gamma_a$, then the dynamical ray at angle $\rho(\theta)$ lands at the point $F_a(w)$.

\subsubsection{Hubbard tree, characteristic angles, and lamination}

\begin{definition}\label{char_para_point}
For a hyperbolic or parabolic map $F_a$, the Fatou component containing the critical value $\infty$ is called the \emph{characteristic Fatou component} of $F_a$. If $F_a$ is parabolic, the parabolic periodic point on the boundary of the characteristic Fatou component is called the \emph{characteristic parabolic point} of $F_a$.
\end{definition}

Suppose that the period of the characteristic Fatou component $U_a$ of a hyperbolic (respectively, parabolic) map $F_a$ is $k$. Then, the sequence of iterates $\{F_a^{\circ kn}\}_n$ forms a normal family on $U_a$. Hence, the $F_a^{\circ k}$-orbit of every point in $U_a$ converges to the attracting (respectively, parabolic) periodic point in $U_a$ (respectively, on $\partial U_a$).

The characteristic Fatou component and the fixed points (of the first return map of the component) on its boundary will be vital in the rest of the section.

\begin{definition}\label{def_root_schwarz}
Let $F_a$ be a hyperbolic (respectively, parabolic) map, and $U_a$ be the characteristic Fatou component. Let $w$ be a boundary point of $U_a$ such that the first return map of $U_a$ fixes $w$. Then we call $w$ a \emph{dynamical root} of $F_a$ if it is a cut-point of $K_a$; otherwise, we call it a \emph{dynamical co-root}.
\end{definition}

We now proceed to define Hubbard trees for super-attracting and parabolic maps.

For a super-attracting map $F_a$, the critical Fatou component $U$ (which has period $k$, say) admits a Riemann map that conjugates the first return map $F_a^{\circ k}$ of $U$ to the map $\overline{z}^2$ on $\mathbb{D}$. Pre-images of radial lines in $\mathbb{D}$ under this Riemann map are called \emph{geodesic rays} in $U$. Pulling the geodesic rays in $U$ back by iterates of $F_a$, we obtain geodesic rays in all the Fatou components of $F_a$. Since the non-escaping set $K_a$ of a hyperbolic map $F_a$ is a locally connected full continuum (see Theorem~\ref{geom_finite_limit_set}), it follows that there exists a unique arc in $K_a$ connecting any two points (of $K_a$) such that the intersection of the arc with every Fatou component is contained in the union of two geodesic rays (compare \cite[Expos{\'e}~II, \S 6, Proposition 6]{orsay}). Such arcs are called \emph{allowable}. The union of the allowable arcs connecting the post-critical set of $F_a$ is a tree, and we call it the \emph{Hubbard tree} of the super-attracting map $F_a$.
 
For a parabolic map $F_a$, one can similarly define a tree in $K_a$ which connects the post-critical set and the parabolic cycle. Note that the tree obtained this way is not uniquely defined in the Fatou components, but they are all homotopic relative to the limit set. Such a tree is called a Hubbard tree of the parabolic map $F_a$ (also compare \cite[Definition~5.4]{HS}).

In either case, the Hubbard tree is an $F_a$-invariant (up to homotopy relative to $\Gamma_a$ in the parabolic case) finite tree in $K_a$ that connects the post-critical orbit (and the parabolic cycle if $F_a$ is parabolic). We denote this tree by $\mathcal{H}_a$. Following the arguments of \cite[Expos{\'e}~IV, \S 4, Proposition~4]{orsay} (in the super-attracting case) or \cite[Lemma~3.5]{S1a} (in the parabolic case), it is easy to see that the critical value $\infty$ is an endpoint of the tree, and every branch points of the tree is a pre-periodic repelling point.

One can now adapt the proof of \cite[Lemma~3.4, Corollary~4.2]{NS1} to deduce the following.

\begin{proposition}\label{prop_root_schwarz}
Let $F_a$ be a super-attracting or parabolic map. Then, every dynamical co-root of $F_a$ has the same exact period as its Fatou component. It is the landing point of exactly one dynamic ray, and this ray has the same exact period as the component.

The characteristic Fatou component $U_a$ of $F_a$ has exactly one dynamical root. If the period of $U_a$ is even, then $U_a$ has no co-root; if the period is odd, it has exactly two co-roots. Moreover, if the period of $U_a$ is an odd integer $k$, then its dynamical root is the landing point of exactly two rays at $2k$-periodic angles.
\end{proposition}

The angles of the two adjacent rays landing at the dynamical root of $F_a$ (bounding a sector of angular width less that $\frac12$) and separating the critical point $0$ from the critical value $\infty$ are called the \emph{characteristic angles} of $F_a$. The hyperbolic geodesic in $\mathbb{D}$ terminating at these two angles is referred to as the \emph{characteristic geodesic}.

The characteristic angles of a hyperbolic or parabolic map $F_a$ will play a crucial role in the combinatorial study of the connectedness locus $\cC(\mathcal{S})$. In particular, we will show that that the pre-periodic lamination of a hyperbolic or parabolic map $F_a$ (which yields a topological model of the map) can be recovered from its characteristic angles. 

We start with a preliminary statement.

\begin{lemma}\label{basin_dense_tree}
Let $F_a$ be a super-attracting map, and $\mathcal{H}_a$ the Hubbard tree of $F_a$. Then, $\Int{K_a}\cap\mathcal{H}_a$ is dense in $\mathcal{H}_a$.
\end{lemma}
\begin{proof}
Let us denote the super-attracting cycle of $F_a$ by $\mathbf{0}$.

\noindent\textbf{Claim:} The union of $\Int{K_a}\cap\mathcal{H}_a$ and the iterated pre-images of $\alpha_a$ on $\mathcal{H}_a$ is dense on $\mathcal{H}_a$.

\noindent\textbf{Proof of claim:} If it were false, then there would be an arc $\gamma_0\in\mathcal{H}_a$ such that none of the images $F_a^{\circ n}(\gamma_0)$ contains $\alpha_a$ or intersects the immediate basin of the critical point $0$. 

We consider the hyperbolic metric on each of the two connected components of $$V:=\widehat{\C}\setminus\left(T_a\cup\mathbf{0}\right),$$ and define the corresponding ``hyperbolic metric" on each connected component $J$ of $\mathcal{H}_a\setminus\{\mathbf{0},\alpha_a\}$ as follows. For $x,y\in J$, a smooth path $\gamma$ connecting $x$ to $y$ in $V$ is called \emph{admissible} if it can be retracted to $\left[x,y\right]\subset\mathcal{H}_a$ in $V$ relative to the end-points (clearly, $\gamma$ must be contained in a connected component of $V$). Then we let $$\displaystyle d_{\mathrm{hyp}}(x,y):=\inf_\gamma \ell_{\mathrm{hyp}}(\gamma),$$ where $\gamma$ runs over all admissible paths connecting $x$ to $y$, and $\ell_{\mathrm{hyp}}$ denotes hyperbolic length in $V$.

Note that if the images $x'=F_a(x)$ and $y'=F_a(y)$ also lie in the same component $J'$ of $\mathcal{H}_a\setminus\{\mathbf{0},\alpha_a\}$, then any admissible path $\gamma'$ connecting $x'$ to $y'$ lifts by $F_a$ to an admissible path $\gamma$ connecting $x$ to $y$. Moreover, $\gamma$ is contained in $\widehat{\C}\setminus\overline{E_a^1}\subset V$. It follows that $\ell_{\mathrm{hyp}}(\gamma')>\ell_{\mathrm{hyp}}(\gamma)$. Hence, 
\begin{equation}
d_{\mathrm{hyp}}(x',y')\geq c\cdot d_{\mathrm{hyp}}(x,y),
\label{expansion}
\end{equation}
with $c=c(\epsilon)>1$ provided $x'$ and $y'$ stay $\epsilon$-away from $\mathbf{0}$, $\frac14$, and $\alpha_a$.
 
Note that by our assumption on $\gamma_0$, all iterates of $\gamma_0$ stay $\epsilon$-away from $\mathbf{0}$. Moreover, no iterate of $\gamma_0$ contains $\alpha_a$. Since $\alpha_a$ repels nearby points in $K_a$, it follows that infinitely many iterates of $\gamma_0$ stay $\epsilon$-away from $\alpha_a$. Finally, as $\gamma_0$ is contained in the Hubbard tree, all iterates of $\gamma_0$ are bounded $\epsilon$-away from $\frac14$ as well. It now follows by Relation~\eqref{expansion} that $d_{\mathrm{hyp}}(F_a^{\circ n}(x),F_a^{\circ n}(y))\to\infty$, for $x,y\in\gamma_0$. But this is impossible as infinitely many iterates of $\gamma_0$ are uniformly bounded away from $\partial V$. This completes the proof of the claim.
\bigskip
 
To complete the proof of the lemma, it now suffices to argue that $\alpha_a$ is in the closure of $\Int{K_a}\cap\mathcal{H}_a$. We will assume the contrary, and arrive at a contradiction. To this end, let us choose a sufficiently small arc $I_0\subset\mathcal{H}_a$ containing $\alpha_a$ and not intersecting $\Int{K_a}$, and define $I_{n}:=F_a^{\circ 2n}\left(I_0\right)$. Since $\alpha_a$ repels nearby points in $K_a$ and $F_a^{\circ 2n}$ ($n\geq1$) is injective on $I_0$, the family $\{I_n\}_n$ is a strictly increasing sequence of arcs in $\mathcal{H}_a$. In particular, the (right) end-points of $I_n$ form a strictly monotone sequence in the compact set $\mathcal{H}_a$, and hence must converge to an attracting fixed point of $F_a^{\circ 2}$. This is clearly impossible unless $a=0$ (which is the only map with a super-attracting $2$-cycle). Moreover, for $a=0$, the above argument shows that the arcs $I_n$ come arbitrarily close to $\mathbf{0}$; i.e., $I_n\cap\Int{K_a}\neq\emptyset$ for $n$ large. As $\Int{K_a}$ is completely invariant, it follows that $I_0\cap\Int{K_a}\neq\emptyset$ as well.
\end{proof}

Recall that the piecewise reflection map $\rho$ was defined in Section~\ref{ideal_triangle}.

\begin{proposition}[Characteristic angles determine lamination]\label{char_angles_lamination}
Let $a$ be a super-attracting parameter. Then, the iterated pre-images of the characteristic geodesic under $\rho$ are pairwise disjoint, and their closure in $\Q/\Z$ is the pre-periodic lamination $\lambda(F_a)$.
\end{proposition}
\begin{proof}
Let $\mathfrak{S}_a$ be an arc in $K_a$ connecting $\frac14$ to $\sigma_a^{-1}(\frac14)$ (it is defined up to homotopy in $\Int{K_a}$). We call $\mathfrak{S}_a$ the \emph{spine} of $F_a$. 

Since $\rho\vert_{\mathbb{T}}$ is expansive and $\Gamma_a$ is locally connected, the arguments of \cite[Proposition~24.15]{L6} can be easily adapted for our setting to show that every cut-point of $\Gamma_a$ eventually falls on the spine of $F_a$. Moreover, every point on the spine eventually falls on the Hubbard tree $\mathcal{H}_a$ of $F_a$.

Therefore, we only need to show that every leaf of the lamination $\lambda(F_a)$ consisting of a pair of rays landing at a cut-point on the Hubbard tree can be approximated by iterated pre-images of the characteristic leaf. We will proceed as in \cite[Theorem~25.58]{L6}.

We denote the critical Fatou component by $U_0$. Let us consider a leaf $L$ of the lamination $\lambda(F_a)$ comprising a pair of dynamical rays landing at the same cut point $w\in\mathcal{H}_a$ and bounding a minimal sector $S$ (so there are no other rays in this sector landing at $w$). By Proposition~\ref{basin_dense_tree}, $w$ can be approximated by components $U_{-k}$ of the basin $\Int{K_a}$ such that $F^{\circ n_k}(U_{-k})=U_0$. There are two possibilities:

1) $w\in\partial U_{-k}$ for some component $U_{-k}\subset S$. Iterating forward, we can assume that $w\in\partial U_0$. But since $0$ is not a branch point of $\mathcal{H}_a$, the immediate basin $U_0$ intersects the Hubbard tree $\mathcal{H}_a$ at most at two points. Moreover, one of these two points $p$ is fixed under the first return map of $\overline{U}_0$, and the other an iterated pre-image of $p$. Hence it is sufficient to consider $w=p$. But then $L$ itself is the characteristic leaf.

2) There is a sequence of components $U_{-k}\subset S$ (such that $n_k\to\infty$) converging to $w$. Let $L_{-k}$ be the corresponding pre-image of the characteristic leaf (landing at iterated pre-image of $p$ on $\partial U_{-k}$). As $k\to\infty$, these curves converge to some curve $L_{-\infty}$ comprising two rays in $\overline{S}$ landing at $w$. But $L$ is the only such curve, so $L_{-\infty}=L$.

In either case, $L$ is approximated by pre-images of the characteristic leaf.
\end{proof}

\subsubsection{Orbit portraits}\label{orb_port_schwarz_subsubsec}

Following Definition~\ref{def_orbit_portrait_anti_quad}, we now introduce the notion of orbit portraits for the Schwarz reflection maps $F_a$.

For a cycle $\mathcal{O} = \{w_1,$ $w_2 ,$ $\cdots,$ $w_p\}$ of $F_a$, let $\mathcal{A}_i$ be the set of angles of dynamical rays landing at $w_i$. The collection $\mathcal{P} = \lbrace \mathcal{A}_1 , \mathcal{A}_2, \cdots, \mathcal{A}_p \rbrace$ is called the \emph{orbit portrait} associated with the periodic orbit $\mathcal{O}$ of $F_a$.

\begin{theorem}\label{complete_antiholomorphic_schwarz}
Let $\mathcal{O} = \lbrace z_1 , z_2 ,\cdots$, $z_p\rbrace$ be a periodic orbit of $F_a$ such that at least one (pre-)periodic dynamical ray (a ray at an angle $\theta\in\mathrm{Per}(\rho)$) lands at some $z_j$, $j\in\{1,\cdots,p\}$. Then the associated orbit portrait $\mathcal{P} = \lbrace \mathcal{A}_1 , \mathcal{A}_2, \cdots, \mathcal{A}_p \rbrace$ (which we assume to be non-trivial; i.e., $\vert\mathcal{A}_i\vert\geq 2$) satisfies the following properties.
\begin{enumerate}
\item Each $\mathcal{A}_j$, $j\in\{1,\cdots,p\}$, is a finite non-empty subset of $\mathrm{Per}(\rho)$.

\item For each $j\in\Z/p\Z$, the map $\rho$ maps $\mathcal{A}_j$ bijectively onto $\mathcal{A}_{j+1}$, and reverses their cyclic order.

\item For each $i\neq j$, the sets $\mathcal{A}_i$ and $\mathcal{A}_j$ are unlinked.

\item Each $\theta \in \mathcal{A}_j$, $j\in\{1,\cdots,p\}$, is periodic under $\rho$, and there are four possibilities for their periods: 
\begin{enumerate}

\item If $p$ is even, then all angles in $\mathcal{P}$ have the same period $rp$ for some $r\geq 1$.

\item If $p$ is odd, then one of the following three conditions holds:
\begin{enumerate}
\item $\vert \mathcal{A}_j\vert = 2$, and both angles have period $p$.

\item $\vert \mathcal{A}_j\vert = 2$, and both angles have period $2p$.

\item $\vert \mathcal{A}_j\vert = 3$; one angle has period $p$, and the other two angles have period $2p$.
\end{enumerate}
\end{enumerate}
\end{enumerate}
\end{theorem}
\begin{proof}
Since $F_a$ is a unicritical anti-holomorphic map whose action on the angles of (landing) dynamical rays is given by $\rho:\R/\Z\to\R/\Z$, which is an orientation-reversing double covering with no attracting periodic point, the proof of \cite[Theorem~2.6]{Sa} carries over verbatim to the present setting (also compare \cite[Lemma~2.3]{M2a}).
\end{proof}

The following definition is analogous to that of formal orbit portraits under the anti-doubling map $m_{-2}$ (see Definition~\ref{def_orbit_portrait}).

\begin{definition}\label{def_orbit_portrait_schwarz}
A finite collection $\mathcal{P} = \{\mathcal{A}_1,$ $\mathcal{A}_2,$ $\cdots,$ $\mathcal{A}_p \}$ of non-empty finite subsets of $\mathrm{Per}(\rho)$ satisfying the conditions of Theorem~\ref{complete_antiholomorphic_schwarz} is called a \emph{formal orbit portrait} under $\rho$ (in short, a $\rho$-FOP).
\end{definition}

Recall from Section~\ref{ideal_triangle} that the map $\mathcal{E}:\R/\Z\to\R/\Z$ is a topological conjugacy between the reflection map $\rho$ (which models the action of $F_a$ on its external dynamical rays) and the anti-doubling map $m_{-2}$ (which models the action of quadratic anti-polynomials on its external dynamical rays). Let $\mathcal{P} = \{\mathcal{A}_1,$ $\mathcal{A}_2,$ $\cdots,$ $\mathcal{A}_p \}$ be a $\rho$-FOP. We define the \emph{push-forward} of $\mathcal{P}$ under $\mathcal{E}$ by $$\mathcal{E}_{\ast}(\mathcal{P}):= \{\mathcal{E}(\mathcal{A}_1), \mathcal{E}(\mathcal{A}_2), \cdots, \mathcal{E}(\mathcal{A}_p)\}.$$ Similarly, we can define the \emph{pull-back} $\mathcal{E}^{\ast}(\mathcal{P})$ of an $m_{-2}$-FOP $\mathcal{P}$ under $\mathcal{E}$.

The following proposition is a consequence of the fact that the maps $\rho$ and $m_{-2}$ are topologically conjugate via $\mathcal{E}$. It is our first step in establishing a combinatorial bijection between the centers of $\cC(\mathcal{S})$ and those of $\mathcal{L}$.

\begin{proposition}[$\mathcal{E}$ preserves orbit portraits]\label{orbit_portraits_preserved}
If $\mathcal{P}$ is a $\rho$-FOP, then $\mathcal{E}_{\ast}(\mathcal{P})$ is an $m_{-2}$-FOP. Conversely, if $\mathcal{P}$ is an $m_{-2}$-FOP, then $\mathcal{E}^{\ast}(\mathcal{P})$ is a $\rho$-FOP.
\end{proposition}

Proposition~\ref{orbit_portraits_preserved} allows us to transfer combinatorial/topological results about $m_{-2}$-FOPs to corresponding results for $\rho$-FOPs. In particular, among all the complementary arcs of the various $\mathcal{A}_j$ of a $\rho$-FOP $\mathcal{P}$, there is a unique one of minimum length. This shortest arc $\mathcal{I}_{\mathcal{P}}$ is called the \emph{characteristic arc} of $\mathcal{P}$, and the two angles $\{t^{-},t^{+}\}$ at the ends of this arc are called its \emph{characteristic angles} (compare \cite[Lemma~3.2]{Sa}). We can assume, without loss of generality, that $0<t^+-t^-<\frac12$.

\subsection{Rigidity theorems}\label{rigidity_hyp_para}

This subsection will be devoted to some combinatorial rigidity results which show that PCF parameters and even-type parabolic parameters of $\cC(\mathcal{S})$ are completely determined by their combinatorics (orbit portraits associated with dynamical root, or laminations).

We start with some preliminary results. Let $a_1,a_2\in\cC(\mathcal{S})\setminus\{-1/12\}$. We will denote the pre-image of $\alpha_{a_i}$ (under $\sigma$) that lies in $\heartsuit$ by $\alpha_{a_i}'$. Recall that by Proposition~\ref{schwarz_group}, $\psi_{a_1}^{a_2}:=\psi_{a_2}^{-1}\circ\psi_{a_1}:T_{a_1}^\infty\to T_{a_2}^\infty$ is a conformal isomorphism. It restricts to a conformal isomorphism between $E_{a_1}^1$ and $E_{a_2}^1$, where $E_{a_i}^1$ is the union of the tiles of rank $0$ and $1$. It extends to a homeomorphism between $\overline{E_{a_1}^1}$ and $\overline{E_{a_2}^1}$ mapping $\frac{1}{4}$, $\alpha_{a_1}$, $\sigma_{a_1}^{-1}(\frac{1}{4})$, and $\alpha_{a_1}'$ to $\frac{1}{4}$, $\alpha_{a_2}$, $\sigma_{a_2}^{-1}(\frac{1}{4})$, and $\alpha_{a_2}'$ respectively (see Figure \ref{psi_extension}).

\begin{figure}[ht!]
\captionsetup{width=0.96\linewidth}
\centering
\includegraphics[scale=0.2]{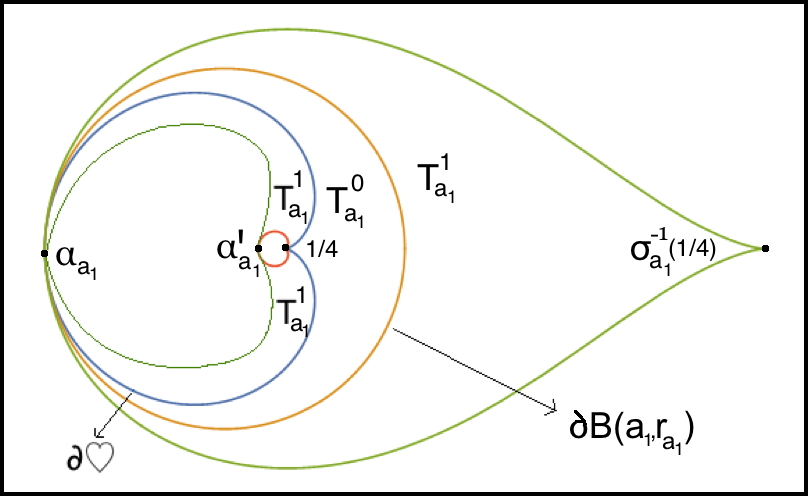} \includegraphics[scale=0.1]{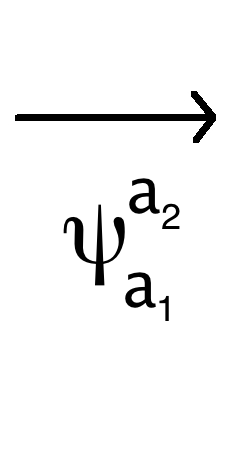} \includegraphics[scale=0.2]{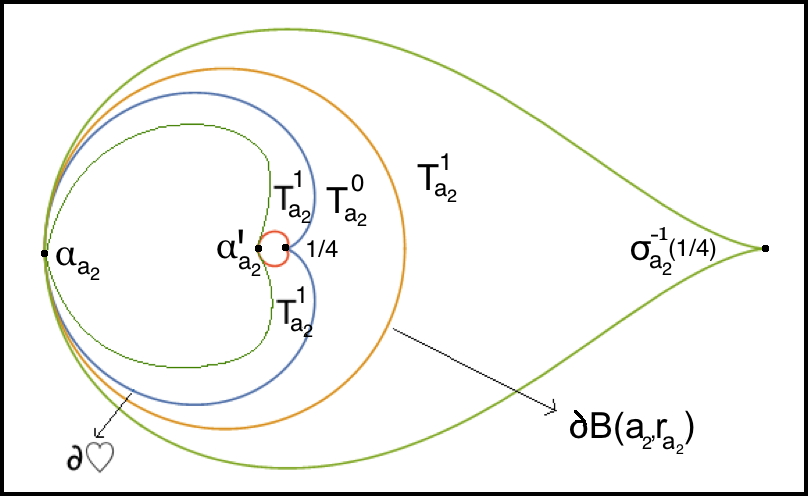}
\caption{$E_{a_i}^1$ is the union of the tiles of rank $1$ and $2$; i.e., $E_{a_i}^1=T_{a_i}^0\cup T_{a_i}^1$. The map $\psi_{a_1}^{a_2}$ induces a homeomorphism between $\overline{E_{a_1}^1}$ and $\overline{E_{a_2}^1}$ mapping $\frac{1}{4}$, $\alpha_{a_1}$, $\sigma_{a_1}^{-1}(\frac{1}{4})$, and $\alpha_{a_1}'$ to $\frac{1}{4}$, $\alpha_{a_2}$, $\sigma_{a_2}^{-1}(\frac{1}{4})$, and $\alpha_{a_2}'$ respectively.}
\label{psi_extension}
\end{figure}

\begin{lemma}\label{qc_conjugacy_alpha}
\noindent\begin{enumerate}
\item $\psi_{a_1}^{a_2}: \overline{E_{a_1}^1}\to\overline{E_{a_2}^1}$ is asymptotically linear near $\frac{1}{4}$, $\alpha_{a_1}$, $\sigma_{a_1}^{-1}(\frac{1}{4})$, and $\alpha_{a_1}'$.

\item $\psi_{a_1}^{a_2}: \overline{E_{a_1}^1}\to\overline{E_{a_2}^1}$ has a quasiconformal extension $\widehat{\psi}$ to $\widehat{\C}$.
\end{enumerate}
\end{lemma}
\begin{proof}
1) In \cite[Proposition~5.38]{LLMM1}, $\psi_{a_i}$ was first defined as (the homeomorphic extension of) a conformal isomorphism between $T_{a_i}$ and $\Pi$. Hence, $\psi_{a_1}^{a_2}: T_{a_1}\to T_{a_2}$ is a conformal isomorphism mapping $\frac{1}{4}$ and $\alpha_{a_1}$ to $\frac{1}{4}$ and $\alpha_{a_2}$ respectively. Moreover, on the tiles of first rank, $\psi_{a_1}^{a_2}$ is equal to $F_{a_2}^{-1}\circ(\psi_{a_1}^{a_2})\vert_{T_{a_1}}\circ F_{a_1}$ (choosing suitable inverse branches of $F_{a_2}$).

Since $a_1,a_2\in\cC(\mathcal{S})\setminus\{-1/12\}$, it follows from \cite[Proposition~5.12,~5.13]{LLMM1} that the asymptotic developments of $F_{a_i}$ near $\frac{1}{4}$ and $\alpha_{a_i}$ are comparable. Moreover, $F_{a_i}$ is anti-conformal near $\sigma_{a_i}^{-1}(\frac{1}{4})$ and $\alpha_{a_i}'$. Therefore, to prove the lemma, it suffices to show that $\psi_{a_1}^{a_2}: T_{a_1}\to T_{a_2}$ is asymptotically linear near $\alpha_{a_1}$ and $\frac{1}{4}$. 
But this follows from Lemmas~\ref{asymp_lin_1_lem} and~\ref{asymp_lin_2_lem} as $\frac14$ is a conformal cusp of type $(3,2)$ on $\partial T_{a_i}$ and $\alpha_{a_i}$ is a regular double point on $\partial T_{a_i}$, $i\in\{1,2\}$.

2) As $\psi_{a_1}^{a_2}$ is asymptotically linear at the cusps $\frac{1}{4}, \sigma_{a_1}^{-1}(\frac{1}{4})$, and at the double points $\alpha_{a_1}, \alpha_{a_1}'$, the proofs of Lemmas~\ref{asymp_lin_1_lem} and~\ref{asymp_lin_2_lem} show that the conformal map $\psi_{a_1}^{a_2}: \overline{E_{a_1}^1}\to\overline{E_{a_2}^1}$ admits quasiconformal extensions to neighborhoods of these singular points. Since the rest of $\partial E_{a_i}^1$ ($i\in\{1,2\}$) is a union of non-singular real-analytic arcs, the map $\psi_{a_1}^{a_2}$ extends conformally in a neighborhood of these points (by the Schwarz reflection principle). It follows that $\psi_{a_1}^{a_2}$ has a quasiconformal extension $\widehat{\psi}$ to $\widehat{\C}$. 
\end{proof}


\begin{remark}
For $a_1,a_2\in\cC(\mathcal{S})\setminus\{-1/12\}$, the map $\psi_{a_1}^{a_2}:E_{a_1}^1\to E_{a_2}^1$ does not necessarily have an analytic extension in a neighborhood of $\alpha_{a_i}$. In fact, the existence of such an analytic extension would imply that the pairs of curve germs $(\partial\heartsuit,\partial B(a_1,r_{a_1}))_{\alpha_{a_1}}$ and $(\partial\heartsuit,\partial B(a_2,r_{a_2}))_{\alpha_{a_2}}$ are conformally equivalent. A description of conformal equivalence classes of pairs of real-analytic smooth curves $(\gamma_1,\gamma_2)$ (with associated Schwarz reflection maps $\sigma_1, \sigma_2$) touching at the origin is given in \cite{Nak} in terms of conformal conjugacy classes of the parabolic germ $\sigma_1\circ\sigma_2$ (also see \cite{Vor1}).
\end{remark}

We are now in a position to prove combinatorial rigidity of super-attracting and even-type parabolic parameters of $\mathcal{S}$.

\begin{proposition}[Rigidity of super-attracting maps]\label{rigidity_center}
Let $a_1$ and $a_2$ be two super-attracting parameters such that their dynamical roots have the same associated orbit portrait. Then, $a_1=a_2$.
\end{proposition}
\begin{proof}
The orbit portrait associated with the dynamical root of a super-attracting parameter completely determines the Hubbard tree of the maps, so the restriction of the maps on their respective Hubbard trees are topologically conjugate. Moreover, there exist B{\"o}ttcher maps conjugating $F_{a_1}$ to $F_{a_2}$ in a neighborhood of the super-attracting cycle. 
 
Let us define a $K$-qc map $\xi_0$ of the sphere that agrees with $\psi_{a_1}^{a_2}$ on $E_{a_1}^1$, and with the B{\"o}ttcher conjugacies in a neighborhood $U$ of the critical cycle. The existence of such a map is guaranteed by Lemma~\ref{qc_conjugacy_alpha}. Note that $\psi_{a_1}^{a_2}$ conjugates $F_{a_1}:\partial E_{a_1}^1\to\partial T_{a_1}^0$ to $F_{a_2}:\partial E_{a_2}^1\to\partial T_{a_2}^0$ (both of which are double coverings). Moreover, $F_{a_i}:\widehat{\C}\setminus\Int{E_{a_i}^1}\to\widehat{\C}\setminus\Int{T_{a_i}^0}$ is a two-to-one branched covering branched at $0$. Since $\xi_0$ sends the critical value of $F_{a_1}$ to that of $F_{a_2}$, we can lift $\xi_0$ via $F_{a_1}$ and $F_{a_2}$ to obtain $K$-qc homeomorphism from $\Int{(\widehat{\C}\setminus E_{a_1}^1)}$ to $\Int{(\widehat{\C}\setminus E_{a_2}^1)}$ that matches continuously with $\xi_0$ on $\partial E_{a_1}^1$. By quasiconformal removability of analytic arcs, we obtain a $K$-qc map $\xi_1$ of the sphere. The lift $\xi_1$ becomes unique once we require $\xi_1(\infty)=\infty$. Moreover, $\xi_0$ and $\xi_1$ agree on $U\cup E_{a_1}^1$, so they are homotopic relative to the union of the super-attracting cycle and the singular points. 

By iterating this lifting procedure and arguing as in \cite[Lemma 38.6]{L6}, we obtain a global $K$-qc map $\xi$ that conjugates $F_{a_1}$ to $F_{a_2}$. Moreover, $\xi$ agrees with $\psi_{a_1}^{a_2}$ on $T_{a_1}^\infty$ and is conformal on $\Int{K_{a_1}}$. Since $\Gamma_{a_1}$ has measure zero (see Theorem~\ref{geom_finite_limit_set}), it follows that $\xi$ is conformal on the sphere. Note that since $\xi$ fixes $0$, $\infty$ and $\frac{1}{4}$, it must be the identity map. Therefore, $\mathrm{id}(F_{a_1}^{\circ 2}(0))=F_{a_2}^{\circ 2}(0)$; i.e., $a_1=a_2$.
\end{proof}

An essentially similar argument combined with the description of hyperbolic components of $\cC(\mathcal{S})$ given in Section~\ref{hyperbolic} yields the following rigidity result for all hyperbolic maps.

\begin{proposition}[Rigidity of hyperbolic maps]\label{rigidity_hyp_prop}
Let $a_1$ and $a_2$ be two hyperbolic parameters such that their dynamical roots have the same associated orbit portrait. Moreover, suppose that the first return map of their characteristic Fatou components are conformally conjugate. Then, $a_1=a_2$.
\end{proposition}

Our next result states that parabolic parameters of even-type can also be recovered from their combinatorics. 

\begin{proposition}[Rigidity of parabolic maps]\label{rigidity_para_even}
Let $a_1$ and $a_2$ be two parabolic parameters such that their characteristic Fatou components have even period. If their dynamical roots have the same associated orbit portrait, then $a_1=a_2$.
\end{proposition}
\begin{proof}
Note that since the parabolic cycles of $F_{a_1}$ and $F_{a_2}$ have the same associated orbit portraits, their parabolic cycles have a common period $k$ and a common combinatorial rotation number when $k$ is even (see \cite[Definition~2.12]{M2a} for the definition of  combinatorial rotation number). Choose an attracting petal containing the critical value $\infty$ and an attracting Fatou coordinate (in the characteristic Fatou component of $F_{a_i}$) that conjugates the first return map of the petal to the translation $\zeta\mapsto\zeta+1$. Since Fatou coordinates are unique up to addition of a complex constant, we can arrange so that the critical values of $F_{a_1}$ and $F_{a_2}$ have the same image under the Fatou coordinates. Hence, the Fatou coordinates induce a conformal conjugacy between the first return maps of the petals that sends $\infty$ to $\infty$. Using $F_{a_1}$ and $F_{a_2}$, we now spread this conjugacy to suitable attracting petals in all the periodic Fatou components such that the domain of the conjugacy (which we denote by $U$) contains the entire post-critical set of $F_{a_1}$. 

We now construct a $K$-qc map $\xi_0$ of the sphere that agrees with $\psi_{a_1}^{a_2}$ on $E_{a_1}^1$, and with the conjugacy on $U$ constructed in the previous paragraph. Then $\xi_0$ lifts to a $K$-qc map $\xi_1$ of the sphere (normalized so that $\xi_1(\infty)=\infty$) agreeing with $\xi_0$ on $E_{a_1}^1\cup U$. In particular, $\xi_1$ is homotopic to $\xi_0$ relative to the union of the post-critical set and the singular points. The rest of the proof is analogous to that of Proposition~\ref{rigidity_center}.  
\end{proof}

\begin{remark}\label{odd_non_rigid}
We will see in Theorem~\ref{parabolic_arcs_schwarz} that parabolic parameters with odd-periodic characteristic Fatou components are not combinatorially rigid; i.e., they admit quasiconformal deformations that preserve combinatorics, and hence cannot be uniquely determined by their parabolic orbit portraits. However, we will prove a slightly weaker rigidity statement for such maps in Proposition~\ref{rigidity_para_odd}.
\end{remark}

We now prove a combinatorial rigidity principle for Misiurewicz parameters $\mathcal{S}$ to the effect that a Misiurewicz parameter is completely determined by its pre-periodic lamination. Since the limit set of a Misiurewicz map $F_a$ is a dendrite (see Theorem~\ref{geom_finite_limit_set}), there exists a unique arc (in $\Gamma_a$) connecting any two points of $\Gamma_a$. The union of such arcs connecting the post-critical set of $F_a$ is a tree, and we call it the \emph{Hubbard tree} of a Misiurewicz map $F_a$.

\begin{proposition}[Rigidity of Misiurewicz parameters]\label{rigidity_misi}
Let $a_1$ and $a_2$ be Misiurewicz parameters with the same pre-periodic lamination. Then, $a_1=a_2$.
\end{proposition}
\begin{proof}
Since $F_{a_1}$ and $F_{a_2}$ have the same pre-periodic lamination, the restriction of the maps on their Hubbard trees are topologically conjugate. In particular, their critical orbits have the same pre-period and period. Moreover, there exist quasiconformal maps defined on a neighborhood of the post-critical set of $F_{a_1}$ conjugating $F_{a_1}$ to $F_{a_2}$. Let us now construct a $K$-qc map $\xi_0$ of the sphere that agrees with $\psi_{a_1}^{a_2}=\psi_{a_2}^{-1}\circ\psi_{a_1}$ on $E_{a_1}^1$, and with the quasiconformal conjugacies in a neighborhood of the post-critical set of $F_{a_1}$. This is possible due to Lemma~\ref{qc_conjugacy_alpha}. Then $\xi_0$ lifts to a $K$-qc map $\xi_1$ of the sphere (normalized so that $\xi_1(\infty)=\infty$) agreeing with $\xi_0$ on the union of $E_{a_1}^1$ and some neighborhood of the post-critical set of $F_{a_1}$. In particular, $\xi_1$ is homotopic to $\xi_0$ relative to the union of the post-critical set and the singular points. 

By iterating this lifting procedure and arguing as in Proposition~\ref{rigidity_center}, we obtain a global $K$-qc map $\xi$ that conjugates $F_{a_1}$ to $F_{a_2}$. Moreover, $\xi$ agrees with $\psi_{a_1}^{a_2}$ on $T_{a_1}^\infty$. Since $\Int{K_{a_1}}=\emptyset$, and $\Gamma_{a_1}$ has measure zero (see Theorem~\ref{geom_finite_limit_set}), it follows that $\xi$ is conformal on the sphere. As $\xi$ fixes $0$, $\infty$ and $\frac{1}{4}$, it must be the identity map. Therefore, $\mathrm{id}(F_{a_1}(\infty))=F_{a_2}(\infty)$; i.e., $a_1=a_2$.
\end{proof}

\section{Parameter rays}\label{odd_period_parabolic}

In this section, we discuss landing/accumulation properties of parameter rays of $\mathcal{S}$ at (pre-)periodic angles. 

In particular, we explore the connection between parabolic (respectively, Misiurewicz) parameters of $\cC(\mathcal{S})$ and parameter rays at $\rho$-periodic (respectively, pre-periodic) angles. This, on the one hand, leads to a complete description of the boundaries of odd period hyperbolic components of $\cC(\mathcal{S})$, and on the other hand, prepares the ground for the proofs of the main theorems of this paper.

\subsection{Odd period parabolics, and period-doubling bifurcations}\label{odd_para_doubling} 

We begin with a preliminary discussion of the boundaries of odd period hyperbolic components and period-doubling bifurcations associated with them. Let us first note that the boundaries of odd period hyperbolic components of $\cC(\mathcal{S})$ consist only of parabolic parameters.

\begin{proposition}[Neutral dynamics of odd period]\label{hyp_odd_parabolic}  
The boundary of a hyperbolic component of odd period $k$ of $\cC(\mathcal{S})$ consists 
entirely of parameters having a parabolic orbit of exact period $k$. In suitable 
local conformal coordinates, the $2k$-th iterate of such a map has the form 
$z\mapsto z+z^{q+1}+\ldots$ with $q\in\{1,2\}$. 
\end{proposition}

\begin{proof} 
See \cite[Lemma~2.5]{MNS} for a proof in the case of unicritical anti-polynomials. The same proof applies to the family $\mathcal{S}$.
\end{proof}

This leads to the following classification of parabolic points of odd period.

\begin{definition}\label{DefCusp_schwarz}
A parameter $a$ will be called a {\em parabolic cusp} if it has a parabolic 
periodic point of odd period such that $q=2$ in the previous proposition. Otherwise, it is called a \emph{simple} parabolic parameter.
\end{definition}

Suppose that $a$ is a parabolic cusp with a $k$-periodic parabolic cycle. Since every cycle of attracting petals of $F_a$ attracts the forward orbit of $0$, it follows that the period of the characteristic Fatou component of $F_a$ is $2k$. By Theorem~\ref{complete_antiholomorphic_schwarz}, the angles of the dynamical rays landing at the $k$-periodic parabolic cycle of $F_a$ have angles $k$ or $2k$. Therefore, there are finitely many possibilities for the orbit portrait associated with the $k$-periodic parabolic cycle of $F_a$.
Hence, Proposition~\ref{rigidity_para_even} implies that there are only finitely parabolic cusps of a given period $k$ in $\cC(\mathcal{S})$.

Let us now fix a hyperbolic component $H$ of odd period $k$, and let $a\in H$. Note that the first return map $F_a^{\circ k}$ of a $k$-periodic Fatou component of $F_a$ has precisely three fixed points (necessarily repelling) on the boundary of the component. As $a$ tends to a simple parabolic parameter on the boundary $\partial H$, the unique attracting periodic point of this Fatou component tends to merge with one of these three repelling periodic points. Similarly, as $a$ tends to a parabolic cusp on the boundary $\partial H$, the unique attracting periodic point of this Fatou component and two of the three boundary repelling periodic points merge together. 

Now let $a$ be a simple parabolic parameter of odd (parabolic) period $k$. The holomorphic first return map $F_a^{\circ 2k}$ of any attracting petal of $F_a$ is conformally conjugate to translation by $+1$ in a right half-plane (see \cite[\S 10]{M1new}). The conjugating map is called an attracting \emph{Fatou coordinate}. Thus the quotient of the petal by the dynamics $F_a^{\circ 2k}$ is isomorphic to a bi-infinite cylinder, called the \emph{attracting {\'E}calle cylinder}. Note that Fatou coordinates are uniquely determined up to addition by a complex constant. 

Since $F_a^{\circ k}$ commutes with $F_a^{\circ 2k}$, it follows that $F_a^{\circ k}$ induces an anti-holomorphic involution of the attracting {\'E}calle cylinder $\C/\Z$. Such a map must fix a horizontal round circle of $\C/\Z$. By using one real additive degree of freedom of the Fatou coordinate, we can assume that this invariant circle is $\R/\Z$. This special Fatou coordinate clearly conjugates the first anti-holomorphic return map $F_a^{\circ k}$ of the attracting petal to the map $\zeta\mapsto\overline{\zeta}+\frac{1}{2}$ (compare Proposition~\ref{normalization of fatou}). This coordinate is unique up to addition of a real constant. The pre-image of the real line (which is invariant under $\zeta\mapsto\overline{\zeta}+\frac{1}{2}$) under this Fatou coordinate is called the \emph{attracting equator}. By construction, the attracting equator is invariant under the dynamics $F_a^{\circ k}$.

The imaginary part of the image of the critical value $\infty$ (whose forward orbit converges to the parabolic cycle, by Proposition~\ref{fatou_critical}) under this special Fatou coordinate is called the \emph{critical {\'E}calle height} of $F_a$ (since this Fatou coordinate is unique up to addition of a real constant, the real part of the image of $\infty$ under this coordinate is not well-defined). It is easy to see that critical {\'E}calle height is a conformal conjugacy invariant of simple parabolic parameters of odd period. One can change the critical {\'E}calle height of simple parabolic parameters by a quasiconformal deformation argument to obtain real-analytic arcs of parabolic parameters on the boundaries of odd period hyperbolic components. 

\begin{theorem}[Parabolic arcs]\label{parabolic_arcs_schwarz}
Let $\widetilde{a}$ be a simple parabolic parameter of odd period. Then $\widetilde{a}$ is on a parabolic arc in the  following sense: there  exists a real-analytic arc of simple parabolic parameters $a(h)$ (for $h\in\mathbb{R}$) with quasiconformally equivalent but conformally distinct dynamics of which $\widetilde{a}$ is an interior point, and the {\'E}calle height of the critical value of $F_{a(h)}$ is $h$. This arc is called a \emph{parabolic arc}.
\end{theorem}
\begin{proof} 
See \cite[Theorem~3.2]{MNS} for a proof in the case of unicritical anti-polynomials. One essentially uses the same deformation in the attracting petals, and Lemma~\ref{schwarz_qcdef} guarantees that the quasiconformal deformations of $F_{\widetilde{a}}$ also lie in the family $\mathcal{S}$.
\end{proof}

The deformation of complex structure giving rise to parabolic arcs is supported on the basin of attraction of the parabolic cycle. Hence, the quasiconformal conjugacy (constructed in Theorem~\ref{parabolic_arcs_schwarz}) between any two maps on the same parabolic arc is conformal on the tiling set. It follows that along a parabolic arc, the angles of the dynamical rays landing at the parabolic cycle remain constant. In other words, the orbit portrait associated with the parabolic cycle remains constant on a parabolic arc, so simple parabolic parameters of odd period of $\cC(\mathcal{S})$ are not combinatorially rigid. However, the next proposition shows that they can be uniquely determined by a combination of combinatorial and analytic data.

\begin{proposition}[Weak rigidity of odd period parabolics]\label{rigidity_para_odd}
Let $a_1$ and $a_2$ be two simple parabolic parameters of odd period such that  their parabolic cycles have the same orbit portrait. Then, they lie on the same parabolic arc. Moreover, if they have equal critical {\'E}calle height, then $a_1=a_2$.
\end{proposition}
\begin{proof}
Suppose that the critical {\'E}calle heights of $F_{a_1}$ and $F_{a_2}$ be $h_1$ and $h_2$ respectively. Using the deformation constructed in \cite[Theorem~3.2]{MNS}, we obtain a quasiconformal conjugacy $\phi^{a_2}_{a_1}$ between the first return maps of $F_{a_i}$ on the attracting petals such that the conjugacy preserves the post-critical set. Let us denote the union of attracting petals (containing the post-critical set) where the conjugacy is defined by $U$. Note that if $h_1\neq h_2$, this conjugacy is not conformal. 

We now construct a $K$-qc map $\xi_0$ of the sphere that agrees with $\psi_{a_1}^{a_2}$ on $E_{a_1}^1$, and with the conjugacy $\phi_{a_1}^{a_2}$ constructed above on $U$ (the existence of such a map is guaranteed by Lemma~\ref{qc_conjugacy_alpha}). Then, $\xi_0$ lifts to a $K$-qc map $\xi_1$ of the sphere that agrees with $\xi_0$ on $E_{a_1}^1\cup U$. In particular, $\xi_1$ is homotopic to $\xi_0$ relative to the post-critical set and the irregular fixed points. One can now proceed as in Proposition~\ref{rigidity_center} to conclude that iterating this lifting procedure yields a quasiconformal homeomorphism $\xi$ of the sphere that agrees with $\psi_{a_1}^{a_2}$ on $T_{a_1}^\infty$, and conjugates $F_{a_1}$ to $F_{a_2}$.

Let us assume that the parameter $a_1$ lies on the parabolic arc $\mathscr{C}$. We denote the critical {\'E}calle height $h_2$ parameter on $\mathscr{C}$ by $a_2'$.  By Theorem~\ref{parabolic_arcs_schwarz}, there exists a quasiconformal conjugacy $\widehat{\xi}$ between $F_{a_1}$ and $F_{a_2'}$ such that $\widehat{\xi}$ is conformal on $T_{a_1}^\infty$. It is now easy to see that $\widehat{\xi}\circ\xi^{-1}$ is a quasiconformal conjugacy between $F_{a_2}$ and $F_{a_2'}$ such that $\widehat{\xi}\circ\xi^{-1}$ agrees with $\psi_{a_2}^{a_2'}:=\psi_{a_2'}^{-1}\circ\psi_{a_2}$ on the tiling set and is conformal on the parabolic basin. Since parabolic limit sets have measure zero (see Theorem~\ref{geom_finite_limit_set}), it follows that $\widehat{\xi}\circ\xi^{-1}$ is a conformal conjugacy between $F_{a_2}$ and $F_{a_2'}$. Moreover, $\widehat{\xi}\circ\xi^{-1}$ fixes $0$, $\infty$, and $\frac14$. Hence, it is the identity map implying that $a_2=F_{a_2}^{\circ 2}(0)=F_{a_2'}^{\circ 2}(0)=a_2'$. Thus, $a_1$ and $a_2$ lie on the same parabolic arc $\mathscr{C}$, proving the first part of the proposition.

Let us now assume that $h_1=h_2$. In this case, the conjugacy $\phi^{a_2}_{a_1}$ between the first return maps of $F_{a_i}$ on the attracting petals (such that it preserves the post-critical set) can be chosen to be conformal. Indeed, both the maps have a simple parabolic cycle of common odd period $k$. Choose an attracting petal containing the critical value $\infty$ and an attracting Fatou coordinate (in the characteristic Fatou component of $F_{a_i}$) that conjugates the first return map of the petal to the glide reflection $\zeta\mapsto\overline{\zeta}+1/2$. These Fatou coordinates are unique up to addition of a real constant. Since $F_{a_1}$ and $F_{a_2}$ have equal critical {\'E}calle height, we can arrange so that the critical values of $F_{a_1}$ and $F_{a_2}$ have the same image under the Fatou coordinates. Hence, the Fatou coordinates induce a conformal conjugacy $\phi^{a_2}_{a_1}$ between the first return maps of the petals that sends $\infty$ to $\infty$. Using $F_{a_1}$ and $F_{a_2}$, we now spread this conjugacy to suitable attracting petals in all the periodic Fatou components such that the domain of the conjugacy contains the entire post-critical set of $F_{a_1}$. 

We can now argue as in the first part of the proposition to obtain a quasiconformal conjugacy $\xi$ between $F_{a_1}$ to $F_{a_2}$ that agrees with $\psi_{a_1}^{a_2}$ on $T_{a_1}^\infty$, and is conformal on the parabolic basin. But this implies that $\xi$ is the identity map, and hence $a_1=a_2$.
\end{proof}

\begin{corollary}\label{para_arcs_no_intersect}
Two distinct parabolic arcs do not intersect.
\end{corollary}

Let us fix a parabolic arc $\mathscr{C}$, and its critical {\'E}calle height parametrization $a:\R\to\mathscr{C}$. Since $\cC(\mathcal{S})$ is bounded, $\mathscr{C}$ accumulates on both ends of $\R$; i.e., $\overline{\mathscr{C}}$ is a  compact connected set in $\C$. It is easy to see that any accumulation point of $\mathscr{C}$ (as the critical {\'E}calle height goes to $\pm\infty$) is a parabolic cusp of the same period (compare \cite[Lemma~3.3]{MNS}). Since there are only finitely many parabolic cusps of a given period, $\mathscr{C}$ limits at parabolic cusp points on both ends. Note also that in the dynamical plane of a parabolic cusp, the double parabolic points are formed by the merger of a simple parabolic point with a repelling point. 

\begin{proposition}[Fixed point index on parabolic arcs]\label{index_to_infinity_schwarz}
Along any parabolic arc of odd period, the holomorphic fixed point index of the parabolic cycle is a real valued real-analytic function that tends to $+\infty$ at both ends.
\end{proposition}
\begin{proof}
The proof is similar to that of \cite[Proposition~3.7]{HS}. Indeed, as we move along a parabolic arc towards one of the cusp points at its end, the simple parabolic cycle merges with a repelling cycle, and the sum of their holomorphic fixed point indices converges to the fixed point index of the double parabolic cycle of the cusp parameter (see Subsection~\ref{bif_odd_per_hyp_subsubsec} for the definition of holomorphic fixed point index). But it is easy to see that, as in the anti-polynomial case, the holomorphic fixed point index of the repelling cycle is real and diverges to $-\infty$ as the parameter converges to a cusp. Since the limiting double parabolic cycle has a finite index, it follows that the holomorphic fixed point index of the simple parabolic cycle (which is also real) must tend to $+\infty$. 
\end{proof}

It now follows by arguments similar to those used in \cite[Theorem~3.8, Corollary~3.9]{HS} that:

\begin{theorem}[Bifurcations along arcs]\label{ThmBifArc_schwarz}
Every parabolic arc of period $k$ intersects the boundary of a hyperbolic component of period $2k$ along an arc consisting of the set of parameters where the parabolic fixed point index is at least $1$. In particular, every parabolic arc has, at both ends, an interval of positive length at which bifurcation from a hyperbolic component of odd period $k$ to a hyperbolic component of period $2k$ occurs.
\end{theorem}

Roughly speaking, when a parameter on such a bifurcating arc is perturbed outside the odd period hyperbolic component, the simple parabolic periodic point splits into two attracting periodic points that lie on the same orbit of $F_a$.

The next proposition contains some partial information about the relation between critical {\'E}calle height and parabolic fixed point index on the bifurcating region of a parabolic arc. For any $h$ in $\mathbb{R}$, we denote the residue fixed point index of the unique parabolic cycle of $F_{a(h)}^{\circ 2}$ by $\ind_{\mathscr{C}}(F_{a(h)}^{\circ 2})$. 

\begin{proposition}\label{index_increasing}
Let $H$ be a hyperbolic component of odd period $k$ in $\cC(\mathcal{S})$, $\mathscr{C}$ be a parabolic arc on $\partial H$, $a:\mathbb{R}\to\mathscr{C}$ be the critical {\'E}calle height parametrization of $\mathscr{C}$, and let $H'$ be a hyperbolic component of period $2k$ bifurcating from $H$ across $\mathscr{C}$. Then there exists some $h_0>0$ such that 
$$
\mathscr{C}\cap\partial H'=a[h_0,+\infty).
$$
Moreover, the function 
$$
\ind_{\mathscr{C}}: [h_0,+\infty) \to [1,+\infty),\quad h \mapsto \ind_{\mathscr{C}}(F_{a(h)}^{\circ 2}).
$$
is strictly increasing, and hence a bijection.
 \end{proposition}
 \begin{proof}
The proof of  \cite[Lemma~2.13 , Corollary~2.21]{IM2} can be applied mutatis mutandis to our setting.
 \end{proof}

\subsection{Parameter rays at periodic angles}\label{para_orbit_portrait}

In this subsection, we will first look at the connection between orbit portraits associated with parabolic parameters on the boundary of an odd period hyperbolic component $H$ and the angles of parameter rays accumulating on $\partial H$. Subsequently, we will discuss the relation between the orbit portrait associated with the root of an even period hyperbolic component $H$ that does not bifurcate from an odd period hyperbolic component and the angles of parameter rays landing at the root of $H$.

We begin with a preliminary result.

\begin{lemma}[Orbit separation lemma]\label{orbit_separation}
Let $F_a$ have a parabolic cycle. Then, the characteristic parabolic point of $F_a$ can be separated from any other point on the parabolic orbit by two (pre-)periodic dynamical rays landing at a common repelling (pre-)periodic point.
\end{lemma}
\begin{proof}
The proof uses the dynamics of $F_a$ on the parabolic Hubbard tree, and is analogous to that of \cite[Lemma~3.7]{S1a}.
\end{proof}

\begin{proposition}[Accumulation points of periodic parameter rays]\label{para_ray_periodic}
Let $\theta\in(1/3,2/3)$ be periodic under $\rho$, and $a_0\in\cC(\mathcal{S})$ be an accumulation point of the parameter ray at angle $\theta$. Then, $F_{a_0}$ has a parabolic cycle of period $k$ dividing the period of $\theta$ such that the corresponding dynamical ray at angle $\theta$ lands at the characteristic parabolic point of $F_{a_0}$.
\end{proposition} 
\begin{proof}
The proof is similar to the classical proof of landing of rational parameter rays of the Mandelbrot set (cf. \cite[Theorem~C.7]{GM1} \cite[Propositions~3.1, 3.2]{S1a}).

By Proposition~\ref{per_rays_land}, the dynamical ray of $F_{a_0}$ at angle $\theta$ lands at a $k$-periodic repelling or parabolic point $w_0\in \Gamma_{a_0}$. Clearly, $k$ divides the period of the angle $\theta$ (under $\rho$).

Let $w_0$ be a repelling periodic point. Then, by the Implicit Function Theorem and the existence of Koenigs linearization coordinates for repelling points, there exist a neighborhood $U$ of $a_0$ in the parameter space and a real-analytic function $w:U\to\widehat{\C}$ such that $w(a_0)=w_0$, the point $w(a)$ is a repelling periodic point of period $k$ of $F_a$, and the dynamical ray of $F_a$ at angle $\theta$ lands at $w(a)$ for all $a\in U$ (cf. \cite[Lemma~B.1]{GM1}). But there are parameters $a$ near $a_0$ and lying on the parameter ray at angle $\theta$. For such parameters $a$, the corresponding dynamical ray at angle $\theta$ bifurcates. This contradiction proves that $w_0$ is a parabolic periodic point.

Let us assume that $w_0$ is not the characteristic parabolic point of $F_{a_0}$. By Lemma~\ref{orbit_separation}, the characteristic parabolic point of $F_{a_0}$ can be separated from $w_0$ by two (pre-)periodic dynamical rays landing at a common repelling (pre-)periodic point. Evidently, this separation line persists under perturbation, and separates $\infty$ from the dynamical ray at angle $\theta$. However, for parameters $a$ near $a_0$ and lying on the parameter ray at angle $\theta$, the critical value $\infty$ lies on the dynamical ray at angle $\theta$. Once again, this is a contradiction, which proves that $w_0$ is the characteristic parabolic point of $F_{a_0}$. 
\end{proof}

Recall the notions of $\rho$-FOP and $m_{-2}$-FOP from Definitions~\ref{def_orbit_portrait_schwarz} and~\ref{def_orbit_portrait}, respectively. Using Proposition~\ref{orbit_portraits_preserved}, one can transfer combinatorial/topological results about $m_{-2}$-FOPs to corresponding results for $\rho$-FOPs. In particular, among all the complementary arcs of the various $\mathcal{A}_j$ of a $\rho$-FOP $\mathcal{P}$, there is a unique one of minimum length. As in the anti-polynomial case, we call this shortest arc $\mathcal{I}_{\mathcal{P}}$ the \emph{characteristic arc} of $\mathcal{P}$, and refer to the two angles $t^{-},t^{+}$ at the ends of this arc as the \emph{characteristic angles} of $\mathcal{P}$ (compare \cite[Lemma~3.2]{Sa}). We can assume, without loss of generality, that $0<t^+-t^-<1/2$.

Let $t^\pm\in(1/3,2/3)$ with $t^-=Q((i_1,i_2,\cdots))$, $t^+=Q((j_1,j_2,\cdots,))$ (see  \cite[\S 2]{LLMM1}). Then, there exists a complementary component of $\left(\mathfrak{T}^{i_1}\cup\mathfrak{T}^{i_1,i_2}\cup\cdots\right)\cup\left(\mathfrak{T}^{j_1}\cup\mathfrak{T}^{j_1,j_2}\cup\cdots\right)\cup\cC(\mathcal{S})$ (see Definition~\ref{para_tiles} for the definition of parameter tiles) that contains the tail of any sequence of tiles determined by any $\theta\in(t^-,t^+)$.  We say that $a$ lies between the parameter rays at angles $t^-$ and $t^+$ if $a$ lies in this component.
 
The next theorem asserts that every $\rho$-FOP with characteristic angles in $(1/3,2/3)$ is realized by some member of $\mathcal{S}$ outside $\cC(\mathcal{S})$. 

\begin{theorem}[Realization of $\rho$-FOP outside $\cC(\mathcal{S})$]\label{realization_orbit_portrait_outside_schwarz}
Let $\mathcal{P} = \{\mathcal{A}_1,$ $\mathcal{A}_2,$ $\cdots,$ $\mathcal{A}_p \}$ be a formal orbit portrait under $\rho$ with its characteristic angles $t^\pm$ in $(1/3,2/3)$. Then, there exists some $a\in\mathbb{C}\setminus\left((-\infty,-1/12)\cup\cC(\mathcal{S})\right)$ such that $F_a$ has a periodic orbit with associated orbit portrait $\mathcal{P}$.
\end{theorem}
\begin{proof}
Adapting the proof of \cite[Lemma~3.4]{Sa}, one can show that in the dynamical plane of every parameter $a$ lying between the parameter rays at angles $t^-$ and $t^+$, the dynamical rays at angles $t^-$ and $t^+$ land at a common point $w\in \Gamma_a$. Using Proposition~\ref{orbit_portraits_preserved}, one then obtains analogues of \cite[Lemma~2.9, Lemma~3.3]{Sa} to the effect that the characteristic angles $t^-$ and $t^+$ essentially determine the actual orbit portrait $\mathcal{P}'$ associated with the periodic point $w$. Finally, one proceeds as in \cite[Theorem~3.1]{Sa} to conclude that for judicious choices of $a$ between the parameter rays at angles $t^-$ and $t^+$, the actual orbit portrait $\mathcal{P}'$ associated with $w$ coincides with $\mathcal{P}$.
\end{proof}

The following proposition, which is an analogue of Theorem~\ref{realization_orbit_portrait_parabolic} for the family $\mathcal{S}$, will play an important role in the sequel.

\begin{theorem}[Realization of $\rho$-FOP at parabolic parameters]\label{realization_orbit_portrait_parabolic_schwarz}
Let $\mathcal{P}= \{\mathcal{A}_1,$ $\mathcal{A}_2,$ $\cdots,$ $\mathcal{A}_p \}$ be a $\rho$-FOP with characteristic angles $t^\pm\in(1/3,2/3)$.

1) Suppose that $p$ is odd, and $t^\pm$ have period $2p$. Then the parameter rays of $\mathcal{S}$ at angles $t^-$ and $t^+$ accumulate on a common root parabolic arc $\mathscr{C}$ such that for every parameter $a\in\mathscr{C}$, the map $F_a$ has a parabolic cycle of period $p$ and the orbit portrait associated with the parabolic cycle of $F_a$ is $\mathcal{P}$. 

2) Suppose that $p$ is even. Then the parameter rays of $\mathcal{S}$ at angles $t^-$ and $t^+$ land at a common parabolic parameter $a$ (whose parabolic cycle has period $p$) such that the orbit portrait associated with the parabolic cycle of $F_a$ is $\mathcal{P}$. 
\end{theorem}

\begin{proof}
1) By Theorem~\ref{realization_orbit_portrait_outside_schwarz}, the dynamical rays at angles $t^-$ and $t^+$ land at a common point of $\Gamma_a$ for all parameters $a\notin\cC(\mathcal{S})$ lying between the parameter rays at angles $t^-$ and $t^+$. One can now employ a standard wake argument, analogous to the one in \cite[Lemma~4.1]{IM1}, to conclude that the parameter rays of $\mathcal{S}$ at angles $t^-$ and $t^+$ must accumulate on a common root parabolic arc $\mathscr{C}$ (such that the rays together with $\mathscr{C}$ separate the plane). The details are similar to the proof of Theorem~\ref{realization_orbit_portrait_parabolic}.

2) Once again, the proof is completely similar to that of Theorem~\ref{realization_orbit_portrait_parabolic}. The fact that the parameter rays at angles $t^-$ and $t^+$ land follows from discreteness of (even-type) parabolic parameters with a prescribed orbit portrait (see Proposition~\ref{rigidity_para_even}).
\end{proof}

Let us now fix a hyperbolic component $H$ of odd period $k$ and center $a_0$. The first return map of the closure of the characteristic Fatou component of $a_0$ fixes exactly three points on its boundary. Only one of these fixed points disconnects the non-escaping set, and is the landing point of two distinct dynamical rays at $2k$-periodic angles (see Proposition~\ref{prop_root_schwarz}). Let the set of the angles of these two rays be $S' = \{\alpha_1,\alpha_2 \}$. Then, $\alpha_2=(-2)^k\alpha_1$, and $S'$ is the set of characteristic angles of the corresponding orbit portrait. Each of the remaining two fixed points is the landing point of precisely one dynamical ray at a $k$-periodic angle; let the collection of the angles of these rays be $S = \{ \theta_1, \theta_2\}$. We can, possibly after renumbering, assume that $0 < \alpha_1 < \theta_1 < \theta_2 < \alpha_2$ and $\alpha_2 - \alpha_1 < \frac{1}{2}$. 

Since parabolic cusps are isolated (see Proposition~\ref{rigidity_para_even}) and $\partial H$ is connected, every parabolic cusp is the common limit point of two distinct parabolic arcs. By Theorem~\ref{ThmBifArc_schwarz}, every parabolic cusp and its nearby simple parabolic parameters are points of bifurcation from $H$ to a hyperbolic component of period $2k$. It follows that every parabolic cusp lies in the interior of $\cC(\mathcal{S})$ (compare \cite[Lemma~5.12]{MNS}). 

One can now argue as in the proof of \cite[Theorem~1.2]{MNS} to give a complete description of $\partial H$. Indeed, Theorem~\ref{realization_orbit_portrait_parabolic_schwarz}, Proposition~\ref{para_ray_periodic} and Proposition~\ref{rigidity_para_odd} imply that there are exactly three parabolic arcs on $\partial H$ which can be numbered in the following way. 

There is a unique parabolic arc (say, $\mathscr{C}_3$) such that the characteristic parabolic point in the dynamical plane of any parameter on $\mathscr{C}_3$ is the landing point of precisely two dynamical rays at angles $\alpha_1$ and $\alpha_2$. Moreover, the parameter rays at angles $\alpha_1$ and $\alpha_2$ (and no other) accumulate on $\mathscr{C}_3$. The rest of the two parabolic arcs (say, $\mathscr{C}_1$ and $\mathscr{C}_2$) on $\partial H$ have the property that the characteristic parabolic point in the dynamical plane of any parameter on $\mathscr{C}_i$ ($i=1,2$) is the landing point of precisely one dynamical ray at angle $\theta_i$. Furthermore, $\mathscr{C}_i$ ($i=1,2$) contains the accumulation set of exactly one parameter ray at angle $\theta_i$. 

At the parabolic cusp on $\partial H$ where $\mathscr{C}_1$ and $\mathscr{C}_2$ meet, the characteristic parabolic point is the landing point of exactly two dynamical rays at angles $\theta_1$ and $\theta_2$. The same is true at the center of the hyperbolic component of period $2k$ that bifurcates from $H$ across this parabolic cusp. Moreover, these angles are the characteristic angles of the corresponding orbit portrait.

On the other hand, at the parabolic cusp where $\mathscr{C}_1$ and $\mathscr{C}_3$ (respectively, $\mathscr{C}_2$ and $\mathscr{C}_3$) meet, the characteristic parabolic point is the landing point of precisely three dynamical rays at angles $\alpha_1$, $\alpha_2$ and $\theta_1$ (respectively, $\alpha_1$, $\alpha_2$ and $\theta_2$). As before, the same is true at the center of the hyperbolic component of period $2k$ that bifurcates from $H$ across this parabolic cusp. The characteristic angles of the corresponding orbit portrait are $\alpha_1$ and $\theta_1$ (respectively, $\theta_2$ and $\alpha_2$). Finally, Proposition~\ref{orbit_portraits_preserved} allows us to translate the second statement of \cite[Lemma~3.5]{Sa} to the current setting implying the following relation 
\begin{equation}
\rho^{\circ k}((\alpha_1,\theta_1))=(\theta_1,\alpha_2),\ \rho^{\circ k}((\theta_2,\alpha_2))=(\alpha_1,\theta_2).
\label{char_rays_relation_2}
\end{equation}

\begin{theorem}[Boundaries of odd period hyperbolic components]\label{Exactly 3}
The boundary of every hyperbolic component of odd period of $\cC(\mathcal{S})$ is a topological triangle having parabolic cusps as vertices and parabolic arcs as sides.
\end{theorem}

Let us briefly carry out a similar analysis for even period hyperbolic components that do not bifurcate from odd period ones. Let $H$ be a hyperbolic component of even period $k$ such that $H$ does not bifurcate from an odd period hyperbolic component. Let $\mathcal{A}_1$ be the set of angles of the dynamical rays landing at the dynamical root of $F_{a}$ (where $a\in H$ or $a$ is the root point of $H$). Then, the first return map of the dynamical root either fixes every angle in $\mathcal{A}_1$ and $\vert\mathcal{A}_1\vert=2$, or permutes the angles in $\mathcal{A}_1$ transitively (by Proposition~\ref{orbit_portraits_preserved} and \cite[Lemma~3.3]{Sa}). Moreover, the characteristic angles $t^\pm$ of the orbit portrait $\mathcal{P}$ generated by $\mathcal{A}_1$ are precisely the two adjacent angles in $\mathcal{A}_1$ (with respect to circular order) that separate $0$ from $\infty$, and bound a sector of angular width less that $\frac12$. It now follows from Theorem~\ref{realization_orbit_portrait_parabolic_schwarz} and Proposition~\ref{rigidity_para_even} that the parameter rays at angles $t^\pm$ land at the root point of $H$.

To conclude this subsection, let us state a generalization of Proposition~\ref{char_angles_lamination}.

\begin{proposition}\label{char_angles_lamination_2}
For any hyperbolic or parabolic map $F_a$, the pre-periodic lamination is completely determined by the characteristic angles of $F_a$.
\end{proposition}
\begin{proof}
It follows from the above discussion that the pre-periodic lamination and the characteristic angles remain unaltered throughout an odd period hyperbolic component and the parabolic arcs on its boundary (respectively, throughout an even period hyperbolic component and its root). The result now follows from Proposition~\ref{char_angles_lamination}.
\end{proof}

\subsection{Parameter rays at pre-periodic angles}\label{para_pre_per} 

In this section, we will study the landing properties of parameter rays of $\mathcal{S}$ at strictly pre-periodic angles. Let $a_0$ be a Misiurewicz parameter, and $\mathcal{A}'$ be the set of angles of the dynamical rays of $F_{a_0}$ landing at the critical point $0$. The set of angles of the dynamical rays of $F_{a_0}$ that land at the critical value $\infty$ is then given by $\mathcal{A}:=\rho(\mathcal{A}')\subset(1/3,2/3)$. Moreover, $\rho$ is two-to-one from $\mathcal{A}'$ onto $\mathcal{A}$, and is injective on all other $\lambda(F_{a_0})-$classes. As for quadratic anti-polynomials, the existence of a unique $\lambda(F_{a_0})$-equivalence class that maps two-to-one onto its image equivalence class (under $\rho$) characterizes pre-periodic laminations of Misiurewicz maps. It is called a pre-periodic lamination of Misiurewicz type.

\begin{proposition}[Landing of parameter rays at pre-periodic angles]\label{misi_para_land}
Let $\theta\in(1/3,2/3)$ be strictly pre-periodic under $\rho$. Then the parameter ray of $\mathcal{S}$ at angle $\theta$ lands at a Misiurewicz parameter such that in the corresponding dynamical plane, the dynamical ray at angle $\theta$ lands at the critical value $\infty$.
\end{proposition}
\begin{proof}
Let $a_0$ be an accumulation point of the parameter ray of $\mathcal{S}$ at angle $\theta$. Arguing as in the second part of \cite[Theorem~37.35]{L6}, we will conclude that $a_0$ is a Misiurewicz parameter such that in the dynamical plane of $F_{a_0}$, the dynamical ray at angle $\theta$ lands at the critical value $\infty$. 

By Proposition~\ref{per_rays_land}, the dynamical ray of $F_{a_0}$ at angle $\theta$ lands at some repelling or parabolic pre-periodic point $w$ (as $\theta$ is strictly pre-periodic under $\rho$, the landing point cannot be periodic). Let us suppose that $F_{a_0}$ is a parabolic map. Note that as the landing point of the dynamical $\theta$-ray of $F_{a_0}$ is not periodic, the ray does not land at the characteristic parabolic point of $F_{a_0}$. Since the limit set of a parabolic map is locally connected and repelling periodic points are dense on the limit set (see Theorem~\ref{geom_finite_limit_set}), it follows that there exists a cut-line through repelling periodic points on $\Gamma_{a_0}$ separating the $\theta$-dynamical ray from the critical value. But such cut-lines remain stable under small perturbation. Therefore, for parameters sufficiently close to $a_0$, the $\theta$-dynamical ray stays away from the critical value. However, this is impossible as there are parameters near $a_0$ on the $\theta$-parameter ray for which the critical value lies on the $\theta$-dynamical ray. This contradiction shows that $F_{a_0}$ is not parabolic; i.e., the dynamical ray of $F_{a_0}$ at angle $\theta$ lands at some repelling pre-periodic point $w$.

We suppose that $w$ is not the critical value of $F_{a_0}$, and will arrive at a contradiction. If $w$ is not a pre-critical point either, then for nearby parameters, the $\theta$-dynamical ray would land at the real-analytic continuation of the repelling pre-periodic point $w$, and would stay away from the critical value. But there are parameters near $a_0$ on the $\theta$-parameter ray for which the critical value lies on the $\theta$-dynamical ray, a contradiction. Hence $w$ must be a pre-critical point implying that the critical point of $F_{a_0}$ is strictly pre-periodic. So $a_0$ is a Misiurewicz parameter. As the limit set of a Misiurewicz map is a dendrite (by Theorem~\ref{geom_finite_limit_set}) and repelling periodic points are dense on it (see Theorem~\ref{geom_finite_limit_set}), the dynamical ray at angle $\theta$ landing at $w$ can be separated from the critical value by a pair of dynamical rays landing at a common repelling periodic point. Once again, this separation line remains stable under perturbation, contradicting the existence of parameters near $a_0$ on the $\theta$-parameter ray. Hence, $w$ must be the critical value of $F_{a_0}$.

We claim that $a_0$ is the unique parameter in $\cC(\mathcal{S})$ with the property that the dynamical ray at angle $\theta$ lands at the critical value $\infty$. Since the limit set of a ray is connected, this will prove that the parameter ray at angle $\theta$ indeed lands at $a_0$.

To prove the claim, let us assume that there exists another parameter $a_1$ with the same property. Note that both the pre-periodic laminations $\lambda(F_{a_0})$ and $\lambda(F_{a_1})$ are of Misiurewicz type. Hence, the formal rational laminations $\mathcal{E}_{\ast}(\lambda(F_{a_0}))$ and $\mathcal{E}_{\ast}(\lambda(F_{a_1}))$ are also of Misiurewicz type. By Theorem~\ref{rat_lam_realized}, there exist Misiurewicz parameters $c_0$ and $c_1$ in $\mathcal{L}$ realizing these rational laminations. By construction, the dynamical ray $R_{c_0}(\mathcal{E}(\theta))$ (respectively, $R_{c_1}(\mathcal{E}(\theta))$) lands at the critical value $c_0$ (respectively, $c_1$) of $f_{c_0}$ (respectively, of $f_{c_1}$). It now follows by Theorem~\ref{Tricorn_para_misi_ray} that the parameter ray $\mathcal{R}_{\mathcal{E}(\theta)}$ lands both at $c_0$ and $c_1$ implying that $c_0=c_1$. Therefore, we have $\lambda(F_{a_0})=\lambda(F_{a_1})$. By Proposition~\ref{rigidity_misi}, we conclude that $a_0=a_1$. This completes the proof.
\end{proof}

Recall that for a Misiurewicz map $f_{c_0}$, the angles of the parameter rays of $\mathcal{T}$ landing at $c_0$ are exactly the external angles of the dynamical rays that land at the critical value $c_0$ in the dynamical plane of $f_{c_0}$ (see Theorem~\ref{Tricorn_para_misi_ray}). We will now prove an analogous statement for Misiurewicz parameters of $\cC(\mathcal{S})$.

\begin{proposition}[Correspondence between dynamical and parameter rays]\label{misi_dyn_para_rays}
Let $a_0\in\cC(\mathcal{S})$ be a Misiurewicz parameter. Then, the angles of the parameter rays (at pre-periodic angles) of $\mathcal{S}$ landing at $a_0$ are exactly the external angles of the dynamical rays that land at the critical value $\infty$ in the dynamical plane of $F_{a_0}$. 
\end{proposition}
\begin{proof}
Let $\mathcal{A}\subset(1/3,2/3)$ be the set of angles of the dynamical rays of $F_{a_0}$ that land at the critical value $\infty$. By Proposition~\ref{misi_para_land}, the angles of the parameter rays (at pre-periodic angles) of $\mathcal{S}$ landing at $a_0$ are contained in $\mathcal{A}$.

Now pick $\theta\in\mathcal{A}$, and let $a_1$ be the landing point of the parameter ray of $\mathcal{S}$ at angle $\theta$. Then, the dynamical ray of $F_{a_1}$ at angle $\theta$ lands at the critical value $\infty$. By the proof of Proposition~\ref{misi_para_land}, we know that there can be at most one parameter in $\cC(\mathcal{S})$ whose dynamical $\theta$-ray lands at the critical value $\infty$. Therefore, $a_0=a_1$, i.e., the parameter ray of $\mathcal{S}$ at angle $\theta$ lands at $a_0$. As $\theta$ was an arbitrary element of $\mathcal{A}$, it follows that all parameter rays at angles in $\mathcal{A}$ land at $a_0$. The proof is now complete.
\end{proof}

\section{Combinatorial straightening}\label{comb_bijec_pcf}

In this section, we prove the main theorems of the paper.

The proof of Theorem~\ref{thm_comb_bijec_pcf} will be carried out in two stages. In Subsection~\ref{comb_bijec_centers}, we will construct a lamination-preserving bijection between the centers of $\cC(\mathcal{S})$ and the centers of $\mathcal{L}$. A lamination-preserving bijection between the Misiurewicz parameters of $\cC(\mathcal{S})$ and $\mathcal{L}$ will be constructed in Subsection~\ref{comb_bijec_misi}. 

\subsection{Combinatorial bijection for hyperbolic and parabolic parameters}\label{comb_bijec_centers}

In this subsection, we will prove Theorem~\ref{thm_comb_bijec_pcf} for the hyperbolic and parabolic parameters of $\cC(\mathcal{S})$. Recall that for every hyperbolic and parabolic parameter $a$, the first return map of the characteristic Fatou component $U_a$ has a unique fixed point on $\partial U_a$ that is a cut-point of $K_a$ (which we call the dynamical root of $F_a$). The orbit portrait (more precisely, its characteristic angles $t^\pm$) associated with the dynamical root of $F_a$ completely determines $\lambda(F_a)$ (see Propositions~\ref{char_angles_lamination} and~\ref{char_angles_lamination_2}). In fact, all the iterated pullbacks of the leaf connecting $t^-$ and $t^+$ under $\rho$ are pairwise disjoint, and their closure in $\Q/\Z$ is equal to $\lambda(F_a)$.

\begin{lemma}\label{thm_comb_bijec_centers}
For every super-attracting map $F_{a_0}\in\cC(\mathcal{S})$ with pre-periodic lamination $\lambda(F_{a_0})$, there exists a unique super-attracting map $f_{c_0}\in\mathcal{L}$ with associated rational lamination $\mathcal{E}_{\ast}(\lambda(F_{a_0}))$. Moreover, this correspondence is a bijection between the super-attracting maps of $\cC(\mathcal{S})$ and $\mathcal{L}$.
\end{lemma}
\begin{proof}
Note that super-attracting maps are precisely the centers of hyperbolic components. 

Let $a_0$ be the center of a hyperbolic component $H$ of even period (other than two) of $\cC(\mathcal{S})$ that does not bifurcate from a hyperbolic component of odd period. Let $\mathcal{P}$ be the orbit portrait associated with the dynamical root of $F_{a_0}$, the characteristic angles of $\mathcal{P}$ be $t^\pm$, and the pre-periodic lamination of $F_{a_0}$ be $\lambda(F_{a_0})$. By Proposition~\ref{orbit_portraits_preserved}, $\mathcal{E}_{\ast}(\mathcal{P})$ is an $m_{-2}$-FOP with characteristic angles $\mathcal{E}(t^\pm)$. Since the orbit period of $\mathcal{P}$ is even, the second statement of Theorem~\ref{realization_orbit_portrait_parabolic} implies that the parameter rays $\mathcal{R}_{\mathcal{E}(t^-)}$ and $\mathcal{R}_{\mathcal{E}(t^+)}$ of the Tricorn land at the root point of some hyperbolic component $H'$ of $\mathcal{T}$. Moreover, the orbit portrait associated with the parabolic cycle of the root of $H'$ is given by $\mathcal{E}_{\ast}(\mathcal{P})$. It follows that in the dynamical plane of the center $c_0$ of $H'$, the orbit portrait associated with the dynamical root is $\mathcal{E}_{\ast}(\mathcal{P})$. In particular, $\mathcal{E}(t^\pm)$ are the characteristic angles of $f_{c_0}$. 

By Proposition~\ref{char_angles_lamination}, the two angles $t^\pm$ (respectively, $\mathcal{E}(t^\pm)$) completely determine the pre-periodic lamination of $F_{a_0}$ (respectively, the rational lamination of $f_{c_0}$); more precisely, all the pullbacks of the leaf connecting $t^+$ and $t^-$ (respectively, $\mathcal{E}(t^+)$ and $\mathcal{E}(t^-)$) under iterations of $\rho$ (respectively, of $m_{-2}$) are pairwise disjoint, and their closure in $\mathrm{Per}(\rho)$ (respectively, in $\Q/\Z$) is equal to $\lambda(F_{a_0})$ (respectively, $\lambda(f_{c_0})$). Therefore, $\lambda(f_{c_0})=\mathcal{E}_{\ast}(\lambda(F_{a_0}))$. Since $1/3\sim2/3$ in $\lambda(F_{a_0})$, it follows that the dynamical rays $R_{c_0}(1/3)$ and $R_{c_0}(2/3)$ land at a common point of $J(f_{c_0})$. Hence, $c_0\in\mathcal{L}$. Finally, by \cite[Theorem~5.1]{MNS} (also compare \cite[Theorem~35.1]{L6}), $f_{c_0}$ is the unique PCF map in $\mathcal{L}$ with rational lamination $\mathcal{E}_{\ast}(\lambda(F_{a_0}))$. 

Thanks to Theorem~\ref{realization_orbit_portrait_parabolic_schwarz}, the previous argument also goes in the opposite direction demonstrating that if $c_0$ is the center of a hyperbolic component of even period of $\mathcal{L}$ that does not bifurcate from a hyperbolic component of odd period, then there exists some super-attracting map $F_{a_0}$ with associated pre-periodic lamination $\mathcal{E}^{\ast}(\lambda(f_{c_0}))$.

We now turn our attention to hyperbolic components of odd period and hyperbolic components of even period bifurcating from them. Let $a_0$ be the center of a hyperbolic component $H$ of odd period $k$. Let $\mathcal{P}$ be the orbit portrait associated with the dynamical root of $F_{a_0}$, and the characteristic angles of $\mathcal{P}$ be $\alpha_1$ and $\alpha_2$. By Subsection~\ref{para_orbit_portrait}, each of the two co-roots of $F_{a_0}$ is the landing point of exactly one dynamical ray at angle $\theta_i$ ($i=1,2$). There are three hyperbolic components of period $2k$ bifurcating from $H$, and the characteristic angles (of the orbit portraits associated with the dynamical roots) of their centers are $\{\theta_1,\theta_2\}$, $\{\alpha_1,\theta_1\}$, and $\{\theta_2,\alpha_2\}$. Moreover, these angles satisfy Relation~\eqref{char_rays_relation_2}.

By Proposition~\ref{orbit_portraits_preserved} and Theorem~\ref{realization_orbit_portrait_parabolic}, the parameter rays $\mathcal{R}_{\mathcal{E}(\alpha_1)}$ and $\mathcal{R}_{\mathcal{E}(\alpha_2)}$ of the Tricorn accumulate on a common root arc $\mathscr{C}$ of $\mathcal{T}$, and for every parameter $c\in\mathscr{C}$, the parabolic orbit portrait is $\mathcal{E}_{\ast}(\mathcal{P})$. Let $\mathscr{C}\subset\partial H'$ (where $H'$ is a hyperbolic component of period $k$ of $\mathcal{T}$), and $c_0$ be the center of $H'$. Then, the orbit portrait associated with the dynamical root of $f_{c_0}$ is $\mathcal{E}_{\ast}(\mathcal{P})$. As in the previous case, this implies that $\lambda(f_{c_0})=\mathcal{E}_{\ast}(\lambda(F_{a_0}))$, and $c_0\in\mathcal{L}$. Moreover, $f_{c_0}$ is the unique PCF map in $\mathcal{L}$ with rational lamination $\mathcal{E}_{\ast}(\lambda(F_{a_0}))$.

It also follows from the above discussion that the angles $\mathcal{E}(\alpha_i)$ and $\mathcal{E}(\theta_i)$ (where $i=1,2$) satisfy Relation~\eqref{char_rays_relation_1}. Hence, the dynamical rays at angles $\mathcal{E}(\theta_i)$ land at the dynamical co-roots of $f_{c_0}$. Therefore, the characteristic angles (of the orbit portraits associated with the dynamical roots) of the centers of the hyperbolic components of period $2k$ bifurcating from $H'$ are given by $\{\mathcal{E}(\theta_1),\mathcal{E}(\theta_2)\}$, $\{\mathcal{E}(\alpha_1),\mathcal{E}(\theta_1)\}$, and $\{\mathcal{E}(\theta_2),\mathcal{E}(\alpha_2)\}$. It follows that the push-forwards of the pre-periodic laminations of the centers of the three hyperbolic components bifurcating from $H$ are precisely the rational laminations of the centers of the three hyperbolic components bifurcating from $H'$.

As in the even period case, one can use Theorem~\ref{realization_orbit_portrait_parabolic_schwarz} and the combinatorial description of odd period hyperbolic components of $\cC(\mathcal{S})$ and $\mathcal{L}$ to conclude surjectivity of the map between centers. 

It remains to discuss hyperbolic components of period two. There is exactly one hyperbolic component $H_2$ (respectively, $H'_2$) of period two in $\cC(\mathcal{S})$ (respectively, in $\mathcal{L}$). The center of $H_2$ (respectively, $H'_2$) is $0$ (respectively, $-1$). In the dynamical planes of the centers of these components, the dynamical rays at angles $1/3$ and $2/3$ land at the $\alpha$-fixed point (which is the dynamical root), and these are the characteristic angles of the corresponding orbit portrait. Moreover, these two angles completely determine the corresponding pre-periodic (respectively, rational) lamination (compare Proposition~\ref{char_angles_lamination}). Since $\mathcal{E}$ fixes $1/3$ and $2/3$, it follows that $\mathcal{E}_{\ast}(\lambda(F_0))=\lambda(f_{-1})$.

Finally, by Proposition~\ref{rigidity_center}, two distinct super-attracting maps in $\cC(\mathcal{S})$ cannot have the same pre-periodic lamination. Hence, the correspondence between the centers of $\cC(\mathcal{S})$ and $\mathcal{L}$ defined above is injective.
\end{proof}

As a corollary of the above proof (combined with our analysis of the hyperbolic components of $\mathcal{S}$ and their boundaries, and the rigidity results of Subsection~\ref{rigidity_hyp_para}), we get a combinatorial bijection between the hyperbolic and parabolic parameters of $\cC(\mathcal{S})$ and those of $\mathcal{L}$. 

\begin{corollary}[Bijection between hyperbolic and parabolic parameters of $\cC(\mathcal{S})$ and $\mathcal{L}$]\label{hyp_para_bijec_cor}
1) If $a\in H\subset\cC(\mathcal{S})$ is a hyperbolic parameter (contained in the hyperbolic component $H$) with associated pre-periodic lamination $\lambda(F_a)$, then there exists a unique hyperbolic parameter $c\in H'\subset\mathcal{L}$ (contained in the hyperbolic component $H'$) with associated rational lamination $\lambda(f_c)=\mathcal{E}_\ast(\lambda(F_a))$ satisfying $\eta_{H'}(c)=\widetilde{\eta}_H(a)$. In particular, the dynamics on the respective periodic Fatou components of $F_a$ and $f_c$ are conformally conjugate. Moreover, this correspondence is a bijection between the hyperbolic parameters of $\cC(\mathcal{S})$ and $\mathcal{L}$.

2) If $a\in\cC(\mathcal{S})$ is a parabolic parameter with associated pre-periodic lamination $\lambda(F_a)$ (and critical {\'E}calle height $h$, if $F_a$ has an odd-periodic simple parabolic cycle), then there exists a unique parabolic parameter $c\in\mathcal{L}$ with associated rational lamination $\mathcal{E}_\ast(\lambda(F_a))$ (with the same critical {\'E}calle height $h$ in the odd-periodic simple parabolic case). In particular, the dynamics on the respective periodic Fatou components of $F_a$ and $f_c$ are conformally conjugate. Moreover, this correspondence is a bijection between the parabolic parameters of $\cC(\mathcal{S})$ and $\mathcal{L}$.
\end{corollary}

\begin{proof}
1) Let $a_0$ be the center of the hyperbolic component $H\ni a$ of $\cC(\mathcal{S})$. By Lemma~\ref{thm_comb_bijec_centers}, there exists a super-attracting parameter $c_0\in\mathcal{L}$ such that $\lambda(f_{c_0})=\mathcal{E}_\ast(\lambda(F_{a_0}))$. We can assume that $c_0$ is the center of the hyperbolic component $H'$ of $\mathcal{L}$.

According to Theorem~\ref{unif_hyp_schwarz}, there exists a homeomorphism $\widetilde{\eta}_H:H\to\mathcal{B}^{\pm}$ (respectively, $\eta_{H'}:H'\to\mathcal{B}^{\pm}$) that maps the center of $H$ (respectively, $H'$) to the super-attracting Blaschke product $B_{0,0}^{\pm}$. We now define the parameter $c:=\eta_{H'}^{-1}\circ\widetilde{\eta}_H(a)$. By construction, the first return maps of $F_a$ and $f_c$ to their respective characteristic Fatou components are conformally conjugate to a common Blaschke product in $\mathcal{B}^{\pm}$, and hence these return maps are conformally conjugate to each other. Thus, the dynamics on the respective periodic Fatou components of $F_a$ and $f_c$ are conformally conjugate. In particular, the Koenigs ratio/multiplier of the corresponding attracting cycles are equal. Moreover, the laminations of the two maps are related in the desired way.

The fact that the above correspondence induces a bijection between the hyperbolic parameters of $\cC(\mathcal{S})$ and $\mathcal{L}$ can now be proved mimicking the arguments of Lemma~\ref{thm_comb_bijec_centers} (and invoking Proposition~\ref{rigidity_hyp_prop}).

2) Given a parabolic parameter $a\in\cC(\mathcal{S})$ with associated pre-periodic lamination $\lambda(F_a)$ (and critical {\'E}calle height $h$, if $F_a$ has an odd-periodic simple parabolic cycle), the existence of a unique parabolic parameter $c\in\mathcal{L}$ with associated rational lamination $\mathcal{E}_\ast(\lambda(F_a))$ (and with critical {\'E}calle height $h$ in the odd-periodic simple parabolic case) follows from the proof of Lemma~\ref{thm_comb_bijec_centers} (cf. \cite[Lemma~5.2]{MNS}). The rigidity statements of Propositions~\ref{rigidity_para_even} and~\ref{rigidity_para_odd} imply that this defines a bijective correspondence between the parabolic parameters of $\cC(\mathcal{S})$ and $\mathcal{L}$.

Note further that if the characteristic Fatou components of $F_a$ and $f_c$ have even period, then the first return maps of $F_a$ and $f_c$ to these components are conformally conjugate to the unique quadratic parabolic Blaschke product $\frac{3z^2+1}{z^2+3}$. Hence, these return maps are conformally conjugate to each other; i.e., the dynamics on the respective periodic Fatou components of $F_a$ and $f_c$ are conformally conjugate. On the other hand, if the characteristic Fatou components of $F_a$ and $f_c$ have odd period, then the first return maps of $F_a$ and $f_c$ to these components are conformally conjugate to a common quadratic parabolic anti-holomorphic Blaschke product, precisely because the maps are required to have equal critical {\'E}calle height. Thus, in this case too, the dynamics on the respective periodic Fatou components of $F_a$ and $f_c$ are conformally conjugate.
\end{proof}

The proof of Lemma~\ref{thm_comb_bijec_centers} also implies that the map $\mathcal{E}$ induces a bijection between the parameter rays at periodic angles landing/accumulating at the parabolic parameters of $\cC(\mathcal{S})$ and $\mathcal{L}$.

\begin{corollary}[Correspondence of parameter rays at periodic angles]\label{para_rays_correspond}
1) If a parabolic parameter $a\in\cC(\mathcal{S})$ of even parabolic period corresponds to the parabolic parameter $c\in\mathcal{L}$ under the above bijection, then the angles of the two parameter rays (at periodic angles) landing at $a$ are precisely the $\mathcal{E}$-pre-images of the angles of the two parameter rays (at periodic angles) landing at $c$.

2) If a root (respectively, co-root) parabolic arc $\mathscr{C}\subset\cC(\mathcal{S})$ corresponds to the arc $\mathscr{C}'\subset\mathcal{L}$, then the angles of the two parameter rays at periodic angles (respectively, the angle of the unique parameter ray at a periodic angle) accumulating on $\mathscr{C}$ are precisely the $\mathcal{E}$-pre-images of the angles of the two parameter rays at periodic angles (respectively, the angle of the unique parameter ray at a periodic angle) accumulating on $\mathscr{C}'$.
\end{corollary}

\subsection{Combinatorial bijection for Misiurewicz parameters}\label{comb_bijec_misi}

We now turn our attention to the Misiurewicz parameters.

\begin{lemma}[Bijection between Misiurewicz parameters of $\cC(\mathcal{S})$ and $\mathcal{L}$]\label{thm_comb_bijec_misi}
For every Misiurewicz parameter $a_0\in\cC(\mathcal{S})$ with pre-periodic lamination $\lambda(F_{a_0})$, there exists a unique Misiurewicz parameter $c_0\in\mathcal{L}$ with associated rational lamination $\mathcal{E}_{\ast}(\lambda(F_{a_0}))$. Moreover, this correspondence is a bijection between the Misiurewicz parameters of $\cC(\mathcal{S})$ and $\mathcal{L}$.
\end{lemma}
\begin{proof}
Pick a Misiurewicz parameter $a_0\in\cC(\mathcal{S})$ with pre-periodic lamination $\lambda(F_{a_0})$. Then by Proposition~\ref{prop_preper_lami}, $\mathcal{E}_{\ast}(\lambda(F_{a_0}))$ is a formal rational lamination of Misiurewicz type. By Theorem~\ref{rat_lam_realized}, there exists a unique Misiurewicz parameter $c_0\in\mathcal{L}$ such that $\lambda(f_{c_0})= \mathcal{E}_{\ast}(\lambda(F_{a_0}))$. Since $1/3\sim2/3$ in $\lambda(F_{a_0})$, the same identification holds in $\lambda(f_{c_0})$ as well (recall that $\mathcal{E}$ fixes $1/3$ and $2/3$). Therefore, the dynamical rays $R_{c_0}(1/3)$ and $R_{c_0}(2/3)$ land at a common point of $\mathcal{J}_{c_0}$; hence $c_0\in\mathcal{L}$.

According to Proposition~\ref{rigidity_misi}, two distinct Misiurewicz parameters of $\cC(\mathcal{S})$ cannot have the same pre-periodic lamination. This shows that the map between Misiurewicz parameters of $\cC(\mathcal{S})$ and $\mathcal{L}$ defined above is injective.

It remains to prove surjectivity of the above map. Pick a Misiurewicz parameter $c_0\in\mathcal{L}$ with rational lamination $\lambda(f_{c_0})$. Suppose that the set of critical value angles of $f_{c_0}$ is $\mathcal{A}$. Then by Theorem~\ref{Tricorn_para_misi_ray}, the parameter $c_0$ is the landing point of the parameter rays of $\mathcal{T}$ at angles in $\mathcal{A}$. Moreover, as $c_0\in\mathcal{L}$, we have that $\mathcal{A}\subset(1/3,2/3)$. Pick $\theta\in\mathcal{E}^{-1}(\mathcal{A})$. Clearly, $\theta\in(1/3,2/3)$. By Proposition~\ref{misi_para_land}, the parameter ray of $\mathcal{S}$ at angle $\theta$ lands at a Misiurewicz parameter $a_0$ such that the dynamical ray at angle $\theta$ of $F_{a_0}$ lands at the critical value $\infty$. By Proposition~\ref{prop_preper_lami}, the push-forward $\mathcal{E}_{\ast}(\lambda(F_{a_0}))$ is a formal rational lamination of Misiurewicz type, and hence is realized as the actual rational lamination of some Misiurewicz parameter $c_1\in\mathcal{T}$. Our construction implies that $\mathcal{E}(\theta)$ is a critical value angle for $f_{c_1}$. Once again by Theorem~\ref{Tricorn_para_misi_ray}, the parameter ray $\mathcal{R}_{\mathcal{E}(\theta)}$ (of $\mathcal{T}$) lands at $c_1$. Since $\mathcal{E}(\theta)\in\mathcal{A}$, it follows that $c_0=c_1$. Hence, $\mathcal{E}_{\ast}(\lambda(F_{a_0}))=\lambda(f_{c_0})$.
\end{proof}

The following important result relates the angles of the parameter rays of $\mathcal{S}$ landing at a Misiurewicz parameter $a_0$ to those of the parameter rays of $\mathcal{T}$ landing at the corresponding parameter $c_0$.

\begin{corollary}[Correspondence of parameter rays at pre-periodic angles]\label{misi_rays_correspond}
Let $a_0$ and $c_0$ be Misiurewicz parameters in $\cC(\mathcal{S})$ and $\mathcal{L}$ (respectively) such that $\mathcal{E}_{\ast}(\lambda(F_{a_0}))=\lambda(f_{c_0})$, and $\mathcal{A}$ be the set of angles of the parameter rays of $\mathcal{S}$ (at pre-periodic angles) landing at $a_0$. Then, $\mathcal{E}(\mathcal{A})$ is precisely the set of angles of the parameter rays of $\mathcal{T}$ (at pre-periodic angles) landing at $c_0$. 
\end{corollary}
\begin{proof}
By Proposition~\ref{misi_dyn_para_rays}, the set of external angles of the dynamical rays that land at the critical value $\infty$ in the dynamical plane of $F_{a_0}$ is precisely $\mathcal{A}$. Since $\mathcal{E}_{\ast}(\lambda(F_{a_0}))=\lambda(f_{c_0})$, the set of external angles of the dynamical rays that land at the critical value $c_0$ in the dynamical plane of $f_{c_0}$ is equal to $\mathcal{E}(\mathcal{A})$. Finally by Theorem~\ref{Tricorn_para_misi_ray}, $\mathcal{E}(\mathcal{A})$ is the set of angles of the parameter rays of $\mathcal{T}$ (at pre-periodic angles) landing at $c_0$.
\end{proof}

Theorem~\ref{thm_comb_bijec_pcf} now readily follows.

\begin{proof}[Proof of Theorem~\ref{thm_comb_bijec_pcf}]
This clearly follows from Corollary~\ref{hyp_para_bijec_cor} and Lemma~\ref{thm_comb_bijec_misi}. 
\end{proof}

Let us give a name to the bijection between the geometrically finite parameters of $\cC(\mathcal{S})$ and $\mathcal{L}$ established in Theorem~\ref{thm_comb_bijec_pcf}.

\begin{definition}\label{comb_straightening}
For a geometrically finite parameter $a_0\in\cC(\mathcal{S})$, we denote by $\chi(a_0)$ the unique geometrically finite parameter $c_0\in\mathcal{L}$ such that $\mathcal{E}_{\ast}(\lambda(F_{a_0}))=\lambda(f_{c_0})$ and the first return maps of the characteristic Fatou components (of $F_{a_0}$ and $f_{c_0}$) are conformally conjugate. The map $f_{\chi(a_0)}$ will be called the \emph{combinatorial straightening} of $F_{a_0}$.
\end{definition}

\section{Conformal mating in the circle-and-cardioid family}\label{sec_PCF_mating}

\subsection{A combinatorial condition for mating}\label{mating_comb_condition}

Let us state and prove a general combinatorial condition that guarantees mateability of a post-critically finite quadratic anti-polynomial and the reflection map $\rho$.

\begin{proposition}\label{lami_correspond_mating}
Let $a_0\in\cC(\mathcal{S})$ and $c_0\in\mathcal{L}$ be post-critically finite maps satisfying $\mathcal{E}_{\ast}(\lambda(F_{a_0}))=\lambda(f_{c_0})$. Then, $F_{a_0}:K_{a_0}\to K_{a_0}$ is topologically conjugate to $f_{c_0}:\mathcal{K}_{c_0}\to\mathcal{K}_{c_0}$ such that the conjugacy is conformal on $\Int{K_{a_0}}$, and $F_{a_0}:T_{a_0}^\infty\setminus \Int{T_{a_0}^0}\to T_{a_0}^\infty$ is conformally conjugate to $\rho:\D\setminus\Int{\Pi}\to\D$.
\end{proposition}

We start with an intermediate lemma, which states that the homeomorphism $\mathcal{E}:\R/\Z\to\R/\Z$ induces a topological conjugacy between the maps $F_{a_0}$ and $f_{c_0}$ on their respective limit and Julia sets. 

\begin{lemma}\label{julia_top_conjugacy}
With $a_0$ and $c_0$ as in Proposition~\ref{lami_correspond_mating}, the homeomorphism $\mathcal{E}$ of the circle induces a topological conjugacy $\mathcal{E}_{a_0}$ between $F_{a_0}:\Gamma_{a_0}\to \Gamma_{a_0}$ and $f_{c_0}:\mathcal{J}_{c_0}\to\mathcal{J}_{c_0}$.
\end{lemma}
\begin{proof}
As $f_{c_0}$ is a post-critically finite map, $\mathcal{J}_{c_0}$ is locally connected. Hence, the inverse of the B{\"o}ttcher conjugacy $\phi_{c_0}:\widehat{\C}\setminus\mathcal{K}_{c_0}\to\widehat{\C}\setminus\overline{\D}$ between $f_{c_0}$ and $\overline{z}^2$ extends continuously to $\R/\Z$, and yields a semi-conjugacy between $m_{-2}:\R/\Z\to\R/\Z$ and $f_{c_0}:\mathcal{J}_{c_0}\to\mathcal{J}_{c_0}$. The fibers of this semi-conjugacy give rise to an equivalence relation $\lambda_{\R}(f_{c_0})$ (on $\R/\Z$), which is the smallest closed equivalence relation containing the rational lamination $\lambda(f_{c_0})$ (see \cite[Lemma~4.17]{kiwi}). Therefore, $f_{c_0}:\mathcal{J}_{c_0}\to\mathcal{J}_{c_0}$ is topologically conjugate to the quotient of $m_{-2}:\R/\Z\to\R/\Z$ by the $m_{-2}$-invariant lamination $\lambda_{\R}(f_{c_0})$.

Again, for a post-critically finite parameter $a_0$, the limit set $\Gamma_{a_0}$ is locally connected (see Theorem~\ref{geom_finite_limit_set}). Hence, the inverse of the external conjugacy $\psi_{a_0}:T_{a_0}^\infty\to\D$ between $F_{a_0}$ and $\rho$ extends continuously to the boundary, and yields a semi-conjugacy between $\rho:\R/\Z\to\R/\Z$ and $F_{a_0}:\Gamma_{a_0}\to \Gamma_{a_0}$. The fibers of this semi-conjugacy give rise to an equivalence relation $\lambda_{\R}(F_{a_0})$ (on $\R/\Z$), which is the smallest closed equivalence relation on $\R/\Z$ containing the pre-periodic lamination $\lambda(F_{a_0})$ (this can be proved following the arguments of \cite[Lemma~4.17]{kiwi}). Therefore, $F_{a_0}:\Gamma_{a_0}\to \Gamma_{a_0}$ is topologically conjugate to the quotient of $\rho:\R/\Z\to\R/\Z$ by the $\rho$-invariant lamination $\lambda_{\R}(F_{a_0})$.

\begin{figure}[ht!]
\captionsetup{width=0.96\linewidth}
\begin{center}
\includegraphics[scale=0.36]{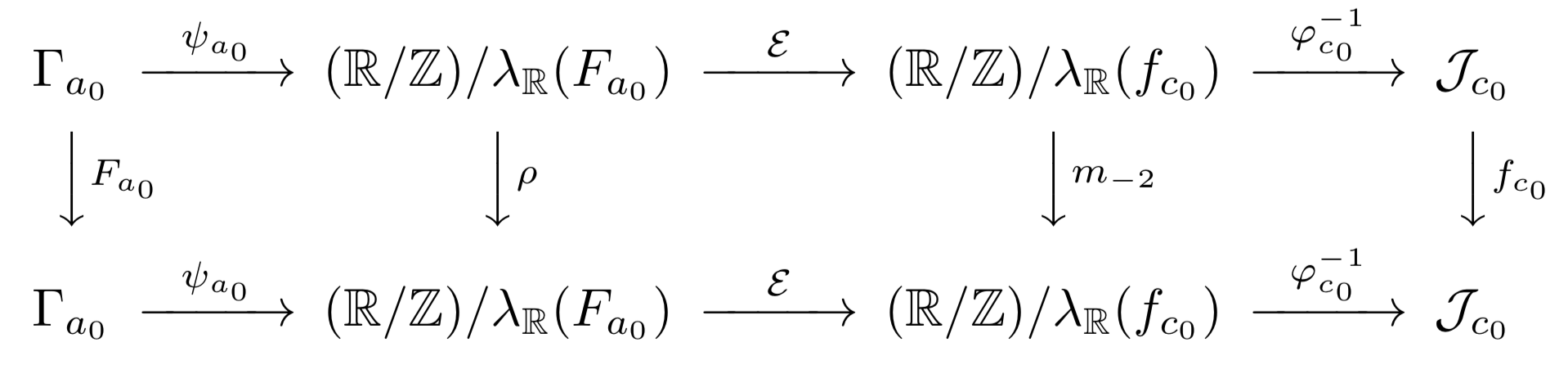}
\end{center}
\caption{The maps $\psi_{a_0}$ and $\phi_{c_0}^{-1}$ in the commutative diagram denote the homeomorphisms  induced by the continuous boundary extensions of the corresponding conformal maps. The composition of the three horizontal maps yields a topological conjugacy between  $F_{a_0}:\Gamma_{a_0}\to \Gamma_{a_0}$ and $f_{c_0}:\mathcal{J}_{c_0}\to\mathcal{J}_{c_0}$.}
\label{top_conjugacy_julia_limit}
\end{figure}

Since $\mathcal{E}_{\ast}(\lambda(F_{a_0}))=\lambda(f_{c_0})$, and $\lambda(F_{a_0})$ (respectively $\lambda(f_{c_0})$) generate $\lambda_{\R}(F_{a_0})$ (respectively, $\lambda_{\R}(f_{c_0})$), we have that $\mathcal{E}_{\ast}(\lambda_{\R}(F_{a_0}))=\lambda_{\R}(f_{c_0})$ (recall that $\mathcal{E}$ is a homeomorphism of the circle). Therefore, the conjugacy $\mathcal{E}$ between $\rho$ and $m_{-2}$ descends to a topological conjugacy between their quotients by $\lambda_{\R}(F_{a_0})$ and $\lambda_{\R}(f_{c_0})$ respectively. It follows that $F_{a_0}:\Gamma_{a_0}\to \Gamma_{a_0}$ is topologically conjugate to $f_{c_0}:\mathcal{J}_{c_0}\to\mathcal{J}_{c_0}$ by the composition of the above three conjugacies (see Figure~\ref{top_conjugacy_julia_limit}). 
\end{proof}

We denote the topological conjugacy between $F_{a_0}:\Gamma_{a_0}\to \Gamma_{a_0}$ and $f_{c_0}:\mathcal{J}_{c_0}\to\mathcal{J}_{c_0}$ by $\mathcal{E}_{a_0}$.

Let us now proceed to prove Proposition~\ref{lami_correspond_mating}, which gives a complete description of the dynamics of the post-critically finite map $F_{a_0}$ as a ``mating'' of the reflection map $\rho$ arising from the ideal triangle group $\mathcal{G}$ and the anti-polynomial $\overline{z}^2+c_0$.

\begin{proof}[Proof of Proposition~\ref{lami_correspond_mating}]
Let us first assume that $a_0$ is a Misiurewicz parameter. Then the corresponding pre-periodic lamination $\lambda(F_{a_0})$ is of Misiurewicz type, and hence, $\lambda(f_{c_0})=\mathcal{E}_{\ast}(\lambda(F_{a_0}))$ is a rational lamination of Misiurewicz type. It follows that $c_0$ is a Misiurewicz parameter of the Tricorn. Thus we have $\Gamma_{a_0}=K_{a_0}$, and $\mathcal{J}_{c_0}=\mathcal{K}_{c_0}$. The theorem (for Misiurewicz parameters) now follows from Lemma~\ref{julia_top_conjugacy} and Proposition~\ref{schwarz_group}.

Now let $a_0$ be a super-attracting parameter. We will construct a conformal conjugacy between $F_{a_0}$ and $f_{c_0}$ in the interior of $K_{a_0}$ that matches continuously with the conjugacy on the limit set constructed in Lemma~\ref{julia_top_conjugacy}.

We assume that $F_{a_0}$ has a super-attracting cycle of odd period $k$ (the even period case is similar). Let $U$ be the Fatou component of $F_{a_0}$ containing the critical point $0$. Since $K_{a_0}$ is full and $\Gamma_{a_0}$ is locally connected, it follows that $U$ is a Jordan disk. Hence, the Riemann uniformization $\mathfrak{b}:\D\to U$ (which is the inverse of a B{\"o}ttcher coordinate) that conjugates $\overline{z}^2\vert_{\D}$ to the first return map $F_{a_0}^{\circ k}\vert_U$ extends homeomorphically to the boundary $\partial\D$. We can normalize the Riemann uniformization so that it sends $1$ to the dynamical root on $\partial U$. Moreover, if $\widetilde{U}$ is the Fatou component of $f_{c_0}$ containing $0$, then the same is true for the Riemann uniformization $\widetilde{\mathfrak{b}}:\D\to\widetilde{U}$ that conjugates $\overline{z}^2\vert_{\D}$ to the first return map $f_{c_0}^{\circ k}\vert_{\widetilde{U}}$. Therefore, $\mathfrak{h}:=\widetilde{\mathfrak{b}}\circ\mathfrak{b}^{-1}:\overline{U}\to\overline{\widetilde{U}}$ conjugates $F_{a_0}^{\circ k}$ to $f_{c_0}^{\circ k}$.

We claim that $\mathfrak{h}$ continuously matches with $\mathcal{E}_{a_0}$ on $\partial U$. Since both $\mathfrak{h}$ and $\mathcal{E}_{a_0}$ conjugate $F_{a_0}^{\circ k}\vert_{\partial U}$ to $f_{c_0}^{\circ k}\vert_{\partial\widetilde{U}}$, it follows that $\mathcal{E}_{a_0}^{-1}\circ\mathfrak{h}:\partial U\to\partial U$ is an orientation-preserving homeomorphism commuting with $F_{a_0}^{\circ k}\vert_{\partial U}$. Moreover, it fixes the dynamical root on $\partial U$. But $F_{a_0}^{\circ k}\vert_{\partial U}$ is topologically equivalent to $\overline{z}^2\vert_{\partial\D}$, and the only orientation-preserving homeomorphism of the circle commuting with $\overline{z}^2$ and having a fixed point is the identity map. Hence, $\mathcal{E}_{a_0}^{-1}\circ\mathfrak{h}\vert_{\partial U}=\mathrm{id}$, and the claim follows.

Let us now consider another component $D$ of $\Int{K_{a_0}}$. Since $\Int{K_{a_0}}$ is the basin of the super-attracting cycle of $F_{a_0}$, there exists some positive integer $n$ such that $F_{a_0}^{\circ n}$ maps $D$ univalently onto $U$. Let $\widetilde{D}$ be the Fatou component of $f_{c_0}$ such that $\mathcal{E}_{a_0}(\partial D)=\partial\widetilde{D}$, and $^{\widetilde{D}}f_{c_0}^{-n}$ be the inverse branch that maps $\overline{\widetilde{U}}$ to $\overline{\widetilde{D}}$. Since $\mathcal{E}_{a_0}$ is a conjugacy on the whole limit set, we have $$\mathcal{E}_{a_0}\vert_{\partial D}=\ ^{\widetilde{D}}f_{c_0}^{-n}\circ\mathcal{E}_{a_0}\vert_{\partial U}\circ F_{a_0}^{\circ n}:\partial D\to\partial\widetilde{D}.$$ We now define $$\mathfrak{h}_D:=\ ^{\widetilde{D}}f_{c_0}^{-n}\circ\mathfrak{h}\circ F_{a_0}^{\circ n}:\overline{D}\to\overline{\widetilde{D}}.$$ Since, $\mathfrak{h}$ agrees with $\mathcal{E}_{a_0}$ on $\partial U$, it now follows that $\mathcal{E}_{a_0}\vert_{\partial D}=\mathfrak{h}_D\vert_{\partial D}$.

Thus, we have extended the conjugacy $\mathcal{E}_{a_0}$ (which was defined between the limit set of $F_{a_0}$ and the Julia set of $f_{c_0}$) conformally and equivariantly to all of $\Int{K_{a_0}}$. Since $\Gamma_{a_0}$ is locally connected, the diameters of the Fatou components of $F_{a_0}$ tend to $0$. It is now easy to verify that the extension is a homeomorphism on $K_{a_0}$. 

This, together with the proof of Proposition~\ref{schwarz_group}, completes the proof of the theorem for super-attracting parameters.
\end{proof}

The proof of Proposition~\ref{lami_correspond_mating} can be easily modified to cover the case of hyperbolic and parabolic maps. 

\begin{proposition}\label{lami_correspond_mating_1}
Let $a_0\in\cC(\mathcal{S})$ be geometrically finite, and $c_0:=\chi(a_0)\in\mathcal{L}$ (see Definition~\ref{comb_straightening}). Then, $F_{a_0}:K_{a_0}\to K_{a_0}$ is topologically conjugate to $f_{c_0}:\mathcal{K}_{c_0}\to\mathcal{K}_{c_0}$ such that the conjugacy is conformal on $\Int{K_{a_0}}$, and $F_{a_0}:T_{a_0}^\infty\setminus \Int{T_{a_0}^0}\to T_{a_0}^\infty$ is conformally conjugate to $\rho:\D\setminus\Int{\Pi}\to\D$.
\end{proposition}

\subsection{Geometrically finite maps are matings}\label{geom_fin_conf_mating_subsec}

\begin{proof}[Proof of Theorem~\ref{filled_julia_top_conjugacy}]
Let $a_0$ be a geometrically finite map in $\cC(\mathcal{S})$. It follows from Theorem~\ref{thm_comb_bijec_pcf} and Proposition~\ref{lami_correspond_mating_1} that $F_{a_0}:K_{a_0}\to K_{a_0}$ is topologically conjugate to $f_{\chi(a_0)}:\mathcal{K}_{\chi(a_0)}\to\mathcal{K}_{\chi(a_0)}$ such that the conjugacy is conformal on $\Int{K_{a_0}}$, and $F_{a_0}:T_{a_0}^\infty\setminus \Int{T_{a_0}^0}\to T_{a_0}^\infty$ is conformally conjugate to $\rho:\D\setminus\Int{\Pi}\to\D$. 

Let us consider the two conformal dynamical systems 
$$f_{\chi(a_0)}:\mathcal{K}_{\chi(a_0)}\to\mathcal{K}_{\chi(a_0)}\quad \mathrm{and}\quad \rho:\overline{\mathbb{D}}\setminus\Int{\Pi}\to\overline{\mathbb{D}}.$$ We use the mating tool $\xi:=\phi_{\chi(a_0)}^{-1}\circ\mathcal{E}:\mathbb{T}\to\mathcal{J}_{\chi(a_0)}$ (where $\phi_{\chi(a_0)}^{-1}:\mathbb{T}\to\mathcal{J}_{\chi(a_0)}$ is the continuous boundary extension of the inverse of the B{\"o}ttcher coordinate of $f_{\chi(a_0)}$) to glue $\overline{\mathbb{D}}$ outside $\mathcal{K}_{\chi(a_0)}$. Note that $\xi$ semi-conjugates $\rho$ to $f_{\chi(a_0)}$.

Denote $$X~=~\overline{\mathbb D}~\vee_{\xi}~\mathcal{K}_{\chi(a_0)}, \qquad Y=X\setminus\Int{\Pi}.$$ (This is a slight abuse of notation. We have denoted the image of $\Int{\Pi}\subset\D$ in $X$ under the gluing by $\Int{\Pi}$.)

We will argue that $X$ is a topological sphere. Since $\mathcal{K}_{\chi(a_0)}$ is homeomorphic to $\overline{\mathbb{D}}/\lambda_{\R}(f_{\chi(a_0)})$ (where $\lambda_{\R}(f_{\chi(a_0)})$ is the real lamination of $f_{\chi(a_0)}$, which has a locally connected Julia set), it follows that $X$ is topologically the quotient of the $2$-sphere by a closed equivalence relation such that all equivalence classes are connected and non-separating, and not all points are equivalent. It follows by Moore's Theorem that $X$ is a topological $2$-sphere \cite[Theorem~25]{Moore}. Moreover, $Y$ is the union of two closed Jordan disks (with a single point of intersection) in $X$.

The well-defined topological map 
$$\eta~\equiv~\rho~\vee_{\xi}~f_{\chi(a_0)}:~ Y\to X$$
is the topological mating between $\rho$ and $f_{\chi(a_0)}$.

The conjugacies obtained in Proposition~\ref{lami_correspond_mating_1} glue together to produce a homeomorphism 
$$
        \mathfrak{H}:  (X,Y) \rightarrow  (\widehat{\C}, \overline{\Omega}_{a_0}) 
$$
which is conformal outside $\mathfrak{H}^{-1}(\Gamma_{a_0})$, and which conjugates $\eta:Y\to X$ to $F_{a_0}: \overline{\Omega}_{a_0}\to\widehat{\C}$ (see Subsection~\ref{C_and_C_subsec} for the definition of $\Omega_{a_0}$ and $\Gamma_{a_0}$).
It endows $X$ with a conformal structure compatible with the one on $X\setminus\mathfrak{H}^{-1}(\Gamma_{a_0})$
that turns $\eta$ into an anti-holomorphic map conformally conjugate to $F_{a_0}$. 

In this sense, $F_{a_0}$ is a conformal mating of the reflection map $\rho$ arising from the ideal triangle group and the anti-polynomial $\overline{z}^2+\chi(a_0)$.   
\end{proof}

Note that the proof of Theorem~\ref{filled_julia_top_conjugacy} uses local connectivity of $\Gamma_{a_0}$ (respectively, of $\mathcal{J}_{\chi(a_0)}$), rigidity of the corresponding maps, and our understanding of the dynamics of $F_{a_0}$ (respectively, of $f_{\chi(a_0)}$) on $K_{a_0}$ (respectively, on $\mathcal{K}_{\chi(a_0)}$) in a crucial way. This is precisely the reason why we restricted ourselves to geometrically finite maps in Theorem~\ref{filled_julia_top_conjugacy}.

We will now provide a weaker (and more combinatorial) mating description for the \emph{periodically repelling} maps in $\cC(\mathcal{S})$.

\begin{definition}\label{per_rep_rigid_def}
1) A map $F_a\in\cC(\mathcal{S})\setminus\{-\frac{1}{12}\}$ is called \emph{periodically repelling} if every periodic point of $F_a$ (except the fixed points $\frac14$ and $\alpha_a$) is repelling.

2) For a periodically repelling map $F_a$ (respectively, $f_c$), we define the \emph{real lamination} $\lambda_{\R}(F_a)$ (respectively, $\lambda_{\R}(f_c)$) of $F_a$ (respectively, of $f_c$) to be the smallest closed equivalence relation in $\R/\Z$ containing $\lambda(F_a)$ (respectively, $\lambda(f_c)$).
\end{definition}

Note that by Proposition~\ref{fatou_classification}, a periodically repelling map $F_a$ has no Fatou components; i.e., $K_a=\Gamma_a$. Thus, the dynamics of a periodically repelling map $F_a$ on its non-escaping set $K_a$ is combinatorially modeled by the quotient of $\rho:\mathbb{T}\to\mathbb{T}$ by the $\rho$-invariant lamination $\lambda_{\R}(F_a)$. By an anti-holomorphic version of \cite[Theorem~1.1]{kiwi} (which can be proved along the lines of Kiwi's arguments), there exists some periodically repelling map $f_c\in\mathcal{L}$ such that $\lambda(f_c)=\mathcal{E}_\ast(\lambda(F_a))$. Clearly, we have that $\mathcal{K}_c=\mathcal{J}_c$, and the dynamics of $f_c$ on its filled Julia set $\mathcal{K}_c$ is combinatorially modeled by the quotient of $m_{-2}:\R/\Z\to\R/\Z$ by the $m_{-2}$-invariant lamination $\lambda_{\R}(f_c)$. Moreover, the combinatorial models of $F_{a}:K_a\to K_a$ and $f_{c}:\mathcal{K}_{c}\to\mathcal{K}_{c}$ are topologically conjugate by a factor of $\mathcal{E}$. 

Therefore, the dynamics of a periodically repelling map $F_a$ can be decomposed into $F_a:T_a^\infty\setminus \Int{T_a^0}\to T_a^\infty$, which is conformally equivalent to $\rho:\D\setminus\Int{\Pi}\to\D$, and $F_a:K_a\to K_a$, which is combinatorially equivalent to $f_{c}:\mathcal{K}_{c}\to\mathcal{K}_{c}$. In this sense, every periodically repelling map $F_a$ is a \emph{combinatorial mating} of a (periodically repelling) quadratic anti-polynomial $f_c$ (restricted to its filled Julia set), and the reflection map $\rho$. 

\begin{proposition}\label{purely_repelling_comb_mating}
Every periodically repelling map $F_a$ is a combinatorial mating of a periodically repelling quadratic anti-polynomial $f_{c}:\mathcal{K}_{c}\to\mathcal{K}_{c}$ and the reflection map $\rho:\D\setminus\Int{\Pi}\to\D$.
\end{proposition}

\begin{corollary}\label{model_honest_lc_cor}
If the limit set of $F_a$ and the Julia set of $f_c$ (appearing in Proposition~\ref{purely_repelling_comb_mating}) are locally connected, then $F_a$ is a conformal mating of $f_c$ (restricted to its Julia set) and the reflection map $\rho$.
\end{corollary}
\begin{proof}
The discussion preceding the proof of Proposition~\ref{purely_repelling_comb_mating}, combined with local connectivity of $\Gamma_a, \mathcal{J}_c$ and \cite[Lemma~4.17]{kiwi}, implies that $F_a\vert_{K_a}$ and $f_c\vert_{\mathcal{K}_c}$ are topologically conjugate by a factor of $\mathcal{E}$. 
\end{proof}

\section{Homeomorphism between topological models of parameter spaces}\label{model_homeo}

In this subsection, we will construct a locally connected model $\widetilde{\cC(\mathcal{S})}$ of $\cC(\mathcal{S})$, and show that it is homeomorphic to the combinatorial model $\widetilde{\mathcal{L}}$ of $\mathcal{L}$. 

The construction of $\widetilde{\cC(\mathcal{S})}$ will be similar to that of $\widetilde{\mathcal{L}}$ (compare Subsection~\ref{sec_basilica_limb}). We first construct an equivalence relation on $\mathrm{Per}(\rho)\cap\partial\D_2$ by identifying the angles of all parameter rays of $\cC(\mathcal{S})$ at pre-periodic angles (under $\rho$) that land at a common (parabolic or Misiurewicz) parameter or accumulate on a common root parabolic arc of $\cC(\mathcal{S})$. We also identify $1/3$ and $2/3$. We then consider the smallest closed equivalence relation on $\partial\D\cap\partial\D_2$ generated by the above relation. Taking the hyperbolic convex hull of each of these equivalence classes in $\overline{\D}$, we obtain a geodesic lamination of $\D_2$ (by hyperbolic geodesics of $\D$). The \emph{abstract connectedness locus} $\widetilde{\cC(\mathcal{S})}$ is defined as the quotient of $\overline{\D}_2$ obtained by collapsing each of these hyperbolic convex hulls to a single point.

\begin{proof}[Proof of Theorem~\ref{thm_model_homeo}]
By Corollaries~\ref{para_rays_correspond} and~\ref{misi_rays_correspond}, the parameter rays of $\mathcal{S}$ at angles $\theta,\theta'\in\mathrm{Per}(\rho)$ land at a common parabolic/Misiurewicz parameter or accumulate on a common root parabolic arc if and only if the parameter rays of $\mathcal{T}$ at angles $\mathcal{E}(\theta),\mathcal{E}(\theta')\in\Q/\Z$ have the same property. Hence, $\mathcal{E}\times\mathcal{E}$ is a bijection between the equivalence relation on $\mathrm{Per}(\rho)\cap\partial\D_2$ induced by the (co-landing or co-accumulation property of) parameter rays of $\mathcal{S}$ at pre-periodic angles (under $\rho$) and the equivalence relation on $\Q/\Z\cap\partial\D_2$ induced by the rational parameter rays of $\mathcal{T}$. Since $\mathcal{E}$ is a homeomorphism, the smallest closed equivalent relations on $\partial\D\cap\partial\D_2$ generated by the above relations are also homeomorphic under $\mathcal{E}\times\mathcal{E}$. It follows that the corresponding geodesic laminations of $\D_2$ are topologically equivalent; i.e., there is a homeomorphism of $\overline{\D}_2$ mapping the leaves and gaps of one lamination to those of the other. Clearly, this homeomorphism descends to a homeomorphism between the quotient spaces $\widetilde{\cC(\mathcal{S})}$ and $\widetilde{\mathcal{L}}$.
\end{proof}

\section{Discontinuity of straightening}\label{straightening_discont}

By construction of the map $\chi$ (see Corollary~\ref{hyp_para_bijec_cor}), it is a homeomorphism between hyperbolic components of $\cC(\mathcal{S})$ and $\mathcal{L}$. Preserving critical {\'E}calle heights, we extended $\chi$ to the boundaries of odd period hyperbolic components. Thus, in light of Theorem~\ref{ThmBifArc_schwarz}, $\chi$ is defined in small neighborhoods of parabolic cusps. The goal of this section is to show that $\chi$ is not always continuous in neighborhoods of parabolic cusps. This is one of the reasons why we construct a homeomorphism between the models of the connectedness loci in a purely combinatorial way.

Let us now assume that $H\subset\cC(\mathcal{S})$ is a hyperbolic component of odd period $k$, and $H'=\chi(H)$ is the corresponding hyperbolic component in $\mathcal{L}$. Recall that $\partial H$ consists of three parabolic arcs and an equal number of parabolic cusps. The map $\chi$ sends parabolic cusps on $\partial H$ to parabolic cusps on $\partial H'$, and simple parabolic parameters on $\partial H$ to simple parabolic parameters on $\partial H'$ (preserving their parabolic orbit portraits and critical {\'E}calle heights).

The next proposition shows that $\chi$ is a homeomorphism from $\overline{H}$ onto $\overline{H}'$.

We denote the Koenigs ratio map of the hyperbolic component $H$ (respectively, $H'$) by $\zeta_H$ (respectively, $\zeta_{H'}$). 

\begin{proposition}\label{homeo_odd_closure}
As the parameter $a$ in $H$ (respectively, $c$ is $H'$) approaches a simple parabolic parameter with critical {\'E}calle height $h$ on the boundary of $H$ (respectively, $H'$), the quantity $\frac{1-\zeta_H(a)}{1-\vert\zeta_H(a)\vert^2}$ (respectively, $\frac{1-\zeta_{H'}(c)}{1-\vert\zeta_{H'}(c)\vert^2}$) converges to $\frac{1}{2}-2ih$. Consequently, $\chi$ maps the closure $\overline{H}$ of the hyperbolic component $H$ homeomorphically onto the closure $\overline{H'}$ of the hyperbolic component $H'$.
\end{proposition}
\begin{proof}
The proof of the first statement is similar to that of \cite[Lemma~6.3]{IM2}. Since $\chi$ preserves Koenigs ratio (of parameters in $H$) and critical {\'E}calle height (of simple parabolic parameters on $\partial H$), it follows that $\chi$ extends continuously to $\partial H$. Since $\chi$ is defined in an injective fashion and it is continuous on $H$, it is a homeomorphism from $\overline{H}$ onto $\overline{H'}$.
\end{proof}

Let $H_1$ be a hyperbolic component of even period $2k$ bifurcating from $H$ across a parabolic arc $\mathscr{C}$. It is now easy to see from the proof of Lemma~\ref{thm_comb_bijec_centers} that the hyperbolic component $\chi(H_1)$ bifurcates from $H'=\chi(H)$ across the parabolic arc $\chi(\mathscr{C})$. We will denote the critical {\'E}calle height parametrization of the parabolic arc $\mathscr{C}$ (respectively, $\chi(\mathscr{C})$) by $a:\R\to\mathscr{C}$ (respectively, $c:\R\to\chi(\mathscr{C})$). Recall that for any $h$ in $\mathbb{R}$, the fixed point index of the unique parabolic cycle of $F_{a(h)}^{\circ 2}$ (respectively of $f_{c(h)}^{\circ 2}$) is denoted by $\ind_{\mathscr{C}}(F_{a(h)}^{\circ 2})$ (respectively $\ind_{\chi(\mathscr{C})}(f_{c(h)}^{\circ 2})$). This defines a pair of real-analytic functions (which we will refer to as index functions) 
$$
\ind_{\mathscr{C}}: \mathbb{R}\to\mathbb{R},\ h\mapsto\ind_{\mathscr{C}}(F_{a(h)}^{\circ 2})\quad \mathrm{and}\quad \ind_{\chi(\mathscr{C})}:\mathbb{R}\to\mathbb{R},\ h\mapsto\ind_{\chi(\mathscr{C})}(f_{c(h)}^{\circ 2}).
$$

We are now in a position to show that continuity of $\chi$ on $\mathscr{C}$ imposes a severe restriction on the index functions $\ind_{\mathscr{C}}$ and $\ind_{\chi(\mathscr{C})}$.

\begin{proposition}[Uniform height-index relation]\label{uniform}
If $\chi$ is continuous on $\mathscr{C}\cap\partial H_1$, then the functions $ \ind_{\mathscr{C}}$ and $\ind_{\chi(\mathscr{C})}$ are identically equal.
\end{proposition}

\begin{figure}[ht!]
\captionsetup{width=0.96\linewidth}
\begin{tikzpicture}
\node[anchor=south west,inner sep=0] at (0,0) {\includegraphics[width=0.46\textwidth]{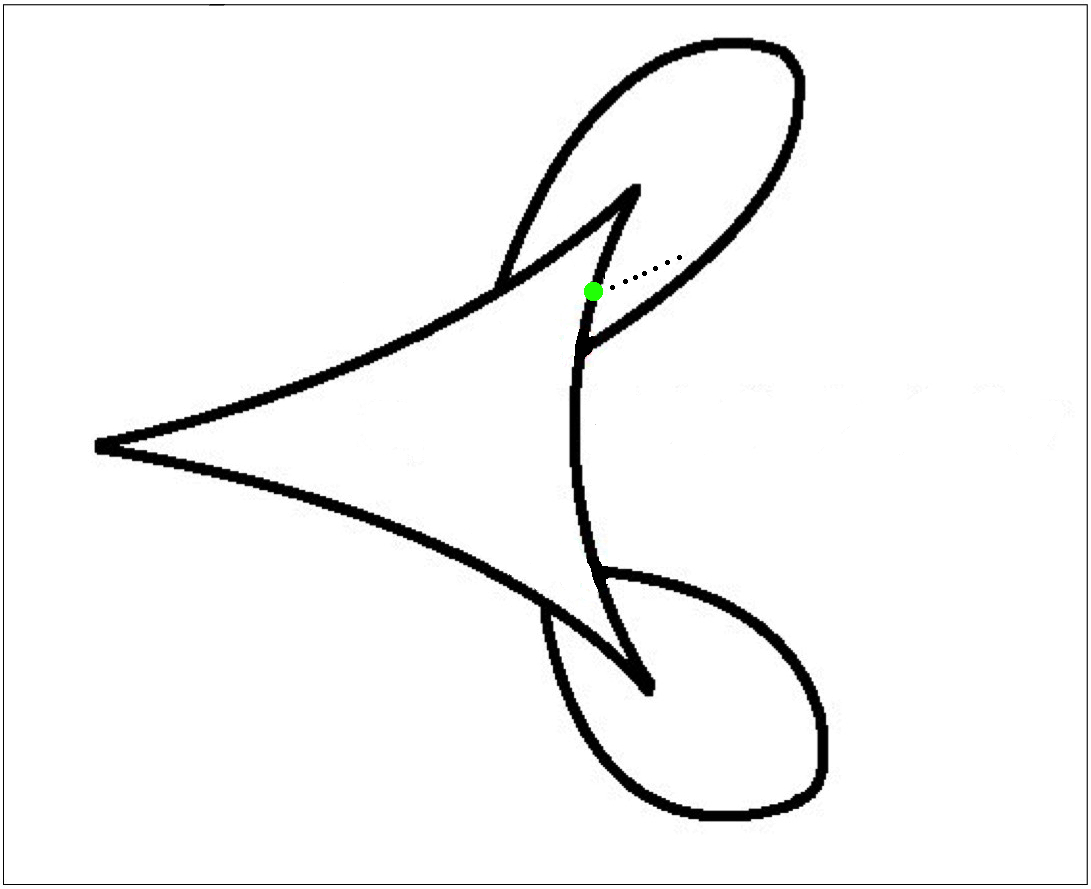}};
\node[anchor=south west,inner sep=0] at (5,0) {\includegraphics[width=0.46\textwidth]{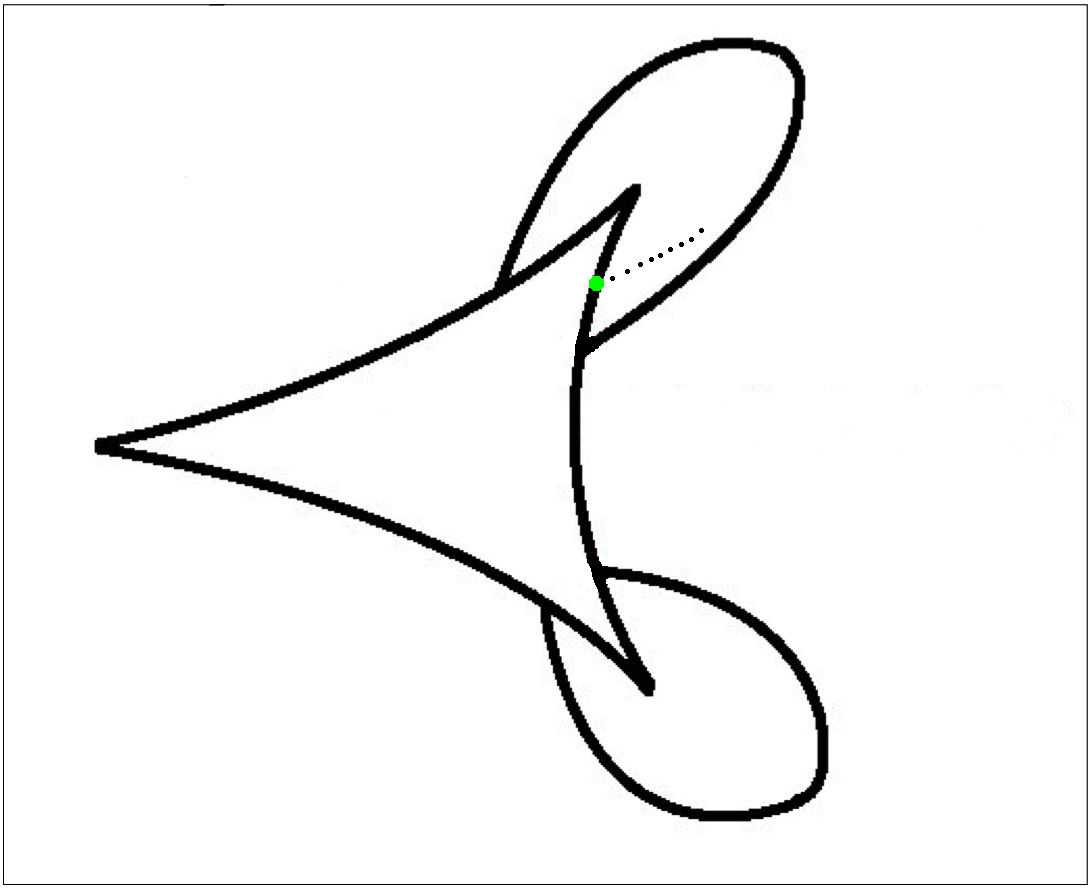}};
\node at (2.2,2.4) {$H$};
\node at (7.36,2.4) {$\chi(H)$};
\node at (3.36,2.4) {$\mathscr{C}$};
\node at (8.56,2.4) {$\chi(\mathscr{C})$};
\node at (3.66,4.12) {\begin{small}$H_1$\end{small}};
\node at (3.54,3.5) {\begin{tiny}$a_n$\end{tiny}};
\node at (2.86,3.08) {\begin{tiny}$a(h)$\end{tiny}};
\node at (8.72,4.12) {\begin{small}$\chi(H_1)$\end{small}};
\node at (7.88,3.1) {\begin{tiny}$c(h)$\end{tiny}};
\node at (9.4,3) {\begin{tiny}$\chi(a_n)$\end{tiny}};
\draw [->,line width=0.5pt] (9.1,3.1) to (8.54,3.36);
\end{tikzpicture}
\caption{The straightening map $\chi$, restricted to $\overline{H}$, is a homeomorphism. On the other hand, continuity of $\chi$ on $\mathscr{C}\cap\partial H_1$ would force parameters with equal critical {\'E}calle height to have the same parabolic fixed point index.}
\label{bifurcation_point_2}
\end{figure}

\begin{proof}
Let us pick a parameter $a(h)\in\mathscr{C}\cap\partial H_1$. Consider a sequence $\{a_n\} \in H_1$ with $a_n\to a(h)$. If $\chi$ is continuous at $a(h)$, then $\chi(a_n)\to\chi(a(h))=c(h)$. We will now show that $\ind_{\mathscr{C}}(F_{a(h)}^{\circ 2})=\ind_{\chi(\mathscr{C})}(f_{c(h)}^{\circ 2})$.

Let, $\ind_{\mathscr{C}}(F_{a(h)}^{\circ 2})=\tau$. For any $n$, the map  $F_{a_n}^{\circ 2}$ has two distinct $k$-periodic attracting cycles (which are born out of the parabolic cycle of $F_{a(h)}^{\circ 2}$) with multipliers $\mu_{a_n}$ and $\overline{\mu}_{a_n}$. Since $a_n$ converges to $a(h)$, we have that

\begin{equation}\label{multiplier_index}
\frac{1}{1-\mu_{a_n}} + \frac{1}{1-\overline{\mu}_{a_n}} \longrightarrow \tau
\end{equation}
as $n\to\infty$. 

Since the multipliers of attracting periodic orbits are preserved by $\chi$, it follows that $f_{\chi(a_n)}^{\circ 2}$ has two distinct $k$-periodic attracting cycles with multipliers $\mu_{a_n}$ and $\overline{\mu}_{a_n}$. As $\{\chi(a_n)\}$ converges to   the odd period parabolic parameter $c(h)$, the same limiting Relation~\eqref{multiplier_index} holds for the fixed point index of the parabolic cycle of $f_{c(h)}^{\circ 2}$ as well. In particular, the parabolic fixed point index of $f_{c(h)}^{\circ 2}$ is also $\tau$ (see Figure~\ref{bifurcation_point_2}). 

Since $h$ was arbitrarily chosen (with $\vert h\vert$ large enough so that $a(h)$ is a bifurcating parameter), we conclude that $\ind_{\mathscr{C}}(h)=\ind_{\chi(\mathscr{C})}(h)$ for all $h$ in an unbounded interval. As the functions $\ind_{\mathscr{C}}$ and $\ind_{\chi(\mathscr{C})}$ are real-analytic, it follows that $\ind_{\mathscr{C}}\equiv\ind_{\chi(\mathscr{C})}$.
\end{proof}

The above condition on index functions seems unlikely to hold in general. It can be numerically verified that the index function of the period $3$ parabolic arc of $\cC(\mathcal{S})$ does not identically agree with the index function of the period $3$ parabolic arc of $\mathcal{L}$. More precisely, it is easy to check that the parabolic cycles of the critical {\'E}calle height $0$ parameters of these two arcs have different fixed point indices. This proves that the map $\chi$ is discontinuous on the period $3$ parabolic arc of $\cC(\mathcal{S})$.

\begin{remark}
For an analogous discussion of discontinuity of straightening maps for the Tricorn, see \cite[Proposition~8.1]{IM2}.
\end{remark}

\appendix

\section{Uniqueness of matings}\label{unique_app}

In this appendix, we will give elementary proofs of uniqueness of the conformal matings of $\rho$ with $\overline{z}^2$ and $\overline{z}^2-1$. We remark that the following proofs do not appeal to conformal removability of the limit sets of the Schwarz reflection maps (cf. \cite[\S 4.4.3]{LLMM1}).

\subsection{Uniqueness of the mating of $\rho$ and $\overline{z}^2$}

The construction of the topological mating in \cite[\S 4]{LLMM1} implies that any conformal mating $\mathfrak{F}$ of $\overline{z}^2\vert_{\overline{\D}}$ and $\rho$ is an anti-holomorphic map defined on the closure of a Jordan domain in the Riemann sphere. Indeed, since $\rho$ is not defined on the interior of the ideal triangle $\Pi$, the map $\mathfrak{F}$ must be defined on the complement of a homeomorphic copy of $\Int{\Pi}$. Moreover, as $\rho$ fixes $\partial \Pi$ pointwise, $\mathfrak{F}$ must fix the boundary of its domain of definition $\mathfrak{D}_{\mathfrak{F}}$. It follows that $\mathfrak{D}_{\mathfrak{F}}$ is the closure of a simply connected quadrature domain, and $\mathfrak{F}$ is the associated Schwarz reflection map. 

By Proposition~\ref{simp_conn_quad}, there exists a rational map $R$ which maps $\widehat{\C}\setminus\overline{\mathbb{D}}$ univalently onto $\Int{\mathfrak{D}_\mathfrak{F}}$. Pre-composing $R$ with a conformal automorphism of $\widehat{\C}\setminus\overline{\D}$, we may assume that $R(\infty)=\infty$. Again, as $\rho:\rho^{-1}(\Int{\Pi})\to\Int{\Pi}$ is a three-to-one covering, $\mathfrak{F}$ must send the preimage of $\widehat{\C}\setminus \mathfrak{D}_\mathfrak{F}$ as a three-to-one cover onto $\widehat{\C}\setminus \mathfrak{D}_\mathfrak{F}$. In light of the commutative diagram defining $\mathfrak{F}$, we have that $R$ is a degree $3$ rational map.

As $\overline{z}^2$ has a unique critical point and this critical point is a fixed point, the same is true for $\mathfrak{F}$. Possibly after conjugating $\mathfrak{F}$ by a M{\"o}bius map, we can assume that the unique critical point of $\mathfrak{F}$ is at $\infty$. As $\infty$ is a (simple) fixed critical point of $\mathfrak{F}$, the commutative diagram defining $\mathfrak{F}$ now shows that $R$ has a pole of order two at $0$. Hence, $R$ is of the form $R(z)=az+b+\frac{c}{z}+\frac{d}{z^2}$, for some $a,d\in\C^*$, and $b,c\in\C$. Possibly after post-composing $R$ with an affine map (which amounts to replacing $\mathfrak{D}_\mathfrak{F}$ by an affine image of it, and conjugating $\mathfrak{F}$ by the same affine map), we may write $R(z)=z+\frac{c}{z}+\frac{d}{z^2}$, for some $c\in\C$ and $d\in\C^*$. 

Note that the cubic anti-rational map $R$ has four critical points (counted multiplicity), one of which is at the origin. Since $\mathfrak{F}$ has only one critical point, the same commutative diagram implies that the other three critical points of $R$ lie on the unit circle (in fact, univalence of $R$ on $\widehat{\C}\setminus\overline{\D}$ implies that these critical points are distinct). Thus, the equation $z^3R'(z)\equiv z^3-cz-2d=0$ has three distinct solutions on $\mathbb{T}$. An elementary algebraic computation involving Vieta's formulas now shows that possibly after conjugating $R$ by a rotation (once again, this amounts to replacing $\mathfrak{D}_\mathfrak{F}$ by a rotated image of it, and conjugating $\mathfrak{F}$ by the same rotation), one can choose $R$ to be the rational map $z+\frac{1}{2z^2}$. Thus, $\mathfrak{D}_\mathfrak{F}$ is the closure of the quadrature domain $\Omega$ (up to post-composition by a M{\"o}bius map), and $\mathfrak{F}$ is the Schwarz reflection map $\sigma$ (up to M{\"o}bius conjugation) introduced in \cite[\S 4.1]{LLMM1}.



\subsection{Uniqueness of the mating of $\rho$ and $\overline{z}^2-1$}

Consider the anti-polynomial $g(z)\equiv g(z)=\overline{z}^2-1$. For notational simplicity, we set $\mathcal{J}:=\mathcal{J}_{-1}$ and $\mathcal{K}:=\mathcal{K}_{-1}$ The unique finite critical point $0$ of $g$ forms a $2$-cycle.

Note that the $1/3$ and $2/3$ rays of $g$ land at a common fixed point $\alpha$ on the Julia set $\mathcal{J}$. We denote the components of $\mathcal{K}\setminus\{\alpha\}$ by $\mathcal{K}^1\ni 0$ and $\mathcal{K}^2\ni -1$. We have the following mapping properties of the action of $g$ on its filled Julia set.
\begin{enumerate}\upshape
\item $g:\mathcal{K}^2\to \mathcal{K}^1$ is a homeomorphism, and 

\item $g: g^{-1}\left(\mathcal{K}^1\right)\cap\mathcal{K}^1\to\mathcal{K}^1$ is a homeomorphism.
\end{enumerate}

Suppose that $\mathfrak{F}$ is a conformal mating of $g$ and $\rho$. Note that $\mathfrak{F}$ is defined outside a simply connected domain whose boundary is a topological triangle with two vertices identified. We denote the domain of $\mathfrak{F}$ by $\mathfrak{D}_\mathfrak{F}$. Clearly, $\Int{\mathfrak{D}_\mathfrak{F}}$ consists of two Jordan domains $\Omega_1$ and $\Omega_2$, and these domains touch at a single point. As $\rho$ fixes $\partial\Pi$ pointwise, $\mathfrak{F}$ must fix the boundary of each $\Omega_i$ pointwise. Hence, each $\Omega_i$ is a simply connected quadrature domain with Schwarz reflection map $\mathfrak{F}_1, \mathfrak{F}_2$, and $\mathfrak{F}$ is the piecewise defined Schwarz reflection map of these quadrature domains. 

After a M{\"o}bius change of coordinates, we can assume that the unique critical point and critical value of $\mathfrak{F}$ are at the origin and at the point at $\infty$, respectively. Possibly after reindexing, we can also assume that $0\in\Omega_1$ and $\infty\in\Omega_2$.

By Proposition~\ref{simp_conn_quad}, there exist rational maps $R_1, R_2$ which map $\D$ univalently onto $\Omega_1, \Omega_2$ respectively. Pre-composing $R_i$ with conformal automorphisms of $\D$, we may assume that $R_1(0)=0, R_2(0)=\infty$. 

The mapping properties of $g$ on $\mathcal{K}_2$ imply that $\mathfrak{F}_2$ carries $\mathfrak{F}_2^{-1}(\Int{\Omega_2^c})$ univalently onto $\Int{\Omega_2^c}$. In light of the commutative diagram defining $\mathfrak{F}_2$, we have that $R_2$ is a degree $1$ rational map; i.e., a M{\"o}bius map. Thus, $\Omega_2\ni\infty$ is a round disk.

On the other hand, the mapping properties of $g$ on $\mathcal{K}_1$ imply that $\mathfrak{F}_1$ carries $\mathfrak{F}_1^{-1}(\Omega_1)$ univalently onto $\Omega_1$. In light of the commutative diagram defining $\mathfrak{F}_1$, we have that $R_1$ is a quadratic rational map. As $0$ is a critical point of $\mathfrak{F}_1$ with associated critical value $\infty$, the same commutative diagram also implies that $\infty$ is a double pole of the quadratic rational map $R_1$. Hence, $R_1$ is a quadratic polynomial. As $R_1(0)=0$, we have that $R_1(z)=az^2+bz$, for $a\in\C\setminus\{0\}$ and $b\in\C$. Since $0$ is the only critical point of $\mathfrak{F}_1$, it follows that the only finite critical point of $R_1$ is on the unit circle. Possibly after pre-composing $R_1$ with a rotation, we can assume that this critical point is at $1$. Furthermore, after conjugating $\mathfrak{F}$ by a map of the form $w\mapsto \mu w$ (which amounts to post-composing $R_1$ with $w\mapsto \mu w$), we can also assume that $R_1(1)=\frac14$. A simple computation now yields that $R_1(z)=z/2-z^2/4$; i.e., $\Omega_1= \heartsuit$.

As $\Omega_2\ni\infty$ is a round disk, we have that $\mathfrak{F}^{\circ 2}(0)$ is the center of the round disk $\overline{\Omega_2}^c$ in $\C$. The fact that $0$ is $2$-periodic under $\mathfrak{F}$ now implies that $\overline{\Omega_2}^c$ is a round disk centered at the origin. Finally, since $\partial\Omega\cap\partial\Omega_2$ is a singleton, we conclude that (up to M{\"o}bius conjugation) $\mathfrak{F}$ is precisely the map $F_a$ in the circle-and-cardioid family $\mathcal{S}$ where $a=0$.

\begin{remark}
The above argument also shows that if a quadratic anti-polynomial $\overline{z}^2+c$ in the real basilica limb of the Tricorn is conformally mateable with the refection map $\rho$, then such a conformal mating necessarily lies in the circle-and-cardioid family (up to M{\"o}bius conjugation).
\end{remark}

\section{Warschawski's theorem and its applications}\label{conf_asymp_appendix}

\subsection*{Cusps and double points}


Let $f:B(0,\epsilon)\to\C,\ f(0)=0$ be a holomorphic map that is univalent on the closure of $B^+:= B(0,\epsilon)\cap\{\re{z}>0\}$ and has a simple critical point at $0$. Then the curve $\gamma:=f(-i\epsilon,i\epsilon)$ has a singularity at the origin, which we refer to as a \emph{conformal cusp} (we will often drop the word conformal and call it a cusp). We note that $f(B^+)$ is a Jordan domain, and by univalence of $f\vert_{B^+}$, the cusp points in the inward direction towards $f(B^+)$.
We further define the \emph{type} of the cusp on $\gamma$  according to the Taylor series expansion of $f$ at $0$. Since $0$ is a simple critical point of $f$,  we can assume (after scaling $f$, if necessary) that $f(w)= w^2 +\sum_{k\geq 3} C_k w^k.$
Then, $\gamma$ can be parametrized near $0$ as $f(it)= (-t^2+O(t^3))+i(ct^n+O(t^{n+1})$, where $t\in(-\epsilon,\epsilon)$, $c\in\R\setminus\{0\}$, and $n\geq 3$. By definition, we say that the cusp at $0$ is of the type $(n,2)$. We note that the arc $\gamma$ disconnects a small neighborhood $U\ni 0$ into two components; one of which (the narrow side) has zero angle at $0$, while the other component (the wide component that the cusp points towards) has a full angle $2\pi$ at $0$. We denote the narrow component by $N$, and the wide component by $W$ (see Figure~\ref{cusp_dp_fig} (left)).
\begin{figure}[ht!]
\captionsetup{width=0.96\linewidth}
\begin{tikzpicture}
\node[anchor=south west,inner sep=0] at (0,0) {\includegraphics[width=0.32\textwidth]{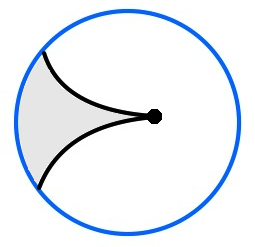}};
\node[anchor=south west,inner sep=0] at (6,0) {\includegraphics[width=0.33\textwidth]{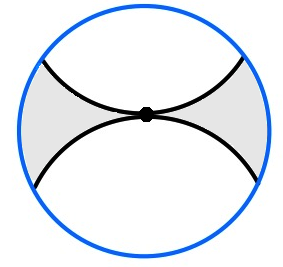}};
\node at (2.75,2.1) {$0$};
\node at (8.2,1.92) {$p$};
\node at (2.25,1) {$W$};
\node at (8.32,0.8) {$W^+$};
\node at (8.32,3.3) {$W^-$};
\node at (0.7,2) {$N$};
\node at (6.75,2.2) {$N^+$};
\node at (9.48,2.2) {$N^-$};
\node at (1.54,2.55) {$\gamma$};
\node at (7.36,2.7) {$\alpha$};
\node at (7.36,1.66) {$\beta$};
\node at (0.72,0.36) {$U$};
\node at (6.72,0.36) {$U$};
\end{tikzpicture}
\caption{Left: The real-analytic arc $\gamma$ has a cusp at the origin, and it divides the neighborhood $U$ of $0$ into a narrow part $N$ and a wide part $W$. Right: The non-singular real-analytic arcs $\alpha,\beta$ touch at the point $p$. These curves divide the neighborhood $U$ of $p$ into a pair of narrow parts $N^\pm$ and a pair of wide parts $W^\pm$.}
\label{cusp_dp_fig}
\end{figure}

Let us now define the notion of a double point singularity. By definition, a \emph{double point} $p$ is a point of tangency  of two distinct non-singular real-analytic arcs $\alpha,\beta$. In particular, the double point $p$ is a non-singular point on each of these arcs. One can classify double points according to the order of contact of these two real-analytic branches. The arcs $\alpha,\beta$ divide a small neighborhood $U\ni p$ into four connected components: two of which (in the tangential directions) have zero angle at $p$, and each of the other two (in the normal directions) has angle $\pi$ at $p$. We denote the zero angle components by $N^\pm$ (the narrow sides), and the $\pi$ angle components by $W^\pm$ (the wide sides). See Figure~\ref{cusp_dp_fig} (right) for an illustration.

\subsection*{Warschawski's theorem on conformal mapping of infinite strips}

We consider the curvilinear strip
$$
\mathfrak{S}:=\{w=u+iv: u>0,\ \nu_-(u)<v<\nu_+(u)\}
$$
where $\nu_\pm:[0,\infty)\to\R$ are $C^2$ maps with $\nu_+>\nu_-$ . Let $\theta_1(u)=\nu_+(u)-\nu_-(u)$ and $\theta_2(u)=(\nu_+(u)+\nu_-(u))/2$. Let $\Phi$ be a conformal map from the curvilinear strip $\mathfrak{S}$ onto the horizontal strip $\{z=x+iy\in\C: x>0, \vert y\vert<\pi/2\}$ such that $\lim_{u\to+\infty} \re{\Phi(w)}=+\infty$.

We suppose that 
\begin{enumerate}
\item the strip $\mathfrak{S}$ has \emph{boundary inclination $0$} at $u=+\infty$; i.e., for $u_2>u_1$, the quotients
$$
\frac{\nu_\pm(u_2)-\nu_\pm(u_1)}{u_2-u_1}
$$
approach the limit $0$, as $u_1\to+\infty$, and


\item the integral $\int_{0}^\infty (\theta_1'(u))^2/\theta_1(u) du$ converges.
\end{enumerate}
Warschawski proved that under the above assumptions,

\begin{theorem}\cite[Theorem~IX]{War}\label{war_thm} 
$$
\Phi(w)= k+\pi\int_{0}^u\frac{1+(\theta_2'(t))^2}{\theta_1(t)} dt +i\pi\frac{v-\theta_2(u)}{\theta_1(u)}+o(1),\quad \textrm{as}\ u\to+\infty,
$$
uniformly with respect to $v$. Here $k$ is a real constant.
\end{theorem}

\subsection*{Applications to boundary behavior of conformal maps between cusps}

Let $\gamma_1, \gamma_2$ be a pair of real-analytic arcs with $\gamma_1(0)=\gamma_2(0)=0$ such that both of them have a cusp of type $(3,2)$ at $0$. Possibly after truncating the curves, we may assume that $\gamma_j\setminus\{0\}$ is non-singular. For open neighborhoods $U_1, U_2$ of $0$ (in the planes of $\gamma_1,\gamma_2$, respectively), we denote the narrow component of $U_j\setminus\gamma_j$ by $N_j$, $j\in\{1,2\}$. We will keep the notation $\gamma_j$ for the truncated arcs $\gamma_j\cap U_j$.

\begin{lemma}\label{asymp_lin_1_lem}
With the setup as above, let $\phi:N_1\to N_2$ be a conformal map whose continuous extension to $\gamma_1$ sends $0$ to $0$. Then, 
\noindent\begin{enumerate}
\item $\phi$ is asymptotic to a linear map at $0$; i.e., $\phi(z)= c z +o(z)$ as $ \overline{N_1}\ni z\to 0$, where $c\in\C^\ast$, and 
\item $\phi$ extends to a quasiconformal map in a neighborhood of $0$.
\end{enumerate}
\end{lemma}
\begin{proof}
1) We can send the cusp of $\gamma_j$ at the origin to the point at $\infty$ by a M{\"o}bius map such that the image of the narrow region $N_j$ is a curvilinear strip $\mathfrak{S}_j$ bounded by the non-singular real-analytic curves. For definiteness, let us work with the strip $\mathfrak{S}_1$, which is bounded by the curves $v=\nu_+(u)=k_1+k_2\sqrt{u}+k_3/\sqrt{u}+o(1/\sqrt{u})$ and $v=\nu_-(u)=k_1-k_2\sqrt{u}-k_3/\sqrt{u}+o(1/\sqrt{u})$, where $u\geq u_0$ for some large $u_0>0$ (with possibly different constants for the two strips). The strips have boundary inclination $0$ at $\infty$ since $\nu_\pm'(u)\to 0$ as $u\to+\infty$. Moreover, real-analyticity of the curve $\gamma_1$ implies that the maps $\nu_\pm$ are real-analytic.
A simple computation shows that the integral convergence condition of Theorem~\ref{war_thm} is also satisfied by the strip. 

Let us now consider a conformal map
$$
\mathfrak{b}:\mathfrak{S}_1\to \{x+iy\in\C:x>0, \vert y\vert<\pi/2\},
$$
such that $\displaystyle\lim_{u\to+\infty} \re{\mathfrak{b}(w)}=+\infty$.
We set $\theta_1(u):=\nu_+(u)-\nu_-(u)$ and $\theta_2(u):=(\nu_+(u)+\nu_-(u))/2$. Theorem~\ref{war_thm} implies that as $u\to+\infty$, the conformal map $\mathfrak{b}$ takes the form
\begin{equation}
w=u+iv\mapsto k_4+\pi\int_{u_0}^u\frac{1+(\theta_2'(t))^2}{\theta_1(t)} dt +i\pi\frac{v-\theta_2(u)}{\theta_1(u)}+o(1)=k_5\sqrt{w}+O(1),
\label{asymp_riemann_cusp}
\end{equation}
uniformly in $v$, where $k_4\in\R$, $k_5\in\C^\ast$. 
It follows that a conformal isomorphism from $\mathfrak{S}_1$ onto $\mathfrak{S}_2$ (that sends $+\infty$ to $+\infty$) admits the asymptotics $k_6 w+O(1)$ as $\re(w)\to+\infty$, for some $k_6\in\C^\ast$ (uniformly in $\im(w)$). Due to analyticity of the boundaries of the above curvilinear strips, this conformal map extends continuously to the boundary. Furthermore, since the above asymptotics is uniform in $\im(w)$, the boundary extension admits the same asymptotics $k_6 w+O(1)$ as $\re(w)\to+\infty$. Now going back to the regions $N_1, N_2$ by the M{\"o}bius maps used earlier, we conclude that $\phi$ is asymptotically linear near the origin; i.e., $\phi(z)= k_7 z +o(z)$ as $\overline{N_1}\ni z\to 0$, where $k_7\in\C^\ast$. 

2) By definition of conformal cusps, for $j\in\{1,2\}$, there exist open sets $B_j\ni 0$ and holomorphic maps $f_j:B_j\to\C,\ f_j(0)=0$, that are univalent on the closure of $B_j^+:= B_j\cap\{\re{z}>0\}$ and have a simple critical point at $0$, such that $\gamma_j=f_j(B_j\cap i\R)$. We set $W_j:=f_j(B_j^+)$. 

Let us now consider the map 
$$
\Psi:\mathcal{I}_1:= B_1\cap i\R \longrightarrow\mathcal{I}_2:=B_2\cap i\R,\quad \Psi\equiv f_2^{-1}\circ\phi\circ f_1.
$$ 
By construction, $\Psi(0)=0$. We will show that $\Psi$ is a quasisymmetric map. Note that $f_j$ admits a conformal extension in a neighborhood of the arc $\mathcal{I}_j\setminus\{0\}$, and $\phi:\overline{N_1}\to\overline{N_2}$ admits a conformal extension (by the Schwarz reflection principle) in a neighborhood of the arc $\gamma_1\setminus\{0\}$. Hence, $\Psi$ is quasisymmetric on $\mathcal{I}_1\setminus\{0\}$. It remains to show that $\Psi$ is quasisymmetric at the origin.
By definition, the map $f_1$ has an asymptotic development
\begin{equation}
f_1(\xi)=c_1 \xi^2+o(\xi^2),
\label{asymp_quad}
\end{equation} 
near $0$, for some constant $c_1\in\C^\ast$. Similarly, $f_2^{-1}$ admits an asymptotic development 
\begin{equation}
f_2^{-1}(\zeta)=c_2 \zeta^{\frac12}+o(\zeta^\frac12),
\label{asymp_root}
\end{equation} 
near $0$, for some constant $c_2\in\C^\ast$ and an appropriate branch of square root. Using these asymptotics of $f_j$, one readily sees that the asymptotic linearity of $\phi: \gamma_1\to\gamma_2$ near $0$ translates to asymptotic linearity of $\Psi$ near $0$. This proves that $\Psi$ is quasisymmetric at the origin.

We continuously extend the quasisymmetric homeomorphism $\Psi: \mathcal{I}_1\to\mathcal{I}_2$ to a quasiconformal homeomorphism $\Psi: B_1^+\to B_2^+$, and define a quasiconformal homeomorphism 
$$
\widetilde{\phi}: W_1\to W_2,\quad \widetilde{\phi}\equiv f_2\circ\Psi\circ \left(f_1\vert_{B_1^+}\right)^{-1}.
$$ 
By construction, $\widetilde{\phi}$ matches continuously with $\phi$ on $\gamma_1$. Thus, we have a homeomorphism in a neighborhood of $0$ that extends $\phi:N_1\to\ N_2$, and is quasiconformal outside $\gamma_1$. Since analytic arcs are quasiconformally removable, it follows that this map is a desired quasiconformal extension of $\phi: N_1\to N_2$.
\end{proof}

We now proceed to prove a similar result for double points. For $j\in\{1,2\}$, let $p_j$ be a point of tangency of two distinct non-singular real-analytic arcs $\alpha_j,\beta_j$. We further assume that $\alpha_j$ and $\beta_j$ have a simple (or non-degenerate) tangency at $p_j$. For an open neighborhood $U_j$ of $p_j$, we denote the (narrow) components of $U_j\setminus(\alpha_j\cup\beta_j)$ that subtend a zero angle at $p_j$ by $N^\pm_j$, and the (wide) components that subtend an angle $\pi$ at $p_j$ by $W^\pm_j$.  We will keep the notation $\alpha_j, \beta_j$ for the truncated arcs $\alpha_j\cap U_j, \beta_j\cap U_j$, for $j\in\{1,2\}$.

\begin{lemma}\label{asymp_lin_2_lem}
With the setup as above, let $\phi:N^+_1\cup N^-_1\to N^+_2\cup N^-_2$ be a conformal map whose continuous extension to $\alpha_1\cup\beta_1$ sends $p_1$ to $p_2$. Then,
\noindent\begin{enumerate}
\item $\phi$ is asymptotic to a linear map at $p_1$; i.e., $\phi(z)= p_2+k (z-p_1) +o(z-p_1)$ as $\overline{N_1^+\cup N_2^-}\ni z\to p_1$, where $k\in\C^\ast$,
\item $\phi$ extend to a quasiconformal map in a neighborhood of $p_1$.
\end{enumerate}
\end{lemma}
\begin{proof}
1) We will show that $\phi:N^+_1\to N^+_2$ is asymptotically linear at $p_1$. The same statement for $\phi:N^-_1\to N^-_2$ can be be proved similarly.
 
The assumption that $\alpha_j$ and $\beta_j$ have a simple tangency at $p_j$ imply that these two arcs have distinct osculating circles at $p_j$. We denote these (distinct) osculating circles by $\mathbf{C}_{\alpha_j}, \mathbf{C}_{\beta_j}$, $j\in\{1,2\}$.
We can send $p_j$ to $\infty$ by a M{\"o}bius map such that the images of $\mathbf{C}_{\alpha_j}, \mathbf{C}_{\beta_j}$ are the horizontal straight lines $y=0$ and $y=1$. Since the two non-singular arcs $\alpha_j,\beta_j$ have at least second order contact with their corresponding osculating circles, the images of $\alpha_j,\beta_j$ under the above M{\"o}bius map are curves of the form $v=\nu_+^j(u)=0+O(1/u)$ and $v=\nu_-^j(u)=1+O(1/u)$ for $u$ large enough. The image of $N_j^+$ under this M{\"o}bius map is a curvilinear strip $\mathfrak{S}_j$ bounded by the real-analytic curves $v=\nu_+^j(u)$ and $v=\nu_-^j(u)$, where $u\geq u_0$ for some large $u_0>0$. As in the proof of Lemma~\ref{asymp_lin_1_lem}, one sees that the maps $(\nu_\pm^j)'$ are real-analytic, the strips $\mathfrak{S}_j$ have boundary inclination $0$ at $\infty$, and the integral convergence condition of Theorem~\ref{war_thm} is satisfied by the strips, for $j\in\{1,2\}$. 

Let us now consider conformal maps $\mathfrak{b}_j:\mathfrak{S}_j\to \{x+iy\in\C:x>0, \vert y\vert<\pi/2\}$, such that $\displaystyle\lim_{u\to+\infty} \re{\mathfrak{b}_j(w)}=+\infty$.
, $j\in\{1,2\}$. Theorem~\ref{war_thm} now implies that as $u\to+\infty$, the conformal map $\mathfrak{b}_j$ has the asymptotic development $w=u+iv\mapsto \pi w+o(w)$, uniformly in $v$. Therefore, the conformal map $\mathfrak{b}_2^{-1}\circ \mathfrak{b}_1:\mathfrak{S}_1\to\mathfrak{S}_2$ is of the form $w+o(w)$ as $u\to+\infty$ (uniformly in $v$). Once again, due to uniformity, the same asymptotics hold up to the boundary. Finally, going back to $N_1^+, N_2^+$ by the M{\"o}bius maps used earlier, we conclude that $\phi:\overline{N^+_1}\to \overline{N^+_2}$ is asymptotically linear near $p_1$; i.e., $\phi(z)= p_2+k (z-p_1) +o(z-p_1)$ as $\overline{N_1^+}\ni z\to p_1$, where $k\in\C^\ast$. 


Moreover, the above proof shows that the linearity constants of the maps\\ $\phi:\overline{N^+_1}\to \overline{N^+_2}$ and $\phi:\overline{N_1^-}\to\overline{N_2^-}$ are the same.


2) Since analytic arcs are quasiconformally removable, it suffices to extend $\phi:N^\pm_1\to N^\pm_2$ to a homeomorphism from a neighborhood of $p_1$ to a neighborhood of $p_2$ that is quasiconformal outside $\alpha_1\cup\beta_1$. 

Observe that $\phi:\overline{N_1^\pm}\to\overline{N_2^\pm}$ admits a conformal extension (by the Schwarz reflection principle) in a neighborhood of $(\alpha_1\cup\beta_1)\setminus\{p_1\}$. Moreover, $\phi:\overline{N^+_1}\to \overline{N^+_2}$ and $\phi:\overline{N^-_1}\to \overline{N^-_2}$ are asymptotically linear at $p_1$ with the same linearity constant. These facts together imply that $\phi$ induces a pair of quasisymmetric maps $\phi:\alpha_1\to\alpha_2$ and $\phi:\beta_1\to\beta_2$. Since $\alpha_j,\beta_j$ are quasi-arcs, these quasisymmetric maps can be continuously extended to quasiconformal maps $\phi:W_1^+\to W_2^+$ and $\phi:W_1^-\to W_2^-$.
This provides a desired quasiconformal extension of $\phi:N^\pm_1\to N^\pm_2$.
\end{proof}

\bibliographystyle{alpha}
\bibliography{schwarz}

\end{document}